\documentclass[10pt,reqno,a4paper, english]{amsart}

\usepackage[a4paper]{geometry}

\usepackage[mathscr]{euscript}

\usepackage[utf8]{inputenc}
\usepackage[main=english, french, german, latin]{babel}

\usepackage[T1]{fontenc}

\usepackage{amsmath, mathtools, amsthm, amsfonts, amssymb,xcolor}
\usepackage{microtype}
\usepackage{hyperref}
\hypersetup{colorlinks={true},linkcolor={blue},citecolor=blue}
\usepackage{comment}
\usepackage{mathrsfs}
\usepackage{stmaryrd}
\usepackage{enumitem}

\numberwithin{equation}{section}

\theoremstyle{plain}
\newtheorem{theorem}{Theorem}[section]
\newtheorem{proposition}[theorem]{Proposition}
\newtheorem{corollary}[theorem]{Corollary}
\newtheorem{lemma}[theorem]{Lemma}

\newtheorem*{question*}{Question}
\newtheorem{remark}[theorem]{Remark}

\theoremstyle{definition}
\newtheorem{definition}[theorem]{Definition}
\newtheorem{notation}[theorem]{Notation}

\DeclareMathOperator{\Id}{Id}

\DeclareMathOperator{\supp}{supp}

\DeclareMathOperator{\e}{e}

\DeclareMathOperator{\hi}{hi}
\DeclareMathOperator{\lo}{lo}

\DeclareMathOperator{\vol}{Vol}
\DeclareMathOperator{\meas}{meas}

\newcommand{\nb}{{\bf n}}
\newcommand{\kb}{{\bf k}}
\newcommand{\lb}{{\bf l}}
\newcommand{\mb}{{\bf m}}

\newcommand{\Nb}{{\bf \mathfrak{N}}}

\newcommand{\N}{\mathbb{N}}
\newcommand{\Z}{\mathbb{Z}}

\newcommand{\R}{\mathbb{R}}
\newcommand{\C}{\mathbb{C}}
\newcommand{\T}{\mathbb{T}}

\newcommand{\im}{\mbox{Im}}
\newcommand{\re}{\mbox{Re}}

\newcommand{\1}{\mathbf{1}}

\newcommand{\Clo}{\mathscr{C}_{\lo}}
\newcommand{\Chigh}{\mathscr{C}_{\hi}}

\newcommand{\Ysup}[1]{Y_{#1}^{\mathrm{sup}}}
\newcommand{\Ylip}[1]{Y_{#1}^{\mathrm{lip}}}

\renewcommand{\epsilon}{\varepsilon}

\makeatletter
\@namedef{subjclassname@2020}{\textup{2020} Mathematics Subject Classification}
\makeatother

\begin{document}


\title[Stability for NLS on flat tori]{Long time stability for cubic nonlinear Schr{\"o}dinger equations on non-rectangular flat tori}

\author{Joackim Bernier}

\address{\small{Nantes Universit\'e, CNRS, Laboratoire de Math\'ematiques Jean Leray, LMJL,
F-44000 Nantes, France
}}

\email{joackim.bernier@univ-nantes.fr}

\author{Nicolas Camps}

\address{\small{Nantes Universit\'e, CNRS, Laboratoire de Math\'ematiques Jean Leray, LMJL,
F-44000 Nantes, France
}}

\email{nicolas.camps@univ-nantes.fr}




\begin{abstract}
We consider nonlinear Schr\"odinger equations on flat tori satisfying a simple and explicit Diophantine non-degeneracy condition. Provided that the nonlinearity contains a cubic term, we prove the almost global existence and stability of most of the small solutions in high regularity Sobolev spaces. 
To this end, we develop a normal form approach designed to
handle general resonant Hamiltonian partial differential equations for which it is possible to modulate the frequencies by using the initial data.

\end{abstract} 

\ \vskip -1cm  \hrule \vskip 1cm \vspace{-8pt}
 \maketitle 
{ \textwidth=4cm \hrule}

\maketitle
\setcounter{tocdepth}{1}
\tableofcontents

\section{Introduction}

\subsection{Context} We consider the nonlinear Schr\"odinger equation
\begin{equation}
\label{eq:nls}
\tag{NLS}
i\partial_{t}u + \mathrm{div} \, G \nabla u = f(|u|^{2})u\,,\quad (t,x)\in\R\times\mathbb{T}^{d}\,
\end{equation}
where $G$ is a real symmetric $d\times d$ positive matrix and $f:\R\to\R$ is a $C^\infty$ function satisfying $f'(0)\neq 0$ (to ensure the existence of a cubic nonlinear term in the equation). Without loss of generality, we assume that $f(0)=0$.

\medskip

This is a convenient way of rewriting nonlinear Schr\"odinger equations on rescaled tori. Indeed, $d$-dimensional flat tori writes $\mathbb{T}_{\mathscr{L}}^d=  \R^d/\mathscr{L}$ where $\mathscr{L}=: V \mathbb{Z}^d$ is a lattice of $\R^d$ ($V$ being a $d\times d$ real invertible matrix) and so by setting $\operatorname{G} = V^{-1} \, ^{t}(\operatorname{V}^{-1})$ and by applying a linear change of coordinate  $x\mapsto Vx$, \eqref{eq:nls} is equivalent to 
\begin{equation*}
i\partial_tu + \Delta u = f(|u|^2)u\,,\quad (t,x)\in \R\times\T_{\mathscr{L}}^d\,.
\end{equation*}

\medskip

By considering small solutions, \eqref{eq:nls} can be seen as a perturbation of the linear integrable system 
\begin{equation}
\label{eq:Schrolin}
i\partial_{t}u + \mathrm{div} \, G \nabla u = 0
\end{equation}
whose actions $I_n(u) = |u_n|^2$  are the square modulus of the Fourier coefficients
$$
 u_{n} = \frac{1}{(2\pi)^d}\int_{\T^{d}}\e^{-in\cdot x}u(x)\, \mathrm{d}x.
$$
Since the actions are constants of motion for the linear evolution \eqref{eq:Schrolin}, a natural question in this context of perturbation theory is then to understand how much the nonlinear flow of \eqref{eq:nls} preserves these actions. Or more formally, by setting
\[
\|u\|_{H^{s}}:=(\sum_{n\in\Z^{d}}\langle n\rangle^{2s}|u_{n}|^{2})^{\frac{1}{2}}\,,\quad \langle n\rangle:= (1+|n|^{2})^{\frac{1}{2}}\,,
\]
for initial data of size $\varepsilon$ in $H^{s}(\mathbb{T}^{d})$, on what time scales $T(\epsilon)$ are the integrable dynamics orbitally stable under the flow of \eqref{eq:nls}, in the sense that the actions of the solution are slow-variables:
\begin{equation}
\label{eq:stab11}
\sup_{|t|\leq T(\epsilon)}\sum_{n\in\Z^d} \langle n\rangle^{2s}
	\Big|
	|u_{n}(t)|^{2}
	-
	|u_{n}(0)|^{2}
	\Big| \ll \varepsilon^{2} ?
\end{equation}
A sub-question  is the one of the stability of the zero solution $u(t,x)=0$. Or more precisely, for initial data of size $\varepsilon$ in $H^{s}(\mathbb{T}^{d})$, on what time scales $T(\epsilon)$ do we have that
\begin{equation}
\label{eq:stab12}
\sup_{|t|\leq T(\epsilon)}\|u(t)\|_{H^{s}} \leq C_s\|u(0)\|_{H^{s}} ?
\end{equation}
(where $C_s>0$ is constant depending only on $s$). The statement \eqref{eq:stab11} is stronger than~\eqref{eq:stab12}. Note that the local Cauchy theory ensures that $T(\varepsilon)$ is larger than or equal to the linear time-scale $\sim\varepsilon^{-2}$, at least when $s>\frac{d}{2}$. 

The above question is a mathematical formulation of the (absence) of transfer of energy from large to small scales of oscillation. Extensive research has been conducted over the past decades to construct energy cascades out of the resonant interactions, and we refer to~\cite{CKSTT,GG22,GK15} for partial results about norm amplification for the cubic Schr{\"o}dinger equation on tori. In this setting, the existence of infinite energy cascades dynamics conjectured by Bourgain~\cite{Bourgain-conj-00} have only been established when the NLS is posed on the waveguide $\R\times\T^{d}$ manifold. The unbounded direction provides a stronger dispersion that reduces the dynamic to the effective one of the resonant system \cite{HPTV15}.  In contract, the second author and Staffilani~\cite{CS24,C25} proved that the Sobolev norms remain bounded in infinite time when NLS is posed on a Diophantine waveguide (namely $\R\times\T_{\mathscr{L}}^{d}$ for $\mathscr{L}$ having the same Diophantine property as in the present paper).

\medskip 

In recent decades, important developments have been made in Birkhoff normal form techniques to prove the stability of small solution over longer time scales, for many different Hamiltonian PDEs (see e.g. \cite{BG06,BDGS07,GIP,Del12,FGL13,BD17,BMM22}). All these results ensure the \emph{almost global existence and stability of the small solutions} in the sense that \eqref{eq:stab12} holds for $T(\varepsilon) \sim \varepsilon^{-r}$ with $r$ arbitrarily large, provided that $s\gg r$. The case of low regularity is still wide open (see e.g. \cite{BG22,BGR23} for results in this direction). However, all these results apply to PDEs associated with non-resonant Hamiltonians: the eigenvalues of the operator associated with the linearized equation (also called the frequencies) must be rationally independent. For \eqref{eq:nls}, the frequencies are
$$
\lambda_n^2:= g(n,n) \quad \mathrm{where} \quad g(a,b):= ^t\operatorname{a}G \operatorname{b}.
$$
They are clearly\footnote{they belong to the finite dimensional $\mathbb{Q}$ vector space generated by the coefficients of $G$.} not rationally independent, and therefore \eqref{eq:nls} is resonant. A common way to overcome this obstacle is to add a random convolution potential $V$ (with real Fourier coefficients) to the equation, and the equation \eqref{eq:nls} becomes
\begin{equation}
\label{eq:nlsVast}
\tag{NLS*}
i\partial_{t}u + \mathrm{div} \, G \nabla u = V\ast u+ f(|u|^{2})u
\end{equation}
See for example \cite{BG06,YZ14,FI19,BMP20,BG22,BFM22}. In this case, the frequencies are modulated and become $\omega_n = \lambda_n^2+ V_n$. They are almost surely rationally independent. For rational tori (i.e. $G\in \mathbb{Q}^{d\times d}$ up to multiplication by a scalar), the almost global existence and stability of the small solutions of \eqref{eq:nlsVast} was proved by \cite{BG06}. For irrational flat tori, small divisors are degenerate which may generate high frequencies instability. Nevertheless, in \cite{BFG20b}, Bernier, Faou and Gr\'ebert developed an approach to prove almost global existence of small solutions for semi-linear equations enjoying such small divisor estimates. More recently, Bambusi, Feola and Montalto \cite{BFM22} proved the almost global stability of small solutions of \eqref{eq:nlsVast}.

\medskip

In this paper, we aim to extend these results by removing the Fourier multiplier $V\ast u$. Since \eqref{eq:nls} is resonant, the previous results do not hold, i.e. the stability \eqref{eq:stab11} or \eqref{eq:stab12} with $T(\varepsilon) \sim \varepsilon^{-r}$ and $r\gg1$ arbitrarily large cannot be derived from the Birkhoff normal form theorem. The situation is even worse, and we expect \eqref{eq:stab11} and \eqref{eq:stab12} to be false after some polynomial time scales $T(\varepsilon) \sim \varepsilon^{-\alpha}$. For example, one can deduce from \cite{CKSTT} that \eqref{eq:stab12} is false on the square torus $\mathbb{T}^2$ (i.e. $G= I_2$) for some initial data after $T(\varepsilon)=\varepsilon^{-2} \log(\varepsilon^{-1})$.

\medskip

The only resonant equations for which stability results were proved have a special algebraic property: the four-waves interactions must be trivial or, in other words, the quartic terms of the Birkhoff normal form of the equation must be integrable. This property allows to modulate the frequencies using the initial data as parameters and leads to the almost global existence and stability of most of the small solutions: an exceptional set of initial data has to be excluded. Such results have been proved for \eqref{eq:nls} in dimension $d=1$ \cite{Bou00,KillBill}, for the generalised Korteweg--de Vries and a modified version of Benjamin--Ono equations \cite{KtA} and for the Kirchhoff equation \cite{BH22} but only for $T(\varepsilon)\sim \varepsilon^{-6}$. We also mention  \cite{LX23,zhiqiang}, who proved stability in Gevrey spaces over  exponentially long times.

\medskip

Kuksin and P\"oschel observed in \cite{Kuksin-Poschel} that four-waves interactions are trivial for \eqref{eq:nls} in dimension $d=1$. The point is that this property is also true in any dimension provided that
\begin{equation}
\label{eq:nondegbase}
g(a,b) \neq 0 , \quad \forall a,b \in \mathbb{Z}^d \setminus \{0\}.
\end{equation}
Indeed, as in dimension $1$, if $n_1-n_2+n_3-n_4 = 0$, we have
$$
 \Omega_{\vec{n}}:= \lambda_{n_1}^2 - \lambda_{n_2}^2 + \lambda_{n_3}^2 - \lambda_{n_4}^2 = 2 g(n_1-n_2,n_1-n_4)
$$
which ensure that if $ \Omega_{\vec{n}}=0$ then $\{n_1,n_3\} = \{n_2,n_4\}$.

\subsection{Main theorem and comments}

In order to establish our main result, we have to make the non-degeneracy condition  \eqref{eq:nondegbase} quantitative.
\begin{definition}[Admissible tori]\label{def:tori} The flat tori $\mathbb{T}^d_{\mathscr{L}}$ is admissible if there exist $C>0$ and $\tau^{\ast}>0$ such that for all $a,b \in \mathbb{Z}^{d} \setminus \{0\}$, 
\begin{equation}
\label{eq:adm}
|g(a,b)| \geq \frac{C}{\|a\|_{2}^{\tau_{\ast}}\|b\|_{2}^{\tau_{\ast}}}\,.
\end{equation}
In particular, $g(a,b)=0$ if and only if $a=0$ or $b=0$.
\end{definition}

\begin{remark}
If the vector $(G_{i,j})_{i\leq j}$ of the coefficients of $G$ is Diophantine then the torus is admissible. Note that the converse is not true (see \eqref{eq:monjoliexemple}). In any case, as a consequence, since this property is true for almost all matrix $G$, we can say that almost all flat torus are admissible (as soon as $\tau_\ast > \frac{d(d+1)}{2}$).
\end{remark}

To state our result, we have to introduce $\Pi_{M}$, $M>0$, the orthogonal projectors onto frequencies smaller than $M$:
\[
\Pi_{M}u = \sum_{|n|\leq M}u_{n}\e^{in\cdot x}\,.
\]
Moreover, by abuse of notation, we denote by $\meas$ the canonical Lebesgue measure on $\Pi_M L^2(\mathbb{T}^d; \mathbb{C})$, for $M>0$.

\begin{theorem}\label{thm:main} Let $\mathbb T_{\mathscr{L}}^{d}$ be an admissible torus in the sense of Definition \ref{def:tori} and $r\gtrsim1$. There exist $\mu_d>0$ and $\nu\lesssim_{r,\mathscr{L}} 1$ such that for all $s\gtrsim_{r}1$ and for all $\varepsilon \lesssim_{r,s} 1$, provided that $M>0$ satisfies
\begin{equation}
\label{eq:contrainte_pour_M}
\varepsilon^{- \mu_d \frac{r}s} \leq M \leq \varepsilon^{-\nu} ,
\end{equation}
there exists an open set $\Theta_{\varepsilon}\subset\Pi_{M}B_{s}(\varepsilon)$
for which the following holds: local solutions to~\eqref{eq:nls} in $H^{s}$ initiated from initial data
\begin{equation}
\label{eq:haut_mode_sympas}
u(0)\in B_{s}(2\varepsilon)\cap\Pi_{M}^{-1}\Theta_{\varepsilon}
\end{equation}
exist in $C([-T_\varepsilon,T_\varepsilon],H^{s})$ for $T_\varepsilon =\varepsilon^{-r}$, and satisfy for all $|t|\leq T_{\epsilon}$
\begin{equation}
\label{eq:thm-dyn}
\|u(t)\|_{H^{s}}\leq 2^{s+1} \| u(0) \|_{H^s}.
\end{equation}
Moreover, these initial data are typical, in the sense that
\begin{equation}
\label{eq:thm-meas}
 \meas(\Theta_{\varepsilon})\geq (1-\varepsilon^{\frac{1}{40}})\meas(\Pi_{M}B_{s}(\varepsilon))\,.
\end{equation}
\end{theorem}

Roughly speaking, for any $r\gtrsim1$, $s\gtrsim_{r}1$ and $\epsilon>0$ sufficiently small, we prove that, under a generic condition on finitely many low Fourier modes, initial data of size $\epsilon$ in $H^{s}(\T^{d})$ lead to stable solutions over time scales of order $\epsilon^{-r}$.

\subsubsection{Comments about the literature}  \begin{itemize} \item Theorem \ref{thm:main} is the first almost global existence result for a resonant PDE in dimension~$d\geq2$. 
\item It extends the results \cite{Bou00,KillBill} in the one-dimensional case which is much more favorable. In particular, the almost global preservation of all the actions (i.e. \eqref{eq:stab11}) holds when $d=1$. When $d\geq2$, however, we prove the stability of the low-frequency actions and of the high-frequency super-actions, as discussed subsection \ref{sec:techni}.
\item In a similar finite dimensional setting the geometric part of the Nekhoroshev theorem allows to prove the long time stability (in a sense weaker than  \eqref{eq:stab11} and \eqref{eq:stab12}) of any small solution, removing the restriction \eqref{eq:haut_mode_sympas}. In contrast, \cite{BouKal} strongly suggests that such a result should not extend to the infinite dimensional setting.  We nevertheless mention \cite{Bam99,BaG23} for results in this direction.

\end{itemize}

\subsubsection{Technical comments}\label{sec:techni}  \begin{itemize} \item We propose a new formulation \eqref{eq:thm-meas} to quantify the proportion of initial data leading to almost global and stable solutions. In this formalism, we only need to impose a condition on the low Fourier coefficients. Moreover, we do not need to assume any additional decay of the Fourier coefficients as it is usually done (e.g. in \cite{Bou00,KillBill,KtA,LX23}). We defer comments on this formulation to subsection \ref{sub:proba}.

\item The constraint $s\gtrsim_r 1$ could be refined in $s \gtrsim r^2$ (which is a classical constraint for this kind of problems, see e.g. \cite{KtA}). It mainly comes from the existence of a number $M$ satisfying \eqref{eq:contrainte_pour_M} and the fact that in the proof we impose $\nu \lesssim r^{-1}$, as specified in \eqref{eq:upsilon}.
\item In dimension $d\geq 2$, we are not able to control the variation of all the actions (it is an open problem even for \eqref{eq:nlsVast}). Actually (in \eqref{eq:a_citer_dans_lintro}), we only control the variation of the low actions 
$$
\sup_{|t|\leq \varepsilon^{-r}}\sum_{ n \in \Clo} \langle n \rangle^{2s} \big| |u_n(t)|^2 -|u_n(0)|^2  \big| \ll \varepsilon^2
$$
where $\Clo \subset \mathbb{Z}^d$ is a finite set, with $\{ n\in \mathbb{Z}^d \ | \ |n| \leq (17r)^{-1} M \} \subset \Clo$. For the high modes, however, as \cite{BFM22}, we only control the variations of the super-actions based on a the cluster decomposition of the frequency space due to \cite{Bou98,Berti-Maspero-19}. We refer to Proposition \ref{prop:ap-hi} for a precise estimate (the truncation parameter $M$ of this proposition in not the same as the one of Theorem \ref{thm:main}). The unusual exponent $2^{s+1}$ in \eqref{eq:thm-dyn} is a consequence of the dyadic structure of the cluster decomposition.  

\end{itemize} 

\subsubsection{Comments on the set of initial data} \label{sub:proba} \begin{itemize} \item To overcome the possible instabilities generated by the resonances of the equation, we
modulate the frequencies by using initial data. This leads (very technical) Diophantine conditions on the set of initial data. Therefore, in order to prove that the stability result is generic (and even not empty), we have to prove that this set of stable initial data is large in some sense. In the finite dimensional setting, we would prove that the set of admissible initial data is open and asymptotically of full Lebesgue measure. 

In the infinite dimensional setting, however, there is no Lebesgue measure and therefore no canonical choice to express such a result. Fortunately, the constraints we impose on the initial datum are quite weak, and it is possible to prove that this set is big for any measure that is not too much degenerated (see e.g. \cite{KillBill,KtA,LX23,zhiqiang} for different possible choices and \cite{KtA} for a discussion about the topology of this set). We do not know if there is an intrinsic way to express that this set is large.

\item In this paper, as in \cite{zhiqiang}, we stressed out the fact that it suffices  to impose a condition on the low Fourier modes of the initial (i.e. $u_n$ with $|n|\leq M$, with $M\sim \varepsilon^{- \mu_d \frac{r}s}$ with $s\gtrsim r$). This is essentially the meaning of  \eqref{eq:haut_mode_sympas}.

\item The dimension of the finite-dimensional space $\Pi_{M}H^{s}(\T^{d}, \C)$, which is of order $M^d \geq \varepsilon^{- d\mu_d \frac{r}s}$, is large with respect to the relative measure of the set of the initial data we exclude ($\varepsilon^{\frac{1}{40}}$). Hence, our probability framework is closer to an infinite dimensional one than a finite dimensional one, even if our statement involves the (finite-dimensional) Lebesgue measure. 
 Indeed, in the limit $\epsilon\to0$ the dimension of $\Pi_{M}H^{s}(\T^{d},\C)$ goes to infinity, and we can show that the Lebesgue measure on $\Pi_{M}B_{s}(\epsilon)$ concentrates its mass on the sphere of radius $\epsilon$:
\begin{equation*}
\begin{split}
\frac{1}{\meas(\Pi_{M}B_{s}(\epsilon))}
\meas
	\Big\{
	u\in\Pi_{M}B_{s}(\epsilon)\ |\ &\|u\|_{H^{s}}\leq \epsilon(1-M^{-\frac{1}{2}})
	\Big\}
	\\
	&=
	 (1- M^{-\frac{1}{2}})^{\mathrm{dim} \Pi_{M}L^2} 
	\ll 
	\epsilon^{\frac{1}{40}}
	\,.
	\end{split}
\end{equation*}
Therefore, to prove Theorem \ref{thm:main} we could assume for free that $\| \Pi_M u(0) \|_{H^s} \geq  \varepsilon (1- M^{-\frac{1}{2}})$. Nevertheless, we do not use such a property in the proof and, to avoid any hidden smallness assumption of this kind on $\| u(0) - \Pi_M  u(0)\|_{H^s}$,  we allow $\|u(0)\|_{H^s}$ to be as large as $2\varepsilon$ in~\eqref{eq:haut_mode_sympas}. 

\item A popular way to draw initial data is to draw the Fourier coefficients independently, see e.g. \cite{Bou00,KillBill,KtA,LX23}. For example, here it could consist in considering random initial data of the form\footnote{Note that $V^{(\varepsilon)} $ also lives on a sphere in the sense that $\|V^{(\varepsilon)}  \|_{\ell^\infty_{s+\alpha}} = c\varepsilon$ almost surely.}
$$
V^{(\varepsilon)} = c \, \varepsilon \sum_{k\in \mathbb{Z}^d} V_k \langle k \rangle^{-s-\alpha}
$$ 
where the random variables $V_k$ are independent and uniformly distributed in the complex unit disk $\mathbb{D}(0,1)$,  $\alpha>\frac{d}{2}$ and $c(\alpha)>0$ is a normalizing constant to ensure that $\| V^{(\varepsilon)} \|_{H^s}< \varepsilon$ almost surely. With minor changes in the proof\footnote{in measure estimates, it suffices to replace round balls $\Pi_M B_s(0,\varepsilon)$ by rectangular ones  $\Pi_M B_{\ell^\infty_{s+\alpha}}(0,c\varepsilon)$.}, we could easily replace the measure estimate \eqref{eq:thm-meas} of Theorem \ref{thm:main} by
 \begin{equation}
 \label{eq:loi_unif}
\mathbb{P}( \Pi_M V^{(\varepsilon)} \in \Theta_{\varepsilon}) \geq 1 -\varepsilon^{\frac{1}{40}}\,.
 \end{equation}
We would deduce from this that the set of the stable initial data $\bigcup_{\varepsilon} B_{s}(2\varepsilon)\cap\Pi_{M}^{-1}\Theta_{\varepsilon}$ is asymptotically of full measure:
$$
\mathbb{P}(   \varepsilon V^{(1)} \in \Pi_{M}^{-1}\Theta_{\varepsilon} ) \geq 1 -\varepsilon^{\frac{1}{40}}.
$$
Approaches of this type have the disadvantage of requiring initial data smoother than necessary. We believe that one of the strengths of our probabilistic formulation in Theorem \ref{thm:main} is to describe dynamics in $H^{s}(\T^{d})$ for solutions no more regular than $H^s(\T^{d})$.

\item In contrast with papers using rational normal forms,  the set $\Theta_{\varepsilon}$ a priori also depends on the angles of the initial data: we do not prove that it is invariant by rotation of the angles, in the sense that
\begin{equation}
\label{eq:inv_angle}
u \in \Theta_\varepsilon \iff \sum_{k\in \mathbb{Z}^d} |u_k| e^{ik\cdot x}.
\end{equation}
That is why contrary to \cite{KillBill,KtA,zhiqiang}, we do not only draw the actions of the initial data randomly but also their angles. We point out that nevertheless, here, \eqref{eq:inv_angle} is true up to conjugation by a diffeomorphism ($\Psi$, defined in \eqref{eq:psi}).

\end{itemize}

\subsubsection{Comments about flat tori} \begin{itemize} \item The strength of our method with respect to rational normal form methods is to only require a very weak and explicit condition on the external parameters $G$ associated to the tori $\mathbb{T}^d_{\mathscr{L}}$. Explicit examples with Diophantine numbers can be easily produced, e.g. 
\begin{equation}
\label{eq:monjoliexemple}
G = \begin{pmatrix} 1 & \sqrt{2} \\
\sqrt{2} & 3
\end{pmatrix} \ \mathrm{is} \ \mathrm{admissible}.
\end{equation}
\item Note that admissible tori are not rectangular but can be chosen arbitrarily close to rectangular tori. For rectangular tori, however, the four-wave interaction set is not trivial. It can be decoupled into two one-dimensional systems (see \cite{Staffilani-Wilson-20}). We believe that it is an interesting problem to understand the situation on a Diophantine rectangular tori. 
\item  Most of our analysis can be easily transferred to the case of a hyperbolic Laplacian (when the symmetric matrix $G$ is indefinite). However, the finite-dimensional reduction, which is based on the separation property of the frequencies (see Lemma \ref{lem:cluster}), may no longer be possible.
\item Theorem \ref{thm:main} is in line with the study of nonlinear Schrödinger equations on irrational tori that has been developed in recent years. The general idea is to mimic the dispersion in compact settings. Under some Diophantine conditions on the torus, the refocusing time for the waves is actually longer \cite{dgg17,dggmr-22} than it is on the square torus, and the four-wave resonant system is smaller \cite{Staffilani-Wilson-20}. On Diophantine rectangular tori, Deng--Germain \cite{dg19} improved polynomial upper bounds for the growth of Sobolev norms, while Deng \cite{Deng-19} achieved polynomial growth in the energy-critical case for small energy solutions (a challenging open problem on the square torus). In a similar spirit, Collot--Germain \cite{CG20} derived the kinetic wave equation for larger set of scalings by considering dispersion relations associated with non-rectangular Diophantine tori. 
\item Our result goes in the same direction, showing stronger stability properties on non-rectangular Diophantine tori, in contrast to the norm amplification observed in \cite{CKSTT}. We stress out that Guilani--Guardia \cite{GG22} proved that the same norm amplification mechanism occurs on rectangular Diophantine tori as well through quasi-resonant quartets, but after exponential time-scales. This is not contradictory to our Theorem \ref{thm:main}, which achieves polynomial time-scales. 
\end{itemize}

\subsection{Discussions on the proof} In this paper we develop a normal form approach for proving stability of solutions to resonant Hamiltonian systems whose modulated frequencies are highly degenerate (more than what the rational normal form approach requires). We believe that beyond the almost global existence result stated in Theorem \ref{thm:main}, this method is one of the main interests of this paper. In the case of \eqref{eq:nls} on flat tori, one advantage of this method is that it results in the simple (and quite minimal) non-degeneracy assumption \eqref{eq:adm} for admissible tori. 

\medskip

Our normal form is inspired by an approach developed by Bourgain in the paper \cite{Bou00} for \eqref{eq:nls} in dimension $d=1$. Apart from the fact that we overcome the degeneracy of linear frequencies which is specific to the dimension $d\geq 2$, the main difference is that we also provide a formalism that allows us to prove and quantify the sense in which most small solutions are stable over very long times. 

\medskip

In his proof, Bourgain describes a generic step of his normal form procedure: after a preliminary change of variable (let us say $\tau^0$) to put the Hamiltonian under resonant normal form, he explains how to construct a new change of variable to get better properties. The point is that (due to the modulated small divisors) the transformation $\tau^0$ itself depends strongly on $u(0)$ (let us denote it by $\tau^0 =: \tau^0_{u(0)}$). Therefore, in order to construct the new transformation, he has to assume that the new initial data $v(0):= \tau^0_{u(0)}(u(0))$ belong to a set $\Xi_{u(0)}$ encoding some Diophantine conditions and depending on $u(0)$. He proves that $\Xi_{u(0)}$ is asymptotically of full measure, but he does not explain why this implies that $v(0)=\tau^0_{u(0)}(u(0))\in \Xi_{u(0)}$ for most real initial data $u(0)$. 

\medskip

This type of problem is quite classical in standard KAM theory and is usually solved by introducing Lipschitz norms to track the dependencies with respect to the internal parameters. This is the strategy we implement in this paper. However, due to the small divisor degeneracy  this is far from obvious and has required the introduction of a whole technical framework (and associated technical estimates), which largely explains the length of this paper. It has also required us to refine some of the fundamental estimates of \cite{Bou00} (see e.g. Remark \ref{rq:comp_bou} for a more detailed discussion on this point). This difficulty is inherent in modulating the frequencies with internal parameters. It is a well known obstacle in the open problem of constructing infinite dimensional invariant tori for \eqref{eq:nls}\footnote{Note that since in dimension $d=1$ we control the variation of all the actions, what we prove in this paper, as well as what is proved in \cite{Bou00,KillBill}, is somehow not so far from such a result: we prove that most of the small solution to \eqref{eq:nls} stay very close to some infinite dimensional tori for very long times.} in dimension $d=1$ (see e.g. \cite{Pos}).

\medskip

 To continue the discussion of the proof, we need to introduce the Hamiltonian structure of \eqref{eq:nls}.  More precisely, \eqref{eq:nls} rewrites  
 \begin{equation}
\label{eq:hamNLS}
i\partial_t u = \nabla H(u) \quad \mathrm{with} \quad H(u) = Z_2(u)+ \frac12 \int_{\mathbb{T}^d} F(|u(x)|^2) \, \mathrm{d}x. 
\end{equation}
 where $F$ is the primitive of $f$ vanishing at the origin and 
$$
Z_2(u) = \frac12 \sum_{n\in \mathbb{Z}^d} \lambda_n^2 |u_n|^2 .
$$
The classical notations (like $\nabla$) are defined in subsection \ref{sec:notations} below.

\subsubsection{Modulated frequencies, small divisors and comparison with rational normal forms} First, let us assume that\footnote{Due to the degeneracy of the small divisor associated with the linear frequencies, this is an open problem in dimension $d\geq 2$ discussed in subsection \ref{sub:resonant}.} \eqref{eq:nls} can be put in Birkhoff normal form (at least up to some high order $2r \gg 1$). This means that there would be a canonical change of variable $\tau$ close to the identity in $H^s$ such that 
\begin{equation}
\label{eq:Birk_reve}
H\circ \tau^{-1}(u) = Z_2(u) + Q(u) + \mathcal{O}(u^{2r}) \quad \mathrm{and} \quad \{ Z_2,Q \} =0.
\end{equation}
Then, using the admissibility condition on $\mathbb{T}^d_{\mathscr{L}}$, the four-waves interactions are integrable and we deduce that (up to a gauge transform), setting $v=e^{it \, \mathrm{div} \, G \nabla }\tau(u)$, \eqref{eq:nls} rewrites
 $$
 i\partial_t v_n = \omega_n(v) v_n + \mathcal{O}(v^5) \quad \mathrm{where} \quad  \omega_n(v) =  - f'(0) |v_n|^2.
 $$ 
Since $f'(0)\neq 0$, this allows to modulate the frequencies using the initial data as parameters. Note, however, that since $v\in H^s$ this modulation is highly  degenerate in the high frequency regime since  $\omega_n$ decays at least like $\langle n\rangle^{-2s}$ with $s\gg 1$. Under some generic conditions on $v$, this provides small divisors estimates of type
 \begin{equation}
 \label{eq:discussion_sd}
\big| \sum_{n\in \mathbb{Z}^d} h_n \omega_{n}(v) \big| \gtrsim \| v\|_{H^s}^2  (\mu_{\max}(h))^{-2|h|_{\ell^1}} (\mu_{\min}(h))^{-2s} ,
 \end{equation}
where $\mu_{\min}(h) = \min \{ \langle n \rangle \ | \ h_n \neq 0 \}$, $\mu_{\max}$ is defined similarly and $h\in \mathbb{Z}^{(\mathbb{Z}^d)}$ is a family of integers with finite support. We prove such an estimate in Section \ref{sec:smd}. This type of small divisors seems to be ubiquitous for resonant Hamiltonian PDEs (e.g. it is also true for KdV, Benjamin--Ono and their generalizations \cite{KtA}, Kirchhoff \cite{BH22}, the pure-gravity water waves \cite{BFP23}) to the notable exception of the Schr\"odinger--Poisson equation (see \cite{KillBill,zhiqiang}).

\medskip

There are three different contributions in the right-hand-side of the small divisor estimate \eqref{eq:discussion_sd}. 
\begin{itemize}
\item The first one, $\| v\|_{H^s}^2$, comes from homogeneity arguments and is harmless. 
\item The second one, $(\mu_{\max}(h))^{-2|h|_{\ell^1}}$, comes from counting estimates and could be very critical but truncation arguments (presented in subsection \ref{sub:resonant}) allow to handle these resonances. 
\item The last contribution $(\mu_{\min}(h))^{-2s}$ is the most critical. In this sense, we say that the small divisor estimate~\eqref{eq:discussion_sd} is degenerate. To explain why it is so critical, let us focus on just one aspect that we think is the most important. 
\end{itemize}

To describe the dynamics over longer timescales one has to conjugate (at least locally) \eqref{eq:nls} to an integrable system up to some higher order terms. In other words, up to a new canonical change of variables $z = \tau^\sharp(v)$,  \eqref{eq:nls} rewrites
 $$
 i\partial_t z_n = \omega_n^\sharp(z) z_n + \mathcal{O}(z^{2q+1}) \quad \mathrm{where} \quad  \omega_n^\sharp(z) = \omega_n(z) + \mathcal{O}(z^4).
 $$ 
and $3\leq q \leq r$. To transform the remainder terms of order $\mathcal{O}(z^{2q+1})$ into terms of order $\mathcal{O}(z^{2q+3})$\footnote{of course, in practice, as we will see in the proof, the gain could be much weaker.}, a naive approach (in the spirit of Birkhoff normal forms) would be to average the terms of order $\mathcal{O}(z^{2q+1})$ by the flow generated by the frequencies $\omega_n(z)$\footnote{i.e. those of $i\partial_t z_n = \omega_n(z) z_n$.} to make them integrable. The point, identified and discussed in \cite{KillBill}, is that usually, when $n$ is large enough,  $\omega_n$ is negligible with respect to $\omega_n^\sharp$, which makes this naive approach fail. More precisely, due to higher order correction terms $\omega_n^\sharp$ has no reason to decay like $\langle n \rangle^{-2s}$ (i.e. like $\omega_n$\footnote{actually for some equations like Beanjamin--Ono \cite{KtA} or the pure gravity water waves \cite{BFP23}, $\omega_n$ contains some terms that do decay like $\langle n \rangle^{-2s}$ but they do not contribute to some of the small divisors.}; for example, for \eqref{eq:nls} in dimension $d=1$, explicit computation of the corrections of order 6 in \cite{KillBill} proves that it actually decays at most like $\langle n \rangle^{-2}$). In any case, more terms of  $\omega_n^\sharp$ have to be considered to average the terms of order $\mathcal{O}(z^{2q+1})$.

\medskip

We discuss two strategies to overcome this problem: either we identify a higher order term correction term allowing to prove that the small divisors are stable by perturbation, or we use that the derivative of the small divisors is stable by perturbation. 
\begin{itemize}
\item The first strategy, used in the papers following the rational normal form approach \cite{KillBill,KtA,LX23}, relies on an explicit computation of the correction terms of order $4$ of $\omega_n^\sharp$ (coming from the terms of order 6 in the Hamiltonian) and on the hope that they are much less degenerate than $\omega_n$. For example, here, we could prove that they are of the form 
$$
\widetilde{\omega}_n(z) = (f'(0))^2 \sum_{m \neq n} \frac{|z_m|^4}{\lambda_{n-m}^2} + \mathrm{other \ terms}.
$$
In the favorable cases it leads to small divisors estimates of the form\footnote{where $C_{|h|_{\ell^1}}>0$ is an explicit function of $|h|_{\ell^1}$ which plays no role.}
\begin{equation}
\label{eq:sd2}
\big| \sum_{n\in \mathbb{Z}^d} h_n (\omega_{n} + \widetilde{\omega}_n)(z) \big| \gtrsim  \| z\|_{H^s}^2 (\mu_{\max}(h))^{-C_{|h|_{\ell^1}}} \max\big( \| z\|_{H^s}^2,  (\mu_{\min}(h))^{-2s} \big) ,
\end{equation}
which are much less degenerate than the one associated with $\omega_n$ (see \eqref{eq:discussion_sd}). Since, as before, the factor $ (\mu_{\max}(h))^{-C_{|h|_{\ell^1}}}$ can be overcome by assuming high-regularity ($s$ large enough), such estimates are stable and imply similar small divisor estimates for the full modulated frequencies $\omega_n^\sharp$. It also implies that it suffices to average the terms of order $\mathcal{O}(z^{2q+1})$ using the flow generated by the modulated frequencies $\omega_n+\widetilde{\omega}_n$. Moreover, it also implies that it is not necessary to impose new Diophantine conditions on the initial data at each step of the normal form procedure, it suffices to ensure that \eqref{eq:sd2} holds for $z=u(0)$.
\item The second strategy, used in \cite{Bou00} and which we implement in this paper, is to average the terms of order $\mathcal{O}(z^{2q+1})$ using the flow generated by all the modulated frequencies $\omega_n^\sharp$. To prove that they satisfy acceptable small divisor estimates, it is important to note that even if $\omega_n$ is not the leading part of $\omega_n^\sharp$, $\partial_{|z_n|^2} \omega_n  = - f'(0)$ is necessarily the leading part of $\partial_{|z_n|^2}  \omega_n^\sharp$. This allows to prove that if $z(0)$ lives in a set of asymptotically full measure, the modulated frequencies $\omega_n^\sharp$ enjoy the same small divisor estimate \eqref{eq:discussion_sd} as $\omega_n$. As discussed above, a significant part of the proof (not considered in \cite{Bou00}) then consists in proving that this corresponds to a set of initial data $u(0)$ of asymptotically full measure.
\end{itemize}
The first strategy suffers from several limitations compared to the second one, which is more flexible and which we follow in this paper. 
\begin{itemize}
\item The higher order correction terms (such as $\widetilde{\omega}_n$) depend on the linear frequencies (here $\lambda_n^2$). To get non-degenerate small divisor estimates of the type \eqref{eq:sd2}, it is then necessary to impose additional assumptions on the linear frequencies (e.g. here we would have to impose additional Diophantine conditions on the numbers $\lambda_n^{-2}$, which would be much less explicit than the admissibility condition \eqref{eq:adm}).
\item For some equations these higher order correction terms may also be degenerate. For the Benjamin--Ono equation they vanish exactly. In \cite{KtA} is was then necessary to assume the existence a term of the form $\partial_x u^3$ in the perturbed equation in order to follow the rational normal form approach.
\item Depending on the equation at hand, and possibly also on preliminary Birkhoff normal form reductions (to remove the non-resonant terms from the Hamiltonian), obtaining explicit expressions for these higher order correction terms requires heavy computation. In \cite{KtA} it was necessary to use of a formal computation software. 
\end{itemize}

\subsubsection{Algebraic framework: rational fractions versus re-centered polynomials} The algebraic framework of the papers based on rational normal forms is quite different from that of this paper and \cite{Bou00}. Rational normal forms are quite close to standard Birkhoff normal form since it consists in replacing polynomial expansions by rational fraction expansions (which explains their name). Conversely, here, as in \cite{Bou00}, we consider expansions in re-centered\footnote{around the actions in the final variables, which is quite implicit.} polynomial expansions which hold much more locally (more in a KAM spirit: see \cite{Bou05,BMP21} where a similar formalism is used in this context). The second approach offers more flexibility in the construction whereas the first one provides a normal form which is more global. 
This choice of algebraic framework does not seem to be related to small divisor considerations: it is likely that we could replace rational fractions by re-centered polynomials in \cite{KillBill,KtA,LX23,zhiqiang}. Nevertheless, here and in \cite{Bou00}, to overcome the small divisor degeneracy (i.e. that we do not have \eqref{eq:sd2}), the normal form is constructed more locally (around some high-dimensional tori in $H^s$) and it is unclear how a formalism based on rational fractions would allow to recover the smallness estimates provided by the fact that the solution remains very close to high-dimensional tori. We refer to the end of Subsection \ref{sub:fin_de_lintro} for further discussions about the way we exploit this extra smallness.

\subsubsection{Reduction to a finite dimensional system}
\label{sub:resonant}

In order to put \eqref{eq:nls} in Birkhoff normal form, the frequencies of the system must satisfy for all $q\geq2$ a Diophantine condition of type
$$
|\lambda_{n_1}^2 -\lambda_{n_2}^2 + \cdots - \lambda_{n_{2q}}^2   | \neq 0 \quad \implies \quad |\lambda_{n_1}^2 -\lambda_{n_2}^2 + \cdots - \lambda_{n_{2q}}^2   | \gtrsim_{q} \mathrm{max}_{3} \langle n \rangle^{-\beta_q}\,,
$$
where $\beta_q>0$, $n_1,\cdots,n_{2q}\in \mathbb{Z}^d$ satisfy the zero momentum condition $n_1-n_2 + \cdots -n_{2q}=0$ and $\mathrm{max}_{3} \langle n \rangle$ denotes the third largest number among $\langle n_1 \rangle, \cdots, \langle n_{2q} \rangle$. Such small divisor estimates are obvious for integer frequencies, and also holds for generic choices of convolution potential in \eqref{eq:nlsVast} on rational tori. In contrast, for Diophantine flat tori this estimate does not hold in general with $\max_{3}\langle n\rangle$ but with $\max_{1}\langle n\rangle$. For this reason the almost global stability of small solution of \eqref{eq:nlsVast} for irrational tori remained an open problem until the recent work of Bambusi--Feola--Montalto \cite{BFM22}.

\medskip

To address this, \cite{BFM22} used cluster decomposition and frequency separation properties between different clusters, initially established by Bourgain~\cite{Bou98} for the square torus $\T^d$ and later generalized by Berti--Maspero~\cite{Berti-Maspero-19} for any flat tori (see Lemma \ref{lem:cluster}). This allows them, while leaving certain quasi-resonant terms in their normal form, to prove the quasi-preservation of the super actions associated with this decomposition and therefore to control the $H^s$ norm of the solution for very long times. 

\medskip

In Section \ref{sec:low-freq} we detail the finite-dimensional approximation. We use the cluster decomposition to prove that, up to a first change of variable $\Phi_\chi^1$ and a bootstrap assumption, the super actions are almost preserved and that the low-frequency part of the solution $\Pi_{\Clo} u(t) $ solves an evolution equation associated with a truncated resonant Hamiltonian $H \circ \Phi_\chi^1 \circ \Pi_{\Clo}$, up to a very small remainder term. However, we have added a number of technical refinements to ensure that the change of variable $\Phi_\chi^1$ does not destroy our measure estimates (of the type \eqref{eq:thm-meas}). The results of Section \ref{sec:low-freq} reduce the proof of Theorem \ref{thm:main} to the almost global stability of a truncated resonant Hamiltonian system, as stated in Theorem \ref{thm:main_low}. Since the system is truncated, it can be put on Birkhoff normal form
(almost) for free and we can fully exploit the fact that 
quartic resonant terms of the Hamiltonian are integrable.

\subsubsection{A normal form with a two parameters scale} \label{sub:fin_de_lintro}
Starting from section \ref{sec:small-div}, as in \cite{Bou00}, we only use the first correction term to modulate the frequencies. As a consequence, the small divisors estimates are of the form \eqref{eq:discussion_sd}. Due to the degeneracy of these small divisors, we believe that an approach in the spirit of rational normal forms, in which we would remove the non-integrable terms degree by degree (like in classical Birkhoff normal forms) would fail. Indeed, with such an approach, there does not seem to exists any reasonable way to absorb the losses due to the factor $(\mu_{\min}(h))^{-2s} $.
Instead, as in \cite{Bou00}, we introduce a two parameters scale to decompose and classify the polynomials appearing in our Taylor expansions. In addition to its presentation given below, we also refer to the Section 3 of the proceeding~\cite{C25} for further explanations.

\medskip

More precisely, following \cite{Bou00} we decompose a polynomial as a sum of monomials of the form
\begin{equation}
\label{eq:recentered_polynomials}
\prod_{n\in \mathbb{Z}^d} u_{n}^{k_{n}}\overline{u_{n}}^{\ell_{n}}(|u_{n}|^{2} - \xi_n )^{m_n}
\end{equation}
where $k,\ell,m\in \mathbb{N}^{\mathbb{Z}^d}$ are some finitely supported families of indices counting the multiplicities satisfying the condition $\ell_n k_n = 0$ for all $n\in \mathbb{Z}^d$ (to ensure the uniqueness of the decomposition). The parameters $\xi_n\in\R$ modulate the frequencies. They are implicitly the actions of the initial data in the final variables. To each non-integrable monomial of the form \eqref{eq:recentered_polynomials} we associate a frequency scale $N_{\alpha}$ such that
\begin{equation}
\label{eq:Nalpha_intro}
 N_{\alpha} \leq \nb_{-} < N_{\alpha+1} \quad \mathrm{where} \quad N_\alpha^{s} = \varepsilon^{-\frac{\alpha}{200}},
\end{equation}
 $\varepsilon$ is the size of the initial data in $H^s$ and $\nb_{-}$ is the size of the smallest index appearing in the associate small divisor, i.e.
$$
\nb_{-}=\min\{\ |n|\quad|\quad n\in\Z^{d}\,,\ k_{n}+\ell_{n}\geq1\}.
$$
The  parameter $\alpha$ classifies polynomials according to the cost of the associated small divisors. Monomials at scale $\alpha$ as in \eqref{eq:Nalpha_intro} satisfy two properties:
\begin{itemize}
\item They preserves the actions $|u_n|^2$ with $|n|<N_{\alpha}$. This type of property is also central in papers in low regularity like \cite{BGR23}.
\item Assuming moreover that they are resonant (or quasi-resonant in some sense to be specified), since they are not integrable and quartic resonant terms are integrable, we have $\sum_n k_n +\ell_n \geq 6$. Consequently, they generate vector fields from $H^s$ to $H^s$ of order at most $\varepsilon^{5}N_\alpha^{-4s}$ (see Proposition \ref{prop:vec-alpha}).
\end{itemize}
The cost of the small divisor is at most of order $\varepsilon^{-2}N_{\alpha+1}^{2s}$ (see \eqref{eq:discussion_sd}). Since, by construction, the ratio between two consecutive scales $(\frac{N_{\alpha+1}}{N_{\alpha}})^{2s}\lesssim \epsilon^{-\frac{1}{100}}$ is not too large, we see that there is room to remove these monomials (they could be associated to quite large coefficients).

\medskip 

The heart of the proof consists in removing, by induction on the frequency scales $N_{\alpha}$, the monomials satisfying $N_{\alpha} \leq \nb_{-} < N_{\alpha+1}$ (see Theorem \ref{thm:nf-alpha}). For this purpose,  we consider for each $\alpha$ a norm $\Ysup{\alpha}$ well suited to polynomials composed of monomials satisfying $N_{\alpha}\leq \nb_{-}$.
 Then imposing new Diophantine conditions on $\xi$, we construct by induction (in Proposition \ref{prop:it-j}) canonical changes of variables to make these norms smaller and smaller in terms of powers of $\varepsilon$ (this is the second scale of parameters) for polynomials composed of monomials satisfying $N_{\alpha} \leq \nb_{-} < N_{\alpha+1}$ . At the end of this second induction, the monomials satisfying $N_{\alpha} \leq \nb_{-} < N_{\alpha+1}$ are associated with very small coefficients, and do not contribute to the dynamics over time scales of order $\epsilon^{-r}$.

\medskip

However, in the second induction, the terms generated by the Poisson bracket with an \emph{integrable monomials with low-indices $|n|\ll N_{\alpha}$}, and only one index at frequency scale $\sim N_{\alpha}$, do not seem to be sufficiently small to compensate for the small divisor loss $\sim N_{\alpha}^{-2s}$.  To address these terms, Bourgain made the following observation: recentered actions are much smaller than the actions themselves (which we prove are expected to be slow variables). This observation provides an extra smallness factor for recentred actions $||u_{n}|^{2}-\xi_n|$ with a gain $N_{\alpha}^{-2s}$, even for low indices $|n|\ll N_{\alpha}$. The definition of the non-resonant neighborhood~\eqref{eq:cal-V} encodes this gain, which is a crucial ingredient to address the degenerate small-divisor losses. At this point, we also see that it is natural to introduce the frequency scale to compensate the small divisor losses thanks to this extra smallness.

\medskip

 When we reach a certain frequency scale $N_{\alpha}$ large enough (with the notations of the paper, when $\alpha=\beta$), we conclude that the actions are almost preserved over very long times using that the only remaining monomials are either integrable or associated with negligible coefficients: they generate vector fields of order at most $N_\beta^{-4s} \ll \varepsilon^{-2r}$ (see estimate \eqref{eq:jeveuxteciterdanslintro}).

\subsection{Notations and functional setting}
\label{sec:notations}
\subsubsection{Functional setting}
\label{sec:fun}
We equip the torus $\mathbb{T}^d = \mathbb{R}^d / 2\pi \mathbb{Z}^d$ with the normalized Lebesgue measure $(2\pi)^{-d}\mathrm{d}x$. Therefore,  the Lebesgue norms $\|\cdot\|_{L^p}$, $1\leq p<\infty$, are defined by density through the formula
$$
\forall u \in C^0(\mathbb{T}^d), \quad \| u\|_{L^p}^p:= (2\pi)^{-d}\int_{\mathbb{T}^d} |u(x)|^p \mathrm{dx}.
$$
We identity each distribution $u$ on $\mathbb{T}^d$ with the sequence of its Fourier coefficients defined by density by
$$
u_k = (2\pi)^{-d} \int_{\mathbb{T}^d} u(x) e^{-ikx} \mathrm{d}x.
$$
Given $1\leq p < \infty$ and $s\in \mathbb{R}$ we set
 $$
 \ell^{p}_s(\mathbb{Z}^d):= \{ u\in \mathbb{C}^{\mathbb{Z}^d} \ | \ \|u\|_{\ell^p_s}^p := \sum_{k\in \mathbb{Z}^d} \langle k \rangle^{ps} |u_k|^p <\infty \}.
 $$
 As usual,  we set
 $$
 H^s(\mathbb{T}^d) = h^s(\mathbb{Z}^d) = \ell^2_s(\mathbb{Z}^d) \quad \mathrm{and} \quad \ell^p(\mathbb{Z}^d) := \ell^p_0(\mathbb{Z}^d).
 $$
Note that with these conventions, the Fourier--Plancherel isometry writes
$$
\forall u \in L^2(\mathbb{T}^d), \quad \| u\|_{L^2}^2 = \|u\|_{\ell^2}^2
$$ 
In all this paper, we consider $L^2$ and $\ell^2$ as real Hilbert spaces equipped with the scalar products
$$
\forall u,v\in L^2(\mathbb{T}^d), \quad (u,v)_{L^2} :=(2\pi)^{-d} \re \int_{\mathbb{T}^d} u(x) \overline{v(x)} \mathrm{d}x = \re \sum_{k\in \mathbb{Z}^d} u_k \overline{v_k} =: (u,v)_{\ell^2}.
$$
Note that if $s \in \mathbb{N}$ is a nonnegative integer, we have
$$
\forall u \in H^s(\mathbb{T}), \quad \|u\|_{H^s}^2 = (2\pi)^{-d} \int_{\mathbb{T}^d} |\partial_x^s u(x)|^2 \mathrm{d}x. 
$$

\subsubsection{Differential calculus and Poisson brackets}

Given $1\leq p<\infty$, $s\in \mathbb{R}$, $\mathcal{U}$ an open subset of $\ell^p_s(\mathbb{Z}^d) $, a smooth function $P :\mathcal{U} \to \mathbb{R}$ and $u\in \ell^p_s(\mathbb{Z}^d)$, its gradient $\nabla P(u)$ is the unique element of $\ell^{p'}_{-s}(\mathbb{Z}^d)$ satisfying
$$
\forall v\in \ell^p_s(\mathbb{Z}^d), \ (\nabla P(u),v)_{L^2} = \mathrm{d}P(u)(v).
$$
It can be checked that
\begin{equation}
\label{eq:formula_grad}
\forall k \in \mathbb{Z}^d, \quad (\nabla P(u))_k = 2\partial_{\overline{u_k}} P(u).
\end{equation}
We equip $L^2(\mathbb{T}^d)$ with the usual symplectic form $(i\cdot,\cdot)_{L^2}$.  Therefore, provided that $\ell^p_s(\mathbb{Z}^d) \subset \ell^2(\mathbb{Z}^d)$, a smooth map $\tau : \mathcal{U} \to \ell^p_s(\mathbb{Z}^d)$ is \emph{symplectic} if
$$
\forall u \in \mathcal{U}, \forall v,w \in \ell^p_s(\mathbb{Z}^d), \quad (iv,w)_{L^2} = (i\mathrm{d}\tau(u)(v),\mathrm{d}\tau(u)(w))_{L^2}.
$$
Moreover, if  $H,K :  \mathcal{U}  \to \mathbb{R}$ are two smooth functions such that $\nabla H$ (or $\nabla K$) is  $ \ell^p_s(\mathbb{Z}^d)$ valued then the \emph{Poisson bracket} of $H$ and $K$ is defined by
$$
\{  H,K\}(u):= (i \nabla H(u),\nabla K(u))_{L^2}.
$$
Note that, as usual, we have
\begin{equation}
\label{eq:pois-brak}
\{  H,K\}
  =2i \sum_{k\in \mathbb{Z}^d} \partial_{\overline{u_k}}H \partial_{u_k} K - \partial_{u_k}H \partial_{\overline{u_k}} K.
\end{equation}
Note that this relation allows to extend the Poisson bracket to the case where $H$ and $K$ are not real valued.

\subsubsection{Subspaces and projections}
Given any subset $\mathcal{A} \subset \mathbb{Z}^d$, we always consider $\mathbb{C}^{\mathcal{A}}$ as the subspace of $\mathbb{C}^{\mathbb{Z}^d}$ of the sequences supported on $\mathcal{A}$. We denote by $\Pi_{\mathcal{A}} : \mathbb{C}^{\mathbb{Z}^d} \to \mathbb{C}^{\mathcal{A}}$ the projection defined by restriction. We extend implicitly any function $F$ on $\mathbb{C}^{\mathcal{A}}$ to a function on $\mathbb{C}^{\mathbb{Z}^d}$ by $F=F\circ \Pi_{\mathcal{A}}$. As a consequence, any function on $\ell^2(\mathcal{A};\mathbb{C}):=\ell^2(\mathbb{Z}^d) \cap \mathbb{C}^{\mathcal{A}}$ can be see as a smooth function in $\ell^2(\mathbb{Z}^d)$ and so all the previous definitions make sense. In particular, we note that if $F : \mathbb{C}^{\mathcal{A}} \to \mathbb{R} $ is $C^1$ then $\nabla F$ is $\mathbb{C}^{\mathcal{A}}$ valued. Given a positive number $N$, we set $\Pi_{\geq N} := \Pi_{ \{ n \in \mathbb{Z}^d \ | \ |n| \geq N \}  }$, $\Pi_{N} := \Pi_{ \{ n \in \mathbb{Z}^d \ | \ |n| \leq N \}  }$ and $\mathbb{Z}^d_N = \{n \in \mathbb{Z}^d \ | \ |n| \leq N \}$.

\medskip

\subsubsection*{Acknowledgements}
The authors would like to thank the anonymous referee for the careful reading of the manuscript and for the valuable comments and suggestions. J.B. would like to thank E. Faou and B. Gr\'ebert for enthusiastic discussions about this problem many years ago.
During the preparation of this work the authors benefited from the support of the Centre Henri Lebesgue ANR-11-LABX-0020-0, the region "Pays de la Loire" through the project "MasCan" and the ANR project KEN ANR-22-CE40-0016. 

\section{Low frequency reduction}
\label{sec:low-freq}
In this section, we reduce the proof of Theorem \ref{thm:main} to that of Theorem \ref{thm:main_low} on the stability of small solutions of a finite-dimensional dynamical system. First, we introduce some notations about polynomials on $\ell^1$ in order to formulate Theorem \ref{thm:main_low}  in a second subsection. Then we prove some a priori estimates for the time-variation of the high-frequency superactions, together with some mismatch lemmas to control the remainder terms generated by the truncations. Finally, we prove that Theorem \ref{thm:main_low} implies Theorem \ref{thm:main}. The dimension $d\geq 1$, the admissible torus $\mathbb{T}^d_{\mathscr{L}}$ and the nonlinearity $f$ of \eqref{eq:nls} are considered as fixed. With a few exceptions, we do not follow in details the dependencies with respect to these constants. 

\medskip

Many of the proof's ingredients are the same as in \cite{BFM22}. However, we use them differently. The main reason is that changes of variables mixing high modes and low modes are, a priori, not compatible with our measure estimates (and in particular that it suffices to draw the low modes of the initial datum randomly to get the almost global existence of the solution to \eqref{eq:nls}). Thus, contrary to \cite{BFM22}, we have to consider two kinds of Birkhoff-like normal forms for \eqref{eq:nls}.
\begin{itemize} \item On the one hand, we consider (in the proof of Proposition \ref{prop:ap-hi}) a normal form in which we have removed all the terms whose small divisors are larger or equal to $1$. It allows us to prove, under a bootstrap assumption, the almost preservation of the high super-actions in Proposition \ref{prop:ap-hi}. It does not require any non-resonance condition but just small and smooth solutions.

\item On the other hand, in Subsection \ref{sub:a-b}, we consider a normal form in which we essentially\footnote{It would be true if $\mathbb{T}^d_{\mathscr{L}}$ satisfied a stronger Diophantine condition. Here, since we did not assume non-resonance conditions for terms of degree larger than or equal to $6$, some quasi-resonant terms may remain. However, these term do not pose significant difficulties and can be ignored on first reading.} remove all non-resonant terms that do not involve high modes.
By construction, the associated change of variable is the identity on high mode which is good for the measure estimates. This normal form allows us to decouple (in the new variables), up to a bootstrap assumption, the low modes dynamics from the the high modes dynamics (thanks to the mismatch lemmas \ref{lem:misI} and \ref{lem:misII}). In other words, it reduces the analysis of the stability of the low modes of \eqref{eq:nls} to the one of a finite (but high) dimensional resonant Hamiltonian system (of which the rest of the paper is devoted).
\end{itemize}

\subsection{A first Hamiltonian formalism} We first introduce a set of multi-indices to describe the polynomials:
\begin{definition}
\label{def:multi}
For $q\geq1$, we define $\mathcal{N}_{2q}$ the set of multi-indices of degree $2q$ with zero momentum 
\[
\mathcal{N}_{2q} = \Big\{ \vec{n} = (n_{1},\cdots,n_{2q})\in (\Z^d)^{2q}\ |\ \sum_{i=1}^{2q}(-1)^{i}n_{i}=0\,\Big\}.
\]
Given $\vec{n}\in\mathcal{N}_{2q}$, we denote the decreasing rearrangement of $(|n_{1}|,\cdots,|n_{2q}|)$ by
\[
n_{1}^{\ast}\geq \cdots \geq n_{2q}^{\ast}.
\]
\end{definition}
Now, we introduce our main class of polynomials.
\begin{definition}[Real homogeneous polynomial]
\label{def:hamb} Let $q\geq 2$ and $\mathcal{A} \subset \mathbb{Z}^d$. The set $\mathcal{H}_{2q}(\mathcal{A})$ of real homogeneous polynomials of degree $2q$ supported on $ \ell^1(\mathcal{A})$ corresponds to the set of functions 
\[
Q(u) = \sum_{\vec{n}\in\mathcal{N}_{2q}} Q_{\vec{n}}u_{n_{1}}\overline{u_{n_{2}}}\cdots \overline{u_{n_{2q}}}=:\sum_{\vec{n}\in\mathcal{N}_{2q}}Q_{\vec{n}}u_{\vec{n}}\,,
\]
where  $(Q_{\vec{n}})\in\C^{\mathcal{N}_{2q}}$ satisfies

\noindent \ (1) (symmetry condition) $Q_{\vec{n}}$ is symmetric in $(n_{1},n_{3},\cdots n_{2q-1})$ and in $(n_{2},n_{4},\cdots n_{2q})$.
 
\noindent \ (2) (reality condition) For all $\vec{n}=(n_{1},n_{2},\cdots,n_{2q})\in\mathcal{N}_{2q}$, 
\[
Q_{(n_{1},n_{2},\cdots n_{2q})} = \overline{Q_{(n_{2q},n_{1},\cdots n_{2q-1})}}. 
\]

\noindent \ (3) (boundedness condition)
\[
\|Q\|_{\infty} := \underset{\vec{n}\in\mathcal{N}_{2q}}{\sup}\ |Q_{\vec{n}}| <+\infty\,.
\]

\noindent \ (4) (support)
$$
Q_{(n_{1},n_{2},\cdots n_{2q})} \neq 0 \quad  \implies\quad \vec{n} \in \mathcal{A}^{2q}.
$$

\end{definition}
\begin{remark}
Thanks to the boundedness condition, these polynomials are smooth functions on $\ell^1(\mathbb{Z}^d)$ and the support condition only means that $Q = Q\circ \Pi_{\mathcal{A}}$.
\end{remark}

\begin{remark} Note that this definition is well suited to the  Hamiltonian $H$ of \eqref{eq:nls} (defined by \eqref{eq:hamNLS}). More precisely, 
in Fourier variables, the Taylor expansion of $H$ writes
$$
H(u) \mathop{=}_{u \to 0} \frac12 \sum_{n\in \mathbb{Z}^d} \lambda_n^2 |u_n|^2 + \frac12 \sum_{q\geq 2} \frac{f^{(q-1)}(0)}{q!} \sum_{\vec{n}\in\mathcal{N}_{2q}}     u_{\vec{n}}
$$
or more quantitatively, for all $r\geq 2$, 
\begin{equation}
\label{eq:expansion_ham}
\begin{split}
H(u) &= \frac12 \sum_{n\in \mathbb{Z}^d} \lambda_n^2 |u_n|^2 + \frac12 \sum_{q=2}^r \frac{f^{(q-1)}(0)}{q!} \sum_{\vec{n}\in\mathcal{N}_{2q}}   u_{\vec{n}}+ R^{(2r+2)} (u) \\ &=:H^{(\leq 2r)}(u) + R^{(2r+2)} (u)
\end{split}
\end{equation}
where $R^{(2r+2)} $ is a remainder term of order $2r+2$ in the sens that for all $u\in B_s(1)$, $s>\frac{d}{2}$
\begin{equation}
\label{eq:rem_NLS}
\| \nabla  R^{(2r+2)} (u) \|_{h^s} \lesssim_{r,s} \| u \|_{h^s}^{2r+1}.
\end{equation}

\end{remark}

\begin{definition}[Resonant function] Given a multi-index $\vec{n}\in\mathcal{N}_{2q}$, the resonant function $\Omega_{\vec{n}}$ is 
\[
\Omega_{\vec{n}} = \sum_{n=1}^{2q}(-1)^{n+1}\lambda_{n}^{2}\,.
\]
\end{definition}

 \begin{definition}[Quasi-resonant Hamiltonian]\label{def:non-kappa} Let $\kappa>0$, $q\geq2$. A real homogeneous polynomial $Q \in \mathcal{H}_{2q}$ 
 is $\kappa$-\emph{resonant} if for all $\vec{n}\in\mathcal{N}_{2q}$, 
 \[
 |\Omega_{\vec{n}}| > \kappa\quad \implies\quad Q_{\vec{n}} = 0\,
 \]
 \end{definition}

\subsection{Dynamics of the low modes : main result}

We set $c_* := 2\tau_\ast +1$ depending only on the geometry of the torus ($\tau_\ast$ is the exponent associated to the admissibility of $\mathbb{T}^d_{\mathscr{L}}$, see Definition \ref{def:tori}).

\begin{theorem}
\label{thm:main_low}
For all $r\geq 9$, there exists $\nu \leq \min((2c_\ast)^{-1},(2d)^{-1})$ such that for all $s>d$, if $\varepsilon \lesssim_{r,s} 1$ is small enough, $M\geq 2$ is a truncation parameter satisfying
$$
M \leq \varepsilon^{-\nu},
$$  
and $H_{\lo}$ is a  real valued polynomial supported on $\ell^1(\mathbb{Z}^d_{\leq M})$ of the form
\begin{equation}
\label{eq:Hlo-0}
H_{\lo}(u) = \frac12 \sum_{|n|\leq M} ( \lambda_n^2 - \frac{f'(0)}2 |u_n|^2) |u_n|^2  + \sum_{j=3}^{2r} P^{(2j)}(u)
\end{equation}
where $P^{(2j)} \in \mathcal{H}_{2j}(\mathbb{Z}^d_{\leq M})$ is a real homogeneous polynomial of degree $2j$ which is 
$\varepsilon^{c_* \nu}$ resonant and satisfies $\| P^{(2j)} \|_{\infty} \leq \varepsilon^{- c_* \nu  j }$, then
there exists an open set $\Theta_{\varepsilon}^\flat \subset \Pi_{M} B_s(\varepsilon)$ such that 
\begin{equation}
\label{eq:density}
\mathrm{meas}(\Theta_{\varepsilon}^\flat ) \geq (1 - \varepsilon^{\frac{1}{39}} )\mathrm{meas}(\Pi_{M} B_s(\varepsilon))
\end{equation}
and if $u \in C^1([-T_{\epsilon},T_{\epsilon}];\Pi_M h^s)$ is such that $u(0) \in \Theta_{\varepsilon}^\flat$ and
$$
\sup_{0\leq t \leq T_{\epsilon}}  \|  i\partial_t u - \nabla H_{\lo} (u)  \|_{h^s} \leq  \varepsilon^{3r}\rho
$$
for some $T_{\epsilon}\leq \varepsilon^{- r}$ and some $\rho>0$ satisfying $\|u(0)\|_{h^s} \leq \rho \leq 2 \varepsilon$ then
\begin{equation}
\label{eq:boo-u}
\sup_{0\leq t \leq T} \sum_{|n|\leq M} \langle n\rangle^{2s} \big| |u_n(t)|^2 - |u_n(0)|^2 \big|\leq \frac{\varepsilon }2 \rho^2\,.
\end{equation}
\end{theorem}

\subsection{A Birkhoff normal form theorem and some basic estimates}

In the energy estimates we will need the following estimate, which is a consequence of Cauchy--Schwarz and Young's convolution inequalities:
\begin{lemma}
\label{lem:young}
Let $q\geq1$, $\iota \in \{-1,1\}^{2q}$ and $a^{(1)},\dots, a^{(2q)}\in \ell^1(\Z^d ; \R_+)$. We have
\[
\sum_{\substack{\vec{n}=(n_1,\dots,n_{2q})\in(\Z^d)^{2q}\\ \iota_1n_1+\iota_2n_2+\dots +\iota_{2q}n_{2q}=0}}a_{n_1}^{(1)}\dots a_{n_{2q}}^{(2q)}\leq\underset{\sigma\in\mathfrak{S}_{2q}}{\min} \| a^{(\sigma(1))} \|_{\ell^2}\| a^{(\sigma(2))} \|_{\ell^2} \prod_{i=3}^{2q}\| a^{(\sigma(i))} \|_{\ell^1}\,.
\]
\end{lemma}
As usual, we deduce some useful estimate for the vector fields.
\begin{corollary}[Vector field estimates]\label{cor:vecfi} Given $q\geq 2$, $P\in \mathcal{H}_{2q}(\mathbb{Z}^d)$, $s\geq 0$, $P$ defines a smooth function on $h^s \cap \ell^1$, $\nabla P$ is smooth from $h^s \cap \ell^1$ to $h^s \cap \ell^1$ and for all  $u\in h^s \cap \ell^1$, we have
$$
\|\nabla P\|_{h^s}(u) \lesssim_{q,s} \| P \|_{\infty} \|u\|_{h^s} \|u\|_{\ell^1}^{2q-2}.
$$
\end{corollary}
Then, as usual, we deduce the local existence of the Hamiltonian flows of the polynomials of which we summarize the properties that will be useful for us in this paper.
\begin{corollary}[Hamiltonian flows]\label{cor:flow} Let $s > \frac{d}{2}$, $r\geq 2$, $C\geq 1$, $\mathcal{A} \subset \mathbb{Z}^d$. For all $\kappa \in (0,1)$, there exists $\varepsilon_* \gtrsim_{r,s,C} \sqrt{\kappa}$ such that for all real polynomial  $\chi$ of degree smaller than or equal $2r$ of the form
$$
\chi = \chi^{(4)} + \cdots + \chi^{(2r)} 
$$
where
$$
\forall j\in \{ 2,\cdots,r\}, \, \chi^{(2j)} \in \mathcal{H}_{2j}(\mathcal{A})\quad \mathrm{satisfies} \quad \| \chi^{(2j)} \|_{\infty} \leq C\kappa^{-j+1},
$$
there exists a smooth map  
$$
\Phi_{\chi} : \left\{ \begin{array}{cll} [-1,1] \times \Pi_{\mathcal{A}}^{-1}\Pi_{\mathcal{A}}B_s(\varepsilon_*)  & \to & h^s(\mathbb{Z}^d) \\
(t,u) & \mapsto & \Phi_\chi^t(u)
\end{array}\right.
$$
such that for all $u\in \Pi_{\mathcal{A}}^{-1}\Pi_{\mathcal{A}}B_s(\varepsilon_*) $
$$
i\partial_t \Phi_\chi^t(u) = (\nabla \chi)\circ \Phi_\chi^t (u) \quad \forall t \in [-1,1],
$$
and for all $t\in [-1,1]$, $\Phi_\chi^t$ is symplectic,
\begin{itemize}
\item close to the identity
$$
\| \Phi_\chi^t(u) - u \|_{h^s} \leq \Big( \frac{\| \Pi_{\mathcal{A}}u \|_{h^s}}{ \varepsilon_* } \Big)^2 \| \Pi_{\mathcal{A}} u\|_{h^s},
$$
\item its derivative is not too big, i.e. 
\begin{equation}
\label{eq:d-varphi}
\| \mathrm{d} \Phi_\chi^t(u) \|_{h^s \to h^s} \leq 2
\end{equation}
\item and provided that $\Phi_\chi^t(u) \in \Pi_{\mathcal{A}}^{-1}\Pi_{\mathcal{A}}B_s(\varepsilon_*)$, we have $\Phi_\chi^{-t}(\Phi_\chi^t(u) ) = u $.
\end{itemize}
\end{corollary} 

\begin{remark}
These projections may seem unusual but they are just a consequence of the block diagonal structure of $\Phi_\chi^t$ :
$$
\Phi_\chi^t = \Pi_{\mathcal{A}} (\Phi_\chi^t)_{| h^s(\mathcal{A})} \oplus \mathrm{id}_{h^s(\mathcal{A}^c)}.
$$
\end{remark}

We also recall the following continuity estimates of the Poisson brackets when applied to polynomials.
\begin{lemma} Let $\mathcal{A} \subset \mathbb{Z}^d$ be a subset of $\mathbb{Z}^d$, $P \in \mathcal{H}_{2p}(\mathcal{A})$ a homogeneous real polynomial of degree $2p\geq 4$ supported on $\ell^1(\mathcal{A})$ and $Q \in \mathcal{H}_{2q}(\mathcal{A})$  a real homogeneous polynomial of degree $2q\geq 4$ supported on $\ell^1(\mathcal{A})$ then their Poisson bracket $\{P,Q\} \in \mathcal{H}_{2(p+q-1)}(\mathcal{A})$ is a real homogeneous polynomial of degree $2(p+q-1)$ supported on $\ell^1(\mathcal{A})$ satisfying the bound
$$
\| \{ P,Q\} \|_{\infty} \lesssim pq \|P\|_{\infty} \| Q\|_{\infty}.
$$
\end{lemma}

As usual, we deduce of these estimates the following Birkhoff normal form theorem.
\begin{theorem}[Quasi-resonant normal form]
\label{thm:res-nf} Let $r\geq2$, $s> \max(\frac{d}{2},1)$, $\mathcal{A} \subset \mathbb{Z}^d$, $\kappa \in (0,1)$ and  $P = P^{(4)} + \cdots + P^{(2r)}$ be a real Hamiltonian of degree $2r$ supported on $\ell^1(\mathcal{A})$ with $P^{(2j)} \in \mathcal{H}_{2j}(\mathcal{A})$ for all $j\in \{ 2, \cdots,r \}$ and set 
\[
H = Z_2^{\mathcal{A}} + P \quad \mathrm{where } \quad Z_2^{\mathcal{A}}(u) = \frac12\sum_{n\in \mathcal{A}} \lambda_n^2 |u_n|^2 \,.
\]
There exist some polynomials $\chi^{(2j)} \in \mathcal{H}_{2j}(\mathcal{A})$, $2\leq j \leq r$, satisfying
$$
\| \chi^{(2j)} \|_{\infty} \lesssim_{ C,r} \kappa^{-j+1} 
$$
where $C = \max_{j\geq 2} \kappa^{2-j}\| P^{(2j)} \|_{\infty}$ such that on $\Pi_{\mathcal{A}}^{-1}\Pi_{\mathcal{A}}B_s(\varepsilon_*) $ (where $\varepsilon_*$ is given by Corollary \ref{cor:flow} )
$$
H\circ \Phi_{\chi}^1 = Z_2^{\mathcal{A}}  + Q^{(4)} + \cdots+ Q^{(2r)} + \Upsilon
$$
where $\chi = \chi^{(4)} + \cdots + \chi^{(2r)}$, the polynomials $Q^{(2j)} \in \mathcal{H}_{2j}(\mathcal{A})$ are some $\kappa$ resonant real homogeneous polynomials satisfying  for all $j\in \{2,\cdots,r\}$ and all $\vec{n} \in (\mathbb{Z}^d)^4$
$$
\| Q^{(2j)} \|_{\infty} \lesssim_{C,r} \kappa^{-j+2} \quad \mathrm{and} \quad Q^{(4)}_{\vec{n}} = \mathbf{1}_{|\Omega_{\vec{n}}| \leq \kappa} P^{(4)}_{\vec{n}}
$$
and $\Upsilon \in C^1(\Pi_{\mathcal{A}}^{-1}\Pi_{\mathcal{A}}B_s(\varepsilon_*);\mathbb{R})$ is a remainder of order $2(r+1)$ in the sense that for all $u\in \Pi_{\mathcal{A}}^{-1}\Pi_{\mathcal{A}}B_s(\varepsilon_*)$\,,
\begin{equation}
\label{eq:R-res}
\|\nabla \Upsilon(u)\|_{h^{s}}\lesssim_{C,r,s} \big(\frac{\| \Pi_{\mathcal{A}} u\|_{h^{s}}}{\sqrt{\kappa}} \big)^{2r}\|\Pi_{\mathcal{A}} u\|_{h^{s}}\,.
\end{equation}

\end{theorem}
This is a formulation of Birkhoff normal form theorem, which is by now standard. For the proof, see Theorem 2.15 in \cite{BG22} or Theorem 2.12 in \cite{BGR23}.

\subsection{Almost preservation of the high-super actions}
In this subsection, we prove Proposition \ref{prop:high_modes}, in which we show that up to a bootstrap assumption we control the $H^s$ norm of the high modes of the solution of \eqref{eq:nls} for very long times. Note that we do not need to use any internal or external parameters. 

\medskip

\begin{proposition}
\label{prop:high_modes} There exist $\delta >0$ and $N_0\geq 2$ depending only on $\mathscr{L}$ such that given $T>0$, $c>1$, $s>d$, $\rho>0$ and $u$ a solution to \eqref{eq:nls} in $C([0,T],h^{s})$ with $\|u(0)\|_{h^{s}} \leq \rho$ and
\begin{equation}
\label{eq:as-ap}
\underset{t\in[0,T]}{\sup}\ \|u(t)\|_{h^{s}}\leq c\rho\,,
\end{equation}
then for all $N\geq N_0$, for all $r\geq2$, provided that $\varepsilon \lesssim_{c,r,s} 1$ is small enough and $t\leq T$,
$$
\| \Pi_{\geq N} u(t) \|_{h^s}^2 \leq  2^{2s} \| u(0) \|_{h^s}^2 +  \rho^3 + TN^{-\frac{\delta}{2}(s-d)} \rho^{3}+T \rho^{2r+1}.
$$
\end{proposition}

The proof is based on Birkhoff normal forms, and on the extension of Bourgain's cluster decomposition lemma \cite{Bou98} to any flat tori, which was proved by Berti and Maspero in \cite[Theorem 2.1]{Berti-Maspero-19}. The use of this decomposition in the context of Birkhoff normal forms was recently initiated by Bambusi, Feola and Montalto in \cite{BFM22}. It is a way to overcome the small divisor degeneracy discussed in subsection \ref{sub:resonant}. In the following, we will mainly focus on the proof of Proposition \ref{prop:ap-hi}, from which Proposition \ref{prop:high_modes} easily follows.

\begin{lemma}[Clustering of eigenvalues of $-\Delta_\mathscr{L}$, Theorem 2.1 in \cite{Berti-Maspero-19}]
\label{lem:cluster}
There exist constants $\delta \equiv \delta(d) \in (0,1)$ and $C(\mathscr{L},d)\geq 2$ and a partition of $\Z^d$  
\[
\Z^d = \bigcup_{\upsilon }\mathscr{C}_\upsilon\,,
\]
satisfying the following properties: 
\begin{enumerate}
\item The sets $\mathscr{C}_\upsilon$ are finite and, up to a bounded set $\mathscr{C}_{\upsilon_0}$ such that
\[
\underset{n\in \mathscr{C}_{\upsilon_0}}{\max} |n| \leq C(\mathscr{L},d) \quad \mathrm{and} \quad  |n| < 2 \implies n\in  \mathscr{C}_{\upsilon_0} \,,
\]
\item They are dyadic in the sense that
\begin{equation}
\label{eq:dydy}
 \underset{n\in\mathscr{C}_{{\upsilon}}}{\max} |n| \leq 2\underset{n\in\mathscr{C}_{{\upsilon}}}{\min} |n| =: 2m_\upsilon\,.
\end{equation}
\item If one denotes $[n_1]$ the class of equivalence of $n_1\in\Z^d$, that is
\[
[n_1] = \mathscr{C}_\upsilon\,,\quad \text{where}\ \upsilon\ \text{is such that}\quad n_1\in\mathscr{C}_\upsilon\,,
\]
then we have the separation property: for all $n_1,n_2\in\Z^d$,
\[
[n_1]\neq[n_2]\quad \implies\quad |n_1-n_2| + |\lambda_{n_{1}}^{2}-\lambda_{n_{2}}^{2}| > (|n_1|+|n_2|)^{\delta(d)}\,.
\]
\end{enumerate}
\end{lemma}

\begin{remark}
As usual, we identify the set of the indices $\upsilon$ with the set of the equivalence classes.
\end{remark}

We stress out that this lemma holds for any flat torus, without any Diophantine assumption on the metric. As in \cite{BFM22}, this decomposition allows us to design some almost conserved quantities of \eqref{eq:nls} called \emph{super-action}.

\begin{definition}[Super actions] \label{def:super_actions} Given a mode $n_0\in \mathbb{Z}^d$ the corresponding super-action is defined by
\[
S_{n_0}(u) := \sum_{n\in[n_{0}]} |u_n|^2\,.
\]
\end{definition}

The cluster being dyadic (see \eqref{eq:dydy}), Proposition \ref{prop:high_modes} can easily be deduced  from the following:
\begin{proposition}[A priori estimates for the high superactions]\label{prop:ap-hi} Let $T,u,c,s,\rho$ be as in Proposition \ref{prop:high_modes}. For all $M \geq 2$,
for all $r\geq2$, provided that $\rho\lesssim_r 1$ is small enough and $t\leq T$, we have
\begin{equation}
\label{eq:ap-hi}
\sum_{m_\upsilon \geq M} m_\upsilon^{2s} \big|S_{\upsilon}(u(t))-S_{\upsilon}(u(0))\big| 
	\lesssim_{r,s,c} \Big(\rho^4 + TM^{-\frac{\delta}{2}(s-d)} \rho^{4}+T \rho^{2(r+1)}\Big)\,,
\end{equation}
where $m_{\upsilon}$ was defined in \eqref{eq:dydy}
\end{proposition}

\begin{proof}[Proof of Proposition \ref{prop:high_modes} from Proposition \ref{prop:ap-hi}] Indeed, to deduce Proposition \ref{prop:high_modes} from Proposition \ref{prop:ap-hi},  thanks to the dyadicity of the clusters (see \eqref{eq:dydy}) to it enough to set $N_0 =  3C(\mathscr{L},d)$ and $M=N/2$ to get that
$$
\| \Pi_{\geq N} u \|_{h^s}^2  \leq 2^{2s} \sum_{|n|\geq N} m_{[n]}^{2s} |u_n|^2 \leq  2^{2s} \sum_{ m_\upsilon \geq M} m_\upsilon^{2n} S_\upsilon(u). 
$$
and so thanks to Proposition \ref{prop:ap-hi} that
$$
\| \Pi_{\geq N} u(t) \|_{h^s}^2  \leq 2^{2s} \sum_{ m_\upsilon \geq M} m_\upsilon^{2n} S_\upsilon(u(0)) + C_{r,s,c} \, \rho \, \Big(\rho^3 + TN^{-\frac{\delta}{2}(s-d)} \rho^{3}+T \rho^{2r+1}\Big)
$$
where $C_{r,s,c}$ is a constant depending only on $(r,s,c)$. We conclude then the proof of Proposition \ref{prop:high_modes}, by using the estimate $\sum_{ m_\upsilon \geq M} m_\upsilon^{2s} S_\upsilon(u(0)) \leq \|u(0) \|_{h^s}^2$ and by absorbing the constant $C_{r,s,c} $ thanks to the smallness of $\rho$.
\end{proof}

Before we  prove Proposition \ref{prop:ap-hi}, let us consider the following technical lemmas.

\begin{lemma}\label{lem:mu3} Let $\vec{n}\in\mathcal{N}_{2q}$ for some $q\geq2$ be such that $|\Omega_{\vec{n}}|\leq 1$. If there exists a frequency cluster $\mathscr{C}_{\upsilon}$ such that
\[
\{ S_{\upsilon}, u_{\vec{n}}\} \neq 0\,,
\]  
then
\[
 n_{1}^{\ast}\geq m_\upsilon \quad \mathrm{and} \quad n_{3}^{\ast}\gtrsim_q (m_\upsilon)^{\delta/2}\,.
\] 
\end{lemma}

\begin{proof}
First, we note that
$$
\{ S_{\upsilon}, u_{\vec{n}}\} = 2i  \Big( \sum_{j=1}^q \mathbf{1}_{n_{2j} \in \mathscr{C}_\upsilon } - \mathbf{1}_{n_{2j-1}  \in \mathscr{C}_\upsilon } \Big) u_{\vec{n}}.
$$
As a consequence, if $\{ S_{\upsilon}, u_{\vec{n}}\} \neq 0$, there exists $j_0 \in \{1,\cdots,2q\}$ and an odd number $\theta\in \mathbb{Z}$ such that $n_{j_0} \in \mathscr{C}_\upsilon$ and $n_{j_0 + \theta} \notin \mathscr{C}_\upsilon$. By symmetry, without loss of generality, we assume that $j_0=1$ and $j_0 + \theta= 2$. Since $n_1 \in \mathscr{C}_\upsilon$, we have $n_{1}^{\ast}\geq m_\upsilon$. Moreover, thanks to the separation property of the clusters, we have
$$
|n_1-n_2| + |\lambda_{n_{1}}^{2}-\lambda_{n_{2}}^{2}| > (|n_1|+|n_2|)^{\delta} \geq m_\upsilon^{\delta}.
$$
As a consequence either (by the zero momentum condition)
$$
 (2q-2) \max_{j\geq 3} |n_j| \geq |n_1-n_2| \geq  \frac{m_\upsilon^{\delta}}2
$$
or (since $|\Omega_{\vec{n}}|\leq 1$)
$$
 (2q-2)   \max_{j\geq 3} |n_j|^2 + 1 \gtrsim_{\mathscr{L}}  |\lambda_{n_{1}}^{2}-\lambda_{n_{2}}^{2}| \geq \frac{m_\upsilon^{\delta}}2.
$$
In any case, we have
$$
n_3^* \geq \max_{j\geq 3} | n_j | \gtrsim_{\mathscr{L},q}  m_\upsilon^{\delta/2}.
$$
\end{proof}

\begin{lemma} Let $s\geq \sigma \geq 0$, $q\geq 2$ and $Q\in\mathcal{H}_{2q}$.  If $Q$ is $1$-resonant then for all $u\in h^{s} \cap \ell^1_\sigma$, we have
$$
\sum_{\upsilon} m_\upsilon^{2s+\frac{\sigma \delta}2 } |\{S_{\upsilon},Q\}(u)| \lesssim_{q,s} \|Q\|_{\infty} \|u\|_{h^{s}}^{2} \| u \|_{\ell^1_\sigma} \|u\|_{\ell^{1}}^{2q-3}.
$$
\end{lemma}
\begin{proof} First, we note that if $\upsilon = \upsilon_0$ then $m_\upsilon = 0$ so we do not have to consider this class in the sum. Therefore, we consider a class $\upsilon \neq \upsilon_0$. Note that it implies that $m_{\upsilon} \geq 2.$ Then, we note that
$$
\{S_{\upsilon},Q\} = \sum_{\vec{n}\in\mathcal{N}_{2q}} Q_{\vec{n}} \{S_\upsilon, u_{\vec{n}} \}.
$$
Using Lemma \ref{lem:mu3}, we deduce that we can reduce the sum to the set of the indices satisfying $n_{1}^{\ast}\geq m_\upsilon \geq 2$ and
\[
n_{3}^{\ast}\gtrsim_{q}m_\upsilon^{\delta/2}.
\]
Moreover, using the zero momentum condition, we note that $n_{2}^{\ast}\geq (2q)^{-1} n_1^\ast \geq (2q)^{-1} m_\upsilon$. It follows that
\begin{align*}
\sum_{\upsilon } m_\upsilon^{2s+\frac{\sigma \delta}2}|\{S_{\upsilon},Q\}(u)|
	&\lesssim_q \sum_{\upsilon} m_\upsilon^{2s+\frac{\sigma \delta}2}   \sum_{\substack{n_{2}^{\ast} \gtrsim_q m_\upsilon \\ n_{3}^{\ast} \gtrsim_{q} m_\upsilon^{\delta/2}   }}|Q_{\vec{n}}||u_{n_1}| \cdots |u_{n_{2q}}|   \\	
	&\lesssim_{q,s} \| Q\|_{\infty} \sum_{\vec{n}\in\mathcal{N}_{2q}}  (n_{1}^{\ast})^{s}  (n_{2}^{\ast})^{s}  (n_{3}^{\ast})^{\delta \sigma /2} |u_{n_{1}}|\cdots |u_{n_{2q}}| \\
	&\lesssim_{q,s} \| Q\|_{\infty} \sum_{\varphi \in \mathfrak{S}_{2q}} \sum_{\vec{n}\in\mathcal{N}_{2q}}  |n_{\varphi_1}|^{s}  |n_{\varphi_2}|^{s}  |n_{\varphi_3}|^{\delta \sigma /2} |u_{n_{1}}|\cdots |u_{n_{2q}}|.
\end{align*}
And so we get the expected estimate by applying the Young estimate of Lemma \ref{lem:young}.
\end{proof}
By applying this lemma with $\sigma  = s-d$ and by considering only the classes such that $m_\upsilon$ is large enough we deduce the following estimate.
\begin{corollary}
\label{cor:super}Let $s\geq d$, $q\geq 2$ and $Q\in\mathcal{H}_{2q}$. If $Q$ is $1$-resonant then for all $u\in h^{s}$ and $M\geq 2$ we have 
\begin{equation}
\label{eq:super}
\sum_{m_\upsilon \geq M} m_\upsilon^{2s} |\{S_{\upsilon},Q\}(u)| \lesssim_{q,s} \|Q\|_{\infty} M^{-\frac{\delta}{2}(s-d)}\|u\|_{h^{s}}^{3} \|u\|_{\ell^{1}}^{2q-3}\,.
\end{equation}
\end{corollary}

Now we start proving Proposition \ref{prop:ap-hi} about the preservation of the super-actions of \eqref{eq:nls}.

\begin{proof}[Proof of Proposition \ref{prop:ap-hi}]
Recall the Hamiltonian formulation \eqref{eq:hamNLS} of \eqref{eq:nls} and the Taylor expansion $H = H^{(\leq 2r)} + R^{(2r+2)}$ of its Hamiltonian \eqref{eq:expansion_ham}. Apply the Birkhoff normal form theorem~\ref{thm:res-nf} with $\kappa=1$ and $\mathcal{A}= \mathbb{Z}^d$ to the Hamiltonian $H^{(\leq 2r)}$ of degree $2r$ to put it in $1$-resonant normal form up to order $2r$. This gives an auxiliary Hamiltonian $\chi = \chi^{(4)} + \cdots + \chi^{(2r)}$ with $\| \chi^{(2j)}\|_\infty \lesssim_r 1$ and, through Corollary \ref{cor:flow}, a radius $\varepsilon_\ast \gtrsim_{r,s} 1$ such that on $B_s(\varepsilon_*)$
\[
\widetilde{H}:=H^{(\leq 2r)}\circ\Phi_\chi^1 =Z_2+ \sum_{q=2}^{r}\widetilde{Q}^{(2q)}+\Upsilon =: \widetilde{H}^{(\leq 2r)} + \Upsilon
\]
where $\widetilde{Q}^{(2q)} \in \mathcal{H}_{2q}$ are 1-resonant  homogeneous polynomials satisfying $\| Q^{(2q)}\|_{\infty} \lesssim_q 1 $ and 
the remainder $\Upsilon$ is of order $2(r+1)$ in the sense that for all $w\in B_s(\varepsilon_*)$
\begin{equation}
\label{eq:bientotnoel}
\|\nabla \Upsilon(w)\|_{h^{s}}\lesssim_{r,s} \| w\|_{h^s}^{2r+1}\,.
\end{equation}
From now assume that the solution $u$ of \eqref{eq:nls} satisfies $\| u(0) \|_{h^s} \leq \rho \leq 10^{-2} c^{-1}\varepsilon_\ast$. It follows that for all $t\in [0,T]$, we have $\|u(t) \|_{h^s} \leq c \rho\leq 10^{-2} \varepsilon_\ast$ and so that $v = \Phi_\chi^{-1}(u) \in C^0([0,T];h^s)$ is well defined for $t\in [0,T]$. Moreover, $ \Phi_\chi^{-1}$ being close to the identity, we have
\begin{equation}
\label{eq:pr-var}
\|u(t)-v(t)\|_{h^{s}}\leq \varepsilon_{\ast}^{-2}\|u(t)\|_{h^{s}}^{3} \lesssim_{r,c} \rho^3  \,,\quad \text{and \ so}\quad \|v(t)\|_{h^{s}}\lesssim_{r,c} \rho \,.
\end{equation}
Without loss of generality, we can assume that\footnote{It can be done by approximating any $u(0)$ by smoother initial data (e.g. trigonometric polynomials) and using the preservation of the regularity for \eqref{eq:nls} (i.e. that while the $\ell^1$ norm of the solution is bounded, its $h^{s+2}$ norm does not blow up). Such an approximation is classical and is done in details, for example, in \cite{BG22}.} $u(0) \in h^{s+2}$ to get $u\in C^1([0,T];h^s)$ and so that $v \in  C^1([0,T];h^s)$. Since  $\Phi_\chi^{-1}$ is symplectic, using the chain rule, we get that
\[
i\partial_{t}v = \nabla \widetilde{H}^{(\leq 2q)}(v) + \nabla \Upsilon (v) + \mathrm{d}\Phi_\chi^{-1}(u)( \nabla R^{(2r+2)}(u))\,.
\] 
As a consequence, one can bound the variation of the superactions as follows:
\begin{multline}
\label{eq:pr-v}
\sum_{m_\upsilon \geq M}m_\upsilon^{2s}\big|\frac{d}{dt}S_{\upsilon}(v(t))\big| 
	\leq \sum_{m_\upsilon \geq M}m_\upsilon^{2s}|\{S_{\upsilon},\widetilde{H}^{(\leq 2r)}\}(v)| 
	+ \|\nabla \Upsilon (v)\|_{h^{s}}\|v\|_{h^{s}}\\+\|\mathrm{d}\Phi_\chi^{-1}(u)\|_{h^{s}\to h^{s}}\|\nabla R^{(2r+2)}(u)\|_{h^{s}}\|v\|_{h^{s}}\,.
\end{multline}
According to Corollary \ref{cor:super} the first term on the right-hand side is bounded (up to a constant depending only on $r$ and $s$) by
\[
\underset{2\leq q \leq r}{\max} \|\widetilde{Q}^{(2q)}\|_{\infty} M^{-\frac{\delta}{2}(s-d)}\|v(t)\|_{h^{s}}^{3} \|v(t)\|_{\ell^{1}}^{2q-3} \lesssim_{r,s,c}  M^{-\frac{\delta}{2}(s-d)} \rho^4 \,.
\]
Moreover, the estimate \eqref{eq:d-varphi} on $\mathrm{d}\Phi_\chi^{t}$  together with the estimate \eqref{eq:rem_NLS} on $\nabla R^{(2r+2)}$ yield 
\[
\|\mathrm{d}\Phi_\chi^{-1}(u)\|_{h^{s}\to h^{s}}\| \nabla R^{(2r+2)} (u)\|_{h^{s}} \lesssim_{r,s,c} \rho^{2r+1}\,.
\]
Using also the estimate \eqref{eq:bientotnoel} on $\nabla \Upsilon$, 
we conclude that 
$$
\sum_{m_\upsilon \geq M}m_\upsilon^{2s}\big|\frac{d}{dt}S_{\upsilon}(v(t))\big| \lesssim_{c,r,s} M^{-\frac{\delta}{2}(s-d)} \rho^4+ \rho^{2r+1}.
$$
Finally, using the fundamental theorem of calculus and that $v$ is close to $u$ (see \eqref{eq:pr-var}; it provides the term $\rho^4$ without factor $T$ in \eqref{eq:ap-hi}),  we get the expected estimate.
\end{proof}

\subsection{Mismatch lemmata} In order to deduce Theorem \ref{thm:main} from Theorem \ref{thm:main_low}, we will have to consider truncated subsystems. In this subsection, we establish two technical lemmata proving that the associated truncation remainder terms are very small.

\begin{lemma}[Mismatch vector field estimate: I]\label{lem:misI}
For all $q\geq 2$, $s>d$, $u\in h^s$, all $M,N\geq 2$ with 
$N<\frac{1}{2q}M\,, $
and for all $Q\in\mathcal{H}_{2q}$ such that $Q\circ \Pi_M = 0$, we have 
\begin{equation*}
\|\Pi_{N} \nabla Q(u)\|_{h^{s}} \lesssim_{q} M^{-(s-d)}
 	\|Q\|_{\infty}
	\|u\|_{h^{s}}^{3}\|u\|_{\ell^{1}}^{2q-4}\,.
\end{equation*}
\end{lemma}
Note that we do not ask for any constraint on the size of the resonance function. The gap between $N$ and $M$ is sufficiently large to directly have the mismatch estimate. 
\begin{proof}
For $n\in \mathbb{Z}^d$ with $|n|\leq N$ and $u\in h^{s}$,
\begin{equation*}
\begin{split}
\langle n\rangle^{s}| \partial_{\overline{u_n}} Q(u)|& \leq 2q \langle n\rangle^{s}
	\sum_{\substack{\vec{n}\in\mathcal{N}_{2q}\\ n_{2q}=n}}	|Q_{\vec{n}}||u_{n_{1}}|\cdots|u_{n_{2q-1}}| \\
	&\leq 2q\|Q\|_{\infty} \langle n\rangle^{s}\sum_{\substack{\vec{n}\in\mathcal{N}_{2q}\\n_{2q}=n}}		\mathbf{1}_{n_{1}^{\ast}\geq M}|u_{n_{1}}|\cdots|u_{n_{2q-1}}|\,.
	\end{split}
\end{equation*}
From the zero momentum condition, we have that $\vec{n}\in\mathcal{N}_{2q}$ contributes only if $n_{1}^{\ast}\geq M$ and
\[
n_{2}^{\ast}\geq \frac{n_{1}^{\ast}}{2q} \geq \frac{M}{2q} > N \geq n\,.
\]
As a consequence, we have $|n|\leq n_{3}^{\ast}$ and so
\begin{equation*}
\begin{split}
\langle n\rangle^{s}|\partial_{\overline{u_n}} Q(u)| &\leq 2q M^{-(s-d)} \|Q\|_{\infty}  
	\sum_{\substack{\vec{n}\in\mathcal{N}_{2q}\\n_{2q}=n}} 
	\langle n_{1}^{\ast}\rangle^{s-d} \langle n_{3}^{\ast}\rangle^{s} 
	|u_{n_{1}}|\cdots|u_{n_{2q-1}}|\\
	&\leq 2q M^{-(s-d)} \|Q\|_{\infty}  
	\sum_{\substack{\vec{n}\in\mathcal{N}_{2q} \\ n=(\ell,n_{2q})}} 
	\langle \ell_{1}^{\ast}\rangle^{s-d} \langle \ell_{2}^{\ast}\rangle^{s} 
	|u_{n_{1}}|\cdots|u_{n_{2q-1}}|\\
	&\lesssim_q M^{-(s-d)} \|Q\|_{\infty}  
	\sum_{\varphi \in \mathfrak{S}_{2q-1}}\sum_{\substack{\vec{n}\in\mathcal{N}_{2q}\\n_{1}=n}} 
	\langle n_{\varphi_1}\rangle^{s-d} \langle n_{\varphi_2}\rangle^{s} 
	|u_{n_{2}}|\cdots|u_{n_{2q}}|.
\end{split}
\end{equation*}
Finally by applying the Young convolution inequality (see Lemma \ref{lem:young}) we get that
\begin{equation*}
\begin{split}
\|\Pi_N\nabla Q(u) \|_{h^s} &\lesssim_q M^{-(s-d)}  \|Q\|_{\infty}   \| u\|_{\ell^1}^{2q-4} \| u \|_{h^s}^2 \| u\|_{\ell^1_{s-d}} \\ &\lesssim_q M^{-(s-d)} \|Q\|_{\infty}    \| u\|_{\ell^1}^{2q-4} \| u \|_{h^s}^3.
\end{split}
\end{equation*}
\end{proof}

\begin{lemma}[Mismatch vector field estimate: II]
\label{lem:misII} 
Let 
$\Clo\cup \Chigh \subset \mathbb{Z}^d$ be a partition of $\mathbb{Z}^d$ compatible with the Bourgain's partition given by Lemma \ref{lem:cluster}, i.e.
$$
n\in \Clo \quad \implies \quad [n] \subset \Clo
$$
such that $\mathscr{C}_{\upsilon_0} \subset \Clo$ and set 
$$
N = \min_{n \in \Chigh} |n|.
$$
For all $q\geq 2$, all $1$-resonant homogeneous polynomial $Q\in\mathcal{H}_{2q}$  of degree $2q$ satisfying $Q\circ \Pi_{\Clo}=0$, and all $u\in h^{s}$ with $s>d$, we have
\begin{equation*}
\label{eq:misII}
\|\Pi_{\Clo} \nabla Q(u)\|_{h^{s}} \lesssim_{q,s} \| Q\|_{\infty} N^{-\frac{\delta}{2}(s-d)}\|u\|_{h^{s}}^{3}\|u\|_{\ell^{1}}^{2q-4}.
\end{equation*} 
\end{lemma}

\begin{proof} For clarity, in this proof, we highlight the dependencies with respect to $\mathscr{L}$. For $n\in\Clo$, by symmetry, we have

\begin{equation*}
\begin{split}
\langle n\rangle^{s}|\partial_{\overline{u_n}} Q(u)| &\leq 2q\langle n\rangle^{s}\sum_{\substack{\vec{n}\in\mathcal{N}_{2q} \\n_{1}=n}}|Q_{\vec{n}}||u_{n_{2}}|\cdots|u_{n_{2q}}| \\
	&\lesssim_q \|Q\|_{\infty}\langle n\rangle^{s} \big(\sum_{\substack{\vec{n}\in\mathcal{N}_{2q}\\n_{1}=n \\ n_2 \in \Chigh \\ |\Omega_{\vec{n}}|\leq 1 }}|u_{n_{2}}|\cdots|u_{n_{2q}}| + \sum_{\substack{\vec{n}\in\mathcal{N}_{2q}\\n_{1}=n \\  n_3 \in \Chigh \\ |\Omega_{\vec{n}}|\leq 1 }}|u_{n_{2}}|\cdots|u_{n_{2q}}| \big) \\ & =: \|Q\|_{\infty} (\mathcal{E}^{(2)}_n + \mathcal{E}^{(3)}_n)\,
\end{split}
\end{equation*}

and so
$$
\|\Pi_{\Clo} \nabla Q(u)\|_{h^{s}}  \lesssim_q \| \Pi_{\Clo} \mathcal{E}^{(2)} \|_{\ell^2} + \| \Pi_{\Clo} \mathcal{E}^{(3)} \|_{\ell^2}.
$$

\noindent{\bf Case 1:} Suppose that $n_3 \in \Chigh$. First, we have
$$
1\geq |\Omega_{\vec{n}}| \geq \sum_{j=1}^q  \lambda_{n_{2j-1}}^2 - \sum_{j=1}^q \lambda_{n_{2j}}^2 \gtrsim_{\mathscr{L}} |n|^2 - q\max_{1\leq j \leq q}  \lambda_{n_{2j}}^2.
$$
As a consequence, we have
$$
\max_{1\leq j \leq q}  \langle n_{2j} \rangle^2 \gtrsim_{\mathscr{L}} \langle n\rangle^2
$$
and so, we have
$$
\mathcal{E}^{(3)}_n \lesssim_{q,\mathscr{L},s} 	\sum_{\substack{\vec{n}\in\mathcal{N}_{2q}\\ n_1=n \\ |n_3| \geq N  }} \langle n_2 \rangle^{s} |u_{n_{2}}|\cdots|u_{n_{2q}}| \lesssim_{q,\mathscr{L}} N^{-(s-d)}		\sum_{\substack{\vec{n}\in\mathcal{N}_{2q}\\ n_1=n}}  \langle n_2 \rangle^{s} \langle n_3 \rangle^{s-d} |u_{n_{2}}|\cdots|u_{n_{2q}}| .
$$
Thus applying the Young convolutional inequality, we get  
$$\| \Pi_{\Clo} \mathcal{E}^{(3)} \|_{\ell^2} \lesssim_{q,\mathscr{L},s}  N^{-(s-d)}\|u\|_{h^{s}}^{3}\|u\|_{\ell^{1}}^{2q-4}.$$

\medskip

\noindent{\bf Case 2:} Suppose that $n_2 \in \Chigh$.  Since $[n_1] \neq [n_2]$, thanks to the separation property of the clusters, we have
$$
 |n_1-n_2| + |\lambda_{n_{1}}^{2}-\lambda_{n_{2}}^{2}| > (|n_1|+|n_2|)^{\delta} \geq N^{\delta}.
$$
As a consequence either
$$
 (2q-2) \max_{j\geq 3} |n_j| \geq |n_1-n_2| \geq \frac{N^{\delta}}2
$$
or
$$
 (2q-2)   \max_{j\geq 3} |n_j|^2 + 1 \gtrsim_{\mathscr{L}}  |\lambda_{n_{1}}^{2}-\lambda_{n_{2}}^{2}| \geq \frac{N^{\delta}}2.
$$
In any case, we have
$$
\max_{j\geq 3} \langle n_j \rangle \gtrsim_{\mathscr{L},q} N^{\delta/2}.
$$
Moreover, as previously, we have $\max_{1\leq j \leq q}  \langle n_{2j} \rangle^2 \gtrsim_{\mathscr{L}} \langle n\rangle^2$. Thus,
since $\delta \leq 1$, it follows that\footnote{for details it is convenient to distinguish the case $|n_2|^2=\max_{1\leq j \leq q}  |n_{2j}|^2$.}
$$
\mathcal{E}^{(2)}_n \lesssim_{q,\mathscr{L},s} \sum_{\substack{\vec{n}=(n,\ell)\in\mathcal{N}_{2q} \\  \langle \ell_2^* \rangle \gtrsim_{\mathscr{L},q} N^{\delta/2}  }} \langle \ell_1^\ast \rangle^{s} |u_{n_{2}}|\cdots|u_{n_{2q}}| 
$$
and reasoning as previously, we get that
$\| \Pi_{\Clo} \mathcal{E}^{(2)} \|_{\ell^2} \lesssim_{q,\mathscr{L}}  N^{-\frac{\delta}2(s-d)}\|u\|_{h^{s}}^{3}\|u\|_{\ell^{1}}^{2q-4}$.
\end{proof}

\subsection{Proof that Theorem \ref{thm:main_low} implies Theorem \ref{thm:main}}
\label{sub:a-b}

We recall that the flat torus $\mathbb{T}^d_{\mathscr{L}}$ is fixed and admissible in the sense of the Definition \ref{def:tori}. It provides an exponent $\tau_\ast$ (which can typically be chosen equal to $d(d+1)/2+1$). We recall the associated constant $c_\ast = 2 \tau_\ast+1$. We will also need the Bourgain decomposition of Lemma \ref{lem:cluster} and the corresponding exponent $\delta \equiv \delta(d)$. We also recall that the nonlinearity $f$ of \eqref{eq:nls} is fixed and satisfies the condition $f'(0) \neq 0$. Then we divide the proof into three main steps : \emph{preparation}, \emph{dynamics} and \emph{measure estimates}.

\medskip

\noindent $\bullet$ \emph{ Step 1 : Preparation.} Let $r\geq 9$, $\nu \leq \min((2c_\ast)^{-1},1/(2d))$ be the associated constant given by Theorem \ref{thm:main_low}, and assume that  
\begin{equation} 
\label{eq:s}
 s_0 := d+20 \delta^{-1} \nu^{-1} r\,.
\end{equation}
 We also set $\varepsilon \leq \varepsilon_0$ where $\varepsilon_0$ is a constant depending only on $r,s$ (and $\mathscr{L}$ and $f$) that will be chosen small enough. Let $M>0$ be a real number such that
\begin{equation}
\label{eq:boundonM}
\varepsilon^{-\frac{8}{\delta}\frac{r}{s-d}}\leq M \leq \varepsilon^{-{\nu}}
\end{equation}
 and define
$$
N = (8r+1)^{-1} M.
$$
Note that the existence of such a $M$ is ensured by construction of $s_0$.
Then we set a partition $\Clo\cup \Chigh \subset \mathbb{Z}^d$ of $\mathbb{Z}^d$ by 
\begin{equation}
\label{eq:Clo}
\Clo := \bigcup_{|n|\leq N}[n] \quad \mathrm{and} \quad \Chigh := \mathbb{Z}^d \setminus \Clo.
\end{equation}
Note that since the clusters are dyadic (see \eqref{eq:dydy}), we have 
\begin{equation}
\label{eq:m-high}
\frac{1}{2}N\leq \underset{n\in\Chigh}{\min}\ |n|\quad \mathrm{and} \quad \underset{n\in\Clo}{\max}\ |n|\leq 2N \leq M \leq  \varepsilon^{-{\nu}}.
\end{equation}
We recall that the Hamiltonian $H$ of \eqref{eq:nls} (defined by  \eqref{eq:hamNLS}) admits the Taylor expansion $H = H^{(\leq 4r)} + R^{(4r+2)} $ (see \eqref{eq:expansion_ham}). Then, setting 
$$
\kappa := \varepsilon^{c_\ast \nu},
$$
 we apply the Birkhoff normal form theorem \ref{thm:res-nf} to $H^{(\leq 4r)}\circ \Pi_M$. We get a polynomial $\chi = \chi^{(4)} + \cdots + \chi^{(4r)}$ supported on $\ell^1(\mathbb{Z}^d_{\leq M})$ (with  $\| \chi^{(2j)}\| \lesssim_r \kappa^{-j+1} $) and $\varepsilon_\ast \gtrsim_{r,s} \sqrt{\kappa}$ (through Corollary \ref{cor:flow})
 such that on $\Pi_{ M}^{-1}\Pi_{ M}B_s(\varepsilon_*) $ 
\begin{equation}
\label{eq:defHloflat}
\begin{split}
(H^{(\leq 4r)}\circ \Pi_M) \circ \Phi_{\chi}^1 &= Z_2\circ \Pi_M  + Q^{(4)} + \cdots+ Q^{(4r)} + \Upsilon \\
& =: Z_2\circ \Pi_M  + Q^{(4)} + Q^{(\geq 6)} +  \Upsilon
\end{split}
\end{equation}
where $Q^{(2j)} \in \mathcal{H}_{2j}(\mathbb{Z}^d_{\leq M})$ is $\kappa$ resonant and satisfies $\| Q^{(2j)} \|_{\infty} \lesssim_r \kappa^{-j+2}$ and $ \Upsilon$ is of order $4r +2$.

Then we set 
\begin{equation}
\label{eq:defHlo}
H_{\lo} := Z_2\circ \Pi_M  + Q^{(4)} + Q^{(\geq 6)} \circ \Pi_{\Clo} -\frac{f'(0)}2 \| \Pi_{M} \cdot\|_{\ell^2}^4.
\end{equation}

As a consequence, if $H_{\lo} $ satisfies the assumptions of Theorem \ref{thm:main_low}, then, provided that $\varepsilon_{0}$ is small enough, we get an open set
$
\Theta_{\varepsilon}^\flat \subset \Pi_{M} B_s(\varepsilon)
$
and we just have to define our set of good initial data by
\begin{equation}
\label{eq:def_lens_des_condinit}
\Theta_{\varepsilon} := \big(\Pi_{ M } B_s(\varepsilon) \big)\cap  (\Phi_{\chi}^{-1})^{-1}   \Theta_{\varepsilon}^\flat.
\end{equation}
Now, we just have to check that  $H_{\lo} $ satisfies the assumptions of Theorem \ref{thm:main_low}. First, it is clear that the homogeneous terms $P^{(2j)} := Q^{(2j)} \circ \Pi_{\Clo}$, $j\geq 3$, of $H_{\lo} $ are $\kappa =  \varepsilon^{c_\ast \nu}$ resonant and satisfy  $\| P^{(2j)} \|_{\infty} \lesssim_r \kappa^{-j+2}$ and so $\| P^{(2j)} \|_{\infty} \leq \varepsilon^{-jc_\ast \nu}$ provided that $\varepsilon_0$ is small enough. So the only thing we have to check is the form of the quartic terms of $H_{\lo}$. To do this, it is enough to compute $Q^{(4)}$. We recall that by the Birkhoff normal form theorem \ref{thm:res-nf}, we have, for all $\vec{n} \in \mathcal{N}_4$
$$
Q^{(4)}_{\vec{n}} = \frac{f'(0)}4 \mathbf{1}_{|\Omega_{\vec{n}}|\leq \kappa} \mathbf{1}_{\vec{n} \in ( \mathbb{Z}^d_{\leq M})^4}.
$$
We are going to see that if $\vec{n} \in (\mathbb{Z}^d_{\leq M})^4$ satisfies $|\Omega_{\vec{n}}|\leq \kappa$ then $\{ n_1,n_3\} = \{n_2,n_4\}$. Note that it directly implies, by the Poincar\'e formula, that,  as expected, for all $u\in \mathbb{C}^{\mathbb{Z}^d_{\leq M}}$
 $$
Q^{(4)}(u) = \frac{f'(0)}2 \big(\sum_{|n|\leq M} |u_n|^2 \big)^2 - \frac{f'(0)}4 \sum_{|n|\leq M} |u_n|^4.
$$
So, let $\vec{n} \in (\mathbb{Z}^d_{\leq M})^4$ be such that $|\Omega_{\vec{n}}|\leq \kappa$. It means that
\begin{equation*}
\begin{split}
\kappa &\geq |\lambda_{n_1}^2 - \lambda_{n_2}^2 + \lambda_{n_3}^2 - \lambda_{n_4}^2| \\
&=  |g(n_1,n_1) - g(n_2,n_2)  + g(n_3,n_3)  - g(n_4,n_4) | \\
&= |g(n_1,n_1) - g(n_2,n_2)  + g(n_3,n_3)  - g(n_1+n_2-n_3,n_1+n_2-n_3) | \\
&= 2|g(n_1-n_2,n_3-n_2)|.
\end{split}
\end{equation*}
Now, assuming by contradiction that $n_1\neq n_2$ and $n_3\neq n_2$, since the torus is admissible we get that
$$
\varepsilon^{ c_\ast \nu} = \kappa \gtrsim |n_1-n_2|^{-\tau_\ast} |n_3-n_2|^{-\tau_\ast} \gtrsim M^{-2\tau_\ast} \gtrsim \varepsilon^{2\tau_\ast \nu}
$$
which is impossible, provided that $\varepsilon_0$ is small enough, by definition of $c_\ast = 2\tau_\ast+1$.
As a consequence, we have proven that $n_1 = n_2$ or $n_3= n_2$ , and so, using the zero momentum condition that $\{ n_1,n_3\} = \{n_2,n_4\}$.

\medskip

\noindent $\bullet$ \emph{ Step 2 : Dynamics.} We now prove stability of the low frequency actions and high-frequency superactions over very long times $T\leq \epsilon^{-r}$. 

\medskip

\noindent $\triangleright$ \emph{ Substep 2.1 : Setting of the bootstrap and low modes reduction.}   We consider  a maximal solution $u\in C^0([-T_-,T_+];h^s)$ of \eqref{eq:nls} such that $u(0) \in B_s(2\varepsilon) \cap \Pi_M^{-1}\Theta_\varepsilon$. We denote 
$$
\rho := \|u(0) \|_{h^s} \leq 2\varepsilon.
$$
Without loss of generality, we focus only on positive times.  As usual, we proceed by bootstrap: we consider $0<T<T_+$ such that
$$
T\leq  \varepsilon^{-r} \quad \mathrm{and} \quad \sup_{0\leq t \leq T} \| u(t) \|_{h^s} \leq 2^{s+2}  \rho
$$ 
and we aim at proving that
\begin{equation}
\label{eq:petit_bootstrap}
 \sup_{0\leq t \leq T} \| u(t) \|_{h^s} \leq 2^{s+1} \rho.
\end{equation}
It will prove that \eqref{eq:petit_bootstrap} holds for $T=\varepsilon^{-r}$. First, since $\Chigh$ contains only high modes, by applying Proposition \ref{prop:high_modes}, we know that, provided that $\varepsilon_0$ is small enough, we have, for $t\in [0,T]$
 $$
\| \Pi_{\Chigh} u(t) \|_{h^s}^2 \leq  \| \Pi_{\geq N/2} u(t) \|_{h^s}^2 \leq  2^{2s} \| u(0) \|_{h^s}^2 +  \rho^3 + T\Big( \frac{N}2 \Big)^{-\frac{\delta}{2}(s-d)} \rho^{3}+T \rho^{2r+1}.
 $$
 Then, using the lower bound on $M$ and the upper bound on $T$, we have
  $$
\| \Pi_{\Chigh} u(t) \|_{h^s}^2 \leq \rho^2 ( 2^{2s} +  \rho + C_{r,s}\rho \varepsilon^{-r} \varepsilon^{2r}+ \varepsilon^{-r} \rho^{2r-1})
 $$
 where $C_{r,s}$ is a constant depending only on $r$ and $s$ (and $d$). It follows that provided that $\varepsilon_0$ is small enough, we have
   \begin{equation}
   \label{eq:high_is_ok}
    \sup_{0\leq t \leq T} \| \Pi_{\Chigh} u(t) \|_{h^s}^2 \leq 2^{2s+1}  \rho^2.
   \end{equation}
To close the bootstrap, it remains to control $ \| \Pi_{\Clo} u \|_{h^s}^2$. To do it, we are going to prove that
\begin{equation}
\label{eq:preservation_des_basses_actions}
\sup_{0\leq t \leq T} \sum_{n\in \Clo} \langle n\rangle^{2s} \big| |u_n(t)|^2 - |u_n(0)|^2 \big|\leq  \varepsilon \rho^2.
\end{equation}
The conclusion \eqref{eq:petit_bootstrap} of the bootstrap, is just then a consequence of the triangular inequality, the Pythagorean identity and \eqref{eq:high_is_ok}. So, from now, we focus on proving \eqref{eq:preservation_des_basses_actions}. 

\medskip

\noindent $\triangleright$ \emph{ Substep 2.2 : Normal form and remainders.} First, we note that since $\nu\leq(2c_\ast)^{-1}$, we have
\begin{equation}
\label{eq:kappa_vs_eps1}
\varepsilon_\ast \gtrsim_{r,s} \kappa = \varepsilon^{c_\ast \nu} \geq \sqrt{\varepsilon},
\end{equation}
where $\epsilon_{\ast}$ was introduced in Corollary ~\ref{cor:flow}. It follows that, provided that $\varepsilon_0$ is small enough, we have $\varepsilon_\ast \gg \varepsilon$ and so (thanks to the bootstrap assumption) it makes sense to define
$$
z := \Phi_\chi^{-1}( u).
$$
As previously, without loss of generality (i.e. up to a standard approximation process), we assume that $u(0) \in h^{s+2}$ to get $u\in C^1([0,T];h^s)$ and so that $z \in  C^1([0,T];h^s)$. Since  $\Phi_\chi^{-1}$ is symplectic, using the chain rule, we get that
$$
i\partial_t  z = \nabla (H^{(\leq 4r)} \circ \Phi_\chi^1)(z)  + \mathcal{E}^{(1)}
$$
where
$$
\mathcal{E}^{(1)} := i\mathrm{d} \Phi_\chi^{-1}(u)(-i \nabla R^{(4r+2)}(u)).
$$
Then using that $\chi$ is supported on $\ell^1(\mathbb{Z}^d_{\leq M})$, we deduce that
$$
i\partial_t \Pi_M z = \nabla (H^{(\leq 4r)}\circ \Pi_M \circ \Phi_\chi^1)(\Pi_M z) + \Pi_M \mathcal{E}^{(1)} + \Pi_M \mathcal{E}^{(2)}
$$
where, setting $H^{>M} := H^{(\leq 4r)} - H^{(\leq 4r)} \circ \Pi_M$
$$
\mathcal{E}^{(2)} := \Pi_M \nabla(H^{>M} \circ \Phi_\chi^1)(z).
$$
By definition of $Q^{(\geq 6)}$ and $\Upsilon$ (see \eqref{eq:defHloflat}), this evolution equation rewrites
$$
i\partial_t \Pi_M z = \nabla (Z_2 + Q^{(4)} + Q^{(\geq 6)})(\Pi_M z) + \nabla \Upsilon(\Pi_M z)+\Pi_M \mathcal{E}^{(1)} + \Pi_M \mathcal{E}^{(2)}.
$$
Then projecting on $\mathbb{C}^{\Clo}$, by definition of $ H_{\lo}$ (see \eqref{eq:defHlo}), we get
$$
i\partial_t \Pi_{\Clo} z = \nabla H_{\lo}(\Pi_{\Clo} z) + 2f'(0) \| \Pi_{M} z \|_{\ell^2}^2 \Pi_{\Clo} z+ \mathcal{E}^{\mathrm{(tot)}}
$$
where
\begin{equation}
\label{eq:def_Etot}
\mathcal{E}^{\mathrm{(tot)}}  :=  \Pi_{\Clo}\nabla \Upsilon(\Pi_M z)+\Pi_{\Clo} \mathcal{E}^{(1)} + \Pi_{\Clo}\mathcal{E}^{(2)} + \Pi_{\Clo}\mathcal{E}^{(3)} 
\end{equation}
and, setting $H^{\mathrm{high}} = Q^{(\geq 6)} - Q^{(\geq 6)}\circ \Pi_{\Clo}$,
$$
 \mathcal{E}^{(3)} := \nabla H^{\mathrm{high}}( \Pi_M z).
$$
Observing that $H_{\lo}$ is invariant by gauge transform (because it commutes with the $\ell^2$ norm) and setting 
$$
w = e^{i 2f'(0) \int_0^t \| \Pi_{M} z(\tau) \|_{\ell^2}^2 \mathrm{d}\tau}  \Pi_{\Clo} z
$$
we have
\begin{equation*}
\label{eq:onw}
i\partial_t w =  \nabla H_{\lo}(w) + e^{i 2f'(0) \int_0^t \| \Pi_{M} z(\tau) \|_{\ell^2}^2 \mathrm{d}\tau}  \mathcal{E}^{\mathrm{(tot)}}.
\end{equation*}
Finally, setting
\begin{equation}
\label{eq:tensorise}
v_n(t) =  \left\{\begin{array}{lll} w_n(t) & \quad & \mathrm{if} \quad n\in \Clo, \\
0 & \quad & \mathrm{if} \quad |n|>M,\\
e^{-it \lambda_n^2 + i tf'(0)  |z_n(0)|^2} z_n(0) & \quad & \mathrm{else},
\end{array}\right.
\end{equation}
we have $v(0) = \Pi_M z(0)$ and as previously
\begin{equation}
\label{eq:onv}
i\partial_t v =  \nabla H_{\lo}(v) + e^{i 2f'(0) \int_0^t \| \Pi_{M} z(\tau) \|_{\ell^2}^2 \mathrm{d}\tau}  \mathcal{E}^{\mathrm{(tot)}}.
\end{equation}

\noindent $\triangleright$ \emph{ Substep 2.3 : Control of the remainder terms.}  To apply Theorem  \ref{thm:main_low} to $v$,  we have to prove that $\| \mathcal{E}^{\mathrm{(tot)}}(t) \|_{h^s} \leq \varepsilon^{3r} \rho $ for all $t\in [0,T]$.
First, using the bound \eqref{eq:d-varphi} on $\mathrm{d}\Phi_\chi^1$ and the bound \eqref{eq:rem_NLS} on $\nabla R^{(4r+2)}$ (and the bootstrap assumption), we have
\begin{equation}
\label{eq:estE1}
\| \mathcal{E}^{(1)}(t) \|_{h^s} \lesssim_{r,s}  \rho^{4r+1}.
\end{equation}
To estimate the other remainder terms, we have to estimate $\| z(t) \|_{h^s}$. To do it, we use that $\Phi_\chi^{-1}$ is close to the identity, to get
\begin{equation}
\label{eq:z_close_to_u}
\| z(t) - u(t) \|_{h^s} \lesssim_{r,s}  \rho^3 \kappa^{-1}
\end{equation}
and so, recalling that $\kappa \geq \sqrt{\varepsilon}$ (provided that $\varepsilon_0$ is small enough)
$$
\| z(t) \|_{h^s} \leq 2  \| u(t) \|_{h^s} \leq 2^{s+2} \rho.
$$
Then  using the estimate
\eqref{eq:R-res}
on $\nabla \Upsilon$, we have
\begin{equation}
\label{eq:estups}
\| \nabla \Upsilon(\Pi_M z(t))\|_{h^s} \lesssim_{r,s} \kappa^{-2r}  \rho \varepsilon^{4r} \lesssim_{r,s} \rho \varepsilon^{3r}.
\end{equation}
To control $\mathcal{E}^{(2)}$ and $\mathcal{E}^{(3)}$, we are going to apply the mismatch lemmata. First, for $\mathcal{E}^{(3)}$, using the definition \eqref{eq:defHloflat} of $Q^{(\geq 6)}$, we have
$$
\Pi_{\Clo} \mathcal{E}^{(3)} = \sum_{j = 3}^{2r} \Pi_{\Clo} \nabla (Q^{(2j)} - Q^{(2j)} \circ \Pi_{\Clo})(\Pi_M(z))
$$ 
and so, since $\Chigh$ contains only high modes (see \eqref{eq:m-high}), by Lemma \ref{lem:misII}, we have
\begin{equation}
\label{eq:estE3}
\begin{split}
\|\Pi_{\Clo} \mathcal{E}^{(3)} (t)\|_{h^{s}} &\lesssim_{r,s}   \sum_{j = 3}^{2r} \| Q^{(2j)}\|_{\infty} \big( \frac{N}2 \big)^{-\frac{\delta}{2}(s-d)}  \rho^{2j-1}\\
&\lesssim_{r,s}  \sum_{j = 3}^{2r}  \kappa^{-j+2} \varepsilon^{4r}  \rho^{2j-1} \lesssim_{r,s} \varepsilon^{3r}  \rho^{5}.
\end{split}
\end{equation}
The estimate on $\mathcal{E}^{(2)}$ is more technical. First, we set
$$
K^{(2j)} := \frac{f^{(j-1)}}{2j!} \| \cdot \|_{L^{2j}}^{2j} \in \mathcal{H}_{2j}(\mathbb{Z}^d)
$$
 and we note that the linear terms vanish in the definition of  $\mathcal{E}^{(2)}$, i.e.  
$$
\mathcal{E}^{(2)}  =  \Pi_M \nabla (K^{>M} \circ  \Phi_\chi^1)  (z)
$$
where $K^{>M} = K^{(4),>M}+ \cdots +  K^{(4r),>M}$ and $K^{(2j),>M} =  K^{(2j)} - K^{(2j)}\circ \Pi_M$. Then, doing the Taylor expansion of $\Phi_\chi^t$ in $t=0$, we get, as usual
$$
K^{>M} \circ  \Phi_\chi^1 = \sum_{q=0}^{2r} \frac{\mathrm{ad}_{\chi}^q K^{>M}}{q!} + \int_0^1 \frac{(1-\tau)^{2r}}{(2r)!}  \mathrm{ad}_{\chi}^{2r+1} K^{>M}\circ \Phi_\chi^\tau \mathrm{d}\tau =:   \sum_{q=0}^{2r} \frac{\mathrm{ad}_{\chi}^q K^{>M}}{q!} + \Gamma.
$$
From the expansions of $\chi$ and $K^{>M}$, we have
$$
\mathrm{ad}_{\chi}^q K^{>M} = \sum_{j=q+2}^{(q+1)(2r-1)+1} L^{q,j}
$$
where, by Lemma \ref{lem:bra},
$$
L^{q,j} := \sum_{ a_1 + \cdots + a_{q+1} = j+q } \mathrm{ad}_{\chi^{(2a_1)}} \cdots \mathrm{ad}_{\chi^{(2a_q)}} K^{(2a_{q+1}),>M}   \in \mathcal{H}_{2j}(\mathbb{Z}^d).
$$
Since $\| \chi^{(2j)} \|_{\infty} \lesssim_r \kappa^{-j+1}$ and $\| K^{(2j),>M}  \|_{\infty} \lesssim_j 1 \lesssim_j \kappa^{-j+2}$, we get by Lemma \ref{lem:bra} that
$$
\| L^{q,j} \|_{\infty} \lesssim_{r}  \kappa^{-j+2}
$$
and so by Corollary \ref{cor:vecfi} 
$$
\| \nabla \mathrm{ad}_{\chi}^q L^{q,j}  \|_{h^s} \lesssim_{r,s} \kappa^{-j+2}  \| \cdot\|_{h^s}^{2j-1}.
$$
As a consequence, if $j\geq 2r+1$, we have
$$
\| \nabla \mathrm{ad}_{\chi}^q L^{q,j} (z(t)) \|_{h^s} \lesssim_{r,s}  \rho \varepsilon^{3r}.
$$
Similarly, since $j\geq q+2$, using that $\Phi_\chi^\tau $ is close to the identity, we have that 
$$
\| \nabla  \Gamma(z(t)) \|_{h^s} \lesssim_{r,s}  \rho \varepsilon^{3r}.
$$
Recalling that $\Clo$ contains only low modes (see \eqref{eq:m-high}), it follows that
$$
\| \Pi_{\Clo} \mathcal{E}^{(2)}(t)\|_{h^s} \lesssim_{r,s}  \rho \varepsilon^{3r} + \sum_{q=0}^{2r-2} \sum_{j=q+2}^{2r} \| \Pi_{2N} \nabla L^{q,j}(z(t)) \|_{h^s}
$$
Now, we note that since the polynomials $\chi^{(2a)}$ are supported on $\ell^1(\mathbb{Z}^d_{\leq M})$ (i.e. they do not depend on $(u_a)_{|a|>M}$) and the polynomials $K^{(2a),>M}$ vanish on $\ell^1(\mathbb{Z}^d_{\leq M})$ (i.e. they do depend on $(u_a)_{|a|>M}$) then the polynomials $L^{q,j}$ vanish on $\ell^1(\mathbb{Z}^d_{\leq M})$, i.e. $L^{q,j}\circ \Pi_M =0$. As a consequence, since $M=(8r+1)N>4rN$, applying the mismatch Lemma \ref{lem:misI} and using the lower bound on $M$ (see \eqref{eq:boundonM}), we have
\begin{equation}
\label{eq:estE2}
\| \Pi_{\Clo} \mathcal{E}^{(2)}(t)\|_{h^s} \lesssim_{r,s}  \rho \varepsilon^{3r}   + \sum_{j=2}^r M^{-(s-d)}
 	\kappa^{-j+2}  \rho^{2j-1} 
 	 \lesssim_{r,s} \rho \varepsilon^{3r} .
\end{equation}
Putting together the estimates \eqref{eq:estE1}, \eqref{eq:estE2}, \eqref{eq:estE3}  and \eqref{eq:estups} on the remainder terms forming $ \mathcal{E}^{\mathrm{(tot)}} $ (defined by \eqref{eq:def_Etot}) and recalling that $r\geq 9$, we have proven that 
$$
\forall t\in [0,T], \quad \| \mathcal{E}^{\mathrm{(tot)}}(t) \|_{h^s} \leq \varepsilon^{3r}  \rho.
$$

\medskip

\noindent $\triangleright$ \emph{ Substep 2.4 : Application of Theorem \ref{thm:main_low} and conclusion.} Since $w$ solves the equation \eqref{eq:onv}, we have proven that
$$
\sup_{0\leq t \leq T}  \|  i\partial_t v - \nabla H_{\lo} (v)  \|_{h^s} \leq  \varepsilon^{3r} \rho.
$$
Now, we aim at checking that $v(0) \in \Theta_{\varepsilon}^\flat$. So we recall that by assumption 
$$
\Pi_M u(0) \in \Theta_\varepsilon = \big(\Pi_{ M } B_s(\varepsilon) \big)\cap  (\Phi_{\chi}^{-1})^{-1}  \Theta_{\varepsilon}^\flat
$$
and so, since $\chi$ is supported on $\ell^1(\mathbb{Z}^d_{\leq M})$,
$$
v(0) = \Pi_{M} z(0) = \Pi_M \Phi_\chi^{-1}  u(0) =  \Phi_\chi^{-1}\Pi_M  u(0) \in  \Theta_{\varepsilon}^\flat .
$$
 Applying Theorem \ref{thm:main_low} and by definition of $v$ (see \eqref{eq:tensorise}), we have that
$$
\sup_{0\leq t \leq T} \sum_{n\in \Clo} \langle n\rangle^{2s} \big| |z_{n}(t)|^2 - |z_{n}(0)|^2 \big| = \sup_{0\leq t \leq T} \sum_{ |n| \leq M} \langle n\rangle^{2s} \big| |v_{n}(t)|^2 - |v_{n}(0)|^2 \big|\leq \frac{\varepsilon}2 \rho^2.
$$
Finally by using that $\| u(t) - z(t) \|_{h^s} \lesssim_{r,s} \rho^{5/2}$ (see \eqref{eq:z_close_to_u}), we get, as expected, (provided that $\varepsilon_0$ is small enough)
\begin{equation}
\label{eq:a_citer_dans_lintro}
\sup_{0\leq t \leq T} \sum_{n\in \Clo} \langle n\rangle^{2s} \big| |u_n(t)|^2 - |u_n(0)|^2 \big|\leq  \varepsilon \rho^2.
\end{equation}

\medskip

\noindent $\bullet$ \emph{ Step 3 : Measure estimates.} First, we note that, by time reversibility of $\Phi_\chi^{t}$ (see the last property of Corollary \ref{cor:flow}),  provided that $\varepsilon$ is small enough, we have
$$
\Phi_\chi^{1} \Theta_\varepsilon^\flat \subset (\Phi_\chi^{-1})^{-1} \Theta_\varepsilon^\flat ,
$$
and therefore
\[
\Phi_{\chi}^{1}\Theta_{\epsilon}^{\flat}\cap \Pi_{M}B_{s}(\epsilon)\subset \Theta_{\epsilon}\,.
\]
Then, since $\Phi_\chi^{1} $ is symplectic, it is volume preserving and so
$$
\mathrm{meas} (\Phi_\chi^{1} \Theta_\varepsilon^\flat) = \mathrm{meas} ( \Theta_\varepsilon^\flat).
$$
Now, since $\Phi_\chi^{1}$ is close to the identity and $\kappa \geq \varepsilon$, we have (provided that $\varepsilon$ is small enough)
$$
\Phi_\chi^{1} \Theta_\varepsilon^\flat  \subset \Pi_M B_s(\varepsilon + \varepsilon^2).
$$
As a consequence, we have
\begin{equation}
\label{eq:measu1}
\mathrm{meas} (\Theta_\varepsilon) \geq \mathrm{meas} (\Theta_\varepsilon^{\flat}) -\meas \mathcal{A}_\varepsilon
\end{equation}
where
$$
\mathcal{A}_\varepsilon = \{ u\in \Pi_M h^s \ | \ \varepsilon \leq \| u\|_{h^s} \leq \varepsilon+ \varepsilon^2  \}.
$$
On the one hand, the measure estimate of Theorem \ref{thm:main_low}, ensures that
\begin{equation}
\label{eq:measu2}
\mathrm{meas}(\Theta_{\varepsilon}^\flat ) \geq (1 - \varepsilon^{\frac{1}{39}} )\mathrm{meas}(\Pi_{M} B_s(\varepsilon)).
\end{equation}
On the other hand, by homogeneity, we have 
\begin{equation}
\label{eq:measu3}
\frac{\mathrm{meas} \mathcal{A}_\varepsilon}{\meas \Pi_M B_s(\varepsilon)} = (1+ \varepsilon)^{2\sharp \mathbb{Z}^d_\leq M} - 1 \leq  (1+ \varepsilon)^{(9M)^d} -1 \leq e^{\varepsilon (9M)^d} - 1  \lesssim \sqrt{\varepsilon}
\end{equation}
the last estimate coming from the upper bound $M\leq \varepsilon^{-\nu}$ and assumption $\nu \leq 1/(2d)$. Putting together the three last estimates \eqref{eq:measu1}, \eqref{eq:measu2} and \eqref{eq:measu3}, we conclude that as expected
$$
\mathrm{meas}(\Theta_{\varepsilon} ) \geq (1 - \varepsilon^{\frac{1}{40}} )\mathrm{meas}(\Pi_{M} B_s(\varepsilon)).
$$

This completes the proof that Theorem \ref{thm:main_low} implies Theorem \ref{thm:main}.

\section{Re-centered polynomials}
\label{sec:small-div}

 The rest of the paper is devoted to the proof of Theorem \ref{thm:main_low}.

\subsection{Set-up}
\label{sec:su}
We start by defining some parameters within the framework of Theorem \ref{thm:main_low}. 

\begin{itemize}
\item In the rest of the paper we only need $s>0$ (the reason for this is that the reduced system is finite dimensional). 
\item We can choose (any) $\nu>0$ such that 
\begin{equation}
\label{eq:upsilon}
10^{8}c_{\ast} d r \nu \leq 10^{-4}\,.
\end{equation}

\item $\epsilon_{\ast}>0$ is a small parameter given by Corollary~\ref{cor:flow} which depends only on $r, s$ and on the torus $\T^{d}_{\mathscr{L}}$.
\item We fix $\epsilon\in(0,\epsilon_{\ast})$. It measures the amplitude of the initial data.
\item The frequency-truncation parameter $M$ depends on $\epsilon$:
\[
2\leq M\leq \epsilon^{-\nu}\,.
\] 
\item We set 
\[
\overline{r} = 10^{5}r\,.
\]
\end{itemize}
Then, $s, \epsilon$, $\nu$ and $M$ are defined once and for all.
\begin{remark} 
In the section \ref{sec:low-freq}, to prove the main Theorem \ref{thm:main} from Theorem \ref{thm:main_low} and in order to control the contribution of high frequencies, we required that 
\[
s\gtrsim r\nu^{-1} \gtrsim r^{2}\,,
\]
where the first (resp. second) inequality is a consequence of \eqref{eq:s} (resp. \eqref{eq:upsilon}). The implicit constants depend on $d$ and on the geometry of the torus $\T_{\mathscr{L}}^{d}$.
\end{remark}

\subsection{Modulation parameters and re-centered actions} In the subsequent analysis, the amplitudes of Fourier coefficients of a function $\phi$ in (say) $\Pi_{M}B_s(20\epsilon)$, serve as internal modulation parameters for the frequencies. We set
\[
\xi_n(\phi) = |\phi_n|^2\,,\quad n\in\Z_{M}^{d}\,. 
\]
We define 
\begin{equation}
 \label{eq:U-eps}
 \mathcal{U}_{s}(\epsilon) = \Big\{ \xi \in \R^{\Z_{M}^{d}}\quad |\quad \sum_{|n|\leq M}\langle n\rangle^{2s}|\xi_{n}| < 400\epsilon^{2}
 	\Big\}
	\,,
\end{equation}
 an open set of $\R^{\Z_{M}^{d}}$ endowed with the natural topology of $\ell_{1}^{2s}(\Z_{M}^{d};\R)$, made of \emph{internal parameters}. In particular, 
 \[
 \phi \in \Pi_{M}B_{s}(20\epsilon)\quad \iff\quad \xi(\phi) \in \mathcal{U}_{s}(\epsilon)\,.
 \]
To modulate the frequencies, we center the actions around the parameters $\xi=(\xi_{n})_{|n|\leq M}\in\mathcal{U}_{s}(\epsilon)$ in the Hamiltonian, and define 
 \[
 y_{n}(\xi,u) = |u_{n}|^{2} - \xi_{n}\,,\quad n\in\Z_{M}^{d}\,.
 \]
Then, the Hamiltonian's coefficients depend on the parameters $\xi$, and we shall introduce another Hamiltonian formalism more convenient for the analysis.
 \subsection{Parameter dependent polynomials}
\label{sub:par} 

In Definition \ref{def:hamb} we denoted $\mathcal{H}_{2q}(\Z_{M}^{d})$ a class of real homogenous polynomials of degree $2q$, supported on $\ell^{1}(\Z_{M}^{d})$. We set 
\[
\mathcal{H}_{\leq 2r}(\Z_{M}^{d}) 
	:= \bigoplus_{0\leq q\leq r }
	\mathcal{H}_{2q}(\Z_{M}^{d})\,,
	\quad \mathcal{H}(\Z_{M}^{d}) 
	:= \bigoplus_{0\leq q}
	\mathcal{H}_{2q}(\Z_{M}^{d})\,,
\]
where $\mathcal{H}_{0}(\Z_{M}^{d})$ is the class of constant real functions. 
\subsubsection{Extended class of real polynomials}
We now motivate the subsequent definitions and notations. In the expansions of polynomials in $\mathcal{H}(\Z_{M}^{d})$ we would like to keep track of the actions and formally see them as new variables: 
\[
u_{n}\,,\quad \overline{u_{n}}\,,\quad |u_{n}|^{2}\,.
\]
To a multi-index $\vec{n}$ in the class $\mathcal{N}$ defined in Definition \ref{def:multi} corresponds a sequence $\nb\in(\N^{\Z_{M}^{d}})^{3}$ where 
\[
\nb = (\kb, \lb,\mb)\,,\quad \text{with}\quad \kb = (k_{n})_{|n|\leq M}\,,\ \lb = (\ell_{n})_{|n|\leq M}\,,\ \mb = (m_{n})_{|n|\leq M}\,.
 \] 
For $n\in\Z_{M}^{d}$, the integer $m_{n}$ (resp. $k_{n}$, $\ell_{n}$) is the multiplicity of $|u_{n}|^{2}$ (resp. $u_{n}$, $\overline{u_{n}}$). Such a decomposition only make sense, and is unique, under a non-pairing condition between $u_{n}$ and $\overline{u_{n}}$.

\begin{definition}
\label{def:multi-bis} 
We define the class of multi-index with zero momentum condition 
\begin{equation}
\label{eq:monomial}
\Nb := \{\nb=(\kb,\lb,\mb) \ | \  \forall n\in\Z_{M}^{d}\,,\ k_n\ell_n=0\,\ \text{and}\  \sum_{n\in\Z_{M}^{d}}k_{n}-\ell_{n}
	=\sum_{n\in\Z_{M}^{d}}n(k_{n}-\ell_{n})=0\}\,.
\end{equation}
The degree of a multi-index $\nb$ is its total degree:
\[
\deg(\nb) := \sum_{n\in\Z_{M}^{d}}2m_{n}+ k_{n} + \ell_{n}\,.
\] 
\end{definition}
\begin{remark}
In the above definition, $k_{n}\ell_{n}=0$ is the non-pairing condition between $u_{n}$ and $\overline{u_{n}}$, $\sum_{|n|\leq M}k_{n}-\ell_{n}=0$ ensures that $u$ and $\overline{u}$ have the same multiplicity and  $\sum_{|n|\leq M}n(k_{n}-\ell_{n}) = 0$ is the zero momentum condition.
\end{remark}

For further use we also denote by $\widetilde{\Nb}$ the set of multi-indices satisfying only the non-pairing condition
\begin{equation}
\label{eq:N2q-t}
\widetilde{\mathcal{\Nb}} = \{\nb=(\kb,\lb,\mb)\quad|\quad k_n\ell_n=0\quad \forall n\in\Z_{M}^{d}\}\,.
\end{equation}
If $q\geq0$, 
\begin{equation}
\label{eq:N2q}
{\Nb}_{2q} = \{ \nb\in{\Nb}\quad |\quad \deg(\nb)= 2q\}\,,\quad {\Nb}_{\leq 2q} = \{ \nb\in{\Nb}\quad |\quad \deg(\nb) \leq 2q\}\,,
\end{equation}
with similar definitions for $\widetilde{\Nb}_{2q}$ and $\widetilde{\Nb}_{\leq 2q}$. 
\begin{notation}
We denote by $\nb_{-}$ the size of the smallest frequency that is not completely paired: 
\begin{equation}
\label{eq:n-}
\nb_{-}=\min\{\ |n|\quad|\quad n\in\Z_{M}^{d}\,,\ k_{n}+\ell_{n}\geq1\}\,.
\end{equation}
\end{notation}

The class of integrable monomials is denoted by
\begin{equation}
\label{eq:I}
\mathfrak{I} = \{\nb=(\kb,\lb,\mb)\in\Nb\quad|\quad \kb=\lb\equiv0\}\,.
\end{equation}
For convenience, given $n\in\Z_{M}^{d}$, we introduce special multi-indices  (that are not in $\Nb$, but in $\widetilde{\Nb}$):
\[
\mathbf{e}_{\mathfrak{m}}(n) := (\delta(n,\cdot)\,,0\,,0)\,,\quad \mathbf{e}_{\mathfrak{k}}(n):=(0,\delta(n, \cdot),0)\,,\quad \mathbf{e}_{\mathfrak{l}}(n) := (0,0,\delta(n,\cdot))\,,
\] 
where $\delta$ is the Kronecker symbol. They correspond to the following monomials:
\[
z_{\mathbf{e}_{\mathfrak{m}}(n)}(u,y) = y_{n}\,,
	\quad z_{\mathbf{e}_{\mathfrak{k}}(n)}(u,y) = u_{n}\,,
	\quad z_{\mathbf{e}_{\mathfrak{l}}(n)}(u,y) := \overline{u_{n}}\,.
\]

\subsubsection{The class of polynomials} We now define the corresponding class of extended polynomials $\mathcal{H}^{\sharp}(\Z_{M}^{d})$.

\begin{lemma} \label{lem:pouette}
For every $H\in\mathcal{H}_{\leq 2q}(\Z_{M}^{d})$ there exists a unique sequence of coefficients $(H_{\nb})\in\C^{\Nb_{\leq2q}}$, such that 
\[
H(u) = \sum_{\nb\in\Nb} H_{\nb} 
	\prod_{n\in\Z_{M}^{d}}
	u_{n}^{k_{n}}\overline{u_{n}}^{\ell_{n}}|u_{n}|^{2m_{n}}\,. 
\]
\end{lemma}
Given a multi-index $\nb=(\kb,\lb,\mb)\in\Nb_{\leq2q}$ we use the short-hand notation:
\[
	u^{\kb} 
	:= \prod_{n\in\Z_{M}^{d}}u_{n}^{k_{n}}\,,
	\quad 
	\overline{u}^{\lb} 
	:= \prod_{n\in\Z_{M}^{d}}\overline{u_{n}}^{\ell_{n}}
	\,,\quad 
	I(u)^{\mb}
	 := \prod_{n\in\Z_{M}^{d}}|u_{n}|^{2m_{n}}\,,
\]
so that, alternatively, 
\[
H(u) = \sum_{\nb\in\Nb} H_{\nb}u^{\kb}\overline{u}^{\lb}I(u)^{\mb}\,.
\]
We can now define the new formalism. 
\begin{definition}[Extended class of real polynomials]
\label{def:ext} Let $q\geq0$. The set $\mathcal{H}^{\sharp}_{\leq2q}(\Z_{M}^{d})$ of real polynomials is the set of functions defined on $\C^{\Z_{M}^{d}}\times\R^{\Z_{M}^{d}}$
\[
Q(u,y) = \sum_{\nb\in\Nb_{\leq2q}} Q_{\nb}z_{\nb}(u,y)\,,
\]
where the coefficients $(Q_{\nb})\in\C^{\Nb_{\leq2q}}$ satisfy the reality condition:
\[
Q_{\kb,\lb,\mb}= \overline{Q_{\lb,\kb,\mb}}\,,
\]
and
\[
z_{\nb}(u,y) :=  u^{\kb}\overline{u}^{\lb}y^{\mb} = \prod_{n\in\Z_{M}^{d}}u_n^{k_n}\overline{u_{n}}^{\ell_n}y_n^{m_n}\,\,.
\]

\end{definition}
\begin{lemma} By associating to $H\in\mathcal{H}_{\leq 2q}(\Z_{M}^{d})$ its coefficients $(H_{\nb})_{\nb\in\Nb_{\leq2q}}$ the following map is an isomorphism 
\[
\begin{array}{ccccc} A & \colon &\mathcal{H}_{\leq2q}(\Z_{M}^{d})&\longrightarrow & \mathcal{H}_{\leq2q}^{\sharp}(\Z_{M}^{d})  \\&&&&\\& & H &\longmapsto & \sum_{\nb\in\Nb} H_{\nb}u^{\kb}\overline{u}^{\lb}y^{\mb}  \,,\\\end{array}
\]
the inverse of $A$ being 
\[
\begin{array}{ccccc} A^{-1} & \colon &\mathcal{H}^{\sharp}_{\leq2q}(\Z_{M}^{d})&\longrightarrow & \mathcal{H}_{\leq2q}(\Z_{M}^{d})  \\&&&&\\& & Q(u,y) &\longmapsto & Q(u,I(u)) \,.\\\end{array}
\]
where $I(u)=(|u_{n}|^{2})_{n\in\Z_{M}^{d}}$.
\end{lemma}
In particular, we have 
\[
H(u) = A(H)(u,I(u))\quad \forall H\in\mathcal{H}_{\leq2q}(\Z_{M}^{d})\,.
\]

\subsubsection{Centering procedure}

With this formalism, we can interpret the centering of a polynomial $H$ with respect to a sequence of parameters $\xi$, as a translation of $\xi$ in the $y$-variable of $A(H)$, in the space $\mathcal{H}^{\sharp}(\Z_{M}^{d})$: 

\begin{definition} Let $\xi\in\mathcal{U}_{s}(\epsilon)$. We define the translation operator $\mathrm{T}_{\xi} : \mathcal{H}^{\sharp}_{\leq 2q}(\Z_{M}^{d}) \ \to\ \mathcal{H}^{\sharp}_{\leq2q}(\Z_{M}^{d})$ such that for all $H\in\mathcal{H}_{\leq2q}^{\sharp}(\Z_{M}^{d})$,
\begin{equation}
\label{eq:tlt}
\mathrm{T}_{\xi}Q(u,y) := Q(u,y-\xi)\,.
\end{equation}
\end{definition}
With these operators at hand, for $H\in\mathcal{H}_{\leq2q}(\Z_{M}^{d})$, centering $H$ around $\xi$ boils down to write\footnote{in the following formula we are just writing $\left|u_{n}\right|^{2}=\left|u_{n}\right|^{2}-\xi_{n}+\xi_{n}$ and then expanding around $\left|u_{n}\right|^{2}-\xi_{n}$ in the expansion of $H$ given by Lemma \ref{lem:pouette}.}
\begin{equation}
\label{eq:12}
H(u) = A(H)(u,I(u)) 
	= (A^{-1}\circ\mathrm{T}_{\xi}\circ\mathrm{T}_{-\xi}\circ A)(H)(u) 
	= (\mathrm{T}_{-\xi}\circ A)(H)(u,I(u)-\xi)\,.
\end{equation}

We stress out that the coefficients of $(\mathrm{T}_{-\xi}\circ A)(H)$ depend polynomially on $\xi$.  Nevertheless, this dependency is quite artificial: the polynomial function $H$ does not really depend on $\xi$.  However, our normal form procedure will naturally generate Hamiltonian functions really depending on $\xi$ (due to the small divisors). For this reason, we introduce a class of parameter-dependent real polynomials.

\begin{definition}[Class of parameter-dependent real polynomials]
\label{def:class}
We define
\begin{equation}
\label{eq:X}
X_{2q}(\epsilon):= C^{\infty}
	\Big(
	\mathcal{U}_{s}(\epsilon),\mathcal{H}_{2q}(\Z_{M}^{d})
	\Big)\,, 
\end{equation}
and 
\[
X_{\leq 2q}(\epsilon):= C^{\infty}
	\Big(
	\mathcal{U}_{s}(\epsilon),\mathcal{H}_{\leq 2q}(\Z_{M}^{d})
	\Big)\,, 	
	\quad 
	X(\epsilon):= C^{\infty}
	\Big(
	\mathcal{U}_{s}(\epsilon),\mathcal{H}(\Z_{M}^{d})
	\Big)\,.
\]
In addition, we define the class of integrable polynomials of degree less than or equal to $2q$ by 
\[
X_{2q,\mathrm{Int}}(\epsilon) := 
	\{
	H\in X_{2q}(\epsilon)\quad |\quad \forall \nb\in\Nb_{\leq2q}\,,\forall \xi\in\mathcal{U}_{s}(\epsilon)\,,\quad (H_{\nb}(\xi)\neq0 \implies \nb\in\mathfrak{J})
	\} \,.
\]
\end{definition}

We can now introduce the main notation.
\begin{notation}
\label{not}
Centreing $H\in X_{\leq2q}(\epsilon)$ around $\xi\in\mathcal{U}_{s}(\epsilon)$ gives a function of $\xi$ valued in $\mathcal{H}_{\leq 2q}^{\sharp}(\Z_{M}^{d})$, denoted $H[\cdot]$ and  defined by
\[
\begin{array}{ccccc} H[\cdot]
		&
		\colon 
		& \mathcal{U}_{s}(\epsilon)&
		\longrightarrow 
		& \mathcal{H}_{\leq 2q}^{\sharp}(\Z_{M}^{d})\\&&&&\\& 
		&\xi&\longmapsto 
		&  H[\xi]:=[\operatorname{T}_{-\xi}\circ A](H(\xi))
		\,.
		\\\end{array}
\]
According to \eqref{eq:12}, a parameter-dependent polynomial $H\in X_{\leq2q}(\epsilon)$, 
\[
H: \xi\in\mathcal{U}_{s}(\epsilon)\ \mapsto\ H(\xi;\cdot) \in \mathcal{H}_{\leq2q}(\Z_{M}^{d}) \,,
\]
once re-centered, can be represented by
\begin{equation}
\label{eq:H-cano}
H(\xi;u) = H[\xi](u,I(u)-\xi)\,.
\end{equation}
In the expanded form, we have 
\[
H(\xi;u) = \sum_{\nb\in\Nb_{\leq2q}}(H[\xi])_{\nb}z_{\nb}(u,I(u)-\xi)\,,
\]
where we recall that $z_{\nb}$ is defined by
\[
z_{\nb}(u,y) = u^{\kb}\overline{u}^{\lb}y^{\mb} = \prod_{n\in\Z_{M}^{d}}u_{n}^{k_{n}}\overline{u_{n}}^{\ell_{n}}y_{n}^{m_{n}} \,.
\]
\end{notation}

\subsubsection{Operations for re-centered polynomials}

\begin{lemma}[Poisson bracket with an action] If $H\in X_{\leq2q}(\epsilon)$ then for all $n\in\Z^{d}$ and $\xi\in\mathcal{U}_{s}(\epsilon)$,
\begin{equation}
\label{eq:brak-I}
\{I_{n},H(\xi;\cdot)\}(u) 
	= 2i\sum_{\nb\in\Nb}(k_{n}-\ell_{n})(H[\xi])_{\nb}u^{\kb}\overline{u}^{\lb}(I(u)-\xi)^{\mb}\,.
\end{equation}
\end{lemma}
By linearity, we deduce the following. 

\begin{corollary}
Given a sequence of real numbers $(\omega_{n})_{n\in\Z_{M}^{d}}$ (the frequencies), if  
\[
Z_{2}(\xi;u) := \sum_{n\in\Z_{M}^{d}}\omega_{n}(|u_{n}|^{2}-\xi_{n})\,,
\]
and $H\in X_{\leq2q}(\epsilon)$, then we have 
\[
\{Z_{2},H\}(u) = 2i\sum_{\nb\in\Nb}\Omega_{\nb}(\omega)(H[\xi])_{\nb}\ u^{\kb}\overline{u}^{\lb}(I(u)-\xi)^{\mb}\,,
\]
where $\Omega_{\nb}(\omega)$ is the resonance function:
\begin{equation}
\label{eq:rf1}
\Omega_{\nb}(\omega) := \sum_{n\in\Z_{M}^{d}}(k_{n}-\ell_{n})\omega_{n}\,.
\end{equation}
\end{corollary}

\section{Functional setting for re-centered polynomials}
\label{sec:main-nf-su}

\subsection{Parameters}
\label{sec:par}
We have remembered the set-up of Theorem \ref{thm:main_low} in subsection \ref{sec:su}. We now introduce the new parameters that play a role in the subsequent finite-dimensional analysis, to control the frequencies smaller than $M$ of the solution. Recall that the small parameter $\epsilon\in(0,\epsilon_{\ast})$ was fixed in paragraph~\eqref{sec:su}.
\begin{itemize}
\item $\eta>0$ is a large portion of $\epsilon$
\[
\eta = \epsilon^{1-\frac{1}{100}}\,.
\]
\item For $\alpha\in\N$, the frequency scale $N_{\alpha}=\epsilon^{-\frac{\alpha}{200s}}$ (which is not necessarily an integer), verifies 
\[
N_{0}=1\,,\quad (\frac{N_{\alpha+1}}{N_{\alpha}})^{s} = \epsilon^{-\frac{1}{200}}\,.
\]
\item We set $\beta = 100r$. In this way, $N_{\beta}$ (which will be our largest scale) verifies
\begin{equation}
\label{eq:s-beta}
\quad N_{\beta}^{-2s} = \epsilon^{\frac{\beta}{100}}  = \epsilon^{r}\,.
\end{equation}
\item $\tau$ (proportional to $s$) is a small parameter compared to $s$:
\[
\tau = \frac{s}{10\beta} = \frac{s}{10^{3}r}\,.
\]
We ensure that for all $\alpha\in\{0,\cdots,\beta\}$, 
\begin{equation}
\label{eq:eta-tau}
(\eta^{-1}\epsilon) N_{\alpha}^{5\tau} = \epsilon^{\frac{1}{100}-\frac{5\alpha\tau}{200s}} \leq \epsilon^{\frac{1}{100}-\frac{5\beta \tau}{200s}}  = \epsilon^{\frac{3}{500}}\leq \epsilon^{\frac{1}{200}}\,. 
\end{equation}
\item $\overline{r}$ (proportional to $r$) is large enough to ensure that 
\begin{equation}
\label{eq:r-bar}
\overline{r}\tau \geq 6s\,.
\end{equation}
In view of the definition of $\tau$, we can set 
\[\overline{r}=10^{5}r\,.\]
\item Recalling the definition \eqref{eq:N2q} of the set of multi-indices $\Nb_{\leq2q}$ observe that, at least when $q\leq\overline{r}^{2}$, 
\begin{equation}
\label{eq:counting}
\sharp\Nb_{\leq2q}\leq \sharp\widetilde{\Nb}_{\leq 2q}\leq \sharp(\Z_{M}^{d})^{2q}\ll \epsilon^{-\frac{1}{10^{6}}}\,.
\end{equation}
\item $\gamma$ (and $\gamma(\alpha)$) are numbers in $(0,1)$, depending on $\epsilon$, which give the large measure of the non-resonant parameter set, but they also appear in the small divisor estimates. The following choice is acceptable (although other choices are also possible):
\begin{equation}
\label{eq:gamma}
 \gamma =  \epsilon^{\frac{1}{30}} \,,\quad \gamma(\alpha) :=4^{\alpha}\gamma \,.
\end{equation}
\end{itemize}

We have the following key relationships between the parameters (if $\epsilon_{\ast}$ is small enough):
\begin{align}
\label{eq:par}
\gamma(\alpha)^{-1}(\epsilon^{-1}\eta)^{2}(\frac{N_{\alpha+1}}{N_{\alpha}})^{2s}
\leq \epsilon^{-\frac{1}{15}}\,.
\end{align}

\subsection{Frequency-scales} 

\begin{definition}[Weights for the coefficient bounds]
\label{def:weight}
For $\alpha\in\{0,\cdots,\beta\}$ and $n\in\Z_{M}^{d}$ set
\begin{align*}
\operatorname{D}(\alpha) &:= \eta^{-2-\frac{1}{5}}N_\alpha^{2s}\,,\\
\operatorname{C}_{n}(\alpha) &:= \eta^{-1}\min(|n|, N_\alpha)^{s}N_{\alpha}^\tau\,.
\end{align*}
Given a multi-index $\nb=(\kb,\lb,\mb)\in\Nb$ defined in \eqref{eq:monomial} we set
\begin{align*}
\mathrm{w}_{\nb}^{0}(\alpha) &:= N_\alpha^{-6s}\eta^{6}\prod_{n\in\Z_{M}^{d}}\operatorname{D}(\alpha)^{m_n}\operatorname{C}_{n}(\alpha)^{k_n+\ell_n}\,,\\
\intertext{and}
\mathrm{w}^1_{\nb}(\alpha) &:= N_\alpha^{-4s}\eta^4\prod_{n\in\Z_{M}^{d}}\operatorname{D}(\alpha)^{m_n}\operatorname{C}_{n}(\alpha)^{k_n+\ell_n}\,.
\end{align*}
\end{definition}
\begin{remark} \label{rq:comp_bou}
Note that our choice of weights $D(\alpha)$ differs slightly from the choice of Bourgain \cite{Bou00}, where (adapting the notation) we would have the following:
\[
D(\alpha) = \eta^{-2-\frac{1}{10}}N_{\alpha}^{2s}\min(|n|, N_{\alpha})^{s}\,.
\]
The above choice has some advantages when moving from the scale $\alpha$ to the scale $\alpha+1$ for the integrable quartic terms generated in the algorithm. In our case, however, we measure these quartic terms in the norm $Z$ (defined below) that does not depend on the scales. This allows us to have smaller weights. We emphasize that such a saving is crucial to obtain the bounds on the symplectic transformations (e.g. Lemma \ref{lem:d2F}) to measure the size of the non-resonant initial data set (a task not addressed in \cite{Bou00}). 
\end{remark}

\subsection{Norms and vector field estimate}

We now define suitable norms for re-centered parameter-dependent polynomials, which are written under the form \eqref{eq:H-cano}. Recall that we defined the set of parameters $\mathcal{U}_{s}(\epsilon)$ in \eqref{eq:U-eps}.
\begin{definition}[Norms for the re-centered polynomials]
\label{def:norm}
Given $\alpha\in\{0\,,\cdots\,,\beta\}$ and a parameter-dependent polynomial $H\in X(\epsilon)$ (defined in \eqref{eq:X}), we define two norms at scale $\alpha$:
$$
\|H\|_{\Ysup{\alpha}}
	:=\sup_{\xi\in\mathcal{U}_{s}(\epsilon)}\max_{\nb\in\Nb}\ \mathrm{w}_{\nb}^{0}(\alpha)^{-1}|(H[\xi])_{\nb}| ,
	$$
	$$
	\|H\|_{\Ylip{\alpha}}
	:= \sup_{\xi\in\mathcal{U}_{s}(\epsilon)}\max_{\nb\in\Nb}\max_{|n|\leq M} \ \mathrm{w}^{1}_{\nb}(\alpha)^{-1}|\partial_{\xi_{n}}(H[\xi])_{\nb}|
	$$
and
$$
	\|H\|_{Z^{\mathrm{sup}}}
	:=\sup_{\xi\in\mathcal{U}_{s}(\epsilon)}\max_{\nb\in\Nb}\ |(H[\xi])_{\nb}|,
	$$
	$$
	\|H\|_{Z^{\mathrm{lip}}}
	:=\sup_{\xi\in\mathcal{U}_{s}(\epsilon)}\ \max_{\nb\in\Nb}\max_{|n|\leq M} |\partial_{\xi_{n}}(H[\xi])_{\nb}|.
$$
\end{definition}
The norm $\Ysup{\alpha}$ is nothing but a weighted $\ell^\infty$-norm on the coefficients of the re-centered polynomial $H[\xi]$ (as introduced in Notation \ref{not}), with scale-dependent weights as in Definition \ref{def:weight}. We will extensively use that for all $\nb\in\Nb$ and $\xi\in\mathcal{U}_{s}(\epsilon)$,
\begin{align*}
|(H[\xi])_{\nb}|&\leq\|H\|_{\Ysup{\alpha}} \mathrm{w}_{\nb}^{0}(\alpha)\,,
\intertext{and, for all $|n|\leq M$,}
|\partial_{\xi_{n}}(H[\xi])_{\nb}|&\leq \|H\|_{\Ylip{\alpha}}\mathrm{w}_{\nb}^{1}(\alpha)\,.
\end{align*}
We say that a Hamiltonian {\it operates at frequency scale} $\alpha$ (see Theorem \ref{thm:nf-alpha}) when its $\Ysup{\alpha}$-norm is small compared to $\epsilon^{-\frac{1}{10^{3}}}$, say. 

\begin{remark}
For all $H$ of order greater than 6 (in the sense that for all $\xi\in\mathcal{U}_{s}(\epsilon)$, $H_{\nb}[\xi]\neq0$ only if $\deg(\nb)\geq6$), we have that
\begin{equation}
\label{eq:alpha=0}
\|H\|_{\Ysup{0}}
	\leq \|H\|_{Z^{\mathrm{sup}}}
	\,,\quad 
	\|H\|_{\Ylip{0}}
	\leq \|H\|_{Z^{\mathrm{lip}}}\,.
\end{equation}
The reason is that $N_{0}=1$.
\end{remark}
Given $\xi$ and $\alpha$  we introduce the annulus  
\begin{equation}
\label{eq:cal-V}
\mathcal{V}_{\alpha,s}(\epsilon,\xi) := \Big\{u\in\Pi_{M}B_{s}(\epsilon)\quad|\quad  \sum_{n\in\Z_{M}^{d}}\langle n\rangle^{2s}||u_n|^2 - \xi_n|\leq \epsilon^{2+\frac{1}{5}}N_\alpha^{-2s}\Big\}\,.
\end{equation}
We now estimate a multilinear quantity for functions in the above neighborhood. The set $\widetilde{\Nb}$ was defined in \eqref{eq:N2q-t}. 
\begin{lemma}[Weighted multilinear estimate]
\label{lem:wz}
For all $\alpha\in\{0,\cdots,\beta\}$, $\xi\in\mathcal{U}_{s}(\epsilon)$, $\nb\in\widetilde{\Nb}$ and  $u\in\mathcal{V}_{\alpha,s}(20\epsilon,\xi)$,
\begin{equation}
\label{eq:wz}
\mathrm{w}_{\nb}^{0}(\alpha)|z_{\nb}(u,I(u)-\xi)|  
	\leq N_{\alpha}^{-6s}\eta^{6}\,.
\end{equation}
\end{lemma}
The above Lemma can be understood as follows. When $u\in\mathcal{V}_{\alpha,s}(20\epsilon,\xi)$, which is an additional smallness condition on the centered actions $(|u_{n}|^{2}-\xi_{n})_{n\in\Z_{M}^{d}}$, we gain a factor $N_{\alpha}^{-6s}\eta^{6}$ in the multilinear estimate \eqref{eq:wz}. We emphasise that the zero momentum condition is not required, and we assume that $\nb$ is in the set $\widetilde{\Nb}$ defined in \eqref{eq:N2q-t}.

\begin{proof} 
Fix $u\in\mathcal{V}_{\alpha,s}(20\epsilon,\xi)$ and $\nb\in\widetilde{\Nb}$. According to the Definition \ref{def:weight} of the weights,
\begin{equation}
\label{eq:wz2}
\mathrm{w}_{\nb}^{0}(\alpha)z_{\nb}(u,I(u)-\xi) = N_{\alpha}^{-6s}\eta^{6}\prod_{n\in\Z_{M}^{d}}\operatorname{C}_{n}(\alpha)^{k_{n}+\ell_{n}}u_{n}^{k_{n}}\overline{u_{n}}^{\ell_{n}}\operatorname{D}(\alpha)^{m_{n}}(|u_{n}|^{2}-\xi_{n})^{m_{n}}\,.
\end{equation}
For all $n\in\Z_{M}^{d}$ and $\alpha\leq\beta$,
\[
\operatorname{C}_{n}(\alpha)|u_{n}|
	\leq \langle n\rangle^{s}N_{\alpha}^{\tau}\eta^{-1}|u_{n}|
	\leq \|\langle n\rangle^{s}u_{n}\|_{\ell^{\infty}}N_{\beta}^{\tau}\eta^{-1} 
	\leq 20\epsilon\eta^{-1} N_{\beta}^{\tau} 
	\leq 1\,,
\]
under the condition \eqref{eq:eta-tau}. In addition, 

\[
\operatorname{D}(\alpha)||u_{n}|^{2}-\xi_{n}|  \leq \eta^{-2-\frac{1}{5}}N_{\alpha}^{2s}\||u_{n}|^{2}-\xi_{n}\|_{\ell^{\infty}} 
	\lesssim (\epsilon\eta^{-1})^{2+\frac{1}{5}}\leq1\,.
\]
Multiplying over the contributions where $m_{n}+k_{n}+\ell_{n}\geq1$, for $n\in\Z_{M}^{d}$, gives the multilinear estimate \eqref{eq:wz}.
\end{proof}

We can now state the key vector field estimate, which motivates the definition of the norm $\Ysup{\alpha}$ and the normal form Theorem \ref{thm:nf-alpha}. 

\begin{proposition}[Vector field estimate at scale $\alpha$] \label{prop:vec-f-e}Given $\alpha\in\{0,\cdots,\beta\}$ and $Q\in X_{\leq\overline{r}^{2}}(\epsilon)$, we have that for all $\xi\in\mathcal{U}_{s}(\epsilon)$ and $u\in\mathcal{V}_{\alpha,s}(20\epsilon,\xi)$,
\begin{equation}
\label{eq:vec-alpha}
\|\nabla Q(\xi;u)\|_{h^{s}}\leq \|Q\|_{\Ysup{\alpha}} N_{\alpha}^{-4s}\epsilon^{4-\frac{1}{4}}\|u\|_{h^{s}}\,.
\end{equation}
\label{prop:vec-alpha}
\end{proposition}

\begin{proof} We fix $\xi\in\mathcal{U}_{s}(\epsilon)$ and, without loss of generality we suppose that $\|Q\|_{\Ysup{\alpha}}=1$.  For $n\in\Z_{M}^{d}$, we have
\[
(\nabla Q(\xi;u))_{n}  = 2\sum_{\nb\in\Nb_{\leq cr}}(Q[\xi])_{\nb}\partial_{\overline{u_{n}}}z_{\nb}(u,I(u)-\xi)\,.
\]
Observe that 
\begin{align*}
\partial_{\overline{u_{n}}}z_{\nb}(u,I(u)-\xi) 
	&= \ell_{n}\frac{z_{\nb}(u,I(u)-\xi)}{\overline{u_{n}}} + m_{n}\frac{u_{n}z_{\nb}(u,I(u)-\xi)}{|u_{n}|^{2}-\xi_{n}}\,.
\end{align*}
Hence, 
\begin{equation*}
\begin{split}
&\|\nabla Q(u)\|_{h^{s}}\\ 
&= 2\Big(\sum_{n\in\Z_{M}^{d}}\langle n\rangle^{2s}\Big|\sum_{\nb\in\Nb_{\leq 2q}}(Q[\xi])_{\nb}\Big(\ell_{n}\frac{z_{\nb}(u,I(u)-\xi)}{\overline{u_{n}}}+m_{n}\frac{u_{n}z_{\nb}(u,I(u)-\xi)}{|u_{n}|^{2}-\xi_{n}}\Big)\Big|^{2}\Big)^{\frac{1}{2}} \\
	&\leq \sharp(\Z_{M}^{d})\sharp(\Nb_{\leq2q})\max_{\substack{n\in\Z_{M}^{d}\\ \nb\in\Nb_{\leq 2q}}}
	\operatorname{w}_{\nb}^{0}(\alpha)\langle n\rangle^{s}
	\Big(|\ell_{n}\frac{z_{\nb}(u,I(u)-\xi)}{\overline{u_{n}}}|+|m_{n}\frac{u_{n}z_{\nb}(u,I(u)-\xi)}{u_{n}^{2}-\xi_{n}}|
	\Big) \\
	&\leq \epsilon^{-\frac{1}{10^{4}}}\max_{n\in\Z_{M}^{d}}
	\max_{\nb\in\Nb_{\leq2q}}  
	\operatorname{w}_{\nb}^{0}(\alpha)\langle n\rangle^{s}\Big(|\ell_{n}\frac{z_{\nb}(u,I(u)-\xi)}{\overline{u_{n}}}|+|m_{n}\frac{u_{n}z_{\nb}(u,I(u)-\xi)}{|u_n|^{2}-\xi_{n}}|\Big) \,.
\end{split}
\end{equation*}
where we used the counting estimate \eqref{eq:counting}. Then, for fixed $n\in\Z_{M}^{d}$ and $\nb\in\Nb_{\leq2q}$, suppose that $\ell_{n}\geq1$. We have from the zero momentum condition that there exists $j\in\Z_{M}^{d}\setminus\{n\}$ with 
\[
\max(k_{j},\ell_{j})\geq1\,,\quad |j|\gtrsim \frac{|n|}{r^{2}}\,.
\]
We only consider the situation when $k_{j}\geq1$ (the situation when $\ell_{j}\geq1$ is analogous), in which case we have
\begin{align*}
\operatorname{w}_{\nb}^{0}(\alpha)\langle n\rangle^{s}\ell_{n}&|\frac{z_{\nb}(u,I(u)-\xi)}{\overline{u_{n}}}| \\
	&\lesssim_{r,s}\operatorname{C}_{n}(\alpha)\operatorname{C}_{j}(\alpha)\langle j\rangle^{s}|u_{j}|\frac{\mathrm{w}_{\nb}^{0}(\alpha)}{\operatorname{C}_{n}(\alpha)\operatorname{C}_{j}(\alpha)}|\frac{z_{\nb}(u,I(u)-\xi)}{u_{j}\overline{u_{n}}}|\\
	&\lesssim_{r,s}\mathrm{C}_{n}(\alpha)\mathrm{C}_{j}(\alpha)\|u\|_{h^{s}}\mathrm{w}_{\nb'}^{0}(\alpha)|z_{\nb'}(u,I(u)-\xi)|\,,
\end{align*}
where 
\[
\nb' = \nb - \mathbf{e}_{\mathfrak{k}}(j) - \mathbf{e}_{\mathfrak{l}}(n) \in \widetilde{\Nb}\,.
\]
Note that
\[
\mathrm{w}_{\nb'}(\alpha)=\frac{\mathrm{w}_{\nb}^{0}(\alpha)}{\operatorname{C}_{n}(\alpha)\operatorname{C}_{j}(\alpha)}\,.
\]
It follows from Lemma \ref{lem:wz} that 
\[
\operatorname{w}_{\nb'}^{0}(\alpha)|z_{\nb'}(u,I(u)-\xi)| \leq N_{\alpha}^{-6s}\eta\,,
\]
and therefore
\begin{align*}
\operatorname{w}_{\nb}^{0}(\alpha)\langle n\rangle^{s}\ell_{n}|\frac{z_{\nb}(u,I(u)-\xi)}{\overline{u_{n}}}|  
	&\lesssim_{r,s} \mathrm{C}_{n}(\alpha)\mathrm{C}_{j}(\alpha) N_{\alpha}^{-6s}\eta^{6} \|u\|_{h^{s}} \\ 
	&\lesssim_{r,s}  N_{\alpha}^{2\tau} \eta^{4}N_{\alpha}^{-4s}\|u\|_{h^{s}} \\
	&\lesssim_{r,s}\eta^{4-\frac{1}{10}}N_{\alpha}^{-4s}\|u\|_{h^{s}}\,,
\end{align*}
according to the constraint \eqref{eq:eta-tau} on the parameters. Similarly, if $m_{n}\geq1$ then 
\begin{align*}
\mathrm{w}_{\nb}^{0}(\alpha)\langle n\rangle^{s}m_{n}|\frac{u_{n}z_{\nb}(u,I(u)-\xi)}{|u_{n}|^{2}-\xi_{n}}| 
	&\lesssim_{r,s}\mathrm{D}(\alpha)\langle n\rangle^{s}|u_{n}| \frac{\mathrm{w}_{\nb}^{0}(\alpha)}{\mathrm{D}(\alpha)}|\frac{z_{\nb}(u,I(u)-\xi)}{|u_n|^{2}-\xi_{n}}|  \\
	&\lesssim_{r,s}\mathrm{D}(\alpha)\|u\|_{h^{s}} \mathrm{w}_{\nb'}^{0}(\alpha)|z_{\nb'}(u,I(u)-\xi)|\,,
\end{align*}
where 
\[
\nb' = \nb - \mathbf{e}_{\mathfrak{m}}(n) \in \Nb\,,\quad \mathrm{w}_{\nb'}^{0}(\alpha)= \frac{\mathrm{w}_{\nb}^{0}(\alpha)}{\mathrm{D}(\alpha)}\,,
\]
and we deduce from Lemma \ref{lem:wz} that 
\[
\mathrm{w}_{\nb}^{0}(\alpha)\langle n\rangle^{s}m_{n}|\frac{u_{n}z_{\nb}(u,I(u)-\xi)}{|u_{n}|^{2}-\xi_{n}}| 
	\lesssim_{r,s}\eta^{4-\frac{1}{5}}N_{\alpha}^{-4s}\|u\|_{h^{s}}\,. 
\]
The desired inequality \eqref{eq:vec-alpha} follows when $\epsilon_{\ast}(s,r)$ is small enough.
\end{proof}

In the next lemma, we provide a fundamental example of a polynomial that operates at scale $\alpha$. Recall that $\nb_{-}=\min\{|n|\ |\ n\in\Z^{d}\,,\ k_{n}+\ell_{n}\geq1\}$ and $\mathfrak{I}$ denotes the set of integrable monomials, as defined in \eqref{eq:I}.

\begin{notation}
Given $\alpha\geq1$ we define  the set of multi-indices with an unpaired frequency smaller than $N_{\alpha}$:
\begin{equation}
\label{eq:scale}
\Lambda_\alpha : = \{\nb\in\Nb\setminus\mathfrak{I}\quad|\quad\nb_{-} <\ N_\alpha\}\,,
\end{equation}
Then, for $H\in X(\epsilon)$ and $\xi\in\mathcal{U}_{s}(\epsilon)$, we define $\Pi_{\Lambda_{\alpha}}H(\xi)\in\mathcal{H}(\Z_{M}^{d})$ by
\begin{equation}
\label{eq:pi-alpha}
\Pi_{\Lambda_{\alpha}}H(\xi;u) 
	:= \sum_{\nb\in\Lambda_{\alpha}}(H[\xi])_{\nb}z_{\nb}(u,I(u)-\xi)\,.
\end{equation}
Moreover, for $q\geq1$, 
\[
\Pi_{\deg=2q}Q(\xi;u) := \sum_{\nb\in\Nb\,:\, \deg{\nb}=2q}(Q[\xi])_{\nb} z_{\nb}(u,I(u)-\xi),\]\[ \Pi_{\deg\leq2q}Q(\xi;u) := \sum_{p=1}^{q}\Pi_{\deg=2p}Q(\xi)
\]
\end{notation}

\begin{lemma} For all $\alpha\geq0$, if $\nb\in\Nb\setminus\Lambda_{\alpha+1}$ with $\deg(\nb)\geq6$, then  
\[
\operatorname{w}_{\nb}^{0}(\alpha) \leq \operatorname{w}_{\nb}^{0}(\alpha+1)\,,\quad \operatorname{w}_{\nb}^{1}(\alpha) \leq \operatorname{w}_{\nb}^{1}(\alpha+1)\,.
\]
\label{lem:scale}
\end{lemma}

\begin{remark}\label{rem:scale} As a consequence, if $Q\in X_{\leq\overline{r}^{2}}(\epsilon)$ satisfies, for all $\xi\in\mathcal{U}_{s}(\epsilon)$
\[
\Pi_{\Lambda_{\alpha}}Q(\xi) = 0\,,\quad \Pi_{\deg\leq4}Q(\xi)=0\,,
\]
(as defined in \eqref{eq:pi-alpha}), then $Q$ operates at scale $\alpha$ in the sense that 
\[
\|Q\|_{\Ysup{\alpha}}\leq \|Q\|_{\Ysup{\alpha-1}} \leq \cdots \leq \|Q\|_{\Ysup{0}}\,.
\]
We deduce from Proposition \ref{prop:vec-alpha} the vector field estimate: for all $\xi\in\mathcal{U}_{s}(\epsilon)$ and $u\in\mathcal{V}_{\alpha,s}(20\epsilon,\xi)$, 
\[
\|\nabla Q(\xi;u)\|_{h^{s}} \leq \|Q\|_{\Ysup{0}}N_{\alpha}^{-4s}\epsilon^{3}\|u\|_{h^{s}}\,.
\]
\end{remark}

\begin{proof}
Recall form Definition \ref{def:weight} that the weight $\mathrm{w}_{\nb}^{0}(\alpha)$ is made of two parts: a prefactor $N_{\alpha}^{-6s}\eta^{6}$, which is a gain, and the product with the coefficients $\mathrm{C}_{n}(\alpha)$ and $\mathrm{D}(\alpha)$.

 When comparing the prefactor from scale $\alpha$ to scale $\alpha+1$ we have a loss:
 \[
 (\frac{N_{\alpha}}{N_{\alpha+1}})^{-6s}\,.
 \] 
 On the other hand, increasing the weights $\mathrm{C}_{n}(\alpha)$ to $\mathrm{C}_{n}(\alpha+1)$ gives a saving when $|n|\geq N_{\alpha+1}$: 
 \[
 \mathrm{C}_{n}(\alpha) = \eta^{-1}N_{\alpha}^{s+\tau} = (\frac{N_{\alpha}}{N_{\alpha+1}})^{s+\tau}\mathrm{C}_{n}(\alpha+1)\,.
 \]
Note also that increasing the weight $\mathrm{D}(\alpha)$ to $\mathrm{D}(\alpha+1)$ always gives a saving:
 \[
 \mathrm{D}(\alpha) 
 	= (\frac{N_{\alpha}}{N_{\alpha+1}})^{2s} \mathrm{D}(\alpha+1)\,.
 \]
We complete the proof of Lemma \ref{lem:scale} from the assumption that $\deg(\nb)\geq6$\,:
\begin{equation*}
\begin{split}
\mathrm{w}_{\nb}^{0}(\alpha) &= (\frac{N_{\alpha}}{N_{\alpha+1}})^{-6s}\prod_{|n|\leq M}(\frac{N_{\alpha}}{N_{\alpha+1}})^{2m_{n}s+(k_{n}+\ell_{n})(s+\tau)}\mathrm{w}_{\nb}^{0}(\alpha+1) \\
&\leq (\frac{N_{\alpha}}{N_{\alpha+1}})^{(\deg(\nb)-6)s}  \mathrm{w}_{\nb}^{0}(\alpha+1)\,.
\end{split}
\end{equation*}
We proceed similarly to prove that $\mathrm{w}_{\nb}^{1}(\alpha) 
	\leq \mathrm{w}_{\nb}^{1}(\alpha+1)$.
\end{proof}

To conclude this subsection, we state a lemma in the same spirit. It is helpful when dealing with monomials of large degree.  

\begin{lemma}[Weights for monomials with large degree]
\label{lem:deg}
For all $\nb\in\Nb$, if 
\begin{equation}
\label{eq:par-tau}
\tau\deg(\nb)\geq 6s\,,
\end{equation}
 then for all $\alpha\geq 0$, 
\[
\mathrm{w}_{\nb}^{0}(\alpha)\leq \mathrm{w}_{\nb}^{0}(\alpha+1)\,.
\]
\end{lemma}

\begin{remark}\label{rem:deg} As a consequence of Proposition \ref{prop:vec-alpha}, if a polynomial $R\in X_{\leq\overline{r}^{2}}(\epsilon)$ has order greater than $6s\tau^{-1}$ then for all $u\in\mathcal{V}_{\beta,s}(20\epsilon,\xi)$,
\[
\|\nabla R(u)\|_{h^{s}}\lesssim_{r} \|R\|_{Z^{\mathrm{sup}}} N_{\beta}^{-4s}\epsilon^{3}\|u\|_{h^{s}}\,.
\] 
\end{remark}

\begin{proof} 
We proceed as in the proof of Lemma \ref{lem:scale}. In this case we leverage on the the factor $N_{\alpha}^{\tau}$ in $\mathrm{C}_{n}(\alpha)$ to get the saving: for all $n\in\Z_{M}^{d}$, 
 \[
 \mathrm{C}_{n}(\alpha) = \eta^{-1}\min(|n|, N_{\alpha})^{s}N_{\alpha}^{\tau} \leq (\frac{N_{\alpha}}{N_{\alpha+1}})^{\tau}\mathrm{C}_{n}(\alpha+1)\,.
 \]
 This gives 
\[
 \mathrm{w}_{\nb}^{0}(\alpha) 
 	= (\frac{N_{\alpha+1}}{N_{\alpha}})^{6s - 2s\deg(\mb)-\tau(\deg(\kb)+\deg(\lb))} \mathrm{w}_{\nb}^{0}(\alpha+1)
	\leq  (\frac{N_{\alpha+1}}{N_{\alpha}})^{6s-\tau\deg(\nb)} \mathrm{w}_{\nb}^{0}(\alpha+1)\,,
\]
where we used that $N_{\alpha+1}\geq N_{\alpha}$ and $s\geq\tau$. Assuming \eqref{eq:par-tau} concludes. 
\end{proof}

\subsection{Poisson bracket and point-wise estimate on the coefficients. }
\label{sec:poiss}

We will prove in Proposition \ref{prop:bracket0} (resp. Proposition \ref{prop:bracket1} ) fundamental bounds on the $\Ysup{\alpha}$-norm ($\Ylip{\alpha}$-norm) of the Poisson bracket of  two polynomials. As a warm up we first state some explicit computations for the Poisson bracket of two re-centered polynomials.
\begin{lemma}[Poisson bracket between two monomials]
For fixed multi-indices $\nb=(\kb,\lb,\mb)$ and $\nb'=(\kb',\lb',\mb')$ one has
\begin{equation}
\begin{split}
\label{eq:expanded-braket}
&\{z_{\nb}(u,I(u)-\xi),z_{\nb'}(u,I(u)-\xi)\} \\
	=& 2i\sum_{n\in\Z_{M}^{d}}
	(k_n'\ell_n-k_n\ell_n')
	\frac{z_{\nb}(u,I(u)-\xi)z_{\nb'}(u,I(u)-\xi)}{|u_n|^2} \\
	&+ 
	(m_n(k_n'-\ell_n')
	+m_n'(\ell_n-k_n))
	\frac{z_{\nb}(u,I(u)-\xi)z_{\nb'}(u,I(u)-\xi)}{|u_{n}|^{2}-\xi_{n}}\,.
\end{split}
\end{equation}
\end{lemma}
In particular, we observe from \eqref{eq:expanded-braket} that a Poisson bracket either kills a pair $u_n,\overline{u_{n}}$ when $k_n'\ell_n-k_n\ell_n'\neq0$, and a term $|u_{n}|^{2}-\xi_{n}$ when $m_n(k_n'-\ell_n')+m_n'(\ell_n-k_n)\neq 0$. 
\begin{proof}
We compute explicitly the Poisson bracket using \eqref{eq:pois-brak}:

\begin{equation*} 
\begin{split}
&\{u^{\kb}\overline{u}^\lb(I(u)-\xi)^{\mb}
	\,,\,
	u^{\kb'}\overline{u}^{\lb'}(I(u)-\xi)^{\mb'}
	\} \\
	=& 2i\sum_{n\in\Z_{M}^{d}}
	\Big(
	\prod_{j\in\Z_{M}^{d}\setminus\{n\}}
	u_j^{k_j+k_j'}
	\overline{u_{j}}^{l_j+l_j'}
	(|u_{j}|^{2}-\xi_{j})^{m_j+m_j'}
	\Big)\\
	&\times \Big[
	(k_n'\ell_n-k_n\ell_n')
	\overline{u_{n}}^{\ell_n+\ell_n'-1}
	u_n^{k_n+k_n'-1}
	(|u_{n}|^{2}-\xi_{n})^{m_n+m_n'}
	 \\
& \quad  + (m_n(k_n'-\ell_n')+m_n'(\ell_n-k_n))
	u_n^{k_n+k_n'}
	\overline{u_{n}}^{\ell_n+\ell_n'}
	(|u_{n}^{2}-\xi_{n}|)^{m_n+m_n'-1}
	\Big]\,.
	\end{split}
\end{equation*}
Organizing the terms gives the identity \eqref{eq:expanded-braket}. 
\end{proof}
The next lemma follows from \eqref{eq:expanded-braket}, and gives a useful expression of the coefficients of the re-centered polynomial generated  by the Poisson bracket of two re-centered polynomials.
\begin{lemma}\label{lem:bra}
If $Q,H\in X(\epsilon)$, with $\min(\deg(Q),\deg(H))=:2q$, then $\{Q,H\}\in X(\epsilon)$ and we have for all $\nb''\in\Nb$ and $\xi\in\mathcal{U}_{s}(\epsilon)$,
\begin{equation}
	\begin{split}
\label{eq:pn}
(\{Q,H\}[\xi])_{\nb''} 
	&= 2i\prod_{|j|\leq M}
	\sum_{0\leq b_{j}\leq a_{j}\leq q} 
	\binom{a_{j}}{b_{j}}\xi_{j}^{a_{j}-b_{j}}
	\sum_{n\in\Z_{M}^{d}}
	\sum_{(\nb,\nb')\in \Nb^{2}}
	(Q[\xi])_{\nb}(H[\xi])_{\nb'}\\
	\Big(
	&\mathbf{1}_{E^{(1)}_{\nb'',\mathbf{a},\mathbf{b}}(n)}(\nb,\nb')(k_{n}'\ell_{n}-k_{n}\ell_{n}') \\
	+&\mathbf{1}_{E^{(2)}_{\nb'',\mathbf{a},\mathbf{b}}(n)}(\nb,\nb') (m_{n}(k_{n}'-\ell_{n'})+m_{n}'(\ell_{n}-k_{n}))
	\Big)\,,
	\end{split}
\end{equation}
where, given $\nb''=(m_{j}'',k_{j}'',\ell_{j}'')_{|j|\leq M}$, $\mathbf{a}=(a_{j})_{|j|\leq M},\ \mathbf{b}=(b_{j})_{|j|\leq M}$ with $0\leq b_{j}\leq a_{j}\leq q$, and $n\in\Z_{M}^{d}$, the sets of multi-indices
\[
E_{\nb'',\mathbf{a},\mathbf{b}}^{(1)}(n),\quad E_{\nb'',\mathbf{a},\mathbf{b}}^{(2)}(n) \subset \Nb^2
\]
are defined as follows:
\begin{enumerate}
\item  $(\nb,\nb')\in E_{\nb'',\mathbf{a},\mathbf{b}}^{(1)}(n)$ if and only if $k_n'\ell_n-k_n\ell_n'\neq 0$, with
\begin{align*}
	a_{n} &= \min(k_{n}+k_{n}'-1,\ell_{n}+\ell_{n}'-1)\,,
\intertext{and}
	m_n'' &= m_n+m_n'+b_{n}\,,\\
	k_n'' &= k_n+k_n'-1 - a_{n}\,,\\
	\ell_n'' &=\ell_n+\ell_n'-1 - a_{n}\,.
\intertext{For all $j\in \Z_{M}^{d}\setminus\{n\}$}
	a_{j} &= \min(k_{j}+k_{j}',\ell_{j}+\ell_{j}')\,, 
\intertext{and} 
	m_j'' &= m_j+m_j'+b_{j}\,,\\
	k_j'' &= k_j+k_j' - a_{j}\,, \\
	\ell_j'' &= \ell_j+\ell_j'- a_{j}\,, 
\end{align*}
\item $(\nb,\nb')\in E_{\nb'',\mathbf{a},\mathbf{b}}^{(2)}(n)$ if and only if  $m_{n}(k_{n}'-\ell_{n}')+m_{n}'(\ell_{n}-k_{n})\neq0$, with
\begin{align*}
	a_{n} &= \min(k_{n}+k_{n}', \ell_{n}+\ell_{n}')\,,
	\intertext{and}
	m_n'' &= m_n+m_n'-1+b_{n}\,,\\
	k_n'' &= k_n+k_n' -a_{n}\,, \\
	\ell_n''&= \ell_n+\ell_n'-a_{n}\,.
\intertext{For all $j\in\Z_{M}^{d}\setminus\{n\}$, we impose
the same conditions as for the above definition  of $E_{\nb'',\mathbf{a},\mathbf{b}}^{(1)}(n)$.}
\end{align*}
\end{enumerate}
\end{lemma}
Let us briefly comment on formula \eqref{eq:pn}. 
\begin{itemize}
\item For fixed $n\in\Z_{M}^{d}$ we collect in $E_{\nb''}^{(1)}(n)$ (resp. $E_{\nb''}^{(2)}(n)$) the contributions of the Poisson bracket between two monomials $(\nb,\nb')$ where a pair $u_n,\overline{u_{n}}$ goes away (resp. a term $y_n$ goes away). 
\item The parameters $\mathbf{a}=(a_{j})_{|j|\leq M}$ and $\mathbf{b}=(b_{j})_{|j|\leq M}$ come form the centreing procedure: $a_{j}$ represents the multiplicity of the pairing $|u_{j}|^{2}$, which is less than $\min(\deg(Q),\deg(H))$, while $b_{j}$ determines the number of terms $y_{j}$. More precisely, when we center $|u_j|^{2a_{j}}=(y_{j}+\xi_j)^{a_{j}}$ we obtain 
\[
|u_j|^{2a_{j}} = \sum_{b_{j}=0}^{a_{j}} \binom{a_{j}}{b_{j}}\xi_j^{a_{j}-b_{j}}y_j^{b_{j}}\,.
\]
\item Observe that the non-pairing condition is preserved: 
\[
k_{j}\ell_{j}=0\quad \text{and}\quad k_{j'}\ell_{j'}=0\quad \forall j\quad \implies \quad k_{j}''\ell_{j}''=0\quad \forall j\,.
\]
\end{itemize}

\begin{lemma} Given $\nb''\in\Nb_{\leq2q}$, we let 
\[
\mathfrak{C}_{\nb''}(q) := \{ (\nb,\nb')\in\Nb_{\leq 2q}^{2}\ |\ \{z_{\nb},z_{\nb'}\}\to z_{\nb''}\}
\]
be the set of multi-indices whose Poisson bracket contributes to $z_{\nb''}$. We have
\begin{equation}
\label{eq:card-c}
\sharp(\mathfrak{C}_{\nb''}(8\overline{r}^{2}))\leq \epsilon^{-\frac{1}{10^{4}}}\,.
\end{equation}
\end{lemma}
To prove this Lemma we count the total number of terms generated by the Poisson bracket between two monomials. The obtained estimate \eqref{eq:card-c} is very crude, but it is sufficient for our needs. 
\begin{proof} Suppose that $(\nb,\nb')\in\Nb_{\leq 2q}$. According to \eqref{eq:expanded-braket} we see that $\{z_{\nb},z_{\nb'}\}$ generates $2\sharp(\Z_{M}^{d})$ terms. Each of them gives at most $\sharp(\Z_{M}^{d})q^{2}$ new terms, up to some universal constant $C$. Summing over all the possible pairs $(\nb,\nb')$ concludes:
\[
\sharp(\mathfrak{C}_{\nb''})\leq C\sharp(\Nb_{\leq 2q})^{2}\sharp(\Z_{M}^{d})^{2}q^{2}\leq C\sharp(\Z_{M}^{d})^{4q+2}q^{2}\,.
\]
The inequality \eqref{eq:card-c} then follows from our choices of parameters.
\end{proof}

The fundamental bilinear estimate is presented in the following proposition. Given two Hamiltonian polynomials $Q$ and $H$, this estimate controls the $\Ysup{\alpha}$-norm of the Poisson bracket between $H$ and $Q$ by the product of the $\Ysup{\alpha}$-norms of $Q$ and $H$.
\begin{proposition}
\label{prop:bracket0} For all $Q,H \in X_{\leq\overline{r}^{2}}(\Z_{M}^{d})$ and for all $\xi\in\mathcal{U}_{s}(\epsilon)$,
\begin{equation} 
\label{eq:braket-alpha}
\|\{Q,H\}\|_{\Ysup{\alpha}} 	
	\leq \eta^{4-\frac{1}{4}}N_{\alpha}^{-4s}
	\|Q\|_{\Ysup{\alpha}}
	\|H\|_{\Ysup{\alpha}}\,.
\end{equation}
Moreover, for $Q\in X_{\leq\overline{r}^{2}}(\epsilon)$ and for all integrable $Z\in X_{4,\mathrm{Int}}(\epsilon)$, we have 
\begin{equation}
\label{eq:braket-alpha-Z}
\|\{Q,Z\}\|_{\Ysup{\alpha}} \leq \eta^{2+\frac{1}{5}-\frac{2}{10^{4}}}N_{\alpha}^{-2s}\|Q\|_{\Ysup{\alpha}}\|Z\|_{Z^{\mathrm{sup}}}\,.
\end{equation}
\end{proposition}
The saving in  $\eta$ and in $N_{\alpha}^{-s}$ is remarquable and it will compensate for the loss from the small divisor. In the second case \eqref{eq:braket-alpha-Z}, this saving is a consequence of the smallness of the centered actions $|u_n|^{2}-\xi_{n}$, which is encoded in the weight $\mathrm{D}(\alpha)$.  

\begin{proof}
Let us first prove \eqref{eq:braket-alpha}. We denote 
\[
P(u) = \{Q,H\} = \sum_{\nb''\in\Nb_{\leq2\overline{r}^{2}}}(P[\xi])_{\nb''}z_{\nb''}(u,I(u)-\xi)\,.
\]
By homogeneity we can assume that 
\[
\|Q\|_{\Ysup{\alpha}} = \|H\|_{\Ysup{\alpha}} =1\,.
\]
In particular, for all $(\nb,\nb')\in\Nb$ we have 
\[
|(Q[\xi])_{\nb}(H[\xi])_{\nb'}| \leq \mathrm{w}_{\nb}^{0}(\alpha)\mathrm{w}_{\nb'}^{0}(\alpha)\,,
\]
and our goal is to prove that for all $\nb''\in\Nb$ and $\xi$,
\[
|(P[\xi])_{\nb''}| 
	\leq 
	\eta^{4-\frac{1}{4}}
	N_{\alpha}^{-4s} 
	\mathrm{w}_{\nb''}^{0}(\alpha)\,.
\]
\paragraph{\it $\bullet$ Reduction to fixed monomials $(\nb,\nb')$.} Formula \eqref{eq:pn} (from Lemma \ref{lem:bra}) expresses $(P[\xi])_{\nb''}$ in terms of the coefficients $(
Q[\xi])_{\nb}$ and $(H[\xi])_{\nb'}$ of the re-centered polynomials associated with  $Q$ and $H$. We deduce from it that for all $\xi\in\mathcal{U}_{s}(\epsilon)$
\begin{multline*}
|(P[\xi])_{\nb'}| \lesssim_{ r}\sharp(\mathfrak{C}_{\nb''}(c\overline{r}^{2}))
	\sup_{\mathbf{a},\mathbf{b},n,\nb,\nb'}
	(\prod_{|j|\leq M}|\xi_{j}|^{a_{j}-b_{j}}) 
	\mathrm{w}_{\nb}^{0}(\alpha)\mathrm{w}_{\nb'}^{0}(\alpha) 
	 \\
	\Big[\mathbf{1}_{E_{\nb''}^{(1)}(n,\mathbf{a},\mathbf{b})}(\nb,\nb')
	|k_{n}'\ell_{n}-k_{n}\ell_{n}'|
	+\mathbf{1}_{E_{\nb''}^{(1)}(n,\mathbf{a},\mathbf{b})}(\nb,\nb')
	|m_{n}(k_{n}'-\ell_{n'})+m_{n'}(\ell_{n}-k_{n})|
	\Big]\,.
\end{multline*}
According to the (crude) counting estimate \eqref{eq:card-c}, we obtain
\begin{multline}
|(P[\xi])_{\nb'}|\lesssim_{r}\epsilon^{-\frac{1}{10^{4}}}
	\sup_{\mathbf{a},\mathbf{b},n,\nb,\nb'}
	(\prod_{|j|\leq M}|\xi_{j}|^{a_{j}-b_{j}})\mathrm{w}_{\nb}^{0}(\alpha)\mathrm{w}_{\nb'}^{0}(\alpha) \\
	\Big[\mathbf{1}_{E_{\nb''}^{(1)}(n,\mathbf{a},\mathbf{b})}(\nb,\nb')
	|k_{n}'\ell_{n}-k_{n}\ell_{n}'|
	+\mathbf{1}_{E_{\nb''}^{(1)}(n,\mathbf{a},\mathbf{b})}(\nb,\nb')
	|m_{n}(k_{n}'-\ell_{n'})+m_{n'}(\ell_{n}-k_{n})|
	\Big]\,.
	\label{eq:pogi} 
\end{multline}
Let us fix $n\in\Z_{M}^{d}$, $\mathbf{a}=(a_{j})_{|j|\leq M}, \mathbf{b} = (b_{j})_{|j|\leq M}$ with $0\leq b_{j}\leq a_{j}\leq \overline{r}$. Let $(\nb,\nb')$ be a pair of indices for which the Poisson bracket $\{z_{\nb},z_{\nb'}\}$ of the associated monomials contribute to $z_{\nb''}$. We organize the terms in two cases. 
\medskip

\noindent $\bullet$ {\it Case 1: A pair $(u_n,\overline{u_{n}})$ goes away.} This corresponds to the case when $(\nb,\nb')\in E_{\nb'',\mathbf{a},\mathbf{b}}^{(1)}(n)$. In such a case, we have
\begin{equation}
\label{eq:pr-w}
\mathrm{w}_{\nb}^{0}(\alpha)\mathrm{w}_{\nb'}^{0}(\alpha)  
	=N_\alpha^{-6s}\eta^{6}\operatorname{C}_{n}(\alpha)^2
	\Big(\prod_{|j|\leq M}\operatorname{C}_{j}(\alpha)^{2(a_{j}-b_{j})}
	\Big(\frac{\operatorname{C}_{j}(\alpha)^{2}}{\operatorname{D}(\alpha)}\Big)^{b_{j}} 
	\Big)
	\mathrm{w}_{\nb''}^{0}(\alpha)\,.
\end{equation}

\noindent $\triangleright$ The killing of the pair $u_n,\overline{u_{n}}$ is responsible for the loss $\operatorname{C}_{n}(\alpha)^2$, whereas the pre-factor $N_{\alpha}^{-6s}\eta^{6}$ is a saving that comes from this prefactor in one of the two weights $\mathrm{w}_{\nb}^{0}(\alpha)$, $\mathrm{w}_{\nb'}^{0}(\alpha)$ . 

\medskip

\noindent $\triangleright$ In addition, the term $\prod_{|j|\leq M}\operatorname{D}(\alpha)^{-b_{j}}(C_\alpha(n))^{2b_{j}}$ is a second saving--- which we will not exploit. It comes from the centering of $b_{j}$-terms $|u_n|^{2}$ that were changed into $|u_{n}|^{2}-\xi_{n}$. 

\medskip

\noindent $\triangleright$ Note also that $(a_{j}-b_{j})$ terms $|u_{j}|^{2}$ are also changed into $\xi_j$, and the extra weights (namely $\prod_{|j|\leq M}\operatorname{C}_{j}(\alpha)^{2(a_{j}-b_{j})}$) will be absorbed by $\prod_{|j|\leq M}\xi_{j}$. We deduce
\begin{align*}
\eqref{eq:pr-w}
	&\leq N_{\alpha}^{-6s}\eta^{6}\operatorname{C}_{n}(\alpha)^2
	\Big(
	\prod_{|j|\leq M}\operatorname{C}_{j}(\alpha)^{2(a_{j}-b_{j})}
	\Big)
	\mathrm{w}_{\nb''}^{0}(\alpha) \\
	&\leq N_\alpha^{-4s+2\tau}\eta^{4}
	\Big(\prod_{|j|\leq M}\operatorname{C}_{j}(\alpha)^{2(a_{j}-b_{j})}
	\Big)
	\mathrm{w}_{\nb''}^{0}(\alpha)\,.
\end{align*}
Hence, for these contributions to \eqref{eq:pogi} we have
\begin{align*}
\eqref{eq:pogi} 
	&\lesssim_{r} \epsilon^{-\frac{1}{10^{4}}}
	N_{\alpha}^{-4s+2\tau}\eta^{4} 
	\Big(\prod_{|j|\leq M}(\mathrm{C}_{j}(\alpha)^{2}|\xi_{j}|)^{a_{j}-b_{j}}
	\Big)
	\mathrm{w}_{\nb''}^{0}(\alpha) \\
	&\lesssim_{r} 
	\epsilon^{-\frac{1}{10^{4}}}
	N_{\alpha}^{-4s+2\tau}\eta^{4}
	\mathrm{w}_{\nb''}^{0}(\alpha)\,.
\end{align*}

\noindent$\bullet${\it Case 2: A term $|u_n|^{2}-\xi_{n}$ goes away}.  In this case $(\nb,\nb')\in E_{\nb'',\mathbf{a},\mathbf{b}}^{(2)}(n)$ and we simply modify the analysis from Case 1. It holds 
\begin{align}
\nonumber
\mathrm{w}_{\nb}^{0}(\alpha)\mathrm{w}_{\nb'}^{0}(\alpha) &=
	 N_{\alpha}^{-6s}\eta^{6}\operatorname{D}(\alpha)
	 \Big(
	 \prod_{|j|\leq M}\operatorname{C}_{j}(\alpha)^{2(a_{j}-b_{j})}\Big(\frac{\operatorname{C}_{j}(\alpha)^{2}}{\operatorname{D}(\alpha)}\Big)^{b_{j}}
	 \Big)
	 \mathrm{w}_{\nb''}^{0}(\alpha)\,.\\ 
	 &\leq N_{\alpha}^{-4s}\eta^{4-\frac{1}{5}}
	 \Big(
	 \prod_{|j|\leq M}C_{j}(\alpha)^{2(a_{j}-b_{j})}
	 \Big)
	 \mathrm{w}_{\nb''}^{0}(\alpha)\,.
\label{eq:pr-w2}
\end{align}
Hence, the contribution of these terms to \eqref{eq:pogi} is controlled by
\begin{align*}
\eqref{eq:pogi}
	&\lesssim_{r}\epsilon^{-\frac{1}{10^{4}}}
	\eta^{4-\frac{1}{5}}
	N_{\alpha}^{-4s}
	\Big(\prod_{|j|\leq M}(\mathrm{C}_{j}(\alpha)^{2}|\xi_{j}|)^{a_{j}-b_{j}}
	\Big)
	\mathrm{w}_{\nb''}^{0}(\alpha) \\
	&\lesssim_{r} 
	\epsilon^{-\frac{1}{10^{4}}}
	\eta^{4-\frac{1}{5}}
	N_{\alpha}^{-4s}
	\mathrm{w}_{\nb''}^{0}(\alpha)\,.
\end{align*}
In both case, by choosing $\epsilon_{\ast}(s,r)$ sufficiently small we conclude that 
\[
\eqref{eq:pogi}  \leq \eta^{4-\frac{1}{4}}N_{\alpha}^{-4s}
			\mathrm{w}_{\nb''}^{0}(\alpha)\,.
\]  
which is conclusive. 

We now turn to the proof of \eqref{eq:braket-alpha-Z}. Once again we may assume that
\[
	\|Q\|_{\Ysup{\alpha}}=1
	\,,\quad 
	\|Z\|_{Z^{\mathrm{sup}}}=1\,.
\]
By expanding $Z$ 
\[
Z = \sum_{|j_{1}|,|j_{2}|\leq M}
	(Z[\xi])_{\mathbf{e}_{\mathfrak{m}}(j_{1})+\mathbf{e}_{\mathfrak{m}}(j_{2})}
	(|u_{j_1}|^2-\xi_{j_{1}})(|u_{j_{2}}|^{2}-\xi_{j_{2}})\,,
\]
we obtain
\begin{align*}
&\{Q,Z\}(\xi;u) \\
	=& 2i\sum_{\substack{\nb\in\Nb_{\leq\overline{r}^{2}} \\ |j_{1}|,|j_{2}|\leq M}}
	(k_{j_{1}}-\ell_{j_{1}})
	(Q[\xi])_{\nb}
	(Z[\xi])_{\mathbf{e}_{\mathfrak{m}}(j_{1})+\mathbf{e}_{\mathfrak{m}}(j_{2})}
	(|u_{j_{2}}|^{2}-\xi_{j_{2}})
	z_{\nb}(u,I(u)-\xi) \\
	=:& \sum_{\nb''}(P[\xi])_{\nb'}z_{\nb''}(u,I(u)-\xi)\,,
\end{align*}
In particular, a term $|u_{j_{1}}|^{2}-\xi_{j_{1}}$ goes away while a term $|u_{j_{2}}|^{2}-\xi_{j_{2}}$ appears. This term provides a saving $\operatorname{D}(\alpha)$. We also stress out that there is no new pairing. 
Organizing the terms gives 
\begin{multline}
\label{eq:poiss-z}
(P[\xi])_{\nb''} \\
	= 2i\sum_{\substack{\nb\in\Nb_{\leq\overline{r}^{2}} \\ |j_{1}|,|j_{2}|\leq M}}
	\mathbf{1}
	_{E_{\nb'',0,0}^{(2)}(j_{1})}
	(
	\nb,
	\mathbf{e}_{\mathfrak{m}}(j_{1})+\mathbf{e}_{\mathfrak{m}}(j_{2})
	)
	(k_{j_{1}}-\ell_{j_{1}})
	(Q[\xi])_{\nb}
	(Z[\xi])_{\mathbf{e}_{\mathfrak{m}}(j_{1})+\mathbf{e}_{\mathfrak{m}}(j_{2})}\,.
\end{multline}
For $|j_{1}|,|j_{2}|\leq M$ and $\nb\in\Nb_{\leq\overline{r}^{2}}$, $\nb''\in\Nb$, we have that 
\begin{equation}
\label{eq:rel}
	(
	\nb,
	\mathbf{e}_{\mathfrak{m}}(j_{1})+\mathbf{e}_{\mathfrak{m}}(j_{2})
	)
	\in
	E_{\nb'',0,0}^{(2)}(j_{1})
	\quad 
	\implies
	\quad 
	\nb''=\nb+\mathbf{e}_{\mathfrak{m}}(j_{2})\,,
\end{equation}
and, in such a case, we have
\[
\mathrm{w}_{\nb}^{0}(\alpha) = \frac{1}{\operatorname{D}(\alpha)}\mathrm{w}_{\nb''}^{0}(\alpha) = N_{\alpha}^{-2s}\eta^{2+\frac{1}{5}}\mathrm{w}_{\nb''}^{0}(\alpha)\,.
\]
We deduce from the explicit formula \eqref{eq:poiss-z} that for all $\xi\in\mathcal{U}_{s}(\epsilon)$,
\begin{equation}
\label{eq:pr-z}
|(P[\xi])_{\nb'}|
	\lesssim_{r}\epsilon^{-\frac{1}{10^{4}}}
	\mathrm{w}_{\nb}^{0}(\alpha)
	\leq 
	N_{\alpha}^{-2s}\eta^{2+\frac{1}{5}-\frac{2}{10^{4}}}\mathrm{w}_{\nb''}^{0}(\alpha) \,.
\end{equation}
This completes the proof of Proposition \ref{prop:bracket0}.
\end{proof}

%
%

We now establish another bilinear estimate in the same spirit, but this time we measure the dependence on $\xi$ of the coefficients of re-centered polynomials. The output polynomial $P$ depends on the coefficients through  the coefficients of $Q$ and $H$ and the  terms $\xi$ coming from the centering of new actions generated by the Poisson bracket. 

\begin{proposition}\label{prop:bracket1}  For all $Q,H\in X_{\leq\overline{r}^{2}}(\epsilon)$ we have  
\begin{equation}
\label{eq:braket1-alpha}
\|\{Q,H\}\|_{\Ylip{\alpha}} 
	\leq
	\eta^{4-\frac{1}{4}}N_{\alpha}^{-4s}
	\Big(
	\|Q\|_{\Ysup{\alpha}}
	\|H\|_{\Ysup{\alpha}} 
	+
	\|Q\|_{\Ylip{\alpha}}
	\|H\|_{\Ysup{\alpha}} 
	+  
	\|Q\|_{\Ysup{\alpha}}
	\|H\|_{\Ylip{\alpha}}
	\Big)\,.
\end{equation}
Moreover, for all integrable quartic term $Z\in X_{4,\mathrm{Int}}(\epsilon)$,
\begin{equation}
\label{eq:braket1-alpha-Z}
\|\{Q,Z\}\|_{\Ylip{\alpha}} 
	\leq 
	\eta^{2+\frac{1}{5}-\frac{2}{10^{4}}}
	N_{\alpha}^{-2s}
	\Big(
	\|Q\|_{\Ylip{\alpha}}
	\|Z\|_{Z^{\mathrm{sup}}} 
	+
	\eta^{2}N_{\alpha}^{-2s}
	\|Q\|_{\Ysup{\alpha}}
	\|Z\|_{Z^{\mathrm{lip}}} 
	\Big)\,.
\end{equation}
\end{proposition}

Since the proof of this Proposition is relatively long (although close to the previous one) we postpone it to the Appendix~\ref{app-lp}.

\section{Small-divisor estimate for the modulated frequencies}\label{sec:smd}
Initially, the frequencies are
\[
\omega_n = \lambda_n^2=g(n,n)\,.
\] 
At each step of the iteration scheme, the centering procedure modulates the frequencies with internal parameters $\xi$ as introduced in the beginning of section \ref{sec:small-div}. These parameters should be viewed as the actions of the initial data $\phi\in\Pi_{M}h^{s}$ in the final coordinates. Recall that $\nb_{-}$, defined in \eqref{eq:n-}, is the size of the smallest index in $\supp\nb$ that is not fully paired. 
\medskip

We also recall that $\overline{r}=10^{5}r$.

\begin{definition}[Non-resonant set of initial data]
\label{def:non-res}Given a sequence of modulated frequencies $\omega(\xi)=(\omega_n(\xi))_{n\in\Z_{M}^{d}}$ 
and a multi-index $\nb\in\Nb_{\leq2\overline{r}}$, we denote the (modulated) resonance function by
\[
\Omega_\nb(\omega(\xi)) := (\kb-\lb)\cdot\omega(\xi) = \sum_{n\in\Z_{M}^{d}}(k_n- \ell_n)\omega_n(\xi)\,. 
\]
For $\gamma\in(0,1)$ and $\mathcal{E}\subset \Nb_{\leq2\overline{r}}$ the \emph{non-resonant set} of internal parameters restricted to $\mathcal{E}$ is
\begin{align}
\label{eq:gmdst}
\Xi_{\epsilon,\gamma}(\omega,\mathcal{E}) 
	&:= \Big\{\xi\in \R^{\Z_{M}^{d}}\ 
	|\ \min_{\nb\in\mathcal{E}}\ |\Omega_\nb(\xi)|
	\nb_-^{2s}
	>
	\gamma\epsilon^2\Big\}\,,
\intertext{and the \emph{non-resonant set} of functions is}
\label{eq:theta-set}
	\operatorname{U}_{\epsilon,\gamma}(\omega,\mathcal{E})&:=\Big\{\phi\in\Pi_{M}B_{s}(\epsilon)\ |\ \xi(\phi)\in\Xi_{\epsilon,\gamma}(\omega,\mathcal{E})\Big\}\,.
\end{align}
When $\mathcal{E}=\Nb_{\leq2\overline{r}}$ we just write $\Xi_{\epsilon,\gamma}(\omega)$ and $\operatorname{U}_{\epsilon,\gamma}(\omega)$.
\end{definition}
\begin{remark}
Note that the non-resonant set $\operatorname{U}_{\epsilon,\gamma}(\omega)$ is an open subset of $\Pi_{M}B_{s}(\epsilon)$. 
\end{remark}
By construction, if $\xi\in\Xi_{\epsilon,\gamma}(\omega)$ then for all $\nb\in\Nb_{\leq2\overline{r}}$ we have
\[
|\Omega_{\nb}(\xi)| >\gamma\epsilon^2\nb_{-}^{-2s}\,.
\]

\begin{lemma}\label{lem:meas-ball} Let $\mathcal{A}\subset{\Z}^{d}$ be a non-empty finite set, let $N=\sharp\mathcal{A}\geq1$ and $\rho>0$. We have 
\[
\meas(\Pi_{\mathcal{A}}B_{s}(\rho)) = \frac{\pi^{N}\rho^{2N}}{(N+1)!}\prod_{n\in\mathcal{A}}\langle n\rangle^{-2s}\,,
\]
where $\meas$ is the (canonical) Lebesgue measure on $\C^{\mathcal{A}}\cong (\R^{2})^{\mathcal{A}} $. 
\end{lemma}

\begin{proof}

Set
\[
\begin{array}{ccccc}F& \colon &\Pi_{\mathcal{A}}B_{s}(\rho)&\longrightarrow & B_{\C^{\mathcal{A}}}(0,1) \\&&&&\\& &(\phi)_{n}&\longmapsto & (x_{n})_{n} = \rho^{-1}(\langle n\rangle^{2s}\phi_{n})_{n} \,.\\\end{array}
\] 
A (linear) change of variables gives
\begin{equation*}
\begin{split}
\meas(\Pi_{\mathcal{A}}B_{s}(\rho)) 
	&= \rho^{2\sharp\mathcal{A}}
	\prod_{n\in\mathcal{A}}\langle n\rangle^{-2s}
	\meas(F(\Pi_{\mathcal{A}}B_{s}(\rho))) \\
	&= \rho^{2N}\prod_{n\in\mathcal{A}}\langle n\rangle^{-2s} \meas(B_{\C^{\mathcal{A}}}(0,1))\,.
\end{split}
\end{equation*}
We conclude by using the explicit formula for the volume of the unit Euclidian ball of $\R^{d}$: 
\[
\vol(B_{\R^{d}}(0,1)) = \frac{\pi^{\frac{d}{2}}}{\Gamma(\frac{d}{2}+1)}\,.
\]
This completes the proof of Lemma \ref{lem:meas-ball}.
\end{proof}
We are now ready to state the measure estimate of the non-resonant set of functions, for a given sequence of modulated frequencies $\omega(\xi):=(\omega_{n}(\xi))_{n\in\Z_{M}^{d}}$. 

\begin{proposition}[Measure estimate of the non-resonant set]
\label{prop:modulation}
Let $\omega \in C^{1}(\mathcal{U}_{s}(\epsilon), \R^{\Z_{M}^{d}})$ be a sequence of  modulated frequencies such that for all $k,n\in\Z_{M}^{d}$,
\begin{equation}
\label{eq:as-omega}
|\partial_{\xi_{k}}(\omega_n(\xi) -\xi_{n})|\leq \epsilon^{\frac{1}{2}}\,.
\end{equation}
Then, provided that $\epsilon_{\ast}(s,r)$ is small enough, we have that for all $\gamma\in(0,1)$,
\[
\meas(\operatorname{U}_{\epsilon,\gamma}(\omega)) \geq \meas(\Pi_{M}B_{s}(\epsilon))(1-\gamma \epsilon^{-\frac{1}{10^{4}}})\,.
\]
\end{proposition}
\begin{proof} 
Fix $\nb\in\Nb_{\leq2\overline{r}}$, $\gamma\in(0,1)$, $\epsilon\in(0,\epsilon_{\ast})$ and set 
\[
\mathcal{S}_{\nb,\gamma,\epsilon}(\omega) 
	:= \Big\{\phi\in\Pi_{M}B_s(\epsilon)\quad |\quad  |\Omega_\nb(\xi(\phi))|
	\leq \gamma\epsilon^2\nb_-^{-2s}\Big\}\,.
\]
Let us prove that there exists $C>0$ such that for all $\nb$ and $\gamma$
\begin{equation}
\label{eq:m1}
\meas(\mathcal{S}_{\nb,\gamma,\epsilon}(\omega))
	\lesssim\gamma(\sharp\Z_{M}^{d})\meas(\Pi_{M}B_s(\epsilon))\,.
\end{equation}
We shall conclude by observing that
\[
\Pi_{M}B_{s}(\epsilon)
	\setminus\operatorname{U}_{\gamma,\epsilon} 
	= \Big\{\phi\in\Pi_{M}B_s(\epsilon) \ | \ \underset{\nb\in\Nb_{\leq2\overline{r}}}{\min}\ |\Omega_\nb(\xi(\phi))|\nb_-^{2s}\leq\gamma\epsilon^2\Big\} 
	= \bigcup_{\nb\in\Nb_{\leq2\overline{r}}}\mathcal{S}_{\nb,\gamma,\epsilon}(\omega)\,, 
\]
and the result follows from \eqref{eq:m1}:
\begin{align*}
\meas(\Pi_{M}B_{s}(\epsilon)
	\setminus\operatorname{U}_{\gamma,\epsilon})
	&\lesssim \gamma
	\sharp(\Nb_{\leq2\overline{r}})\sharp(\Z_{M}^{d})
	\meas(\Pi_{M}B_s(\epsilon)) \\
	&\leq \gamma \epsilon^{-\frac{1}{10^{4}}}
	\meas(\Pi_{M}B_s(\epsilon))\,,
\end{align*}
where we used the bound \eqref{eq:counting}. It remains to prove \eqref{eq:m1}, for fixed $\nb\in\Nb_{\leq2\overline{r}}$. Take $n_0\in\supp(\nb)$ such that
\[
|n_0| = \nb_-\,,\quad |k_{n_{0}}-\ell_{n_{0}}|\geq1\,.
\]
To decide whether a function belongs to $\mathcal{S}_{\nb,\gamma,\epsilon}(\omega)$ or not, it suffices to move the its $n_{0}$-th action $\xi_{n_0}(\phi)$. Hence, we will freeze the other variables and, abusing notation, we write
\[
\omega(\xi) = \omega(\xi_{n_0})\,,
\]
and see the frequencies $\omega_{n}$ as functions 
\[
\xi_{n_{0}} \in I_{n_{0}}\ \mapsto\ \omega_{n}(\xi_{n_{0}})\in\R\,,\quad 
	I_{n_{0}}= [-\langle n_{0}\rangle^{-2s}400\epsilon^{2},\langle n_{0}\rangle^{-2s} 400 \epsilon^{2}]\subset\R\,. 
\]
Once the variables $(\xi_{n})_{n\neq n_{0}}$ are fixed, we deduce from the mean value theorem and from the assumption \eqref{eq:as-omega} that for all $(\xi_{n_{0}} , \xi_{n_{0}}') \in I_{n_{0}}^{2}$,
\begin{align*}
|\omega_{n_{0}}(\xi_{n_{0}}) - \omega_{n_{0}}(\xi_{n_{0}}') - (\xi_{n_{0}}-\xi_{n_{0}}')| 
	&\leq \epsilon^{\frac{1}{2}}|\xi_{n_{0}}-\xi_{n_{0}}'|\,,
\intertext{and, for all $n\neq n_{0}$,}
|\omega_{n}(\xi_{n_{0}}) - \omega_{n}(\xi_{n_{0}}')| 
	&\leq \epsilon^{\frac{1}{2}}|\xi_{n_{0}}-\xi_{n_{0}}'|\,.
\end{align*}
We deduce from the triangle inequality that
\begin{align*}
|\Omega_\nb(\xi_{n_0})-\Omega_\nb(\xi_{n_0}')| \geq& |k_{n_0}-\ell_{n_0}||\omega_{n_0}(\xi_{n_0})-\omega_{n_0}(\xi_{n_0}')| 
	\\ &- \sum_{n\in\Z_{M}^{d}\setminus\{n_0\}}|k_n-\ell_n| |\omega_{n}(\xi_{n_0})-\omega_{n}(\xi_{n_0}')| \\
	\geq& |\xi_{n_0}-\xi_{n_0}'|
	(|k_{n_0}-\ell_{n_0}| -2\overline{r}\epsilon^\frac{1}{2})\\
	\geq& \frac{1}{2}|\xi_{n_0}-\xi_{n_0}'|\,.
\end{align*}
It follows that for all $\phi,\phi'\in\Pi_{M}B_{s}(\epsilon)$ such that $\phi_n = \phi_n'$ when $n\neq n_{0}$,
\begin{equation}
\label{eq:dif}
\phi, \phi'\in E_{\nb,\gamma,\epsilon}(\omega)\quad \implies\quad \frac{1}{2}|\xi_{n_0}(\phi)-\xi_{n_0}(\phi')|\leq |\Omega_{\nb}(\xi(\phi))-\Omega_{\nb}(\xi(\phi'))|\leq2\gamma\nb_-^{-2s}\epsilon^2\,.
\end{equation}
We deduce from this that 
\begin{equation}
\label{eq:dif2}
\sup_{z\in \Pi_{\Z_{M}^{d}\setminus\{n_{0}\}}B_{s}(\epsilon)}\ 
	\int_\C\mathbf{1}_{\mathcal{S}_{\nb,\gamma,\epsilon}(\omega)}(z ; \phi_{n_{0}}) d\phi_{n_0} 
	\leq 4\pi\gamma\nb_-^{-2s}\epsilon^2\,.
\end{equation}
Indeed, using polar coordinates for $\phi_{n_0}$ and using that the condition to belong to $\mathcal{S}_{\nb,\gamma,\epsilon}$ only depends on $|\phi_{n_{0}}|^{2}$,  we obtain, for fixed $z\in\Pi_{\Z_{M}^{d}\setminus\{n_{0}\}}$, that
\begin{equation*}
\begin{split}
\int_{\C}\mathbf{1}_{\mathcal{S}_{\nb,\gamma,\epsilon}(\omega)}(z; |\phi_{n_{0}}|^{2})d\phi_{n_0} 
	&= \int_{(0,2\pi)}d\theta\int_{(0,1)}\mathbf{1}_{\mathcal{S}_{\nb,\gamma,\epsilon}(\omega)}(z ; \rho^{2})\rho d\rho \\
	&= \pi\int_{(0,1)}\mathbf{1}_{\mathcal{S}_{\nb,\gamma,\epsilon}(\omega)}(z ; \rho^{2})d\rho^{2}\,.
\end{split}
\end{equation*}
Then, to prove \eqref{eq:dif2} we just observe from \eqref{eq:dif} that  if $\rho,\rho'\in(0,1)^2$ are such that
\[
\mathbf{1}_{\mathcal{S}_{\nb,\gamma,\epsilon}(\omega)}(z;\rho^2) =\mathbf{1}_{\mathcal{S}_{\nb,\gamma,\epsilon}(\omega)}(z ;(\rho')^{2})=1\,,
\]
 then 
\[
|\rho^{2} -( \rho')^{2}| \leq 2\gamma\nb_-^{-2s}\epsilon^2\,.
\]
To conclude the proof of \eqref{eq:m1}, we freeze the variables $z=(\phi_{n})_{n\in\Z_{M}^{d}\setminus\{n_{0}\}}$  in the measure estimate by applying Fubini's Theorem:
\begin{equation*}
\begin{split}
\meas(\mathcal{S}_{\nb,\gamma,\epsilon}(\omega)) 
	&= \int_{\Pi_{M}B_{s}(\epsilon)}
	\mathbf{1}_{\mathcal{S}_{\nb,\gamma,\epsilon}(\omega)}(\phi)d\phi \\
	&\leq \int_{\Pi_{\Z_{M}^{d}\setminus\{n_{0}\}}B_{s}(\epsilon)}
	\Big(\int_{\C}
	\mathbf{1}_{\mathcal{S}_{\nb,\gamma,\epsilon}(\omega)}(z;\phi_{n_{0}})d\phi_{n_{0}}\Big)dz\,.
\end{split}
\end{equation*}
According to \eqref{eq:dif2}, we deduce that
\[
\meas(\mathcal{S}_{\nb,\gamma,\epsilon}(\omega)) \leq 4\pi\gamma\nb_{-}^{-2s}\epsilon^{2}\meas(\Pi_{\Z_{M}^{d}\setminus\{n_{0}\}}B_{s}(\epsilon))\,.
\]
Lemma \ref{lem:meas-ball} gives 
\[
\meas(\Pi_{\Z_{M}^{d}\setminus\{n_{0}\}}B_{s}(\epsilon))
	 =\frac{\sharp\Z_{M}^{d}}{\pi \epsilon^{2}}\langle n_{0}\rangle^{2s}
	 \meas(\Pi_{M}B_{s}(\epsilon))\,,
\]
and we obtain 
\[
\meas(\mathcal{S}_{\nb,\gamma,\epsilon}(\omega)) \leq 2\gamma (\sharp\Z_{M}^{d})\nb^{-2s}\langle n_{0}\rangle^{2s} \meas(\Pi_{M}B_{s}(\epsilon)) \lesssim\gamma(\sharp\Z_{M}^{d})\meas(\Pi_{M}B_{s}(\epsilon))\,.
\]

This completes the proof of Proposition \ref{prop:modulation}.
\end{proof}


\section{Normal form with internal parameters}
\label{sec:main-nf}


We state in this section our main normal form theorem from which we will deduce Theorem \ref{thm:main_low} in section \ref{sec:pr-thm}. This normal form is based on an induction on the frequency scales $N_{\alpha}$, for $\alpha\in\{0,\cdots,\beta\}$.

\subsection{Preparation} \label{sec:init}
Let $H_{\lo}\in\mathcal{H}_{\leq4r}(\Z_{M}^{d})$ be a real polynomial of degree less than or equal to $4r$ as in the assumptions of Theorem \ref{thm:main_low} where, without loss of generality \footnote{Since we are in the small data regime the sign of $f'(0)$ has indeed no role to play}, we supposed that $f'(0)$ is suitably chosen such that  
\begin{equation}
\label{eq:mars}
 H_{\lo}(u) = \frac{1}{2}\sum_{n\in\Z_{M}^{d}} (\lambda_{n}^{2}+ \frac{1}{2}|u_{n}|^{2}) |u_{n}|^{2} + \sum_{j=3}^{2r}P^{(2j)}(u)\,.
\end{equation}
 For $j\in\{3,\cdots,2r\}$, the real homogeneous polynomial  $P^{(2j)}\in\mathcal{H}_{2j}(\Z_{M}^{d})$ satisfies
\begin{equation}
 \label{eq:froid}
 \|P^{(2j)}\|_{\infty} \leq \epsilon^{-4c_{\ast}\nu r}\,.
\end{equation}
Moreover, it is $\epsilon^{c_{\ast}\nu}$-resonant in the sense that 
\[
P^{(2j)}(u) 
	=
	\sum_{\vec{n}\in\mathcal{N}_{2j}} 
	P^{(2j)}_{\vec{n}}u_{\vec{n}}\,,
	\quad P^{(2j)}_{\vec{n}}\neq0\quad  
	\implies 
	|\sum_{i=1}^{2j}(-1)^{i}\lambda_{n_{i}}^{2}| \leq \epsilon^{c_{\ast}\nu}\,.
\]
We re-center $H_{\lo}$ around parameters $\xi$, and we introduce a decomposition into quadratic terms, quartic integrable terms and terms of order greater than or equal to 6. 

\begin{lemma} Let $H_{\lo}$ be as in \eqref{eq:mars}. There exist two polynomials $Q\in X_{\leq 4r}(\epsilon)$, $Z_{4}\in X_{4,\mathrm{Int}}(\epsilon)$, and some modulated frequencies, such that the following decomposition holds for all $\xi\in\mathcal{U}_{s}(\epsilon)$: 
\begin{equation}
\label{eq:H-loo}
H_{\lo}(u) = \frac{1}{2}\sum_{n\in\Z_{M}^{d}}\omega_{n}(\xi)(|u_n|^{2}-\xi_{n} 
)	
	+ Z_{4}(\xi;u)
	+ Q(\xi;u) 
\end{equation}
where
$$
 \Pi_{\deg=2}Q =  \Pi_{\deg=4}Q=0.
$$
We have the bounds
\[
\|Q\|_{Z^{\mathrm{sup}}}
	+\|Q\|_{Z^{\mathrm{lip}}}
	+
	\|Z_{4}\|_{Z^{\mathrm{sup}}}
	+\|Z_{4}\|_{Z^{\mathrm{lip}}}
	\lesssim \epsilon^{-4c_{\ast}\nu r} \leq \epsilon^{-\frac{1}{10^{4}}}\,.
\]
The modulated frequencies are
\[
\omega_{n}(\xi) = \lambda_{n}^{2} +\xi_{n}+h_{n}(\xi)\,,
\]
where the last term $h_{n}(\xi)$ is a polynomial function of $(\xi_{\ell})_{\ell\in\Z_{M}^{d}}$. For all $\xi\in\mathcal{U}_{s}(\epsilon)$ and $|k|,|n|\leq M$,
\[
|\partial_{\xi_{k}}(\omega_{n}(\xi)-\xi_{n})| \lesssim \epsilon^{2-\frac{1}{10^{4}}} \leq \epsilon\,.
\]
\end{lemma}
\begin{proof}
We set
\[
\quad Z_{4}[\xi]:= \Pi_{\deg=4}H_{\lo}[\xi]\,,\quad Q[\xi] := H_{\lo}[\xi] - Z_{4}[\xi] - Z_{2}[\xi]\,,
\] 
where
\[
Z_{2}[\xi](u,y) = \frac{1}{2}\sum_{n\in\Z_{M}^{d}}\omega_{n}(\xi)y_{n}\]
and we define
\[
Q(\xi;u):= Q[\xi](u,I(u)-\xi)\,,\quad Z_{4}(\xi;u):= Z_{4}[\xi](u,I(u)-\xi)\,. 
\]
Then, the decomposition \eqref{eq:H-loo} follows from the identity 
\[
H_{\lo}(u) = H_{\lo}[\xi](u,I(u)-\xi)\,.
\]
Since $P^{(2j)}$ is $\epsilon^{-c_{\ast}\nu}$-resonant, we have from Step 1 in paragraph~\ref{sub:a-b} that the quartic terms are integrable, which proves that $Z_{4}\in X_{4,\mathrm{Int}}(\epsilon)$. 
\medskip

The bounds on the polynomial norms and on the modulated frequencies follow from the estimates~\eqref{eq:froid} on the coefficients of $P^{(2j)}$ (before the re-centering around $\xi$), and from the fact that $\xi$ is a small parameter. 
\end{proof}

\subsection{Statement}

In the previous section, we described the structure of the initial Hamiltonian $H_{\lo}\in X_{\leq2r}(\epsilon)$, written in \eqref{eq:H-loo}.  We are now ready to formulate the main normal form theorem, which is based on an induction on frequency scales. The current iteration step is $\alpha\in\{0,\cdots,\beta\}$.  Before giving the results we recall that the formalism was described in section \ref{sec:small-div} and in particular in notation \ref{not}. We also recall that the non-resonant set of parameters~$\Xi(\omega)$ was defined in Definition~\ref{def:non-res}.
\medskip

For $\alpha\in\{0,\cdots,\beta\}$ and $c>0$ a universal constant to be determined later, we introduce
\begin{align}
\label{eq:eps-alpha}
\epsilon_{\alpha}&:=10\epsilon- c\alpha r\epsilon^{\frac{3}{2}}\,,
\intertext{Given the parameter~$\gamma(\alpha)=4^{\alpha}\gamma$, the modulated frequencies~$\omega^{(\alpha)}(\xi)$ defined iteratively in Theorem~\ref{thm:nf-alpha}, and the non-resonant set defined in \eqref{eq:theta-set}, we fix}
\label{eq:tw}
\Xi_{\alpha}&:= \Xi_{\epsilon,\gamma(\alpha)}(\omega^{(\alpha)},\Lambda_{\alpha+1})\,.
\end{align}
We are now ready to state the main normal form Theorem. 
\begin{theorem}[Normal form at scale $\alpha$]\label{thm:nf-alpha}  

For all $\alpha\in\{1,\cdots,\beta\}$ and $\xi\in\mathcal{U}_{s}(\epsilon)$, there exists a symplectic transformation $\tau_{\alpha,\xi}:\ h^{s}(\Z_{M}^{d})\,\to\,h^{s}(\Z_{M}^{d})$ such that
\[
H_{\lo}\circ\tau_{\alpha,\xi}(u) 
	=Z_{2}^{(\alpha)}(\xi;u)
	+ Z_{4}^{(\alpha)}(\xi;u) 
	+ Q^{(\alpha)}(\xi;u) 
	+ R^{(\alpha)}(\xi;u)\,,
\]
where the different terms are organized as follows: 

\begin{enumerate}
\item  $Z_{2}^{(\alpha)}$ collects the quadratic terms: for all $\xi\in\mathcal{U}_{s}(\epsilon)$
\[
Z_{2}^{(\alpha)}(\xi;u)=\frac{1}{2}\sum_{n\in\Z_{M}^{d}}\omega_{n}^{(\alpha)}(\xi)(|u_{n}|^{2}-\xi_{n})\,,
\]
where $\omega^{(\alpha)}_{n}\in C^{1}(\mathcal{U}_{s}(\epsilon),\R)$ for all $|n|\leq M$, with for all $|k|\leq M$ and $\xi\in\mathcal{U}_{s}(\epsilon)$,
\begin{align}
\label{eq:omega}
 |\partial_{\xi_{k}}(
 \omega_{n}^{(\alpha)}(\xi) - \xi_n )|
 	&\leq (\alpha+1)\epsilon\,.
\intertext{In addition, the modulated frequencies are stable in the sense that (when $\alpha\geq1$)}
\label{eq:st-fr}
 |\omega_{n}^{(\alpha)}(\xi)-\omega_{n}^{(\alpha-1)}(\xi)| &\leq N_{\alpha-1}^{-4s}\epsilon^{3}\,.
\end{align}
\item $Z_{4}^{(\alpha)}\in X_{4,\mathrm{Int}}(\epsilon)$ collects quartic terms, and
\begin{equation}
\label{eq:boZ}
\max(\|Z_4^{(\alpha)}\|_{Z^{\mathrm{sup}}}\,,\,\|Z_4^{(\alpha)}\|_{Z^{\mathrm{lip}}})
	\lesssim_{\alpha} \epsilon^{-\frac{1}{10^{4}}}\,.
\end{equation}
\item $Q^{(\alpha)}\in X_{\leq2\overline{r}}(\epsilon)$ operates at scale $\alpha$:   
\begin{equation}
\label{eq:boQ}
\|Q^{(\alpha)}\|_{\Ysup{\alpha}}
 	\lesssim_{\alpha} \epsilon^{-\frac{1}{10^{4}}}\,,\quad \|Q^{(\alpha)}\|_{\Ylip{\alpha}}\lesssim_{\alpha}\epsilon^{-\frac{1}{15}}\,.
\end{equation}
It is at least of order 6, and $(cr)^{c\alpha r}\epsilon^{c_{\ast}\nu}$-resonant for some universal constant 
$c>0$: 
\begin{equation}
\label{eq:c-4}
(Q^{(\alpha)}[\xi])_{\nb}\neq0
	\quad 
	\implies
	\quad \deg(\nb)\geq6\,,\quad
	 |\sum_{n\in\Z_{M}^{d}}(k_{n}-\ell_{n})\lambda_{n}^{2}|\leq (cr)^{c\alpha r}\epsilon^{c_{\ast}\nu}\,.
\end{equation}
\item $R^{(\alpha)}$ is a smooth map on $\mathcal{U}_{s}(\epsilon)\times h^{s}(\Z_{M}^{d})$, and for all $\xi\in\Xi_{\alpha-1}$ and $u\in \mathcal{V}_{\alpha,s}(\epsilon_{\alpha},\xi)$,
\begin{equation}
\label{eq:boR}
\|\nabla_{u}R^{(\alpha)}(\xi;u)\|_{h^{s}}
	\lesssim_{\alpha,r} \epsilon^{2r+1}\|u\|_{h^{s}}
	\,.
\end{equation}
\item The symplectic transformation satisfies, for all $\xi\in\mathcal{U}_{s}(\epsilon)$ and $u\in\mathcal{V}_{\alpha,s}(\epsilon_{\alpha},\xi)$,
\begin{align}
\label{eq:tau}
\|\tau_{\alpha,\xi}(u)-u\|_{h^{s}}
	&\lesssim_{\alpha,r}\epsilon^{\frac{3}{4}}\|u\|_{h^{s}}\,,\\
\label{eq:d-tau}
\|\mathrm{d}\tau_{\alpha,\xi}(u)-\Id\|_{h^{s}\to h^{s}}
	&\lesssim_{\alpha,r}\epsilon^{\frac{3}{4}}\,.
\intertext
{Moreover, for all $\phi\in\Pi_{M}B_{s}(\epsilon_{\alpha})$ and $u\in\mathcal{V}_{\alpha,s}(\epsilon_{\alpha},\xi(\phi))$,}
\label{eq:d-phi-xi}
	\|
	\mathrm{d}_{\phi}\tau_{\alpha,\xi(\phi)}(u) 
	 \|_{h^{s}\to h^{s}}
	 &\lesssim_{\alpha,r}\epsilon^{1-\frac{2}{5}}\,.
\end{align}
\end{enumerate}
\end{theorem}

Here are some comments on the above Theorem. 

\medskip

\noindent {$\triangleright$\it Modulated frequencies:} In the iterative procedure of increasing the scales $\alpha$, the successive polynomials we generate have new integrable terms that modulate the frequencies when we re-centre the actions around $\xi$. The condition \eqref{eq:omega} guarantees that the modulated frequencies satisfy the Lipschitz assumption \eqref{eq:as-omega} from Proposition \ref{prop:modulation} (which is essential to prove the small-divisor bounds for parameters in a set with an asymptotically full measure). 

\medskip

\noindent{$\triangleright$ \it Terms of type $Q$:} These terms (once re-centered) are of order greater than 6. They are either integrable, or associated to multi-indices $\nb\in\Lambda_{\alpha}$ (the unpaired indices are larger than $N_{\alpha}$) (even if this property is not explicitly stated in Theorem \ref{thm:nf-alpha}). These polynomials operate at scale $\alpha$ (in the sense that their $\Ysup{\alpha}$-norm is controlled by a small negative power of $\epsilon$). If $N_{\alpha}$ is very large, we see from the vector field estimate of Proposition \ref{prop:vec-alpha} that they generate stable dynamics.

\medskip

\noindent{$\triangleright$ \it Non-resonance condition and terms of order four:} The condition \eqref{eq:c-4} ensures that every term of order less than or equal to 4 are integrable. Indeed, we proved in subsection \ref{sub:a-b} Step 1 that $(cr)^{c\alpha r}\epsilon^{c_{\ast}\nu}$-resonant terms of order 4 are integrable.

\medskip

\noindent {$\triangleright$\it Modulation parameters:} The non-resonance condition for  the modulation parameters $\xi\in\Xi_{\alpha-1}$ ensures that $R^{(\alpha)}$ is a remainder order $r$ (this is encoded by \eqref{eq:boR}). Recall that this condition is dictated by the modulated frequencies $\omega^{(\alpha-1)}$ at scale $\alpha-1$, and allows us to use a Birkhoff normal form to remove  $\Pi_{\Lambda_{\alpha}}Q^{(\alpha-1)}$ (up to a term of order $r$), which do not commute with the actions with a frequency lower than $N_{\alpha}$ (and which a priori do not operate at scale $\alpha$). We will prove in Lemma \ref{lem:cut} that the stability condition \eqref{eq:st-fr} and the definition \eqref{eq:tw} of $\Xi_{\alpha}$ imply that 
\[
\Xi_{\alpha} \subset \Xi_{\alpha-1}\,.
\]
This ensures that non-resonant parameters $\xi$ at scale $\alpha$ (in the sense of \eqref{eq:tw}) are also non-resonant on the previous scale $\alpha-1$.

\medskip

\noindent{$\triangleright$\it The symplectic transformations:} At scale $\alpha$, the symplectic transformations (which are well defined on the finite dimensional space $\Pi_{M}h^{s}$) depend on the modulation parameter $\xi\in\mathcal{U}_{s}(\epsilon)$. In addition, $\tau_{\alpha,\xi}$ is close to the identity on the annulus  $\mathcal{V}_{\alpha,s}(\epsilon_{\alpha},\xi)$. 

\begin{remark} We observe from \eqref{eq:tau} that 
\[
\tau_{\alpha,\xi}(\mathcal{V}_{\alpha,s}(\epsilon_{\alpha},\xi))\subset \mathcal{V}_{0,s}(10\epsilon,\xi)\,,
\]
and from \eqref{eq:d-tau} that for all $u\in\mathcal{V}_{\alpha,s}(\epsilon_{\alpha},\xi)$,
\[
\|\mathrm{d}\tau_{\alpha,\xi}(u)\|_{h^{s}\to h^{s}}+\|(\mathrm{d}\tau_{\alpha,\xi}(u))^{-1}\|_{h^{s}\to h^{s}} \lesssim1\,.
\]
In particular, if $u:= \tau_{\alpha,\xi}^{-1}(v) \in \mathcal{V}_{0,s}(10\epsilon,\xi)$, then, for fixed $\xi$, 
\begin{equation}
\label{eq:d--1}
\|\mathrm{d}\tau_{\alpha,\xi}^{-1}(u)\|_{h^{s}\to h^{s}} 
	= 
	\|(\mathrm{d}\tau_{\alpha,\xi}(v))^{-1}\|_{h^{s}\to h^{s}} 
	\lesssim1\,.
\end{equation}
\end{remark}

%


The rest of this section is devoted to the proof of the normal form Theorem \ref{thm:nf-alpha}. It consists in an induction on the frequency scales, and we first give an outline of the proof in the paragraph below.

\subsection{Proof: Induction on the frequency scales}.
\label{sec:proof-nf}
We fix a scale $\alpha\in\{1\,,\,\cdots\,,\,\beta-1\}$, and we assume that Theorem \ref{thm:nf-alpha} is proved up to scale $\alpha$. In this paragraph we briefly introduce the strategy we use to get from the scale $\alpha$ to the scale $\alpha+1$. For this purpose, we are going to remove non-integrable terms from $Q^{(\alpha)}$ which are associated to multi-indices in $\Lambda_{\alpha+1}$ (defined in \eqref{eq:scale}) 
\[
\nb\in\Lambda_{\alpha+1}\,,
\]
up to some new terms that are remainders if $\xi\in\Xi_{\alpha}$ is a non-resonant parameter. Using Lemma \ref{lem:scale} and remark \ref{rem:scale}, this will imply that the new Hamiltonian 
$Q^{(\alpha+1)}$ operates at scale $\alpha+1$. 

\medskip

The algorithm to get from step $\alpha$ to step $\alpha+1$ requires at most $\kappa\leq 100r$ steps, which are detailed in subsections \ref{sec:it-j} and \ref{sec:+1}. At each step, we introduce a new Lie transform, whose properties are presented in subsection \ref{sec:lie}. To control the terms generated by the Taylor expansion of the Hamiltonian in the new coordinates, we will use the main multilinear estimates for the Poisson bracket with respect to the $\Ysup{\alpha}$-norm, already proved in Section \ref{sec:poiss}.

\medskip

If we reach the scale $\alpha=\beta$, where $\beta$ is sufficiently large, then the vector field generated by $Q^{(\beta)}$ is smoothing when it is restricted to the open set $\mathcal{V}_{\alpha,s}(10\epsilon,\xi)$, and the saving of $N_{\beta}^{-4s}$ in the vector field estimate is sufficient to control the actions over a very long time $T\sim N_{\beta}^{2s}=\epsilon^{-r}$.

\subsection{The Lie transforms}
\label{sec:lie}
We introduce and study the Lie transforms used to remove the undesirable terms from the Hamiltonian (namely the ones with a frequency smaller than $N_{\alpha+1}$ that is not paired).


\subsubsection{Truncation argument}

\label{sec:trun}

 Let $\alpha\in\{0,\cdots,\beta-1\}$ and $\Xi_{\alpha}$ be the admissible sets of parameters defined in \eqref{eq:tw}, just before the statement of Theorem \ref{thm:nf-alpha}.  We first introduce a cutoff function, which essentially ensures that the Lie transform is defined for all $\xi$, but is only effective for non-resonant parameters $\xi\in\Xi_{\alpha}$.  
 
 \begin{definition}[Truncation]\label{def:trun} Fix $\varphi\in C_{c}^{\infty}(\R)$ valued in $[0,1]$ such that 
 \[
 \supp\varphi\subset [-1,1]\,,\quad \varphi\equiv 1 \quad \text{on}\ [-\frac{1}{2},\frac{1}{2}]\,,\quad \varphi<1\ \text{on}\ \R\setminus[-\frac{1}{2},\frac{1}{2}]\,.
 \]
 Then, given $\omega^{(\alpha)}$ and $\gamma(\alpha)=4^{\alpha}\gamma$ (constructed at scale $\alpha$ in Theorem \ref{thm:nf-alpha} by the recurrence asumption) we define the cutoff-function 
 \[
 \mathfrak{h}^{(\alpha)}: \xi\in\mathcal{U}_{s}(\epsilon)\ \mapsto\ \prod_{\nb\in\Lambda_{\alpha+1}\,,\ \deg(\nb)\leq2\overline{r}} 
 	\Big(
	1-	
 	\varphi
	\Big(
	\gamma(\alpha)^{-1}\nb_{-}^{2s}\epsilon^{-2}\Omega_{\nb}(\omega^{(\alpha)}(\xi))
	\Big)
	\Big)\,,
 \]
 where $\Lambda_{\alpha+1}$ was defined in \eqref{eq:scale}.
 \end{definition}
 
 We now prove important properties of the cutoff function $\mathfrak{h}^{(\alpha)}$. 
 
 \begin{lemma}\label{lem:cut}
 Let $\alpha\in\{1,\cdots,\beta\}$, and suppose that Theorem \ref{thm:nf-alpha} is proved up to scale $\alpha$. Then, 
  \begin{align}
   \label{eq:s-dg}
 \xi\in\Xi_{\alpha}\quad 
 	&\implies \quad \mathfrak{h}^{(\alpha)}(\xi)=1\,.
	 \intertext{Moreover, for all $\xi\in\mathcal{U}_{s}(\epsilon)$}
 \label{eq:frak-h}
 \xi\in\supp \mathfrak{h}^{(\alpha)} \quad 	
 	&\implies\quad  \forall \nb\in\Lambda_{\alpha+1}\,,\quad |\Omega_{\nb}(\omega^{(\alpha)}(\xi))| > \frac{1}{2} \gamma(\alpha)\nb_{-}^{-2s}\epsilon^{2}\,,
 \intertext{and}
 \label{eq:s-gd}
\xi\in\supp \mathfrak{h}^{(\alpha)} \quad 
 	&\implies \xi\in\Xi_{\alpha-1}\,.
 \end{align}
 
 \end{lemma}

\begin{remark} In particular, we can deduce from \eqref{eq:s-dg} and \eqref{eq:frak-h} that $\Xi_{\alpha}\subset\Xi_{\alpha-1}$. 
\end{remark}

\begin{proof} The proof of \eqref{eq:s-dg} follows from the construction: suppose that $\xi\in\Xi_{\alpha}$, in which case for all $\nb\in\Lambda_{\alpha+1}$
\[
|\Omega_{\nb}(\omega^{(\alpha)}(\xi))| 
	> 
	\gamma(\alpha)\nb_{-}^{-2s}\epsilon^{2}\,.
\]
 Hence, we have that for all $\nb\in\Lambda_{\alpha+1}$, 
 \[
 \varphi
 	\Big(
	\gamma(\alpha)^{-1}\nb_{-}^{2s}\epsilon^{-2}\Omega_{\nb}(\omega^{(\alpha)}(\xi))
	\Big) = 0\,,
 \]
 and we deduce that $\mathfrak{h}^{(\alpha)}(\xi)=1$\,.

\medskip

We now show \eqref{eq:frak-h}, and suppose that $\xi$ is in the support of $\mathfrak{h}^{(\alpha)}$. Then, we have from the construction of $\varphi$ and $\mathfrak{h}$ that for all $\nb\in\Lambda_{\alpha+1}$,
\[
|\Omega_{\nb}(\omega^{(\alpha)}(\xi))|
	>
	\frac{1}{2} \gamma(\alpha)\nb_{-}^{-2s}\epsilon^{2}\,.
\]
This proves \eqref{eq:frak-h}. Lastly, to show \eqref{eq:s-gd}, we observe from the stability property for the modulated frequencies claimed in \eqref{eq:st-fr} that 
\begin{align*}
|\Omega_{\nb}(\omega^{(\alpha-1)}(\xi))|
	&\geq
	|\Omega_{\nb}(\omega^{(\alpha)}(\xi))|-
	|\Omega_{\nb}(\omega^{(\alpha-1)}(\xi))-\Omega_{\nb}(\omega^{(\alpha)}(\xi))|
	\\
	&\geq \frac{1}{2}\gamma(\alpha)\nb_{-}^{-2s}\epsilon^{2} - 	
	2\overline{r}N_{\alpha-1}^{-4s}\epsilon^{3}\,.
\end{align*}
When $\nb\in\Lambda_{\alpha+1}$ then $\nb_{-}\leq N_{\alpha+1}$ and we obtain that for all $\nb\in\Lambda_{\alpha+1}$, 
\[
|\Omega_{\nb}(\omega^{(\alpha-1)}(\xi))| > \frac{1}{4}\gamma(\alpha)\nb_{-}^{-2s}\epsilon^{2} = \gamma(\alpha-1)\nb_{-}^{-2s}\epsilon^{2}\,,
\]
under the condition 
\[
(\frac{N_{\alpha+1}}{N_{\alpha-1}})^{2s}N_{\alpha-1}^{-2s}\epsilon
	<\frac{1}{4}\gamma(\alpha)\,, 
\]
which, under our choice of parameters in subsection \ref{sec:par}, is equivalent to 
$\epsilon^{1 - \frac{1}{100}}N_{\alpha}^{-2s} < 4^{\alpha-1}\epsilon^{\frac{1}{30}}$.
This shows that $\xi\in\Xi_{\alpha-1}$, which is precisely the claim \eqref{eq:s-gd}. 
 \end{proof}

\begin{lemma}[Small divisor estimate] 
\label{lem:sde}
Let $\alpha\in\{0,\cdots,\beta\}$ such that Theorem \ref{thm:nf-alpha} is proved up to scale $\alpha$. Suppose that $(\widetilde{w}_{n}(\xi))_{n\in\Z_{M}^{d}}$ satisfies, for all $\xi\in\mathcal{U}_{s}(\epsilon)$, and $k,n\in\Z_{M}^{d}$, 
\begin{align}
\label{eq:o-tilde0}
	|\omega_{n}^{(\alpha)}(\xi)-\widetilde{\omega}_{n}(\xi)| 
	&\lesssim_{r}
	\epsilon^{3}N_{\alpha}^{-4s}\,,
	\\
	|\partial_{\xi_{k}}(\omega_{n}^{(\alpha)}(\xi) - \widetilde{\omega}_{n}(\xi))|
	&\lesssim_{r} \epsilon N_{\alpha}^{-2s}\,.
\label{eq:o-tilde1}
\end{align}
Then for all $\xi\in\mathcal{U}_{s}(\epsilon)$ and $\nb\in\Lambda_{\alpha+1}$, 
\begin{align}
\label{eq:smd0}
	\frac{\mathfrak{h}^{(\alpha)}(\xi)}{|\Omega_{\nb}(\widetilde{\omega}(\xi))|} 
	&\leq4\gamma(\alpha)^{-1}\epsilon^{-2}\nb_{-}^{2s}\,,
	\intertext{and, for all $k\in\Z_{M}^{d}$,}
\label{eq:smd1}
	\Big|
	\partial_{\xi_{k}} \Big(
	\frac{\mathfrak{h}^{(\alpha)}(\xi)}{\Omega_{\nb}(\widetilde{\omega}(\xi))}
	\Big)
	\Big|
	&\leq 2\epsilon^{-\frac{1}{10^{4}}}
	(\gamma(\alpha)^{-1}\epsilon^{-2}\nb_{-}^{2s})^{2}
	\,,
\end{align}
\end{lemma}

\begin{proof}
First, we proceed as in the proof of \eqref{eq:s-gd} to deduce from \eqref{eq:frak-h} together with the stability assumption \eqref{eq:o-tilde0} that
\begin{equation}
\label{eq:76}
\xi\in\supp \mathfrak{h}^{(\alpha)}\quad \implies \quad \forall \nb\in\Lambda_{\alpha+1}\,,\quad |\Omega_{\nb}(\widetilde{\omega}(\xi))|\geq\frac{1}{4}\gamma(\alpha)\nb_{-}^{-2s}\epsilon^{2}\,.
\end{equation}
from which we deduce \eqref{eq:smd0}. To show \eqref{eq:smd1}, we take $n\in\Z_{M}^{d}$ and we first observe that, thanks to \eqref{eq:omega}, for all $\nb\in\Lambda_{\alpha+1}$ and $\xi\in\mathcal{U}_{s}(\epsilon)$,
\[
|\partial_{\xi_{n}}\Omega_{\nb}(\omega^{(\alpha)}(\xi))| 
	+
	|\partial_{\xi_{n}}\Omega_{\nb}(\widetilde{\omega}(\xi))| \lesssim r\,.
\]
We deduce that for all $\xi\in\mathcal{U}_{s}(\epsilon)$,
\begin{equation*}
\begin{split}
\label{eq:dhf}
|\partial_{\xi_{n}}\mathfrak{h}^{(\alpha)}(\xi)| 
	&\leq \sum_{\nb\in\Lambda_{\alpha+1}}\gamma(\alpha)^{-1}\nb_{-}^{2s}\epsilon^{-2}
	\Big|
	\partial_{\xi_{n}}\Omega_{\nb}(\omega^{(\alpha)}(\xi))
	\varphi'
	\Big(
	\gamma(\alpha)^{-1}\nb_{-}^{2s}\epsilon^{-2}\Omega_{\nb}(\omega^{(\alpha)}(\xi))
	\Big)
	\Big| \\
	&\leq \epsilon^{-\frac{1}{10^{4}}}\gamma(\alpha)^{-1}\epsilon^{-2}N_{\alpha+1}^{2s}\,.
	\end{split}
\end{equation*}
Hence, for $n\in\Z_{M}^{d}$, $\nb\in\Lambda_{\alpha+1}$ and $\xi\in\mathcal{U}_{s}(\epsilon)$,
\[
\Big|\partial_{\xi_{n}}\Big(
		\frac{\mathfrak{h}^{(\alpha)}(\xi)}{\Omega_{\nb}(\widetilde{\omega}(\xi))}
		\Big)
		\Big|
		\leq 
		\Big|
		\frac{\partial_{\xi_{n}}\mathfrak{h}^{(\alpha)}(\xi)}{\Omega_{\nb}(\widetilde{\omega}(\xi))}
		\Big|
		+ 
		\mathfrak{h}^{(\alpha)}(\xi)
		\frac{|\partial_{\xi_{n}}\Omega_{\nb}(\widetilde{\omega}(\xi))|}{\Omega_{\nb}(\widetilde{\omega}(\xi))^{2}}
		\leq 2\epsilon^{-\frac{1}{10^{4}}}(\gamma(\alpha)^{-1}\epsilon^{-2}N_{\alpha+1}^{2s})^{2} \,,
\]
which concludes the proof of Lemma \ref{lem:sde}.
\end{proof}
\subsubsection{Definition of the transformation}
\label{sec:Lie} Let us introduce linear operator used to solve the cohomological equations.

\begin{definition}
\label{def:Lie}
Given some modulated frequencies $(\widetilde{\omega}_{n}(\xi))_{n\in\Z_{M}^{d}}$ satisfying \eqref{eq:o-tilde0} and \eqref{eq:o-tilde1}. We define the linear operator

\begin{equation}
\label{eq:op}
\mathcal{L}_{\alpha,\widetilde{\omega}} = \left\{ 
\begin{array}{ccc}
X_{\leq2\overline{r}}(\epsilon) & \to &  X_{\leq2\overline{r}}(\epsilon) \\
H  & \mapsto & \frac{i}{2}\mathfrak{h}^{(\alpha)}(\xi)	
	\displaystyle{\sum_{\nb\in\Lambda_{\alpha+1}}}
	\frac{(H[\xi])_{\nb}}{\Omega_{\nb}(\widetilde{\omega}(\xi))}
	z_{\nb}(u,I(u)-\xi)
\end{array} \right.
\end{equation}

\end{definition}
When the context is clear we may abuse notations and simply write
\[
\mathcal{L} := \mathcal{L}_{\alpha,\widetilde{\omega}}\,.
\]
The purpose of $\mathcal{L}$ is to solve a cohomological equation.
\begin{lemma}
\label{lem:24}
For all $Q\in X_{\leq2\overline{r}}(\epsilon)$, and frequencies $(\widetilde{\omega}_{n})_{n\in\Z_{M}^{d}}$, if 
\[
Z_{2}(\widetilde{\omega},\xi;u) 
	:= \sum_{|n|\leq M} \widetilde{\omega}_{n}(\xi)(|u_n|^{2}-\xi_{n})\,,
\]
then for all $\xi\in\mathcal{U}_{s}(\epsilon)$ the polynomial $\mathcal{L}(Q)(\xi)\in \mathcal{H}_{\leq2\overline{r}}(\Z_{M}^{d})$ solves  the cohomological equation 
\begin{equation}
\label{eq:rf3}
Q(\xi;u) + \{Z_{2}(\widetilde{\omega}),\mathcal{L}(Q)\}(\xi;u) = (\Id-\mathfrak{h}^{(\alpha)}(\xi)\Pi_{\Lambda_{\alpha+1}})Q(\xi;u) \,.
\end{equation}
In particular, when $\xi\in\Xi_{\alpha}$ then 
\[
Q(\xi;u) + \{Z_{2}(\widetilde{\omega}),\mathcal{L}(Q)\}(\xi;u) = (\Id-\Pi_{\Lambda_{\alpha+1}})Q(\xi;u) \,.
\]
\end{lemma}
\begin{proof}
The equality \eqref{eq:rf3} follows directly from the Definition \ref{def:Lie} of $\mathcal{L}_{\alpha,\widetilde{\omega}}$:
\begin{equation}
\begin{split}
\label{eq:rf2}
\{Z_{2}(\widetilde{\omega}),\mathcal{L}(Q)(u)\}(\xi;u)& = - \mathfrak{h}^{(\alpha)}(\xi)\sum_{\nb\in\Lambda_{\alpha+1}}(Q[\xi])_{\nb}z_{\nb}(u,I(u)-\xi)\\& = - \mathfrak{h}^{(\alpha)}(\xi)\Pi_{\Lambda_{\alpha+1}}Q(\xi;u)\,.
\end{split}
\end{equation}
The second equality is a consequence of the construction of the truncation $\mathfrak{h}^{(\alpha)}$, which satisfies $\mathfrak{h}^{(\alpha)}(\xi)=1$ for all $\xi\in\Xi_{\alpha}$.  
\end{proof}

\begin{lemma} 
\label{lem:23}
For all $Q\in X_{\leq2\overline{r}}(\epsilon)$, we have 
\begin{align*}
	\|\mathcal{L}_{\alpha,\widetilde{\omega}}(Q)\|_{\Ysup{\alpha}}	
	&\lesssim \gamma(\alpha)^{-1}\epsilon^{-2}N_{\alpha+1}^{2s}\|\Pi_{\Lambda_{\alpha+1}}Q\|_{\Ysup{\alpha}} \,,\\
	\|\mathcal{L}_{\alpha,\widetilde{\omega}}(Q)\|_{\Ylip{\alpha}}
	&\lesssim 
	\gamma(\alpha)^{-1}
	\epsilon^{-2}N_{\alpha+1}^{2s}
	(
	\epsilon^{-\frac{1}{30}-\frac{3}{100}-\frac{1}{10^{4}}}\|\Pi_{\Lambda_{\alpha+1}}Q\|_{\Ysup{\alpha}}
	+
	\|\Pi_{\Lambda_{\alpha+1}}Q\|_{\Ylip{\alpha}}
	)\,.
\end{align*}
\end{lemma}

\begin{proof}
The proof of the $\Ysup{\alpha}$-bound is a direct consequence of the small divisor estimates  \eqref{eq:smd0} from Lemma \ref{lem:sde}. For the second bound, we also use the Leibniz rule: recall that for $\nb\in\Nb$ and $\xi\in\mathcal{U}_{s}(\epsilon)$, 
\[
	(
	\mathcal{L}_{\alpha,\widetilde{\omega}}(Q)[\xi]
	)_{\nb}
	=
	\mathbf{1}_{\Lambda_{\alpha+1}}(\nb)\frac{\mathfrak{h}^{(\alpha)}(\xi)}{\Omega_{\nb}(\widetilde{\omega}(\xi))}(Q[\xi])_{\nb}\,, 
\]
and given $k\in\Z_{M}^{d}$ we have from the Leibniz rule that for $\nb\in\Lambda_{\alpha+1}$ and $\xi\in\mathcal{U}_{s}(\epsilon)$,
\[
\partial_{\xi_{k}}(
	\mathcal{L}_{\alpha,\widetilde{\omega}}(Q)[\xi]
	)_{\nb}
	=
	\partial_{\xi_{k}}(\frac{\mathfrak{h}^{(\alpha)}(\xi)}{\Omega_{\nb}(\widetilde{\omega}(\xi))})(Q[\xi])_{\nb}
	+
	\frac{\mathfrak{h}^{(\alpha)}(\xi)}{\Omega_{\nb}(\widetilde{\omega}(\xi))}
	\partial_{\xi_{k}}(Q[\xi])_{\nb}\,.
\]
According to the small divisor bounds \eqref{eq:smd0} and \eqref{eq:smd1}, we obtain 
\begin{align*}
&|\partial_{\xi_{k}}(
	\mathcal{L}_{\alpha,\widetilde{\omega}}(Q)[\xi]
	)_{\nb}
	| \\
	&\leq
	2\epsilon^{-\frac{1}{10^{4}}}(\gamma(\alpha)^{-1}\epsilon^{-2}\nb_{-}^{2s})^{2}
	|(Q[\xi])_{\nb}|
	+
	4\gamma(\alpha)^{-1}\epsilon^{-2}\nb_{-}^{2s}
	|\partial_{\xi_{n}}(Q[\xi])_{\nb}| \\
	&
	\leq
	2\epsilon^{-\frac{1}{10^{4}}}(\gamma(\alpha)^{-1}\epsilon^{-2}\nb_{-}^{2s})^{2}
	\mathrm{w}_{\nb}^{0}(\alpha)\|Q\|_{\Ysup{\alpha}}
	+
	4\gamma(\alpha)^{-1}\epsilon^{-2}\nb_{-}^{2s}
	\mathrm{w}_{\nb}^{1}(\alpha)\|Q\|_{\Ylip{\alpha}}\\
	&\lesssim
	\mathrm{w}_{\nb}^{1}(\alpha)
	\Big(
	\epsilon^{-\frac{1}{10^{4}}}(\gamma(\alpha)^{-1}\epsilon^{-2}\nb_{-}^{2s})^{2}
	N_{\alpha}^{-2s}\eta^{2}\|Q\|_{\Ysup{\alpha}}
	+
	\gamma(\alpha)^{-1}\epsilon^{-2}\nb_{-}^{2s}
	\|Q\|_{\Ylip{\alpha}}
	\Big)\,.
\end{align*}
Therefore, we have from the the assumption~$\nb\in\Lambda_{\alpha+1}$ and the relation \eqref{eq:par} between the parameters that
\[
(\mathrm{w}_{\nb}^{1}(\alpha))^{-1}
	|
	\partial_{\xi_{k}}
	(
	\mathcal{L}_{\alpha,\widetilde{\omega}}(Q)[\xi]
	)_{\nb}
	|
	\lesssim \gamma(\alpha)^{-1}\epsilon^{-2}N_{\alpha+1}^{2s}(
	\epsilon^{-\frac{1}{30}-\frac{3}{100}-\frac{1}{10^{4}}}
	\|Q\|_{\Ysup{\alpha}}
	+
	\|Q\|_{\Ylip{\alpha}}
	)\,.
\]
This completes the proof of the second estimate, recalling  the Definition \ref{def:norm} of the $\Ylip{\alpha}$-norm. 
\end{proof}
\begin{corollary}
\label{cor:ad}
For all $Q\in X_{\leq2\overline{r}}(\epsilon)$, $H\in X_{\leq\overline{r}^{2}}(\epsilon)$ and $Z_{4}\in X_{4,\mathrm{Int}}(\epsilon)$ we have 
\begin{align*}
\|\{H,\mathcal{L}_{\alpha,\widetilde{\omega}}(Q)\}\|_{\Ysup{\alpha}} 
	&\leq \epsilon^{\frac{3}{2}}N_{\alpha}^{-2s}
	\|H\|_{\Ysup{\alpha}}
	\|\Pi_{\Lambda_{\alpha+1}}Q\|_{\Ysup{\alpha}}
	\,,\\
	\|\{Z_{4},\mathcal{L}_{\alpha,\widetilde{\omega}}(Q)\}\|_{\Ysup{\alpha}} 
	&\leq \epsilon^{\frac{1}{15}}
	\|Z_{4}\|_{Z^{\mathrm{sup}}}
	\|\Pi_{\Lambda_{\alpha+1}}Q\|_{\Ysup{\alpha}}\,.
\end{align*}
Moreover, if we suppose that
\begin{equation}
\label{eq:as-14}
	\|\Pi_{\Lambda_{\alpha+1}}Q\|_{\Ysup{\alpha}}
	+\|H\|_{\Ysup{\alpha}}
	+\|Z_{4}\|_{Z^{\mathrm{sup}}}
	+
	\|Z_{4}\|_{Z^{\mathrm{lip}}}
	\leq \epsilon^{-\frac{1}{10^{3}}}\,,
\end{equation}
and
\begin{equation}
\label{eq:as-15}
\|\Pi_{\Lambda_{\alpha+1}}Q\|_{\Ylip{\alpha}}+\|H\|_{\Ylip{\alpha}}\lesssim \epsilon^{-\frac{1}{15}}\,,
\end{equation}
then
\[
\|\{H,\mathcal{L}_{\alpha,\widetilde{\omega}}(Q)\}\|_{\Ylip{\alpha}}  
	\leq \epsilon N_{\alpha}^{-2s}\,,\quad
\|\{Z_{4},\mathcal{L}_{\alpha,\widetilde{\omega}}(Q)\}\|_{\Ylip{\alpha}}  
	\leq \epsilon^{\frac{1}{50}}
	\,.
\]
\end{corollary}
\begin{proof}
It suffices to combine the (bilinear) estimate on the Poisson bracket between two polynomials given by Proposition \ref{prop:bracket0} with the operator bound from Lemma \ref{lem:23}, recalling the relation \ref{eq:par} between the parameters:
\begin{align*}
\|\{H,\mathcal{L}_{\alpha,\widetilde{\omega}}(Q)\}\|_{\Ysup{\alpha}}
	&\leq 
	\eta^{4-\frac{1}{4}}N_{\alpha}^{-4s}
	\|H\|_{\Ysup{\alpha}}
	\|
	\mathcal{L}_{\alpha,\widetilde{\omega}}(Q)
	\|_{\Ysup{\alpha}} \\
	&\lesssim
	\gamma(\alpha)^{-1}(\eta\epsilon^{-1})^{2}(\frac{N_{\alpha+1}}{N_{\alpha}})^{2s}\eta^{2-\frac{1}{4}}N_{\alpha}^{-2s}
	\|H\|_{\Ysup{\alpha}}
	\|\Pi_{\Lambda_{\alpha+1}}Q\|_{\Ysup{\alpha}} \\
	&\leq
	\epsilon^{\frac{3}{2}}N_{\alpha}^{-2s}
	\|H\|_{\Ysup{\alpha}}
	\|\Pi_{\Lambda_{\alpha+1}}Q\|_{\Ysup{\alpha}}\,.
\end{align*}  
As for the Lipschitz estimates, we have from the assumptions and  from Lemma \ref{lem:23} that:
\[
	\|
	\mathcal{L}_{\alpha,\widetilde{\omega}}(Q)
	\|_{\Ysup{\alpha}} 
	\lesssim \gamma(\alpha)^{-1}\epsilon^{-2}N_{\alpha+1}^{2s}\epsilon^{-\frac{1}{10^{3}}} \,,
	\quad 
	\|
	\mathcal{L}_{\alpha,\widetilde{\omega}}(Q)
	\|_{\Ylip{\alpha}}
	\lesssim 
	\gamma(\alpha)^{-1}\epsilon^{-2}N_{\alpha+1}^{2s}\epsilon^{-\frac{1}{15}}\,.
\]
Then, Proposition \ref{prop:bracket1} gives
\begin{align*}
&\|\{H,\mathcal{L}_{\alpha,\widetilde{\omega}}\}\|_{\Ylip{\alpha}} \\
	&\leq \eta^{4-\frac{1}{4}}N_{\alpha}^{-4s}
	\gamma(\alpha)^{-1}\epsilon^{-2}N_{\alpha+1}^{2s}(
	\epsilon^{-\frac{1}{10^{3}}}\|H\|_{\Ysup{\alpha}}
	+
	\epsilon^{-\frac{1}{15}}
	\|H\|_{\Ysup{\alpha}}
	+
	\epsilon^{-\frac{1}{10^{3}}}
	\|H\|_{\Ylip{\alpha}}
	)\\
	&\leq
	\epsilon N_{\alpha}^{-2s}\,.
\end{align*}
Similarly, 
\begin{align*}
&\|\{Z_{4},\mathcal{L}_{\alpha,\widetilde{\omega}}\}\|_{\Ylip{\alpha}} \\
	&\lesssim \eta^{2+\frac{1}{5}-\frac{2}{10^{4}}}N_{\alpha}^{-2s}
	\gamma(\alpha)^{-1}\epsilon^{-2}N_{\alpha+1}^{2s}
	(
	\epsilon^{-\frac{1}{15}}\|Z\|_{Z_{4}^{\mathrm{sup}}}
	+\eta^{2}N_{\alpha}^{-2s}\epsilon^{-\frac{1}{10^{3}}}\|Z_{4}\|_{Z^{\mathrm{lip}}})\\
	&\leq\epsilon^{\frac{1}{50}}\,.
\end{align*}
This completes the proof of Corollary \ref{cor:ad}. 
\end{proof}

\begin{notation}
Given $Q\in X_{\leq\overline{r}^{2}}$, and $t\in\R$, we denote by $\Phi_{\mathcal{L}(Q)}^{t}$ the Hamiltonian flow on $h^s(\Z_{M}^{d})$ generated by $\mathcal{L}(Q)(\xi)$. 

\medskip

\noindent For $\xi\in\mathcal{U}_{s}(\epsilon)$ and $u\in h^{s}(\Z_{M}^{d})$ we denote by $v(\xi;t)$ (or implicitly $v(t)$) the solution of the finite-dimensional ODE
\begin{equation}
\label{eq:odev}
\begin{cases}
i\partial_t v = \nabla(\mathcal{L}(Q))(\xi)(v)\,,\\
v(0) = u\,.
\end{cases}
\end{equation}
\end{notation}

\begin{remark}
Using the preservation of the $\ell^{2}$-mass, we note that the ODE \eqref{eq:odev} is globally well-posed, and $\Phi_{\mathcal{L}(Q)}^{t}$ is well-defined on $h^s(\Z_{M}^{d})$ for all $t\in\R$. Moreover, 
\[
\Phi_{\mathcal{L}(Q)}^{-t} = (\Phi_{\mathcal{L}(Q)}^{t})^{-1}\,. 
\]
\end{remark}

\subsection{Estimates on the Lie transforms}
\label{sec:lie-prop}
We fix some frequencies $(\widetilde{\omega}_{n})_{n\in\Z_{M}^{d}}$ satisfying \eqref{eq:o-tilde0} and \eqref{eq:o-tilde1}, and we consider the linear operator $\mathcal{L}=\mathcal{L}_{\alpha,\widetilde{\omega}}$ as in Definition \ref{def:Lie}. We prove important properties on the flow $(\Phi_{\mathcal{L}(Q)}^t)_{t\in[-1,1]}$, when $Q\in X_{\leq2\overline{r}}(\epsilon)$ operates at scale $\alpha$ (in the sense that its $\Ysup{\alpha}$-norm is no too large).

\begin{lemma}[Stability]\label{lem:stab} Let $Q\in X_{\leq2\overline{r}}(\epsilon)$ be such that 
\begin{equation}
\label{eq:ass-Q}
\|Q\|_{\Ysup{\alpha}}
	\leq \epsilon^{-\frac{1}{10^{3}}}\,.
\end{equation}
For all $\xi\in\mathcal{U}_{s}(\epsilon)$, the transformation $\Phi_{\mathcal{L}(Q)(\xi)}^{t}$ is close to the identity on the annulus $\mathcal{V}_{\alpha,s}(10\epsilon,\xi)$: for all $t\in[-1,1]$ and $u\in\mathcal{V}_{\alpha,s}(10\epsilon,\xi)$,
\begin{equation}
\label{eq:stab}
\|\Phi_{\mathcal{L}(Q)}^{t}(u)-u\|_{h^{s}}
	\leq\epsilon^{\frac{3}{2}}\|u\|_{h^{s}}N_{\alpha}^{-2s}\,.
\end{equation}
\end{lemma}
\begin{corollary}
We deduce from \eqref{eq:stab} the bound
\begin{equation}
\label{eq:stab-a}
\sum_{n\in\Z_{M}^{d}}\langle n\rangle^{2s}||(\Phi_{\mathcal{L}(Q)}^{t}(u))_{n}|^{2}-|u_{n}|^{2}|
	\lesssim \epsilon^{\frac{3}{2}}\|u\|_{h^{s}}^{2}N_{\alpha}^{-2s}\,,
\end{equation}
which gives the stability of the annulus $\mathcal{V}_{\alpha,s}(10\epsilon,\xi)$ under the transformations $\Phi_{\mathcal{L}(Q)(\xi)}^{t}$.
\end{corollary}

\begin{proof}[Proof of Lemma \ref{lem:stab}] We denote
\[
v(t) = \Phi_{\mathcal{L}_{\alpha,\widetilde{\omega}}(Q)}^t(u)\,,
\] 
solution to the ODE \eqref{eq:odev}. We have from the Duhamel's integral formula that for all $t$,
\[
v(t) = u - i\int_{0}^{t}[\nabla\mathcal{L}(Q)(\xi)](v(t'))dt'\,.
\]
Moreover, we deduce from the vector field estimate of Proposition \ref{prop:vec-alpha} and from the operator bound for $\mathcal{L}$ obtained in Lemma \ref{lem:23} that 
\begin{align*}
\|[\nabla \mathcal{L}(Q)(\xi)](v)\|_{h^{s}}
	&\leq 
	N_{\alpha}^{-4s}\epsilon^{4-\frac{1}{4}}
	\|
	\mathcal{L}(Q)
	\|_{\Ysup{\alpha}}
	\|
	v
	\|_{h^{s}} \\
	&\leq
	4\gamma(\alpha)^{-1}(\frac{N_{\alpha+1}}{N_{\alpha}})^{2s}N_{\alpha}^{-2s}\epsilon^{2-\frac{1}{4}}
	\|
	Q
	\|_{\Ysup{\alpha}}
	\|v
	\|_{h^{s}} \\
	&\leq
		4\gamma(\alpha)^{-1}(\frac{N_{\alpha+1}}{N_{\alpha}})^{2s}N_{\alpha}^{-2s}\epsilon^{2-\frac{1}{4}-\frac{1}{10^{3}}}
	\|v
	\|_{h^{s}}	\,.
\end{align*}
Under the relation \eqref{eq:par} for the parameters we deduce that 
\[
\|[\nabla \mathcal{L}(Q)(\xi)](v)\|_{h^{s}} 
	\leq N_{\alpha}^{-2s}\epsilon^{\frac{3}{2}}\|v\|_{h^{s}}\,.
\]
Therefore, for all $t$,
\[
\|v(t)-u\|_{h^{s}} \leq N_{\alpha}^{-2s}\epsilon^{\frac{3}{2}} |\int_{0}^{t}\|v(t')\|_{h^{s}}dt'|\,.
\]
We conclude from a bootstrap argument, using that $\|v(t')\|_{h^{s}}\lesssim\epsilon$ when \eqref{eq:stab} is satisfied. 
\end{proof}

In the next lemma we estimate the differential of the vector field. Since the proof is quite long and does not requires new ideas, we postpone it to Appendix \ref{sub:d2F}.  
\begin{lemma}
\label{lem:d2F} For all $Q\in X_{\leq2\overline{r}}$, $\xi\in\mathcal{U}_{s}(\epsilon)$ and $u\in\mathcal{V}_{\alpha,s}(20\epsilon,\xi)$, we have
\[
 \| \mathrm{d}\nabla\mathcal{L}(Q)(\xi)(u)\|_{h^s\to h^s}
 	\leq
	\epsilon^{1-\frac{1}{8}}
	\|Q\|_{\Ysup{\alpha}}\,.
\]

\end{lemma}

We now state and prove two important corollaries. 

\begin{corollary}\label{cor:dphi} Let $Q\in X_{\leq2\overline{r}}(\epsilon)$ be such that 
\[
\|Q\|_{\Ysup{\alpha}}\leq \epsilon^{-\frac{1}{10^{3}}}\,. 
\]  
We have that for all $\xi\in\mathcal{U}_{s}(\epsilon)$, $u\in\mathcal{V}_{\alpha,s}(10\epsilon,\xi)$ and $t\in[-1,1]$, 
\[
\|\mathrm{d}\Phi_{\mathcal{L}(Q)(\xi)}^t(u) - \Id \|_{h^{s}\to h^{s}} \leq \epsilon^{\frac{3}{4}}\,.
\]
\end{corollary}
\begin{proof}
We use the equation, the previous Lemmas \ref{lem:stab}, \ref{lem:d2F}, and a bootstrap argument. Let $\varphi \in h^{s}$, and $u\in\mathcal{V}_{\alpha,s}(10\epsilon,\xi)$. Set 
\[
v(t) := \Phi^t_{\mathcal{L}(Q)}(u)\,,\quad \widetilde{\varphi}(t) := \mathrm{d}\Phi^{t}_{\mathcal{L}(Q)}(u)\varphi\,.
\]
We have from the equation that 
\[
i\partial_{t}\widetilde{\varphi}(t) = 
	\Big[
	\mathrm{d}\nabla\mathcal{L}(Q)(\xi)(v(t))
	\Big]
	(\widetilde{\varphi}(t))\,.
\]
Since $\Phi^{0}_{\mathcal{L}(Q)}=\operatorname{Id}$, we have
\[
\widetilde{\varphi}(0)= \varphi\,.
\]
Hence, the Duhamel's integral formula gives 
\[
\|\widetilde{\varphi}(t) - \varphi\|_{h^{s}} 
	\leq 
	\Big|
	\int_0^t 
	\|
	[\mathrm{d}_{u}\nabla\mathcal{L}(Q)(v(t'))
	]
	(\widetilde{\varphi}(t'))
	\|_{h^{s}}dt'
	\Big|\,.
\]
We have from Lemma \ref{lem:stab} (and its consequence \eqref{eq:stab-a}) that for all $t\in[-1,1]$, 
\[
v(t)\in\mathcal{V}_{\alpha,s}(11\epsilon,\xi)\,.
\]
According to Lemma \ref{lem:d2F}, 
\[
\|\widetilde{\varphi}(t)-\varphi\|_{h^{s}} 
	\leq \epsilon^{1-\frac{1}{8}}\Big|\int_0^t 
	\|Q\|_{\Ysup{\alpha}}
	\|\widetilde{\varphi}(t')
	\|_{h^{s}}dt'\Big|\,,
\]
and, under the assumption that $\|Q\|_{\Ysup{\alpha}}\leq \epsilon^{-\frac{1}{10^{3}}}$ we can conclude from a bootstrap argument. More precisely, we use that if 
\[
\underset{|t'|\leq |t|}{\sup}\|\widetilde{\varphi}(t)-\varphi\|_{h^{s}}
	\leq
	\epsilon^{\frac{3}{4}}
	\|\varphi\|_{h^{s}}\,,
\]
then 
\[
\underset{|t'|\leq |t|}{\sup}\ 
	\|\widetilde{\varphi}(t')
	\|_{h^{s}}
	\leq
	2\|\varphi\|_{h^{s}}\,.
\]
This concludes the proof of Corollary \ref{cor:dphi}. 
\end{proof}
The second corollary is more technical, but in the same spirit. It is crucial to estimate the measure of the non-resonant initial data in the original coordinates (i.e. before the successive Lie transformations). 
\begin{corollary}\label{cor:dxi} Let $Q\in X_{\leq2\overline{r}}(\epsilon)$ be such that for all $\xi\in\mathcal{U}_{s}(\epsilon)$
\[
\|Q\|_{\Ysup{\alpha}}\leq\epsilon^{-\frac{1}{10^{3}}}\,,\quad \|Q\|_{\Ylip{\alpha}} \leq \epsilon^{-\frac{1}{15}}\,,
\]
Then for all $t\in[-1,1]$, $\phi\in\Pi_{M}B_{s}(10\epsilon)$ and $u\in\mathcal{V}_{\alpha,s}(10\epsilon,\xi(\phi))$ then
\begin{equation}
\label{eq:dxi}
	\|
	\mathrm{d}_{\phi}\Phi^{t}_{\mathcal{L}(Q)(\xi(\phi))}(u)
	\|_{h^{s} \to h^{s}} 
	\lesssim \epsilon^{1-\frac{2}{5}}\,.
	\end{equation}
\end{corollary}

\begin{proof}
Let $\varphi\in h^{s}(\Z_{M}^{d})$, $\phi\in\Pi_{M}B_{s}(10\epsilon)$ and $u\in\mathcal{V}_{\alpha,s}(10\epsilon,\xi(\phi))$. For $|t|\leq 1$, set 
\[
v(t) = \Phi^{t}_{\mathcal{L}(Q)(\xi(\phi))}(u)\,,\quad \widetilde{\varphi}(t) :=   [\mathrm{d}_{\phi}v(t,\phi)](\varphi)\,. 
\]
In particular, $v$ is solution in $C([-1,1],h^{s}(\Z_{M}^{d}))$ to the Cauchy problem 
\[
\begin{cases}
i\partial_{t}v(t) = \nabla [\mathcal{L}(Q)(\xi(\phi))](v(t))\,,\\
v(0)=u\,.
\end{cases}
\]
We deduce from the stability Lemma \ref{lem:stab} (and more precisely from \eqref{eq:stab-a}) that $v(t)\in\mathcal{V}_{\alpha,s}(20\epsilon,\xi)$ for all $t$. In particular 
\[
\|v(t)\|_{h^{s}}\leq 2\|u\|_{h^{s}}\leq20\epsilon\,.
\]
To control $\widetilde{\varphi}(t)$ in $h^{s}(\Z_{M}^{d})$ we use the equation and a bootstrap: differentiating \eqref{eq:odev} in $\phi$ and applying the Leibniz rule and the chain rule gives
\begin{equation}
\label{eq:dim}
\begin{cases}
i\partial_{t}\widetilde{\varphi}(t) 
	= \nabla H^{(1)}(\xi(\phi),\varphi)(v) + \nabla H^{(2)}(\xi(\phi),\varphi)(v) + \Big[\mathrm{d}\nabla\Big[\mathcal{L}(Q)(\xi(\phi))\Big](v)\Big](\widetilde{\varphi}(t))\,,\\
	\widetilde{\varphi}(0)=0\,,
\end{cases}
\end{equation}
where we collected in $H^{(1)}$ (resp. $H^{(2)}$) the contributions when $\mathrm{d}_{\phi}$ falls on the re-centered variable $z_{\nb}(u,I(u)-\xi(\phi))$ (resp. on the coefficient of $\mathcal{L}(Q)(\xi(\phi))$), namely
\begin{align*}
	H^{(1)}(\xi(\phi),\varphi;v) 
	&= -\sum_{\nb\in\Lambda_{\alpha+1}}\sum_{|k|\leq M}\varphi_{k}\overline{\phi}_{k}
	\frac{\mathfrak{h}^{(\alpha)}(\xi(\phi))}{\Omega_{\nb}(\widetilde{\omega}(\xi(\phi)))}
	Q[\xi(\phi)]_{\nb}
	m_{k}z_{\nb-\mathbf{e}_{\mathfrak{m}}(k)}(v,I(v)-\xi)\,,\\
	H^{(2)}(\xi(\phi),\varphi ; v) 
	&= \sum_{\nb\in\Nb_{\leq2\overline{r}}} 
	[H^{(2)}(\xi,\varphi)]_{\nb}z_{\nb}(v,I(v)-\xi)\,,
\end{align*}
with, for fixed $\nb\in\Nb_{\leq2\overline{r}}$, 
\begin{equation*}
\begin{split}
&[H^{(2)}(\xi(\phi),\varphi)]_{\nb} \\
	&=  
	\sum_{|k|\leq M} 
	\varphi_{k}\overline{\phi_{k}}
	\partial_{\xi_{k}}
	\Big(
	\frac{\mathfrak{h}^{(\alpha)}(\xi(\phi))}
	{\Omega_{\nb}(\widetilde{\omega}(\xi(\phi)))}
	\Big)
	(Q[\xi(\phi)])_{\nb} 
	+ 
	\varphi_{k}\overline{\phi}_{k}
	\frac{\mathfrak{h}^{(\alpha)}(\xi(\phi))}
	{\Omega_{\nb}(\widetilde{\omega}(\xi(\phi)))}\partial_{\xi_{k}}
	(Q[\xi(\phi)])_{\nb}\,.
\end{split}
\end{equation*}
First, we easily control the last term on the right-hand side of \eqref{eq:dim} from Lemma \ref{lem:d2F}: for all $t$,
\[
\Big\|\Big[\mathrm{d}\nabla\mathcal{L}(Q)(v(t))\Big](\widetilde{\varphi}(t))\Big \|_{h^{s}}
	 \leq 
	 \epsilon^{1-\frac{1}{8}}
	 \|Q\|_{\Ysup{\alpha}}
	 \|\widetilde{\varphi}(t)\|_{h^{s}} 
	 \leq 
	 \epsilon^{1-\frac{2}{5}}
	 \|\widetilde{\varphi}(t)\|_{h^{s}} \,.
\]
Then, we deduce from the vector field estimate of Proposition \ref{prop:vec-alpha} that, for $i\in\{1,2\}$,
\begin{equation}
\label{eq:aprem}
\| \nabla H^{(i)}(\xi(\phi),\varphi)(v)\|_{h^{s}} 
	\leq  
	\|
	H^{(i)}(\varphi)
	\|_{\Ysup{\alpha}}
	N_{\alpha}^{-4s}\epsilon^{3}\|v(t)\|_{h^{s}}
	\lesssim
	\|
	H^{(i)}(\varphi)
	\|_{\Ysup{\alpha}}
	N_{\alpha}^{-4s}\epsilon^{4}
	\,.
\end{equation}
Let us first control the second contribution ($i=2$), when the derivative falls on the coefficients. According to the small divisor bounds \eqref{eq:smd0} and \eqref{eq:smd1}, we obtain that for all $\nb\in\Nb$, 
\begin{align*}
&|[H^{(2)}(\xi(\phi),\varphi)]_{\nb}| 
	\\
	\leq& 
	\|\varphi\|_{h^{s}}\|\phi\|_{h^{s}}
	\gamma(\alpha)^{-1}\epsilon^{-2}N_{\alpha+1}^{2s}
	\\ &\times \Big(
	2\epsilon^{-\frac{1}{10^{4}}}
	\gamma(\alpha)^{-1}\epsilon^{-2}N_{\alpha+1}^{2s}
	|(Q[\xi(\phi)])_{\nb}| 
	+ 
	\max_{|k|\leq M}|\partial_{\xi_{k}}(Q[\xi])_{\nb}| 
\Big) \\
	\lesssim&
	\mathrm{w}_{\nb}^{0}(\alpha)
	\|\varphi\|_{h^{s}}\|\phi\|_{h^{s}}
	\gamma(\alpha)^{-1}(\epsilon^{-2}N_{\alpha+1}^{2s})^{2}
	\Big(
	\gamma(\alpha)^{-1}\epsilon^{-\frac{1}{10^{4}}}\|Q\|_{\Ysup{\alpha}}
	+
	\|Q\|_{\Ylip{\alpha}}
	\Big)
	\\
	\lesssim_{\alpha}&
	\mathrm{w}_{\nb}^{0}(\alpha)
	\|\varphi\|_{h^{s}}
	\epsilon^{1-\frac{1}{30}}
	(\epsilon^{-2}N_{\alpha+1}^{2s})^{2}
	(
	\epsilon^{-\frac{1}{30}-\frac{1}{10^{4}}-\frac{1}{10^{3}}}+\epsilon^{-\frac{1}{15}}
	)
	\,.
\end{align*}
We deduce from the Definition \ref{def:norm} of the norms  that 
\[
 	\|H^{(2)}(\phi)\|_{\Ysup{\alpha}} 
 	\lesssim_{\alpha} 
	\|\varphi\|_{h^{s}}\epsilon^{1-\frac{1}{30}-\frac{1}{15}}(\epsilon^{-2}N_{\alpha+1}^{2s})^{2}\,,
\]
and therefore 
\begin{equation*}
\begin{split}
\|\nabla H^{(2)}(\xi(\phi),\varphi)(v)\|_{h^{s}}
	\lesssim_{\alpha} \|\varphi\|_{h^{s}}(\frac{N_{\alpha+1}}{N_{\alpha}})^{4s}\epsilon^{1-\frac{1}{30}-\frac{1}{15}} 
	&\lesssim_{\alpha}\|\varphi\|_{h^{s}}
	\epsilon^{1-\frac{1}{30}-\frac{1}{15}-\frac{1}{50}} \\
	&\leq \epsilon^{1-\frac{2}{5}}\|\varphi\|_{h^{s}}\,.
\end{split}
\end{equation*}
Similarly, we prove that
\begin{align*}
\|H^{(1)}(\xi(\phi),\varphi)\|_{\Ysup{\alpha}}
	&\leq \|\varphi\|_{h^{s}}\|\phi\|_{h^{s}}(\gamma(\alpha)^{-1}\epsilon^{-2}N_{\alpha+1}^{2s})\operatorname{D}(\alpha)\|Q\|_{\Ysup{\alpha}} \\
	&\lesssim_{\alpha}
	\|\varphi\|_{h^{s}}
	\epsilon^{1-\frac{1}{30}-\frac{1}{5}-\frac{1}{10^{3}}}(\epsilon^{-2}N_{\alpha+1}^{2s})^{2}\,,
\end{align*}
and therefore 
\begin{equation*}
\begin{split}
\|\nabla H^{(1)}(\xi(\phi),\varphi)(v)\|_{h^{s}}
	\lesssim 
	(\frac{N_{\alpha+1}}{N_{\alpha}})^{4s}\epsilon^{1-\frac{1}{30}-\frac{1}{5}-\frac{1}{10^{3}}} 
	\|\varphi\|_{h^{s}} 
	&\lesssim_{\alpha}
	\epsilon^{1-\frac{1}{30}-\frac{1}{5}-\frac{1}{10^{3}}-\frac{1}{50}} 
	\|\varphi\|_{h^{s}} \\
	&\leq 
	\epsilon^{1-\frac{2}{5}}\|\varphi\|_{h^{s}}\,.
\end{split}
\end{equation*}
Hence, we obtain from the Duhamel's integral formula that for all $t\in[-1,1]$\,,
\[
	\|\widetilde{\varphi}(t)
	\|_{h^{s}} 
	\leq 
	2\epsilon^{1-\frac{2}{5}}\|\varphi\|_{h^{s}} +\epsilon^{1-\frac{2}{5}}\Big|\int_{0}^{t}\|\widetilde{\varphi}(t')\|_{h^{s}}\mathrm{d}t'\Big|\,,
\]
which in turn gives \eqref{eq:dxi} from the Gronwall inequality. 
\end{proof}


\subsection{The iteration at scale $\alpha$}\label{sec:it-j} In this paragraph, $\alpha\in\{1,\cdots,\beta-1\}$ is fixed and we suppose that Theorem \ref{thm:nf-alpha} is proved up to scale $\alpha$. We detail the main iteration to remove the the monomials associated to the multi-indices $\nb\in\Lambda_{\alpha+1}$, in order to eventually go to scale $\alpha+1$. 

\medskip

We now suppose that  Theorem \ref{thm:nf-alpha} is proved up to the step $\alpha\in\{0,\cdots,\beta-1\}$: for all $\xi\in\mathcal{U}_{s}(\epsilon)$, there exists $\tau_{\alpha,\xi}$ such that 
\begin{equation}
\label{eq:19h}
H_{\lo}\circ\tau_{\alpha,\xi}(u) = Z_{2}^{(\alpha)}(\xi;u) + Z_{4}^{(\alpha)}(\xi;u)+Q^{(\alpha)}(\xi;u) + R^{(\alpha)}(\xi;u)\,,
\end{equation}
satisfying points $(1)-(5)$ of Theorem \ref{thm:nf-alpha}. We define 
\[
\kappa:=40r\,,
\]
and, for $j\in\{0,\cdots,\kappa\}$,
\[
\epsilon_{j,\alpha}:=\epsilon_{\alpha}-j\epsilon^{\frac{3}{2}} = 10\epsilon - (c\alpha r +j)\epsilon^{\frac{3}{2}}\,.
\]

The strategy to pass from the normal form at scale $\alpha$ to the scale $\alpha+1$ essentially consists in gaining {\it homogeneity} in $\epsilon$ by applying successively the Lie transformation introduced and analyzed in Section \ref{sec:lie}. After $\kappa$ iterations  (the number $\kappa$ is comparable to a multiple of $r$) all the terms that are not under the normal form at scale $\alpha+1$ (namely, which contain unpaired indices smaller than $N_{\alpha+1}$) eventually come with a factor $\epsilon^{2r}$, provided $\xi$ is a non-resonant parameter, and are therefore acceptable remainders. The next Proposition states this general second induction scheme (on the homogeneity), performed between two steps of the main induction scheme (on the frequency scales~$N_{\alpha}$). 
\begin{proposition}[Iteration at step $\alpha$]
\label{prop:it-j}
 For all  $\xi\in\mathcal{U}_{s}(\epsilon)$ and $j\in\{0,\cdots,\kappa\}$, there exists a symplectic transformation $\tau_{j,\alpha,\xi}$ on $h^{s}(\Z_{M}^{d})$ such that
 \[
H_{\lo}\circ\tau_{j,\alpha,\xi}(u) 
	=Z_{2}^{(j,\alpha)}(\xi;u) + Z_{4}^{(j,\alpha)}(\xi;u)+Q^{(j,\alpha)}(\xi;u)+R^{(j,\alpha)}(\xi;u)\,,
\]
where for all $j\in\{1,\cdots,\kappa\}$,
\begin{equation}
\label{eq:it-alpha-2}
	\|\mathbf{1}_{\Xi_{\alpha}}(\xi)\Pi_{\Lambda_{\alpha+1}}Q^{(j,\alpha)}\|_{\Ysup{\alpha}} \leq \epsilon^{\frac{j}{30}-\frac{1}{10^{4}}}\,.
\end{equation}
Moreover, the new Hamiltonian and $\tau_{j,\alpha,\xi}$ satisfy Theorem \ref{thm:nf-alpha} at scale $\alpha$:
\begin{enumerate}
\item The quadratic term is
\[
Z_{2}^{(j,\alpha)}(\xi;u) =  \frac{1}{2}\sum_{n\in\Z_{M}^{d}}\omega_{n}^{(j,\alpha)}(\xi)(|u_n|^{2}-\xi_{n})\,,
\]
where for all $\xi\in\mathcal{U}_{s}(\epsilon)$ and $(k,n)\in(\Z_{M}^{d})^{2}$,
\begin{align}
\label{eq:j-omega}
|\omega_{n}^{(j,\alpha)}(\xi)-\omega_{n}^{(j-1,\alpha)}(\xi)|
	&\leq 
	\epsilon^{3+\frac{1}{10}}
	N_{\alpha}^{-4s}\,,\\
\label{eq:j-domega}
|
	\partial_{\xi_{k}}
	(\omega_{n}^{(j,\alpha)}(\xi) - 	
	\omega_{n}^{(j-1,\alpha)}(\xi) )
|
	&\leq \epsilon^{1+\frac{1}{10}}N_{\alpha}^{-2s}\,.
\end{align}
\item The quartic integrable terms collected in $Z_{4}^{(j,\alpha)}\in X_{4,\mathrm{Int}}(\epsilon)$ verify 
\begin{equation}
\label{eq:bo4r}
	\max(\|Z_{4}^{(j,\alpha)}\|_{Z^{\mathrm{sup}}}
	\,,\,
	\|Z_{4}^{(j,\alpha)}\|_{Z^{\mathrm{lip}}}
	)
	\lesssim_{\alpha} \epsilon^{-\frac{1}{10^{4}}}
	\end{equation}
\item $Q^{(j,\alpha)}\in X_{\leq2\overline{r}}(\epsilon)$ operates at scale $\alpha$:
\begin{equation}
\label{eq:boQr}
\|Q^{(j,\alpha)}\|_{\Ysup{\alpha}}
	\lesssim_{\alpha}\epsilon^{-\frac{1}{10^{4}}}\,,
\quad
	\|Q^{(j,\alpha)}\|_{\Ylip{\alpha}} 
	\lesssim_{\alpha} \epsilon^{-\frac{1}{15}}\,,
\end{equation}
and for all $\xi\in\mathcal{U}_{s}(\epsilon)$,
\begin{equation}
\label{eq:c-4r}
(Q^{(j,\alpha)}[\xi])_{\nb}\neq0\quad \implies\quad \deg(\nb)\geq6\,,\quad |\sum_{n\in\Z_{M}^{d}}(k_{n}-\ell_{n})\lambda_{n}^{2}|\leq (cr)^{c\alpha r+j}\epsilon^{-c_{\ast}\nu}\,.
\end{equation}
\item If $\xi\in\Xi_{\alpha-1}$ then for all $u\in\mathcal{V}_{\alpha,s}(\epsilon_{j,\alpha},\xi)$,
\begin{equation}
\label{eq:boRr}
\|\nabla R^{(j,\alpha)}(\xi;u)\|_{h^{s}} \lesssim_{r,\alpha}\epsilon^{2r+1}\|u\|_{h^{s}}\,.
\end{equation}
\item The symplectic transformation satisfies, for all $\xi\in\mathcal{U}_{s}(\epsilon)$ and $u\in \mathcal{V}_{\alpha,s}(\epsilon_{j,\alpha},\xi)$, 
\begin{align}
\label{eq:phi-j}
\|\tau_{j,\alpha,\xi}(u) - u\|_{h^{s}}
	&\lesssim_{r,\alpha}\epsilon^{\frac{3}{4}}\|u\|_{h^{s}}\,,\\
\label{eq:d-phi-j}
\|\mathrm{d}_{u}\tau_{j,\alpha,\xi}(u)-\Id\|_{h^{s}\to h^{s}}
	&\lesssim_{r,\alpha}\epsilon^{\frac{3}{4}}\,,
\intertext{Moreover, for all 
$\phi\in\Pi_{M}B_{s}(\epsilon_{j,\alpha})$ and $u\in\mathcal{V}(\epsilon_{j,\alpha},\xi)$,}
\label{eq:d-phi-xi-j}
	\|
	\mathrm{d}_{\phi}\tau_{j,\alpha,\xi(\phi)}(u) 
	 \|_{h^{s}\to h^{s}}
	 &\lesssim_{r,\alpha}\epsilon^{1-\frac{2}{5}}\,.
\end{align}
\end{enumerate}

\end{proposition} 

Let us make a few comments.

\begin{itemize}
\item  The main idea is to remove from $Q^{(\alpha)}$ the monomials associated to $\nb\in\Lambda_{\alpha+1}$, when $\xi$ is in the non-resonant set $\Xi_{\alpha}$. Specifically, after $\kappa=40r$ steps we are able to replace $\Pi_{\Lambda_{\alpha+1}}Q^{(\alpha)}$ by a term which is a remainder when $\xi\in\Xi_{\alpha}$, as claimed in \eqref{eq:it-alpha-2}. Recalling Definition \ref{def:norm} of the norms, the condition \eqref{eq:it-alpha-2}  can be formulated this way: for all $\nb\in\Nb$,
\begin{equation}
\label{eq:it-alpha-2bis} 
\nb\in\Lambda_{\alpha+1}\quad \text{and}
	\quad \xi\in \Xi_{\alpha}
	\quad \implies\quad |(Q^{(j,\alpha)}[\xi])_{\nb}| \leq 
	\epsilon^{\frac{j}{30}-\frac{1}{10^{4}}}\mathrm{w}_{\nb}^{0}(\alpha) \,.
\end{equation}
\item According to Lemma \ref{prop:vec-alpha}, after at most $\kappa$ steps, the bound \eqref{eq:it-alpha-2} implies that $\Pi_{\Lambda_{\alpha+1}}Q^{(j,\alpha)}$ is a remainder in the sense of \eqref{eq:boR}. On the other hand, the polynomial $(\mathrm{Id}-\Pi_{\Lambda_{\alpha+1}})Q^{(j,\alpha)}$ operates at scale $\alpha+1$ according to Lemma \ref{lem:scale}  and Remark \ref{rem:scale}.
\item At each iteration $j$ we generate additional terms of type $Z_{2}$, $Z_{4}$ and $Q$, and we need to propagate the estimates for their norms. 
\end{itemize}

 \begin{proof} We proceed by induction on $j$, and start with the Hamiltonian from Theorem \ref{thm:nf-alpha} at scale $\alpha$, as written in \eqref{eq:19h}.
\medskip

\noindent{$\bullet$ \it Initialization:} The inputs are $\tau_{0,\alpha,\xi} := \tau_{\alpha,\xi}$ and
\[
\omega^{(0,\alpha)} := \omega^{(\alpha)}\,,\quad Z_{4}^{(0,\alpha)} := Z_{4}^{(\alpha)}\,,\quad Q^{(0,\alpha)}:=Q^{(\alpha)}\,,\quad R^{(0,\alpha)}= R^{(\alpha)}\,.
\]
At step $j=0$ we ask nothing more than the estimates claimed in Theorem \ref{thm:nf-alpha} at scale $\alpha$.

\medskip

\noindent {$\bullet$ \it Iteration $j\to j+1$:} 

To pass from step $j$ to step $j+1$ we need to remove from $Q^{(j,\alpha)}$ the monomials associated to multi-indices in $\Lambda_{\alpha+1}$, up to some new terms  $\epsilon^{\frac{1}{30}}$-smaller. To achieve this we consider the Lie transformation introduced and studied in subsection \ref{sec:lie}. 

\medskip

{$\triangleright$\it The cohomological equation:}  We note from the recurrence assumptions \eqref{eq:j-omega} and \eqref{eq:j-domega} that the modulated frequencies $(\omega_{n}^{(j,\alpha)})$ satisfy the stability estimate \eqref{eq:o-tilde0} and \eqref{eq:o-tilde1}, and we let, for $\xi\in\mathcal{U}_{s}(\epsilon)$, 
\[
\widetilde{\omega}_{n}(\xi) := \omega^{(j,\alpha)}(\xi)\,.
\]
We introduce the corresponding auxiliary Hamiltonian 
\begin{equation}
\label{eq:coho}
\chi(\xi) 
	:= \mathcal{L}_{\alpha,\widetilde{\omega}}(Q^{(j,\alpha)})(\xi) 
	=: \mathcal{L}(Q^{(j,\alpha)})(\xi)\,.
\end{equation}
Lemma \ref{lem:sde} provides the small divisor estimates \eqref{eq:smd0} and \eqref{eq:smd1} for the multiplication operator $\mathcal{L}$, defined in Definition \ref{def:Lie}. In addition, Lemma \ref{lem:24} gives that for all $\xi\in\mathcal{U}_{s}(\epsilon)$, 
\begin{equation}
\label{eq:cohoQ}
Q^{(j,\alpha)}(\xi;u)
	+
	\{Z_{2}^{(j,\alpha)}(\widetilde{\omega}),\chi\}(\xi;u)
	=(\operatorname{Id}-\mathfrak{h}^{(\alpha)}(\xi)\Pi_{\Lambda_{\alpha+1}})Q^{(j,\alpha)}(\xi;u)\,.
\end{equation}
Hence, when $\xi$ is non-resonant (i.e. $\xi\in\Xi_{\alpha}$, and therefore $\mathfrak{h}^{(\alpha)}(\xi)=1$) the Hamiltonian $\chi$ is tuned to solve the cohomological equation see \eqref{eq:rf3}.
In this way, the monomials associated to multi-indices in $\Lambda_{\alpha+1}$ can be filtered and removed. Once again we stress out that 
\[
\mathcal{L}(Q^{(j,\alpha)})
	=
	\mathcal{L}(\Pi_{\Lambda_{\alpha+1}}Q^{(j,\alpha)})\,.
\]
\medskip

$\triangleright${\it The new variables:} The Lie transformation was defined in Definition \ref{def:Lie}: for $t\in[-1,1]$, let 
\[
\Phi_{j,\alpha,\xi}^{t} := \Phi_{\chi}^{t}\,, 
\] 
where the auxiliary Hamiltonian $\chi=\chi(j,\alpha,\xi)$ was defined above \eqref{eq:coho}.  According to the recurrence assumption (more precisely the bound \eqref{eq:boQr}), the $\Ysup{\alpha}$-norm of $Q^{(j,\alpha)}$ is bounded by $\epsilon^{-\frac{1}{10^{3}}}$. Applying Lemma \ref{lem:stab}, we obtain that for all $\xi\in\mathcal{U}_{s}(\epsilon)$, $u\in\mathcal{V}_{\alpha,s}(10\epsilon,\xi)$ and $t\in[-1,1]$
\begin{equation}
\label{eq:55}
\|\Phi_{j,\alpha,\xi}^{t}(u) - u \|_{h^{s}} 
	\leq \epsilon^{\frac{3}{2}}\|u(t)\|_{h^{s}}N_{\alpha}^{-2s}
	\,.
\end{equation}
In particular, recalling that $\epsilon_{j,\alpha}:=10\epsilon-(c\alpha r +j)\epsilon^{\frac{3}{2}}$, we have that for all $\xi\in\mathcal{U}_{s}(\epsilon)$,
\begin{equation}
\label{eq:inc-v}
\Phi_{j,\alpha,\xi}^{t}\Big(\mathcal{V}_{\alpha,s}(\epsilon_{j+1,\alpha},\xi)
	\Big)
	\subset \mathcal{V}_{\alpha,s}(\epsilon_{j,\alpha},\xi)\,.
\end{equation}
Corollary \ref{cor:dphi} gives that for all $\xi\in\mathcal{U}_{s}(\epsilon)$ and $u\in\mathcal{V}_{\alpha,s}(10\epsilon,\xi)$,
\begin{align}
\label{eq:33}
	\|
\mathrm{d}_{u}\Phi_{j,\alpha,\xi}^{t}(u) - \operatorname{Id}
	\|_{h^{s}\to h^{s}}
	\leq \epsilon^{\frac{3}{4}}\,,
\intertext{and Corollary \ref{cor:dxi} gives that for all $\phi\in\Pi_{M}B_{s}(10\epsilon)$ and $u\in\mathcal{V}_{\alpha,s}(10\epsilon,\xi(\phi))$,}
\label{eq:34}
	\|
	\mathrm{d}_{\phi}
	\Phi_{j,\alpha,\xi(\phi)}^{t}(u)
	\|_{h^{s}\to h^{s}}
	\lesssim\epsilon^{1-\frac{2}{5}}\,.
\end{align}
We now define the new transformation, at scale $j+1$, by 
\[
\tau_{j+1,\alpha,\xi} : = \tau_{j , \alpha, \xi}\circ \Phi_{j,\alpha,\xi}^{1}\,,
\]
and we will deduce from the above bounds and from the recurrence assumption the bounds \eqref{eq:phi-j}, \eqref{eq:d-phi-j} and \eqref{eq:d-phi-xi-j} at step $j+1$. According to \eqref{eq:inc-v} and the recurrence assumption, we have that 
\begin{equation}
\label{eq:sam}
\tau_{j+1,\alpha,\xi}(\mathcal{V}_{\alpha,s}(\epsilon_{j+1,\alpha},\xi))
	\subset
	\tau_{j,\alpha,\xi}
	(
	\mathcal{V}_{\alpha,s}(\epsilon_{j,\alpha}, \xi)
	)
	\subset \mathcal{V}_{\alpha,s}(10\epsilon,\xi)\,.
\end{equation}
The bound \eqref{eq:phi-j} at scale $j+1$ follows from the triangle inequality:\[
	\|\tau_{j+1,\alpha,\xi}(u)
		-u\|_{h^{s}}
	\leq 
	\|\tau_{j,\alpha,\xi}(\Phi_{j,\alpha,\xi}^{1}(u)) - \Phi_{j,\alpha,\xi}^{1}(u)\|_{h^{s}}
	+
	 \|\Phi_{j,\alpha,\xi}^{1}(u)
	 	- u\|_{h^{s}}\,.
\]
When $u\in\mathcal{V}_{\alpha,s}(\epsilon_{\alpha,j+1},\xi)$, 
 the property \eqref{eq:inc-v} allows to apply the estimate \eqref{eq:phi-j} at step $j$ and we obtain from the bound \eqref{eq:55} that   
\[
	\|\tau_{j+1,\alpha,\xi}(u)
		-u\|_{h^{s}}	
	\lesssim_{r,\alpha,j}\epsilon^{\frac{3}{4}}\|\Phi_{j,\alpha,\xi}^{1}(u)\|_{h^{s}}
	+
	\epsilon^{\frac{3}{2}}\|u\|_{h^{s}}N_{\alpha}^{-2s} 
	\lesssim_{r,\alpha,j}
	\epsilon^{\frac{3}{4}}\|u\|_{h^{s}}\,.
\]
This shows \eqref{eq:phi-j}. The bound \eqref{eq:d-phi-j} is proved similarly, but using also the chain rule:
\[
\operatorname{d}\tau_{j+1,\alpha,\xi}(u) - \operatorname{Id} 
	= (\operatorname{d}\tau_{j,\alpha,\xi}(\Phi_{j,\alpha,\xi}^{1}(u))-\mathrm{Id})
	\circ
	\mathrm{d}\Phi_{j,\alpha,\xi}^{1}(u) 
	+ 
	\mathrm{d}
	\Phi_{j,\alpha,\xi}^{1}(u)
	- \operatorname{Id}\,. 
\]
Thanks to \eqref{eq:sam}, we can apply the recurrence bound \eqref{eq:d-phi-j} at step $j$ together with \eqref{eq:33} to obtain 
\begin{align*}
	&\|
	\operatorname{d}\tau_{j+1,\alpha,\xi}(u) - \operatorname{Id}
	\|_{h^{s}\to h^{s}} \\
	&\leq
	\|
	\operatorname{d}\tau_{j,\alpha,\xi}(\Phi_{j,\alpha,\xi}^{1}(u)) - \operatorname{Id}
	\|_{h^{s}\to h^{s}}
	\|
	 \mathrm{d}\Phi_{j,\alpha,\xi}^{1}(u)
	\|_{h^{s}\to h^{s}}
	+
	\|
	 \mathrm{d}\Phi_{j,\alpha,\xi}^{1}(u)
	- \operatorname{Id}
	\|_{h^{s}\to h^{s}} \\
	&\lesssim_{r,\alpha,j}\epsilon^{\frac{3}{4}}\|
	 \mathrm{d}_{u}\Phi_{j,\alpha,\xi}^{1}(u)
	\|_{h^{s}\to h^{s}} + \epsilon^{\frac{3}{4}}
	\lesssim_{r,\alpha,j} 
	\epsilon^{\frac{3}{4}}\,,
\end{align*}
which proves \eqref{eq:d-phi-j} at step $j+1$. The proof of  \eqref{eq:d-phi-xi-j} goes along the same lines and we do not give the details. 

\medskip

{$\triangleright$\it The new expansion:} We do the Taylor expansion at $t=0$ of the Hamiltonian in the new coordinates 
 \[
 H_{\lo}^{(j+1,\alpha)}(\xi;u):=H_{\lo}(\xi;\tau_{j,\alpha,\xi}\circ\Phi_{j,\alpha,\xi}^{1}(u))\,.
 \]
 Recall the standard notation 
\[
\operatorname{ad}_{\chi}^{0}(H)=H\,,\quad \operatorname{ad}^{1}_{\chi}(H) = \{H,\chi\}\,,\quad \operatorname{ad}^{\ell}_{\chi}(H) = \{\operatorname{ad}_{\chi}^{\ell-1}(H),\chi\}\,,\quad \ell\geq1\,.
\]
First, we expand 
\begin{multline}
\label{eq:exp2}
	Z_{2}^{(j,\alpha)}(\xi;\Phi_{j,\alpha,\xi}^{1}(u)) 
		= Z_{2}^{(j,\alpha)}(\xi;u) + \{Z_{2}^{(j,\alpha)},\chi\}(\xi;u) +  \sum_{\ell=2}^{\kappa-1}\frac{1}{\ell!}\operatorname{ad}_{\chi}^{\ell}(Z_{2}^{(j,\alpha)})(\xi;u)\\
		+ \int_{0}^{1}\frac{(1-t)^{\kappa-1}}{(\kappa-1)!}\operatorname{ad}_{\chi}^{\kappa}(Z_{2}^{(j,\alpha)})(\xi;\Phi_{j,\alpha,\xi}^{t}(u))\mathrm{d}t\,, 
\end{multline}
\begin{multline}
\label{eq:exp4}
	Z_{4}^{(j,\alpha)}(\xi;\Phi_{j,\alpha,\xi}^{1}(u))
		= Z_{4}^{(j,\alpha)}(\xi;u)  + \sum_{\ell=1}^{\kappa-1}\frac{1}{\ell!}\operatorname{ad}_{\chi}^{\ell}(Z_{4}^{(j,\alpha)})(\xi;u)\\
		+ \int_{0}^{1}\frac{(1-t)^{\kappa-1}}{(\kappa-1)!}\operatorname{ad}_{\chi}^{\kappa}(Z_{4}^{(j,\alpha)})(\xi;\Phi_{j,\alpha,\xi}^{t}(u))\mathrm{d}t\,,
\end{multline}
\begin{multline}
\label{eq:expQ}
	Q^{(j,\alpha)}(\xi;\Phi_{j,\alpha,\xi}^{1}(u))
		= Q^{(j,\alpha)} +  \sum_{\ell=1}^{\kappa-1}\frac{1}{\ell!}\operatorname{ad}_{\chi}^{\ell}(Q^{(j,\alpha)})(\xi;u) \\
		+ \int_{0}^{1}\frac{(1-t)^{\kappa-1}}{(\kappa-1)!}\operatorname{ad}_{\chi}^{\kappa}(Q^{(j,\alpha)})(\xi;\Phi_{j,\alpha,\xi}^{t}(u))\mathrm{d}t\,.
\end{multline}
We collect the integral remainders in 
\[
I^{(j,\alpha)}(\xi;u) := \int_{0}^{1}\frac{(1-t)^{\kappa}}{\kappa!}\operatorname{ad}_{\chi}^{\kappa}\Big(Z_{2}^{(j,\alpha)} + Z_{4}^{(j,\alpha)}+ Q^{(j,\alpha)}\Big)(\xi;\Phi_{j,\alpha,\xi}^{t}(u))\mathrm{d}t\,,
\]
and the higher order polynomial terms in 
\[
P^{(j,\alpha)}(\xi;u) := \{Z_{4}^{(j,\alpha)}+Q^{(j,\alpha)},\chi\}(\xi;u)  
	+ \sum_{\ell=2}^{\kappa-1}\frac{1}{\ell!}\operatorname{ad}_{\chi}^{\ell}\Big(Z_{2}^{(j,\alpha)}+Z_{4}^{(j,\alpha)}+Q^{(j,\alpha)}\Big)(\xi;u)\,.
\]
Note that $P^{(j,\alpha)}\in X(\epsilon)$ and, for all $\xi\in\mathcal{U}_{s}(\epsilon)$, $\chi(\xi)$ has degree less than $2\overline{r}$, we have for all $\xi\in\mathcal{U}_{s}(\epsilon)$, 
\begin{equation}
\label{eq:d-P}
\deg(P^{(j,\alpha)}(\xi))\leq  2\overline{r}+\kappa(2\overline{r}-2) \leq 2\overline{r}^{2}\,.
\end{equation}
We obtain the expansion  
\begin{multline}
\label{eq:expansion2}
H_{\lo}^{(j+1,\alpha)}(\xi;u)
	= Z_{2}^{(j,\alpha)}(\xi;u) + Z_{4}^{(j,\alpha)}(\xi;u) +(\mathrm{Id} -\mathfrak{h}^{(\alpha)}(\xi)\Pi_{\Lambda_{\alpha+1}})Q^{(j,\alpha)}(\xi;u) \\ 
	+ P^{(j,\alpha)}(\xi;u) + I^{(j,\alpha)}(\xi;u) + R^{(j,\alpha)}(\xi;\Phi_{j,\alpha,\xi}^{1}(u))\,.
\end{multline}
We recall that when $\xi\in\Xi_{\alpha}$ then $\mathfrak{h}^{(\alpha)}(\xi)=1$, which removes $\Pi_{\Lambda_{\alpha+1}}Q^{(j,\alpha)}$ from $H_{\lo}^{(j+1,\alpha)}$. 

\medskip

Let us now use the multilinear estimates obtained in Proposition \ref{prop:bracket0} and the recurrence assumption to control the norms of the terms $P^{(j,\alpha)}$ and $I^{(j,\alpha)}$ generated in the expansions. We start with \eqref{eq:exp2}. When $\ell=1$, recall that 
\begin{equation}
\label{eq:adchi2}
\operatorname{ad}_{\chi}(Z_{2}^{(j,\alpha)})(\xi;u)
	= -\mathfrak{h}^{(\alpha)}(\xi)\sum_{\nb\in\Lambda_{\alpha+1}}
	(Q^{(j,\alpha)}[\xi])_{\nb}z_{\nb}(u,I(u)-\xi)\,.
\end{equation}
In particular, 
\[
\|\operatorname{ad}_{\chi}(Z_{2}^{(j,\alpha)})\|_{\Ysup{\alpha}} \leq\|\Pi_{\Lambda_{\alpha+1}}Q^{(j,\alpha)}\|_{\Ysup{\alpha}}\,.
\]
According to the recurrence assumptions \eqref{eq:boQr} and \eqref{eq:it-alpha-2}
\begin{equation}
\label{eq:foden}
	\|
	\Pi_{\Lambda_{\alpha+1}}Q^{(j,\alpha)}
	\|_{\Ysup{\alpha}} 
	\lesssim\epsilon^{-\frac{1}{10^{4}}}
	\,,\quad
	\|\mathbf{1}_{\Xi_{\alpha}}(\xi)\Pi_{\Lambda_{\alpha+1}}Q^{(j,\alpha)}\|_{\Ysup{\alpha}}
	\leq\epsilon^{\frac{j}{30}-\frac{1}{10^{4}}}\,.
\end{equation}
We deduce from Corollary \ref{cor:ad} that for all $\ell\in\{2,\cdots,\kappa\}$,
\begin{align}
\nonumber
\|\operatorname{ad}_{\chi}^{\ell}(Z_{2}^{(j,\alpha)})\|_{\Ysup{\alpha}} 
	&\leq \epsilon^{\frac{3}{2}}
	\|\operatorname{ad}_{\chi}^{\ell-1}(Z_{2}^{(j,\alpha)})\|_{\Ysup{\alpha}}
	\|\Pi_{\Lambda_{\alpha+1}}Q^{(j,\alpha)}\|_{\Ysup{\alpha}}
	\\
	\nonumber
	&\leq \epsilon^{\frac{3}{2}(\ell-1)}
	\|\operatorname{ad}_{\chi}(Z_{2}^{(j,\alpha)})\|_{\Ysup{\alpha}} 
	\|\Pi_{\Lambda_{\alpha+1}}Q^{(j,\alpha)}\|_{\Ysup{\alpha}}^{\ell-1}
	\\
	&\lesssim_{\ell}\epsilon^{\ell(\frac{3}{2}-\frac{1}{10^{4}})-\frac{3}{2}-\frac{1}{10^{4}}}\,,
	\label{eq:ad2}
\intertext{and, for all $\ell\in\{2,\cdots,\kappa\}$,}
\label{eq:ad2g}
\|\1_{\Xi_{\alpha}}(\xi)\operatorname{ad}_{\chi}^{\ell}(Z_{2}^{(j,\alpha)})\|_{\Ysup{\alpha}} 
	&\lesssim_{\ell}
	\epsilon^{\ell(\frac{3}{2}-\frac{1}{10^{4}}+\frac{j}{30})-\frac{3}{2}-\frac{1}{10^{4}}}\,.
\end{align}
Similarly, we obtain form Corollary \ref{cor:ad} that for $\ell\in\{1,\cdots,\kappa\}$,
\begin{equation}
\begin{split}
\|\operatorname{ad}_{\chi}^{\ell}(Z_{4}^{(j,\alpha)})\|_{\Ysup{\alpha}} 
	&\leq \epsilon^{\frac{1}{15}+\frac{3}{2}(\ell-1)}
	\|Z_{4}^{(j,\alpha)}\|_{Z^{\sup}}
	\|
	\Pi_{\Lambda_{\alpha+1}}Q^{(j,\alpha)}
	\|_{\Ysup{\alpha}}^{\ell} \\
	&\lesssim_{r}
	\epsilon^{\ell(\frac{3}{2}-\frac{1}{10^{4}})+\frac{1}{15}-\frac{3}{2}-\frac{1}{10^{4}}}\,,
	\label{eq:ad4}
\end{split}
	\end{equation}
and for all $\ell\in\{2,\cdots,\kappa\}$,
\begin{equation}
\label{eq:ad4g}
\|\mathbf{1}_{\Xi_{\alpha}}(\xi)\operatorname{ad}_{\chi}^{\ell}(Z_{4}^{(j,\alpha)})\|_{\Ysup{\alpha}} 
	\lesssim
	\epsilon^{\ell(\frac{3}{2}-\frac{1}{10^{4}}+\frac{j}{30})+\frac{1}{15}-\frac{3}{2}-\frac{1}{10^{4}}}\,.
\end{equation}
Finally, for all $\ell\in\{1,\cdots,\kappa\}$, 
\begin{equation}
\|\operatorname{ad}_{\chi}^{\ell}(Q^{(j,\alpha)})\|_{\Ysup{\alpha}} 
	\leq \epsilon^{\frac{3}{2}\ell}\|\Pi_{\Lambda_{\alpha+1}}Q^{(j,\alpha)}(\xi)\|_{\Ysup{\alpha}}^{\ell}\|Q^{(j,\alpha)}(\xi)\|_{\Ysup{\alpha}}
	\lesssim_{r}\epsilon^{\ell(\frac{3}{2}-\frac{1}{10^{4}})-\frac{1}{10^{4}}}\,,
\label{eq:adQ}
\end{equation}
and, for all $\ell\in\{1,\cdots,\kappa\}$,
\begin{equation}
\label{eq:adQg}
\|\mathbf{1}_{\Xi_{\alpha}}(\xi)\operatorname{ad}_{\chi}^{\ell}(Q^{(j,\alpha)})\|_{\Ysup{\alpha}} 
	\leq
	\epsilon^{\ell(\frac{3}{2}-\frac{1}{10^{4}}+\frac{j}{30})-\frac{1}{10^{4}}}\,.
\end{equation}
Collecting these contributions we deduce that
\begin{align}
\label{eq:Ptilde}
\|P^{(j,\alpha)}\|_{\Ysup{\alpha}}
	&\lesssim_{r}\epsilon^{\frac{1}{15}-\frac{2}{10^{4}}} \leq \epsilon^{\frac{1}{20}}\,,
\intertext{and}
\label{eq:Ptilde2}
\|\mathbf{1}_{\Xi_{\alpha}}(\xi)P^{(j,\alpha)}\|_{\Ysup{\alpha}}
	&\leq\epsilon^{\frac{1}{20}+\frac{j}{30}}\,.
\end{align}
Corollary \ref{cor:ad} provides the analog estimates for the $\Ylip{\alpha}$ norms of the different terms (but this time we do not need the extra gain when $\xi\in\Xi_{\alpha}$). We deduce form the Leibniz rule that 
\[
\|\operatorname{ad}_{\chi}(Z_{2}^{(j,\alpha)})\|_{\Ylip{\alpha}}\leq\epsilon^{-\frac{1}{15}}\,.
\] 
Indeed, for $|k|\leq M$ and $\nb\in\Lambda_{\alpha+1}$,
\[
|\mathfrak{h}^{(\alpha)}(\xi)\partial_{\xi_{k}}(Q^{(j,\alpha)}[\xi])_{\nb}|
	\leq
	\|Q^{(j,\alpha)}\|_{\Ylip{\alpha}}\mathrm{w}_{\nb}^{1}(\alpha)\,,
\]
and, on the other hand, when $\partial_{\xi_{k}}$ falls on the truncation function $\mathfrak{h}^{(\alpha)}$ then we use the bound \eqref{eq:dhf} (from the proof of the small divisor bound) to obtain
\begin{align}
\nonumber
|\partial_{\xi_{k}}\mathfrak{h}^{(\alpha)}(\xi)Q^{(j,\alpha)}_{\nb}(\xi)|
	&\leq 
	\epsilon^{-\frac{1}{10^{4}}}\gamma(\alpha)^{-1}\epsilon^{-2}N_{\alpha+1}^{2s}
	\|Q^{(j,\alpha)}\|_{\Ysup{\alpha}}\mathrm{w}_{\nb}^{0}(\alpha) \\
	\nonumber
	&\lesssim\epsilon^{-\frac{2}{10^{4}}}\gamma(\alpha)^{-1}
	(\epsilon^{-1}\eta)^{2}
	(\frac{N_{\alpha+1}}{N_{\alpha}})^{2s}
	\mathrm{w}_{\nb}^{1}(\alpha) \\
	\nonumber
	&\lesssim_{\alpha}
	\epsilon^{-\frac{2}{10^{4}}-\frac{1}{30}-\frac{1}{50}-\frac{1}{100}} \mathrm{w}_{\nb}^{1}(\alpha)
	\leq
	\epsilon^{-\frac{1}{15}}\mathrm{w}_{\nb}^{1}(\alpha)\,.
	\label{eq:toulon}
\end{align}
In particular $\mathrm{ad}_{\chi}(Z_{2}^{(j,\alpha)})$ verifies the assumptions \eqref{eq:as-14} and \eqref{eq:as-15} of Corollary \ref{cor:ad}, and we deduce by recurrence that for all $\ell\geq2$, 
\begin{equation}
\label{eq:adZ1}
\|\mathrm{ad}_{\chi}^{\ell}(Z_{2}^{(j,\alpha)})\|_{\Ylip{\alpha}}
	=
	\|
	\mathrm{ad}_{\chi}^{\ell-1}(\mathrm{ad}_{\chi}(Z_{2}^{(j,\alpha)}))
	\|_{\Ylip{\alpha}}
	\leq 
	\epsilon N_{\alpha}^{-2s}\,.	
\end{equation}
Similarly, using Corollary \ref{cor:ad} we show that 
\[
\|\operatorname{ad}_{\chi}(Z_{4}^{(j,\alpha)})\|_{\Ylip{\alpha}} \leq \epsilon^{\frac{1}{50}}\,, 
\]
and therefore, for all $\ell\in\{2,\cdots,\kappa\}$,
\begin{equation}
\label{eq:adZ41}
\|\operatorname{ad}_{\chi}^{\ell}(Z_{4}^{(j,\alpha)})\|_{\Ylip{\alpha}} \leq \epsilon N_{\alpha}^{-2s}\,, 
\end{equation}
Finally, we deduce from Corollary \ref{cor:ad} that for all $\ell\geq1$, 
\begin{equation}
\label{eq:adQ1}
\|\mathrm{ad}_{\chi}^{\ell}(Q^{(j,\alpha)})\|_{\Ylip{\alpha}}
	\leq \epsilon N_{\alpha}^{-2s}\,.
\end{equation}
In particular, we proved that
\begin{equation}
\label{eq:Ptilde1}
\|P^{(j,\alpha)}\|_{\Ylip{\alpha}}\lesssim \epsilon^{\frac{1}{50}}\,.
\end{equation}

 Let us now explain how to organize the new terms generated from the expansions \eqref{eq:exp2}, \eqref{eq:exp4} and \eqref{eq:expQ}, which are collected in $P^{(j,\alpha)}$, and the new remainder terms. 
\medskip

\noindent{$\triangleright$ \it New terms of type $Q$:}
We let
\[
Q^{(j+1,\alpha)}(u) := (\mathrm{Id}-\mathfrak{h}^{(\alpha)}\Pi_{\Lambda_{\alpha+1}})Q^{(j,\alpha)}(u) + \Pi_{6\leq\deg\leq\overline{r}}P^{(j,\alpha)}(u)\,.
\]
Since $\mathfrak{h}^{(\alpha)}(\xi)=1$ for all $\xi\in\Xi_{\alpha}$, we have that  for all $\xi\in\Xi_{\alpha}$,
\[
\Pi_{\Lambda_{\alpha+1}}Q^{(j+1,\alpha)}(\xi;u) 
	= \Pi_{{\Lambda}_{\alpha+1}}\Pi_{6\leq\deg\leq\overline{r}}P^{(j,\alpha)}(\xi;u)\,,
\]
which in turn implies from the bound \eqref{eq:Ptilde2} that 
\[
	\|\mathbf{1}_{\Xi_{\alpha}}(\xi)\Pi_{\Lambda_{\alpha+1}}Q^{(j+1,\alpha)}(\xi) \|_{\Ysup{\alpha}} 
	\leq \epsilon^{\frac{1}{20}+\frac{j}{30}} \leq \epsilon^{\frac{j+1}{30}}\,.
\]
This proves \eqref{eq:it-alpha-2} at step $j+1$. We now show that $Q^{(j+1,\alpha)}$ still satisfies \eqref{eq:boQr}. On the one hand, since $\mathfrak{h}^{(\alpha)}$ is valued in $[0,1]$ we have 
\[
\|(\mathrm{Id}-\mathfrak{h}^{(\alpha)}(\xi)\Pi_{\Lambda_{\alpha+1}})Q^{(j,\alpha)}\|_{\Ysup{\alpha}} \leq \|Q^{(j,\alpha)}\|_{\Ysup{\alpha}}\,,
\]
and we conclude from the bound \eqref{eq:Ptilde} that 
\[
\|Q^{(j+1,\alpha)}\|_{\Ysup{\alpha}} 
	\leq \|Q^{(j,\alpha)}\|_{\Ysup{\alpha}} 
	+\|P^{(j,\alpha)}\|_{\Ysup{\alpha}}
	\leq 
	\|Q^{(j,\alpha)}\|_{\Ysup{\alpha}} + \epsilon^{\frac{1}{20}} \lesssim_{\alpha}\epsilon^{-\frac{1}{10^{4}}}\,.
\]
On the other hand, we have from the chain rule and bound that 
\[
\|(\mathrm{Id}-\mathfrak{h}^{(\alpha)}(\xi)\Pi_{\Lambda_{\alpha+1}})Q^{(j,\alpha)}\|_{\Ylip{\alpha}} 
	\leq \|Q^{(j,\alpha)}\|_{\Ylip{\alpha}} 
	+ \epsilon^{-\frac{1}{15}}\,.
\]
We deduce from the bound \eqref{eq:Ptilde1} that 
\[
\|Q^{(j+1,\alpha)}\|_{\Ylip{\alpha}} \leq \|Q^{(j,\alpha)}\|_{\Ylip{\alpha}} 
	+ \epsilon^{-\frac{1}{15}} + \|P^{(j,\alpha)}\|_{\Ylip{\alpha}}
	\leq 
	\|Q^{(j,\alpha)}\|_{\Ylip{\alpha}} 
	+
	\epsilon^{-\frac{1}{15}}
	+
	\epsilon^{\frac{1}{50}} \lesssim_{\alpha}\epsilon^{-\frac{1}{15}}\,,
\]
which gives the bound \eqref{eq:boQr}. Finally, to prove \eqref{eq:c-4r}, we observe that the resonance function of the unmodulated frequencies is additive under the Poisson bracket of two monomials. Since there are at most $\kappa$ Poisson brackets in the new terms, we obtain \eqref{eq:c-4r}. 
\medskip

\noindent{$\triangleright$ \it New modulated frequencies:} The polynomial $P^{(j,\alpha)}$  contains terms of degree $2$, corresponding to $\nb''=\mathbf{e}_{\mathfrak{m}}(n)$ for $n\in\Z_{M}^{d}$:
\[
\Pi_{\deg=2}P^{(j,\alpha)}(\xi;u)
	:= \sum_{n\in\Z_{M}^{d}}(P^{(j,\alpha)}[\xi])_{\mathbf{e}_{\mathfrak{m}}(n)}(|u_n|^{2}-\xi_{n})\,.
\]
These terms modulate the frequencies. For $n\in\Z_{M}^{d}$, we set 
\[
\omega_{n}^{(j+1,\alpha)}(\xi) 
	:= \omega_{n}^{(j,\alpha)}(\xi) 
	+ (P^{(j,\alpha)}[\xi])_{\mathbf{e}_{\mathfrak{m}}(n)}(j,\xi)\,.
\]
By Definition \ref{def:norm} and from the bound \eqref{eq:Ptilde}, of the $\Ysup{\alpha}$-norm we have, for all $n\in\Z_{M}^{d}$ and $\xi\in\mathcal{U}_{s}(\epsilon)$
\[
|(P^{(j,\alpha)}[\xi])_{\mathbf{e}_{\mathfrak{m}}(n)}| 
	\leq \mathrm{w}_{\mathbf{e}_{\mathfrak{m}}(n)}^{0}(\alpha)\|P^{(j,\alpha)}\|_{\Ysup{\alpha}}
	\leq \epsilon^{\frac{1}{20}}
	N_{\alpha}^{-6s}\eta^{6}\mathrm{D}(\alpha)
	\leq \epsilon^{3+\frac{1}{10}}N_{\alpha}^{-4s}\,.
\] 
Similarly we deduce from the bound \eqref{eq:Ptilde1} that for $(k,n)\in(\Z_{M}^{d})^{2}$ and $\xi\in\mathcal{U}_{s}(\epsilon)$,
\[
|\partial_{\xi_{k}} (P^{(j,\alpha)}[\xi])_{\mathbf{e}_{\mathfrak{m}}(n)}|
	\leq 
	\mathrm{w}_{\mathbf{e}_{\mathfrak{m}}(n)}^{1}(\alpha)
	\|P^{(j,\alpha)}\|_{\Ylip{\alpha}} 
	\lesssim \epsilon^{\frac{1}{50}}\eta^{4}N_{\alpha}^{-4s}\operatorname{D}(\alpha)
	\leq
	\epsilon^{1+\frac{1}{10}}N_{\alpha}^{-2s}\,.
\]
This prove the stability estimates \eqref{eq:j-omega} and \eqref{eq:j-domega} for the frequencies.

\medskip

\noindent{$\triangleright$ \it New integrable quartic terms:} The polynomial $P^{(j,\alpha)}$ also contains terms of degree 4 corresponding to monomials $\mathbf{e}_{\mathfrak{m}}(n_1)+\mathbf{e}_{\mathfrak{m}}(n_2)$ for $n_{1},n_{2}\in\Z_{M}^{d}$. To estimate this contributions, and especially their $\Ylip{\alpha}$- norms, we need to be slightly more precise. Actually, the terms
\[
\operatorname{ad}_{\chi}^{\ell}(Z_{2}^{(j,\alpha)})\,,\quad \operatorname{ad}_{\chi}^{\ell}(Z_{4}^{(j,\alpha)})
\]
generate new integrable terms only when $\ell\geq2$. Therefore, we let
\begin{equation}
\label{eq:4+1}
Z_{4}^{(j+1,\alpha)} := Z_{4}^{(j,\alpha)} + \Pi_{\deg=4}\widetilde{P}^{(j,\alpha)}\,,
\end{equation}
where 
\[
\widetilde{P}^{(j,\alpha)} 
	:= P^{(j,\alpha)} - \{Z_{4}^{(j,\alpha)},\chi\}
	=
	\mathrm{ad}_{\chi}(Q^{(j,\alpha)})
	+\sum_{\ell=2}^{\kappa-1}
	\frac{1}{\ell!}
	\mathrm{ad}_{\chi}^{\ell}
	\Big(
	Z_{2}^{(j,\alpha)}+Z_{4}^{(j,\alpha)}+Q^{(j,\alpha)}
	\Big)\,.
\]
The reason for introducing $\widetilde{P}^{(j,\alpha)}$ is that the estimates \eqref{eq:Ptilde} and \eqref{eq:Ptilde1} are slightly better for $\widetilde{P}(j,\xi)$ than they are for $P(j,\xi)$ (we gain an important factor $\epsilon$ for the $\Ylip{\alpha}$-norm): indeed, we actually proved that
\[
\|\widetilde{P}(j,\xi)\|_{\Ysup{\alpha}}	
	\leq
	\epsilon^{\frac{3}{2}-\frac{2}{10^{4}}}\,,
	\quad 
	\|\widetilde{P}(j,\xi)\|_{\Ylip{\alpha}}
	\leq
	\epsilon N_{\alpha}^{-2s}\,.
\]
We write
\begin{multline*}
\Pi_{\deg=4}\widetilde{P}^{(j,\alpha)}(\xi;u)
	:= 
	2\sum_{n_{1},n_{2}\in\Z_{M}^{d}}
	(\widetilde{P}^{(j,\alpha)}[\xi])_{\mathbf{e}_{\mathfrak{m}}(n_1)+\mathbf{e}_{\mathfrak{m}}(n_2)}
	(|u_{n_{1}}|^{2}-\xi_{n_{1}})(|u_{n_{2}}|^{2}-\xi_{n_{2}})
	\\
	-
	\sum_{n\in\Z_{M}^{d}}
	(\widetilde{P}^{(j,\alpha)}[\xi])_{\mathbf{e}_{\mathfrak{m}}(n)+\mathbf{e}_{\mathfrak{m}}(n)}
	(|u_{n}|^{2}-\xi_{n})^{2}
\,.
\end{multline*}
For all $\xi\in\mathcal{U}_{s}(\epsilon)$ and $(n_{1},n_{2})\in(\Z_{M}^{d})^{2}$, we have from the definition of the $\Ysup{\alpha}$-norm that  
\begin{multline*}
|(\widetilde{P}^{(j,\alpha)}[\xi])_{\mathbf{e}_{\mathfrak{m}}(n_1)+\mathbf{e}_{\mathfrak{m}}(n_2)}| 
	\leq 
	\mathrm{w}_{\mathbf{e}_{\mathfrak{m}}(n_1)+\mathbf{e}_{\mathfrak{m}}(n_2)}^{0}(\alpha)
	\|\widetilde{P}^{(j,\alpha)}\|_{\Ysup{\alpha}} 
	\\
	\leq
	\epsilon^{\frac{3}{2}-\frac{2}{10^{4}}}
	N_{\alpha}^{-6s}\eta^{6}
	\mathrm{D}(\alpha)^{2} 
	\leq\epsilon^{\frac{3}{2}-\frac{2}{5}-\frac{2}{10^{4}}}N_{\alpha}^{-2s}\eta^{2}\,.
\end{multline*}
Similarly, for all $k\in\Z_{M}^{d}$,
\begin{multline*}
|\partial_{\xi_{k}}
	(\widetilde{P}^{(j,\alpha)}[\xi])_{\mathbf{e}_{\mathfrak{m}}(n_1)+\mathbf{e}_{\mathfrak{m}}(n_2)}| 
	\leq 
	\mathrm{w}_{\mathbf{e}_{\mathfrak{m}}(n_1)+\mathbf{e}_{\mathfrak{m}}(n_2)}^{1}(\alpha)
	\|\widetilde{P}^{(j,\alpha)}\|_{\Ylip{\alpha}} \\
	\leq \epsilon N_{\alpha}^{-2s}N_{\alpha}^{-4s}\eta^{4}\mathrm{D}(\alpha)^{2}
	\leq 
	\epsilon^{1-\frac{2}{5}}N_{\alpha}^{-2s}\,.
\end{multline*}
This shows that for all $\xi\in\mathcal{U}_{s}(\epsilon)$,
\[
\|\Pi_{\deg=4}\widetilde{P}^{(j,\alpha)}\|_{\Ysup{\alpha}} 
	\leq \epsilon^{\frac{3}{2}} N_{\alpha}^{-2s}\,,
	\quad 
	\|\Pi_{\deg=4}\widetilde{P}^{(j,\alpha)}\|_{\Ylip{\alpha}} 
	\leq \epsilon^{1-\frac{2}{5}}N_{\alpha}^{-2s}\,,
\]
and therefore we have from the expression \eqref{eq:4+1} of $Z_{4}^{(j+1,\alpha)}$ that under the recurrence assumption~\eqref{eq:bo4r}
\begin{align*}
\|Z_{4}^{(j+1,\alpha)}\|_{Z^{\mathrm{sup}}}
	&\leq
	\|Z_{4}^{(j,\alpha)}\|_{Z^{\mathrm{sup}}}+
	 \epsilon^{\frac{3}{2}}N_{\alpha}^{-2s}
	 \lesssim_{\alpha}\epsilon^{-\frac{1}{10^{4}}}\,,
\intertext{and}
 	\|Z_{4}^{(j+1,\alpha)}\|_{Z^{\mathrm{lip}}}
 	&\leq 
	\|Z_{4}^{(j,\alpha)}\|_{Z^{\mathrm{lip}}}+\epsilon^{1-\frac{2}{5}}N_{\alpha}^{-2s}
	\lesssim_{\alpha}\epsilon^{-\frac{1}{10^{4}}}\,.
\end{align*}

\medskip

\noindent{$\triangleright$ \it New remainder terms:} Four types of terms contribute to the remainder: the former remainder in the new coordinates, the terms with degree larger than $\overline{r}$, the integral remainders in the Taylor expansions and the constants. We define
\begin{equation*}
\begin{split}
R^{(j+1,\alpha)}(\xi;u) :=& R^{(j,\alpha)}(\xi;\Phi_{j,\alpha,\xi}^{1}(u)) 
	+ \Pi_{\deg>\overline{r}}P^{(j,\alpha)}(\xi;u) \\
	&+ I^{(j,\alpha)}(\xi;u)\, +  \Pi_{\deg=0}P^{(j,\alpha)}(\xi;u),
\end{split}
\end{equation*}
and we justify that these terms are actually some remainders in the sense of the vector field estimate \eqref{eq:boRr}, when $\xi\in\Xi_{\alpha}$ and $u\in\mathcal{V}_{\alpha,s}(\epsilon_{j+1,\alpha},\xi)$, with $\epsilon_{j,\alpha}=10\epsilon-(\alpha\kappa+j)\epsilon^{\frac{3}{2}}$. 

\medskip

 Fix $\xi\in\Xi_{\alpha}$ and $u\in\mathcal{V}_{\alpha,s}(\epsilon_{\alpha,j+1},\xi)$. First, we note that by definition 
 $$\nabla \Pi_{\deg=0}P^{(j,\alpha)}(\xi;u)=0$$ and so that this term can be ignored. Then we deduce from the estimate on the symplectic transformation \eqref{eq:inc-v} that for all $t\in[-1,1]$,
\[
v(t) := \Phi_{j,\alpha,\xi}^{t}(u)\in\mathcal{V}_{\alpha,s}(\epsilon_{j,\alpha},\xi)\,.
\]
Using that $\Phi_{j,\alpha,\xi}^{1}$ is symplectic, we obtain 
\begin{align*}
\nabla R^{(j,\alpha)}(\xi;\Phi_{j,\alpha,\xi}^{1}(u)) 
	&= (\mathrm{d}\Phi_{j,\alpha,\xi}^{1}(u))^{\ast}
	\nabla R^{(j,\alpha)}(\xi;v(1)) \\
	&= -i(\mathrm{d}\Phi_{j,\alpha,\xi}^{1}(u))^{-1}i
	\nabla R^{(j,\alpha)}(\xi;v(1)) \\
	&= -i\mathrm{d}\Phi_{j,\alpha,\xi}^{-1}(v)i
	\nabla R^{(j,\alpha)}(\xi;v(1))\,.
\end{align*}
We deduce from the estimate on the symplectic transformation \eqref{eq:33} and from the vector field estimate \eqref{eq:boRr} at step $j$ on $\mathcal{V}_{\alpha,s}(\epsilon_{j,\alpha},\xi)$ that for all $\xi\in\mathcal{U}_{s}(\epsilon)$, 
\[
\begin{split}
\|\nabla R^{(j,\alpha)}(\xi;
	\Phi_{j,\alpha,\xi}^{1}(u))\|_{h^{s}} 
	&\leq 
	\|\mathrm{d}\Phi_{j,\alpha,\xi}^{-1}(v(1))\|_{h^{s}\to h^{s}}
	\|\nabla R^{(j,\alpha)}(\xi;v(1))\|_{h^{s}}  \\
	&\lesssim_{r,\alpha} (1+\epsilon^{\frac{1}{2}})\epsilon^{2r+1}\,.
\end{split}
\]
Similarly, given $H\in X_{\leq2\overline{r}}(\epsilon)$ we have that for all $t\in[0,1]$,
\[
\nabla
	\operatorname{ad}_{\chi}^{\ell}(H)(\xi;\Phi_{j,\alpha,\xi}^{t}(u))
	= -i \mathrm{d}\Phi_{j,\alpha,\xi}^{-t}(v)i\nabla \operatorname{ad}_{\chi}^{\ell}(H)(\xi;v(t))\,.
\]
The vector field estimate then follows from Proposition \ref{prop:vec-alpha} and from the bounds \eqref{eq:ad2}, \eqref{eq:ad4} and \eqref{eq:adQ}: for $\xi\in\mathcal{U}_{s}(\epsilon)$ and $u\in\mathcal{V}_{\alpha,s}(\epsilon_{j+1,\alpha},\xi)$,
\begin{align*}
&\|\nabla I^{(j,\alpha)}(\xi;u)\|_{h^{s}}\\
&\lesssim  
	\sup_{t\in[0,1]}\|\mathrm{d}\Phi_{j,\alpha,\xi}^{-t}(v(t))\|_{h^{s}\to h^{s}}
	\|\nabla \mathrm{ad}_{\chi}^{\kappa}
	(Z_{2}^{(j,\alpha)}
	+Z_{4}^{(j,\alpha)}
	+Q^{(j,\alpha)})
	(\xi;v(t))\|_{h^{s}}
	\\
	&\lesssim 
	\| \mathrm{ad}_{\chi}^{\kappa}
	(Z_{2}^{(j,\alpha)}+Z_{4}^{(j,\alpha)}+Q^{(j,\alpha)})
	\|_{\Ysup{\alpha}}N_{\alpha}^{-4s}\epsilon^{3}
	\sup_{t\in[0,1]}\|v(t)\|_{h^{s}}
	\leq\epsilon^{3+\frac{\kappa}{20}}\|u\|_{h^{s}}\,,
\end{align*}
which is conclusive provided $\kappa\geq 40r$. As for the terms with large degree, we exploit Lemma \ref{lem:deg} and the assumption $\overline{r}\tau\geq 6s$. We get
\[
	\|
	\Pi_{\deg>\overline{r}}P^{(j,\alpha)}\|_{\Ysup{\beta}} 
	\leq
	\|\Pi_{\deg>\overline{r}}P^{(j,\alpha)}
	\|_{\Ysup{\alpha}}
	\leq
	\epsilon^{\frac{1}{20}}\,,
\]
where we used the bound obtained in \eqref{eq:Ptilde}. On the other, since $\deg(P^{(j,\alpha)})\leq\overline{r}^{2}$, we deduce from the vector field estimate at scale $\beta$ of Proposition \ref{prop:vec-alpha} that for all $u\in\mathcal{V}_{\alpha,s}(10\epsilon,\xi)$, 
\[
\|\nabla \Pi_{\deg>\overline{r}}P^{(j,\alpha)}(\xi;u)\|_{h^{s}}
	\leq  \|\Pi_{\deg>\overline{r}}P^{(j,\alpha)}\|_{\Ysup{\beta}}
	N_{\beta}^{-4s}\epsilon^{3}\|u\|_{h^{s}}
	\leq 
	\epsilon^{\frac{1}{20}+3}N_{\beta}^{-4s}\|u\|_{h^{s}}\,,
\]
and we conclude that under our choice of parameter \eqref{eq:s-beta} (at scale $\beta$,  $N_{\beta}^{-2s}=\epsilon^{r}$) 
\[
\|\nabla \Pi_{\deg>\overline{r}}P^{(j,\alpha)}(\xi;u)\|_{h^{s}}
	\leq  
	\epsilon^{2r+3}\|u\|_{h^{s}}\,.
\]
This concludes the proof of  \eqref{eq:boRr} at step $j+1$ and completes the proof of Proposition \ref{prop:it-j}.
 \end{proof}

\subsection{From step $\alpha$ to step $\alpha+1$:}
\label{sec:+1}
We now have all the ingredients to deduce Theorem \ref{thm:nf-alpha} at scale $\alpha+1$ from Theorem  \ref{thm:nf-alpha} at scale $\alpha$.

\begin{proof}[Proof of Theorem \ref{thm:nf-alpha}:] The initialization $\alpha=0$ was done in section \ref{sec:init}, with $R^{(0)}=\Pi_{\deg=0}Q$.  We fix $\alpha\in\{0,\cdots,\beta-1\}$ and we suppose that Theorem \ref{thm:nf-alpha} is proved up to scale $\alpha$. The inputs are the transformation $\tau_{\alpha}$ and the Hamiltonian functions 
\[
Z_{2}^{(\alpha)}\,,\quad Z_{4}^{(\alpha)}\,,\quad Q^{(\alpha)}\,,\quad R^{(\alpha)}
\]
where $Z_{2}^{(\alpha)}$ has modulated frequencies $\omega^{(\alpha)}$. We constructed the new objects at scale $\alpha+1$ in Proposition \ref{prop:it-j}: for $\epsilon\in(0,10^{-2}\epsilon_{\ast})$ and $\xi\in\mathcal{U}_{s}(\epsilon)$, we set 
\[
\tau_{\alpha+1,\xi}:= \tau_{\kappa,\alpha,\xi}\,,\quad Z_{2}^{(\alpha+1)} := Z_{2}^{(\kappa,\alpha)}\,,\quad Z_{4}^{(\alpha+1)} := Z_{4}^{(\kappa,\alpha)}\,.
\]
We also set
\[
Q^{(\alpha+1)} :=(\operatorname{Id}- \Pi_{\Lambda_{\alpha+1}})Q^{(\kappa,\alpha)}\,,
	\quad 
	R^{(\alpha+1)} := \Pi_{\Lambda_{\alpha+1}}Q^{(\kappa,\alpha)}
	+ R^{(\kappa,\alpha)}\,.
\]
By construction
\[
H_{\lo}\circ\tau_{\alpha+1,\xi}(\xi;u) = 
	Z_{2}^{(\alpha+1)}(\xi;u)
	+ Z_{4}^{(\alpha+1)}(\xi;u)
	+ Q^{(\alpha+1)}(\xi;u)
	+R^{(\alpha+1)}(\xi;u)\,.
\]
Then, we make sure that the new objects satisfy the properties claimed in Theorem \ref{thm:nf-alpha} at scale $\alpha+1$.  First, note that 
\[
\epsilon_{\kappa,\alpha} := 10\epsilon - (\alpha r + \kappa)\epsilon^{\frac{3}{2}} 
	\geq
	10\epsilon - cr(\alpha+1)\epsilon^{\frac{3}{2}}
	 =\epsilon_{\alpha+1}\,, 
\]
provided that $c$ is large enough to ensure that $\kappa\leq cr$ ($c\geq40$ is acceptable). Hence, $\mathcal{V}_{\alpha+1,s}(\epsilon_{\alpha+1},\xi) \subset \mathcal{V}_{\alpha,s}(\epsilon_{\kappa,\alpha},\xi)$.

\medskip

\noindent{$\triangleright$ \it The new variables:} We deduce from the above observation and from the estimates on the transformation \eqref{eq:phi-j}, \eqref{eq:d-phi-j} and \eqref{eq:d-phi-xi-j} proved at step $\kappa$, that the estimates \eqref{eq:tau}, \eqref{eq:d-tau} and \eqref{eq:d-phi-xi} (claimed in  the normal form Theorem \ref{thm:nf-alpha}) hold at scale $\alpha+1$. 

\medskip

\noindent{$\triangleright$ \it Modulated frequencies at scale $\alpha+1$:} We deduce the stability property from the bound \eqref{eq:j-omega}: for all $n\in\Z_{M}^{d}$,
\[
|\omega_{n}^{(\alpha+1)}(\xi)-\omega_{n}^{(\alpha)}(\xi)| 
	\leq \sum_{j=0}^{\kappa-1}
	|\omega_{n}^{(j+1,\alpha)}(\xi) - \omega_{n}^{(j,\alpha)}(\xi)|
	\leq \kappa N_{\alpha}^{-4s}\epsilon^{3+\frac{1}{10}}
	\leq N_{\alpha}^{-4s}\epsilon^{3}\,.
\]
This proves \eqref{eq:st-fr} at scale $\alpha+1$.  Similarly, we deduce from \eqref{eq:j-domega} and \eqref{eq:omega} at scale $\alpha$ that for all $k\in\Z_{M}^{d}$,
\begin{align*}
|\partial_{\xi_{k}}(\omega_{n}^{(\alpha+1)}(\xi)-\xi_{n})|
	&\leq |\partial_{\xi_{k}}(\omega_{n}^{(\alpha)}(\xi)-\xi_{n})|
	+ 
	\sum_{j=0}^{\kappa-1}
	|\partial_{\xi_{k}}(\omega_{n}^{(j+1,\alpha)}(\xi)-\omega_{n}^{(j,\alpha)}(\xi))| \\
	&\leq \alpha\epsilon + \kappa\epsilon^{1+\frac{1}{10}}N_{\alpha}^{-2s} 
	\leq (\alpha+1)\epsilon\,.
\end{align*}
This proves \eqref{eq:omega} at scale $\alpha+1$. 
\medskip

\noindent{$\triangleright$ \it Quartic integrable terms at scale $\alpha+1$:} According to \eqref{eq:bo4r}, 
\begin{align*}
\max(\|Z_{4}^{(\alpha+1)}\|_{Z^{\mathrm{sup}}}
	\,,\,
	\|Z_{4}^{(\alpha+1)}\|_{Z^{\mathrm{lip}}}
	)
	&=\max(\|Z_{4}^{(\kappa,\alpha)}\|_{Z^{\mathrm{sup}}}
	\,,\,
	\|Z_{4}^{(\kappa,\alpha)}\|_{Z^{\mathrm{lip}}}
	)\\
	&\lesssim_{\alpha}\epsilon^{-\frac{1}{10^{4}}} + \alpha+\kappa\epsilon^{\frac{1}{2}}
	\lesssim_{\alpha}\epsilon^{-\frac{1}{10^{4}}}\,,
\end{align*}
and this proves \eqref{eq:boZ} at scale $\alpha+1$.
 
 \medskip

 \noindent{$\triangleright$ \it Terms of type $Q$ at scale $\alpha+1$:} We now verify that $Q^{(\alpha+1)}$ operates at scale $\alpha+1$: we have from \eqref{eq:boQr} that  
\[
	\|Q^{(\kappa,\alpha)}\|_{\Ysup{\alpha}} \lesssim_{\alpha} \epsilon^{-\frac{1}{10^{4}}}
	\,,
	\quad
	\|Q^{(\kappa,\alpha)}\|_{\Ylip{\alpha}}
	\lesssim
	\epsilon^{-\frac{1}{20}}\,.
\]
According to Lemma \ref{lem:scale},
\[
\|Q^{(\alpha+1)}\|_{\Ysup{\alpha+1}}=
	\|
	(\operatorname{Id}-\Pi_{\Lambda_{\alpha+1}})
	Q^{(\kappa,\alpha)}
	\|_{\Ysup{\alpha+1}}
	\leq
	\|
	(\operatorname{Id}-\Pi_{\Lambda_{\alpha+1}})Q^{(\kappa,\alpha)}
	\|_{\Ysup{\alpha}}
	\lesssim_{\alpha}\epsilon^{-\frac{1}{10^{4}}}\,,
\]
and, similarly, 
\[
\|Q^{(\alpha+1)}\|_{\Ylip{\alpha+1}}=
	\|
	(\operatorname{Id}-\Pi_{\Lambda_{\alpha+1}})
	Q^{(\kappa,\alpha)}
	\|_{\Ylip{\alpha+1}}
	\leq
	\|
	(\operatorname{Id}-\Pi_{\Lambda_{\alpha+1}})Q^{(\kappa,\alpha)}
	\|_{\Ylip{\alpha}}
	\lesssim_{\alpha}\epsilon^{-\frac{1}{20}}\,.
\]
This proves \eqref{eq:boQ} at scale $\alpha+1$. 
\medskip

\noindent{$\triangleright$ \it Remainder term at scale $\alpha+1$:} As for the remainder term, when $\xi\in\Xi_{\alpha}$ then $\xi\in\Xi_{\alpha-1}$ and we deduce from \eqref{eq:boRr} at step $\kappa$ that for all $u\in\mathcal{V}_{\alpha+1,s}(\epsilon_{\alpha+1}, \xi)$, 
\[
\|\nabla R^{(\kappa,\alpha)}(\xi;u)\|_{h^{s}} 
	\lesssim_{\alpha,r}
	\epsilon^{2r+1}
	\|u\|_{h^{s}}\,.
\]
Moreover, we have from \eqref{eq:it-alpha-2} that 
\[
\|\Pi_{\Lambda_{\alpha+1}}Q^{(\kappa,\alpha)}
	\|_{\Ysup{\alpha}}
	\leq \epsilon^{\frac{\kappa}{30}-\frac{1}{10^{4}}} \leq \epsilon^{2r}\,,
\]
provided, say, that $\kappa\geq100r$. We conclude from Proposition \ref{prop:vec-alpha} that for all $u\in\mathcal{V}_{\alpha+1,s}(\epsilon_{\alpha+1},\xi)$, 
\[
\|\nabla \Pi_{\Lambda_{\alpha+1}}Q^{(\kappa,\alpha)}(\xi;
	u)\|_{h^{s}}
	\leq \epsilon^{2r+1}\|u\|_{h^{s}}\,.
\]
This proves \eqref{eq:boR} at scale $\alpha+1$ and completes the proof of Theorem \ref{thm:nf-alpha}. 
\end{proof}

\section{Proof of the Theorem}
\label{sec:pr-thm}

In this section we apply the normal form Theorem \ref{thm:nf-alpha} at scale $\beta$ together with the measure estimates from section \ref{sec:smd} to prove Theorem \ref{thm:main_low}. Then, the finite dimensional reduction strategy exposed in section \ref{sec:low-freq} completes the proof of the main Theorem \ref{thm:main}. 
\medskip

Let us start with the construction of the non-resonant set of initial data $\Theta_{\epsilon}^{\flat}$ and its measure estimate. We finally handle the dynamics of the low-frequency actions in subsection \ref{sec:dyn-low}.

\subsection{The good modulation data set, and its measure estimate}
\label{sec:meas}
Let $\tau_{\xi}:=\tau_{\beta,\xi}$ be the parameter-dependent symplectic transformations obtained in Theorem \ref{thm:nf-alpha} after $\beta$ steps. We set
\begin{equation}
\label{eq:psi}
\begin{array}{ccccc}
\Psi &\colon &\Pi_{M}B_s(9\epsilon) &\longrightarrow & h^s(\Z_{M}^{d}) \\
&&&&\\
& &\phi& \longmapsto & \tau_{\beta,\xi(\phi)}(\phi)\,.\\
\end{array}
\end{equation}
Let us state some properties of $\Psi$ that we  essentially deduce from the properties of $\tau_{\xi,\beta}$ claimed in Theorem \ref{thm:nf-alpha}. 
\begin{lemma}
\label{lem:psi-lip} 
The function $\Psi$ is a $C^{1}$-diffeomorphism from $\Pi_{M}B_{s}(9\epsilon)$  into it image which is close to the identity: for all $\phi\in \Pi_{M}B_{s}(9\epsilon)$, 
\begin{align}
\label{eq:psi-id}
\|\Psi(\phi)-\phi\|_{h^{s}}
	&\lesssim_{r}\epsilon^{\frac{3}{4}}\|\phi\|_{h^{s}}\,,\\
\label{eq:d-psi-id}
\|\mathrm{d} \Psi(\phi) -\Id\|_{h^s\to h^s} 
	&\leq \epsilon^\frac{1}{2}\,. 
\end{align}
\end{lemma}
\begin{proof}
Observe that $9\epsilon\leq \epsilon_{\beta}$ and, by construction, that $\phi\in\mathcal{V}_{\beta,s}(\xi(\phi),\epsilon_{\beta})$. Hence, the bound \eqref{eq:psi-id} corresponds to the bound \eqref{eq:tau}. Applying the chain rule, we obtain that for all $\phi\in \Pi_{M}B_{s}(9\epsilon)$, 
\[
\mathrm{d}\Psi(\phi) = \mathrm{d}_{\phi}\tau_{\beta,\xi(\phi)}(u)\mid_{u=\phi}+\mathrm{d}\tau_{\beta,\xi(\phi)}(\phi)\,.
\]
Hence, \eqref{eq:d-psi-id} is a consequence of  \eqref{eq:d-phi-xi} at scale $\beta$, with $u=\phi\in\mathcal{V}_{s}(\epsilon_{\beta},\xi(\phi))$, and of \eqref{eq:d-tau}. 
\medskip

It follows from \eqref{eq:d-psi-id} is that $\Psi$ is a local diffeomorphism. According to the mean value Theorem applied on the convex set $\Pi_{M}B_{s}(9\epsilon)$, 
$\Psi-\Id$ is $\epsilon^{\frac{1}{2}}$-Lipschitz, which in turn implies that $\Psi$ is an injective function on $\Pi_{M}B_{s}(9\epsilon)$. This proves that $\Psi$ is a $C^{1}$-diffeomorphism on $\Pi_{M}B_{s}(9\epsilon)$.
\end{proof}
Note that another consequence of \eqref{eq:psi-id} is that 
\[
\Psi(\Pi_{M}B_{s}(9\epsilon)) \subset \Pi_{M}B_{s}(10\epsilon)\,.
\]
The next Lemma is useful to estimate the measure of sets transported by the function $\Psi$. 
\begin{lemma}
\label{lem:meas-psi}
For all $\phi\in\Pi_{M}B_{s}(9\epsilon)$,
\begin{equation}
\label{eq:jac-psi}
	\Big|
	|\det(\mathrm{d}\Psi(\phi))|-1 
	\Big| 
	\leq
	\epsilon^\frac{1}{20}\,.
\end{equation}
\end{lemma}
\begin{proof}
Seeing $\Pi_{M}B_{s}$ as a $\R$-vectorial space of dimension $d_{M}:=2\sharp(\Z_{M}^{d})$, we deduce from \eqref{eq:d-psi-id} that 
\[
(1-\epsilon^{\frac{1}{2}})^{d_{M}}\leq|\det(\mathrm{d}\Psi(\phi))| \leq (1+\epsilon^{\frac{1}{2}})^{d_{M}}\,.
\]
It remains to prove that
\begin{equation}
\label{eq:epsilon}
|(1\pm\epsilon^\frac{1}{2})^{d_{M}}-1| \leq \epsilon^\frac{1}{20}\,.
\end{equation}
We have
\begin{equation}
\label{eq:pluie}
|(1\pm\epsilon^\frac{1}{2})^{d_{M}}-1| 
	= |\exp(d_{M}\log(1\pm\epsilon^\frac{1}{2}))-1| \leq 2d_{M}|\log(1\pm\epsilon^\frac{1}{2})|\,.
\end{equation}
Recall that 
\[
d_{M} \leq 2\epsilon^{-\frac{1}{10^{4}}}\,.
\]
Using that for all $x\in[0,1)$, 
\[
|\log(1-x)|\leq \frac{x}{1-x}\,,\quad \log(1+x)\leq x\,,
\]
we obtain 
\[
d_{M}|\log(1\pm\epsilon^\frac{1}{2})|\leq 2\epsilon^{\frac{1}{2}-\frac{1}{10^{4}}} \leq \frac{1}{2}\epsilon^{\frac{1}{20}}\,.
\]
We conclude from \eqref{eq:pluie} that 
\[
|(1\pm\epsilon^\frac{1}{2})^{d_{M}}-1| \leq \epsilon^\frac{1}{20}\,,
\]
which gives \eqref{eq:epsilon} and completes the proof of \eqref{eq:jac-psi}.
\end{proof}
\begin{definition}[Non-resonant initial data]
\label{def:u-frak}
Let $\omega^{(\beta-1)}$ be the modulated frequencies at step $\beta-1$ of the normal form iteration scheme.  The set of non-resonant internal parameters is $\Xi_{\beta-1}$ as defined in \eqref{eq:tw} and the set of non-resonant data in final coordinates is
\begin{equation}
\label{eq:U}
\operatorname{U}_{\beta}
	= \operatorname{U}_{\epsilon,\gamma(\beta)}(\omega^{(\beta-1)},\Lambda_{\beta}) 
	= \Big\{\phi\in\Pi_{M}B_{s}(\epsilon)\ |\ \xi(\phi)\in\Xi_{\beta-1}\Big\}\,.
\end{equation}
\end{definition}
We first deduce from Proposition \ref{prop:modulation} an estimate on the density of $\operatorname{U}_{\beta}$ in $\Pi_{M}B_{s}(\epsilon)$. 
\begin{lemma}
\label{lem:meas-}
For all $\epsilon\in(0,\epsilon_{\ast})$ we have 
\begin{equation}
\label{eq:lem-meas}
\meas(\operatorname{U}_{\beta})
	\geq
	(1 - \epsilon^{\frac{1}{38}})
	\meas(\Pi_{M}B_{s}(\epsilon))\,.
\end{equation}
\end{lemma}
\begin{proof}
For all $\xi\in\mathcal{U}_{s}(\epsilon)$, the modulated frequencies $(\omega_{n}^{(\beta-1)}(\xi))_{n}$ satisfy \eqref{eq:omega} at scale $\beta-1$. Hence, they satisfy the assumption of Proposition \ref{prop:modulation}, and we obtain that the associated non-resonant set of data $\operatorname{U}_{\beta}$ has large density in $\Pi_{M}B_{s}(\epsilon)$: 
\[
\meas\Big(
	\Pi_{M}B_{s}(\epsilon)\setminus \operatorname{U}_{\beta}
	\Big) 
	\leq
	\gamma(\beta)\epsilon^{-\frac{1}{10^{4}}}\meas(\Pi_{M}B_{s}(\epsilon))\,.
\] 
Under our choice of parameters (in particular $\gamma(\alpha)=4^{\alpha}\gamma=4^{\alpha}\epsilon^{\frac{1}{30}}$) we deduce that 
\[
\meas(\Pi_{M}B_{s}(\epsilon)
	\setminus\operatorname{U}_{\beta}) 
	\leq 
	\epsilon^{\frac{1}{38}}\meas(\Pi_{M}B_{s}(\epsilon))\,.
\]
This gives \eqref{eq:lem-meas} and concludes the proof of the lemma.
\end{proof}
We are now ready to estimate from below the density of the set of non-resonant data in the coordinates obtained after the preliminary transformation $\tau_{0}$ of section \ref{sec:low-freq}.
\begin{proposition}
\label{prop:xi0} For all $\epsilon\in(0,\epsilon_{\ast})$, let
\begin{equation}
\label{eq:theta}
\Theta_{\epsilon}^{\flat}
	:= 
	\Psi(\operatorname{U}_{\beta})
	\cap
	\Pi_{M}B_{s}(\epsilon)
	\,.
\end{equation}
Then, $\Theta_{\epsilon}^{\flat}$ is an open subset of $\Pi_{M}B_{s}(\epsilon)$, with
\begin{equation}
\meas(\Theta_{\epsilon}^{\flat})\geq(1-\epsilon^{\frac{1}{39}})\meas(\Pi_{M}B_{s}(\epsilon))\,.
\end{equation}
\end{proposition}
\begin{proof} 
Since $\operatorname{U}_{\beta}$ is also an open subset of $\Pi_{M}B_{s}(\epsilon)$ and $\Psi$ is bi-Lipschitz, $\Theta_{\epsilon}^{\flat}$ is an open subset of $\Pi_{M}B_{s}(\epsilon)$. 

\medskip

As for the measure estimate, we estimated the measure of $\operatorname{U}_{\beta}$ in Lemma \ref{lem:meas-} and it essentially remains to estimate its image by $\Psi$, intersected with $\Pi_{M}B(\epsilon)$. Note that, according to \eqref{eq:psi-id}, we have 
\[
\Psi(\Pi_{M}B_{s}(\epsilon))\subset \Pi_{M}B_{s}(\epsilon+C(r)\epsilon^{\frac{7}{4}})
	\subset \Pi_{M}B_{s}(\epsilon+\epsilon^{\frac{3}{2}})\,, 
\]
provided $\epsilon_{\ast}(r)$ is small enough, and therefore, defining 
\[
\mathcal{A}_{\epsilon} = \{ u\in \Pi_{M}h^{s}\ |\ \epsilon \leq \|u\|_{h^{s}} \leq \epsilon + \epsilon^{\frac{3}{2}}\}\,.
\]
we obtain that
\[
\Psi(\operatorname{U}_{\beta}) 
	=
	\Big(
	\Psi(\operatorname{U}_{\beta})
	\cap
	\Pi_{M}B_{s}(\epsilon)
	\Big)
	\sqcup 
	\Big(
	\Psi(\operatorname{U}_{\beta})\cap \mathcal{A}_{\epsilon}
	\Big)
	= \Theta_{\epsilon}^{\flat} 
	\sqcup 
	\Big(
	\Psi(\operatorname{U}_{\beta})\cap \mathcal{A}_{\epsilon}
	\Big)\,,
\]
where the second equality follows from the definition of the set $\Theta_{\epsilon}^{\flat}$. We deduce that 
\begin{equation}
\label{eq:sqcup}
\meas(\Theta_{\epsilon}^{\flat}) 
	\geq 
	\meas(\Psi(\operatorname{U}_{\beta}))
	- \meas(\mathcal{A}_{\epsilon})\,.
\end{equation}
Then, we do a change of variable  
\[
\begin{split}
\meas(\Psi(\operatorname{U}_{\beta})) 
	= 
	\int_{\Psi(\operatorname{U}_{\beta})}
	\mathrm{d}u 
	&= 
	\int_{\operatorname{U}_{\beta}}
	|\det(\mathrm{d}\Psi(\phi))|\mathrm{d}\phi \\
	&\geq 
	\meas(\operatorname{U}_{\beta})
	\min_{\phi\in \Pi_{M}B_s(\epsilon)}
	|\det(\mathrm{d}\Psi(\phi))| \,,
	\end{split}
\]
and we deduce from \eqref{eq:jac-psi} and \eqref{eq:lem-meas} that
\[
\meas
	(
	\Psi(\operatorname{U}_{\beta})
	)
	\geq 
	(1-\epsilon^{\frac{1}{38}})(1-\epsilon^{\frac{1}{20}})\meas(\Pi_{M}B_s(\epsilon))\,.
\]
On the other hand,  we deduce from the bound \eqref{eq:epsilon} that, by homogeneity,
\[
	\frac{\meas(\mathcal{A}_{\epsilon})}
	{\meas(\Pi_{M}(\epsilon))} 
	=
	(1+\epsilon^{\frac{1}{2}})^{d_{M}}
	-1
	\leq \epsilon^{\frac{1}{20}}\,.
\]
We conclude from \eqref{eq:sqcup} that 
\[
\meas(\Theta_{\epsilon}^{\flat}) \geq (1-\frac{1}{2}\epsilon^{\frac{1}{39}}- \epsilon^{\frac{1}{20}})\meas(\Pi_{M}B_{s}(\epsilon)) \geq (1-\epsilon^{\frac{1}{39}})\meas(\Pi_{M}B_{s}(\epsilon))\,.
\]
This completes the proof of Proposition \ref{prop:xi0}. 
\end{proof}

\subsection{Dynamics} 
\label{sec:dyn-low}
It remains to prove the dynamical part of Theorem \ref{thm:main_low}, which is precisely the stability of the low frequency actions for initial data in the non-resonant set $\Theta_{\epsilon}^{\flat}$ defined in \eqref{eq:theta}.

\begin{proof}[Proof of Theorem \ref{thm:main_low}] Let $u(0)\in\Pi_{M}B_{s}(\epsilon)$, with $\|u(0)\|_{h^{s}}\leq \rho \leq \epsilon$. We suppose that $0<T\leq \epsilon^{-r}$ and $u\in C^{1}([-T,T],\Pi_{M}h^{s})$ is such that 
\[
i\partial_{t}u = \nabla H_{\lo}(u) + f(t)\,,
\]
with $u(t=0)=u(0)$ and, for all $|t|\leq T$,
\[
\|f(t)\|_{h^{s}}\leq \epsilon^{3r}\rho\,.
\]
The Hamiltonian $H_{\lo}$ is given by \eqref{eq:Hlo-0}. We constructed in Theorem \ref{thm:nf-alpha} for each $\xi\in\mathcal{U}_{s}(\epsilon)$ a symplectic transformation $\tau_{\beta,\xi}$, and we defined the bi-Lipschitz function $\Psi$ in \eqref{eq:psi}. The set of non-resonant parameters $\Theta_{\epsilon}^{\flat}\subset \Pi_{M}B_{s}(\epsilon)$ was introduced in \eqref{eq:theta}.

\medskip

We proved in Proposition \ref{prop:xi0} the density estimate \eqref{eq:density} of $\Theta_{\epsilon}^{\flat}$ in $\Pi_{M}B_{s}(\epsilon)$, and we now suppose that $u(0)\in\Theta_{\epsilon}^{\flat}$. By construction of the set $\Theta_{\epsilon}^{\flat}$ (see \eqref{eq:theta}) one can find $\phi\in\operatorname{U}_{\beta}$ such that 
\[
u(0) = \Psi(\phi) = \tau_{\beta,\xi(\phi)}(\phi)\,.
\]
Then, we set for $|t|\leq T$
\[
\xi:=\xi(\phi)=(|\phi_{n}|^{2})_{|n|\leq M}\,,\quad v(t) := \tau_{\beta,\xi}^{-1}(u(t))\,.
\]
In particular, $v$ is solution in $C([-T,T],\Pi_{M}h^{s})$ to the Cauchy problem
\begin{equation}
\label{eq:cau-v}
\begin{cases}
i\partial_{t}v = \nabla H^{(\beta)}(\xi;v(t)) - i(\mathrm{d}\tau_{\beta,\xi}(v(t)))^{-1}(if(t))\,,\\
v(0) = \phi\,,
\end{cases}
\end{equation}
where 
\[
H^{(\beta)}(\xi;u) := H_{\lo}^{(\beta)}(\xi;\tau_{\beta,\xi}(v) = Z_{2}^{(\beta)}(\xi;u)+Z_{4}^{(\beta)}(\xi;u)+Q^{(\beta)}(\xi;u)+R^{(\beta)}(\xi;u)\,,
\]
as in Theorem \ref{thm:nf-alpha}. Note that we used 
\[
\mathrm{d}\tau_{\beta,\xi}^{-1}(u(t)) = (\mathrm{d}\tau_{\beta,\xi}(v(t)))^{-1}\,.
\]
We now perform a bootstrap argument to show that for all $|t|\leq T\leq \epsilon^{-r}$, 
\begin{equation}
\label{eq:boo-v}
\sum_{|n|\leq M}\langle n\rangle^{2s}||v_{n}(t)|^{2} - \xi_{n}| \leq N_{\beta}^{-2s}\epsilon^{\frac{1}{2}}\rho^{2}\,. 
\end{equation}
This in particular implies that 
\begin{equation}
\label{eq:boo-v-c}
v(t)\in\mathcal{V}_{\beta,s}(2\epsilon,\xi)\,,\quad \|v(t)\|_{h^{s}}\leq 2\rho\,,
\end{equation}
and, according to the bound \eqref{eq:tau} on the symplectic transformation $\tau_{\beta,\xi}$, we deduce that
\begin{align*}
\sum_{|n|\leq M}&\langle n\rangle^{2s}||u_{n}(t)|^{2} - |u_{n}(0)|^{2}| \\
	\leq&
	\sum_{|n|\leq M}\langle n\rangle^{2s}||u_{n}(t)|^{2} - |v_{n}(t)|^{2}|+ \sum_{|n|\leq M}\langle n\rangle^{2s}||v_{n}(t)|^{2} - |v_{n}(0)|^{2}|  \\
	&+
	\sum_{|n|\leq M}\langle n\rangle^{2s}||v_{n}(0)|^{2} - |u_{n}(0)|^{2}| \\
	\lesssim_{r}&
	\epsilon^{\frac{3}{2}}(\|v(t)\|_{h^{s}}^{2}+\|\phi\|_{h^{s}}^{2}) + N_{\beta}^{-2s}\epsilon^{\frac{1}{2}}\rho^{2}  \leq \frac{\epsilon}{2}\rho^{2}\,,
\end{align*}
uniformly in $|t|\leq T$. This is precisely the desired bound $\eqref{eq:boo-u}$ for the dynamics of $u(t)$.

\medskip

It remains to prove \eqref{eq:boo-v}. At time $t=0$, since $v(0)=\phi$ and $\xi_{n}=|\phi_{n}|^{2}$ for all $|n|\leq M$ the left-and-side of \eqref{eq:boo-v} vanishes. For $|t|\leq T$, we deduce from the equation \eqref{eq:cau-v} a priori bounds for $v$ under the assumption \eqref{eq:boo-v}, from which \eqref{eq:boo-v} follows by a standard bootstrap argument.

\medskip

Since, by assumption, $\phi\in\operatorname{U}_{\beta-1}$ then $\xi(\phi)\in\Xi_{\beta-1}$ and, assuming that $v(t)\in\mathcal{V}_{\beta,s}(2\epsilon,\xi)$, we obtain
\[
\|\nabla R^{(\beta)}(\xi,v(t))\|_{h^{s}}\lesssim_{r}\epsilon^{2r+1}\|v(t)\|_{h^{s}}\,.
\]
where we used the bound \eqref{eq:boR} for the remainder $R$. 
In addition, when $v(t)\in\mathcal{V}_{\beta,s}(2\epsilon,\xi)$, we obtain from Proposition \ref{prop:vec-alpha} and the bound \eqref{eq:boQ} at scale $\beta$ that 
\begin{equation}
\label{eq:jeveuxteciterdanslintro}
\|\nabla Q^{(\beta)}(\xi;v(t))\|_{h^{s}} \leq \|Q^{(\beta)}\|_{\Ysup{\beta}}N_{\beta}^{-4s}\epsilon^{3}\|v(t)\|_{h^{s}}
	\lesssim N_{\beta}^{-4s}\epsilon^{3-\frac{1}{10^{4}}}\rho\,.
\end{equation}
Using the equation \eqref{eq:cau-v} we obtain that for all $|n|\leq M$, 
\begin{multline*}
\frac{d}{dt}|v_{n}(t)|^{2} 
= 2\im
	\Big(
	\partial_{\overline{v_{n}}}(Z_{2}^{(\beta)}+Z_{4}^{(\beta)}+Q^{(\beta)}+R^{(\beta)})(\xi;v(t))\overline{v_{n}(t)}
	\Big) \\
	+
	2\im
	\Big(i
	\Big((\mathrm{d}\tau_{\beta,\xi}(v(t)))^{-1}(if(t))\Big)_{n}\overline{v_{n}(t)}
	\Big)
\end{multline*}
Using that $Z_{2}^{(\beta)},Z_{4}^{(\beta)}\in X_{\mathrm{Int}}(\epsilon)$ is integrable (and, consequently, does not contribute to the dynamics of the actions), we deduce that 
\[
\frac{d}{dt}|v_{n}(t)|^{2} 	=2\im
	\Big(
	\partial_{\overline{v_{n}}}(Q^{(\beta)}+R^{(\beta)})(\xi;v(t))\overline{v_{n}}
	\Big)
	+
	2\im
	\Big(i\Big(\mathrm{d}\tau_{\beta,\xi}(v(t)))^{-1}(if(t))\Big)_{n}\overline{v_{n}}
	\Big)\,.
\]
We deduce from the above an energy estimate that
\begin{multline*}
\sum_{|n|\leq M}\langle n\rangle^{2s}|\frac{d}{dt}|v_{n}(t)|^{2}| \leq \|v(t)\|_{h^{s}}
	\\
	\Big(
	\|\nabla Q^{(\beta)}(\xi;v(t))\|_{h^{s}}
	+
	\|\nabla R^{(\beta)}(\xi;v(t))\|_{h^{s}}
	+
	\|(\mathrm{d}\tau_{\beta,\xi}(v(t)))^{-1}\|_{h^{s}\to h^{s}}\|f(t)\|_{h^{s}}
	\Big)
\end{multline*}
 Under the bootstrap assumption \eqref{eq:boo-v} and its consequences \eqref{eq:boo-v-c}, we have therefore 
\begin{align*}
\sum_{|n|\leq M}\langle n\rangle^{2s}|\frac{d}{dt}|v_{n}(t)|^{2}| 
	&\lesssim_{r}
	\rho^{2}
	\Big(
	N_{\beta}^{-4s}\epsilon^{3-\frac{1}{10^{4}}}
	+
	\epsilon^{2r+1}
	+\epsilon^{3r}
	\Big) \\
	&\lesssim_{r}
	\epsilon^{r}\rho^{2}
	\Big(
	N_{\beta}^{-4s}\epsilon^{-r}\epsilon^{3-\frac{1}{10^{4}}}+\epsilon^{r+1} + \epsilon^{2r}
	\Big)\,.
\end{align*}
where we used the bound \eqref{eq:d--1} on the symplectic transformation and  the assumption on $f(t)$. Under the conditions of subsection \eqref{sec:par} on the parameters (in particular the choice \eqref{eq:s-beta} of $N_{\beta}$ such that $N_{\beta}^{-2s}=\epsilon^{r}$), we deduce that 
\[
\sum_{|n|\leq M}\langle n\rangle^{2s}|\frac{d}{dt}|v_{n}(t)|^{2}|
	\lesssim_{r} N_{\beta}^{-2s}\epsilon^{r+1}\rho^{2}\,.
\]
Hence, integrating in time (with $|t|\leq T\leq \epsilon^{-r}$) gives 
\[
\sum_{|n|\leq M}\langle n\rangle^{2s}||v_{n}(t)|^{2}-\xi_{n}|
	\lesssim_{r} N_{\beta}^{-2s}\epsilon\rho^{2}\,. 
\]
This is a stronger bound than the bootstrap assumption \eqref{eq:boo-v}, and this completes the proof of Theorem \ref{thm:main_low}.
\end{proof}

\medskip
\appendix
\section{Appendix}

\subsection{Proof of Proposition \ref{prop:bracket1}}
\label{app-lp}

\begin{proof}[Proof of Proposition \ref{prop:bracket1}] The proof follows the same lines as the proof of Proposition \ref{prop:bracket0}, with some adjustments. We first prove \eqref{eq:braket1-alpha}: given $Q$ and $ H$ we set
\[
P(\xi;u) := \{Q,H\}(\xi;u) = \sum_{\nb''\in\Nb_{\leq2\overline{r}^{2}}}(P[\xi])_{\nb''}z_{\nb''}(u,I(u)-\xi)\,.
\]
As in the proof of Proposition \ref{prop:bracket0} we first reduce to fixed monomials $(\nb,\nb')$: for all $\nb''$ and $k\in\Z_{M}^{d}$ we have 
\begin{multline}
|\partial_{\xi_{k}}(P[\xi])_{\nb''}|
	\leq \epsilon^{-\frac{1}{10^{4}}}\sup_{\mathbf{a},\mathbf{b},n,\nb,\nb'}
	\Big|
	\partial_{\xi_{k}}
	\Big(\prod_{|j|\leq M}\xi_{j}^{a_{j}-b_{j}}
	(Q[\xi])_{\nb}(H[\xi])_{\nb'}
	\Big)
	\Big|  \\
	\Big[
	\mathbf{1}_{E_{\nb'',\mathbf{a},\mathbf{b}}^{(1)}(n)}(\nb,\nb')|k_{n}'\ell_{n}-k_{n}\ell_{n}'|
	+\mathbf{1}_{E_{\nb'',\mathbf{a},\mathbf{b}}^{(2)}(n)}(\nb,\nb') 
	|m_{n}(k_{n'}-\ell_{n}')+m_{n'}(k_{n}-\ell_{n})|
	\Big]
	\label{eq:pr-dxi}
\end{multline}
We separate three main cases, depending on where the derivative $\partial_{\xi_{k}}$ falls.  
\medskip

\noindent$\triangleright$ {\it Case 1:} $\partial_{\xi_{k}}$ falls on a term $\xi_{k}$ coming from the re-centered actions. In this case, we have
\begin{multline}
	\Big|
	\partial_{\xi_{k}} 
	\Big(
	\prod_{|j|\leq M}\xi_{j}^{a_{j}-b_{j}}
	\Big)
	(Q[\xi])_{\nb}(H[\xi])_{\nb'}
	\Big| \\
	\leq 
	\|Q\|_{\Ysup{\alpha}}\|H\|_{\Ysup{\alpha}} 
	\Big(
	\mathbf{1}_{a_{k}-b_{k}\geq1}(a_{k}-b_{k})\xi_{k}^{a_{k}-b_{k}-1}
	\prod_{j\neq k}\xi_{j}^{a_{j}-b_{j}}\Big)
	\mathrm{w}_{\nb}^{0}(\alpha)
	\mathrm{w}_{\nb'}^{0}(\alpha)\,.
	\label{eq:nuit}
\end{multline}
Hence, when $a_{k}-b_{k}\geq1$, a term $\xi_{k}$ is removed, and this costs a factor $\operatorname{C}_{k}(\alpha)^{2}$. Following the proof of Proposition \ref{prop:bracket0} we deduce that the contribution of \eqref{eq:nuit} to \eqref{eq:pr-dxi} is bounded by
\begin{align*}
\eqref{eq:nuit}
	&\leq
	\|Q\|_{\Ysup{\alpha}}\|H\|_{\Ysup{\alpha}}
	\eta^{4-\frac{1}{5}}N_{\alpha}^{-4s}
	\mathrm{w}_{\nb''}^{0}(\alpha)
	\operatorname{C}_{k}(\alpha)^{2}
	(a_{k}-b_{k})
	 \\
	&\lesssim_{r}
	\|Q\|_{\Ysup{\alpha}}\|H\|_{\Ysup{\alpha}}
	\eta^{2-\frac{1}{5}}N_{\alpha}^{-2(s-\tau)}
	\mathrm{w}_{\nb''}^{0}(\alpha)
	\\
	&=
	\|Q\|_{\Ysup{\alpha}}\|H\|_{\Ysup{\alpha}}
	\eta^{4-\frac{1}{5}}N_{\alpha}^{-4s+2\tau}
	\mathrm{w}_{\nb''}^{1}(\alpha)
	\\
	&\leq
	\|Q\|_{\Ysup{\alpha}}\|H\|_{\Ysup{\alpha}}
	 \eta^{4-\frac{1}{4}}N_{\alpha}^{-4s}\mathrm{w}_{\nb''}^{1}(\alpha) \,,
\end{align*}
where we used Definition \ref{def:weight} for the weights and the condition \eqref{eq:eta-tau} on the parameters. 
\medskip

\noindent$\triangleright${\it Case 2:} In this case, $\partial_{\xi_{k}}$ falls on either $(Q[\xi])_{\nb}$ or $(H[\xi])_{\nb'}$. Since the two situations are symmetric, we suppose that it falls on $(Q[\xi])_{\nb}$. In Case 2 we have to control the following quantity:
\begin{multline}
\epsilon^{-\frac{1}{10^{4}}}
	\|Q\|_{\Ylip{\alpha}}\|H\|_{\Ysup{\alpha}}
	\sup_{\mathbf{a},\mathbf{b},n,\nb,\nb'}
	(\prod_{|j|\leq M}\xi_{j}^{a_{j}-b_{j}})\mathrm{w}_{\nb}^{1}(\alpha)\mathrm{w}_{\nb'}^{0}(\alpha) \\
	\Big[\mathbf{1}_{E_{\nb''}^{(1)}(n,\mathbf{a},\mathbf{b})}(\nb,\nb')
	|k_{n}'\ell_{n}-k_{n}\ell_{n}'|
	+\mathbf{1}_{E_{\nb''}^{(1)}(n,\mathbf{a},\mathbf{b})}(\nb,\nb')
	|m_{n}(k_{n}'-\ell_{n'})+m_{n'}(\ell_{n}-k_{n})|
	\Big]\,.
\label{eq:pinot} 
\end{multline}
In comparison to \eqref{eq:pogi}, $\mathrm{w}_{\nb}^{0}(\alpha)$ is substituted with $\mathrm{w}_{\nb}^{1}(\alpha)$ in \eqref{eq:pinot}.  Note that the weight $\mathrm{w}_{\nb''}^{1}(\alpha)$ has the same pre-factor as $\mathrm{w}_{\nb}^{1}(\alpha)$, so it suffices to prove that $\mathrm{w}_{\nb'}^{0}(\alpha)$ absorbs the losses. Reproducing the analysis performed  in Cases 1 and 2 in the proof of Proposition \ref{prop:bracket0} gives 
\[
|\eqref{eq:pinot}| \leq 
	\|Q\|_{\Ylip{\alpha}}\|H\|_{\Ysup{\alpha}}\eta^{4-\frac{1}{4}}N_{\alpha}^{-4s}\mathrm{w}_{\nb''}^{1}(\alpha)\,,
\]
which is conclusive. 
\medskip

Let us now turn to the proof of \eqref{eq:braket1-alpha-Z}. Given an integrable quartic term $Z\in X_{4}(\epsilon)$, the Poisson bracket $\{Q,Z\}=:P$ is explicitly written in \eqref{eq:poiss-z}. Once again we stress out that $P$ has no new integrable terms (hence there is no Case 1 as above). Differentiating  the coefficient $(P[\xi])_{\nb''}$ with respect to $\xi_{k}$ gives
\begin{multline}
\label{eq:midi}
\partial_{\xi_{k}}(P[\xi])_{\nb''} 
	= 2i\sum_{|j_{1}|,|j_{2}|\leq M}
	\sum_{\nb\in\Nb_{\leq\overline{r}^{2}}}
	\mathbf{1}
	_{E_{\nb'',0,0}^{(2)}(j_{1})}
	(
	\nb,
	\mathbf{e}_{\mathfrak{m}}(j_{1})+\mathbf{e}_{\mathfrak{m}}(j_{2})
	)
	(k_{j_{1}}-\ell_{j_{1}}) \\
	\Big(
	\partial_{\xi_{k}}(Q[\xi])_{\nb}
	(Z[\xi])_{\mathbf{e}_{\mathfrak{m}}(j_{1})+\mathbf{e}_{\mathfrak{m}}(j_{2})}
	+
	(Q[\xi])_{\nb}
	\partial_{\xi_{k}}
	(Z[\xi])_{\mathbf{e}_{\mathfrak{m}}(j_{1})+\mathbf{e}_{\mathfrak{m}}(j_{2})}
	\Big)\,.
\end{multline}
As detailed in the proof of Proposition \ref{prop:bracket0}, we have the relation \eqref{eq:rel} and therefore 
\[
\mathrm{w}_{\nb}^{1}(\alpha) = \frac{1}{\operatorname{D}(\alpha)}\mathrm{w}_{\nb''}^{1}(\alpha) = N_{\alpha}^{-2s}\eta^{2+\frac{1}{5}}\mathrm{w}_{\nb''}^{1}(\alpha)\,.
\]
We deduce that when $\partial_{\xi_{k}}$ falls on $(Q[\xi])_{\nb}$ the contribution is bounded by 
\[
\lesssim_{r}
	\|Q\|_{\Ylip{\alpha}}\|Z\|_{Z^{\mathrm{lip}}}	
	\epsilon^{-\frac{1}{10^{4}}}
	\mathrm{w}_{\nb}^{1}(\alpha) 
	\leq 
	\|Q\|_{\Ylip{\alpha}}\|Z\|_{Z^{\mathrm{sup}}}
	N_{\alpha}^{-2s}\eta^{2+\frac{1}{5}}
	\mathrm{w}_{\nb''}^{1}(\alpha)\,.
\]
On the other hand, when $\partial_{\xi_{k}}$ falls on $(Z[\xi])_{\mathbf{e}_{\mathfrak{m}}(j_{1})+\mathbf{e}_{\mathfrak{m}}(j_{2})}$ we reproduce the proof of \eqref{eq:pr-z} to deduce that this contribution is bounded by 
\[	
	\leq
	\|Q\|_{\Ysup{\alpha}}
	\| Z\|_{Z^{\mathrm{lip}}}
	N_{\alpha}^{-2s}\eta^{2+\frac{1}{5}-\frac{2}{10^{4}}}
	\mathrm{w}_{\nb''}^{0}(\alpha)
	\leq 
	\|Q\|_{\Ysup{\alpha}}
	\| Z\|_{Z^{\mathrm{lip}}}
	N_{\alpha}^{-4s}\eta^{4+\frac{1}{5}-\frac{2}{10^{4}}}
	\mathrm{w}_{\nb''}^{1}(\alpha)
	\,.
\]
This concludes the proof of Proposition \ref{prop:bracket1}.  
\end{proof}

\subsection{Proof of Lemma \ref{lem:d2F}}
\label{sub:d2F}
\begin{proof}[Proof of Lemma \ref{lem:d2F}]
We consider $Q\in X_{\leq2\overline{r}}(\epsilon)$. By homogeneity we can assume that
\[
\|Q\|_{\Ysup{\alpha}}=1\,. 
\]
Given $\xi\in\mathcal{U}_{s}(\epsilon)$, $u\in\mathcal{V}_{\alpha,\xi}(20\epsilon,\xi)$, and a normalized $w\in h^{s}(\Z_{M}^{d})$ with $\|w\|_{h^{s}}=1$, we have
\[
[\mathrm{d}\nabla \mathcal{L}(Q)(\xi)(u)](w) = \sum_{\nb\in\Lambda_{\alpha+1}}\frac{\mathfrak{h}^{(\alpha)}(\xi)}{\Omega_{\nb}(\widetilde{\omega}(\xi))}(Q[\xi])_{\nb}[\mathrm{d}\nabla z_{\nb}(u,I(u)-\xi)](w)\,.
\]
The small-divisor estimate \eqref{eq:smd0} together with the definition of the $\Ysup{\alpha}$-norm and the counting estimate \eqref{eq:counting} give
\begin{multline}
\label{eq:d2F0}
	\|[\mathrm{d}\nabla \mathcal{L}(Q)(\xi)(u)](w)\|_{h^{s}}
	\\
	\lesssim \epsilon^{-\frac{1}{10^{4}}}\gamma(\alpha)^{-1}\epsilon^{-2}N_{\alpha+1}^{2s}
	\max_{\nb\in\Lambda_{\alpha+1}\,,\ \deg(\nb)\leq 2\overline{r}} 
	\mathrm{w}_{\nb}^{0}(\alpha)
	\|
	[\mathrm{d}\nabla z_{\nb}(u,I(u)-\xi)](w)
	\|_{h^{s}}\,.
\end{multline}
We will prove that 
\begin{equation}
\label{eq:d2F}
\max_{\nb\in\Lambda_{\alpha+1}\,,\ \deg(\nb)\leq 2\overline{r}} 
	\mathrm{w}_{\nb}^{0}(\alpha)
	\|
	[\mathrm{d}\nabla z_{\nb}(u,I(u)-\xi)](w)
	\|_{h^{s}}
	\lesssim_{r}\epsilon^{-\frac{1}{10^{4}}}\eta^{4-\frac{2}{5}}N_{\alpha}^{-2s}\,.
\end{equation}
and conclude from \eqref{eq:d2F0} and the relation \eqref{eq:par} between the parameters:
\[
|\eqref{eq:d2F0}|
	\lesssim_{r}\epsilon^{-\frac{2}{10^{4}}}\gamma(\alpha)^{-1}
	(\frac{N_{\alpha+1}}{N_{\alpha}})^{2s}(\eta\epsilon^{-1})^{2}\eta^{2-\frac{2}{5}}
	\leq \epsilon^{1-\frac{1}{8}}\,.
\]
Let us show \eqref{eq:d2F}. Using that the Hamiltonian is real, we have 
\begin{align}
\nonumber
	\|
	[\mathrm{d}\nabla z_{\nb}(u,I(u)-\xi)](w)
	\|_{h^{s}}
	&\lesssim
	\Big(\sum_{n'\in\Z_{M}^{d}}\langle n'\rangle^{2s}|\sum_{n\in\Z_{M}^{d}}	\partial_{u_{n}}\partial_{\overline{u_{n'}}}
	z_{\nb}(u,I(u)-\xi)w_{n}|^2
	\Big)^\frac{1}{2} \\
	\nonumber
	&+
	\Big(\sum_{n'\in\Z_{M}^{d}}\langle n'\rangle^{2s}|\sum_{n\in\Z_{M}^{d}}	\partial_{u_{n}}\partial_{\overline{u_{n'}}}
	z_{\nb}(u,I(u)-\xi)\overline{w_{n}}|^2
	\Big)^\frac{1}{2} \\
	\label{eq:wasa}
	&\lesssim \epsilon^{-\frac{1}{10^{4}}}
	\max_{|n|,|n'|\leq M}
	\langle n'\rangle^{s}
	|\partial_{u_{n}}\partial_{\overline{u_{n'}}}z_{\nb}(u,I(u)-\xi)| 
	|w_{n}|\,.
\end{align}
For $\nb\in\Lambda_{\alpha+1}$ and $(n,n')\in(\Z_{M}^{d})^{2}$ we have
\begin{align}
\label{eq:df2-1}
	\partial_{u_n}\partial_{\overline{u_{n'}}}
	z_{\nb}(u,I(u)-\xi) 
	&= k_{n}\ell_{n'}
	\frac{z_{\nb}(u,I(u)-\xi)}{u_{n}\overline{u_{n'}}}\\
	\label{eq:df2-2}
	&+m_n\ell_{n'} 
	\frac{z_{\nb}(u,I(u)-\xi)}{(|u_{n}|^{2}-\xi_{n})\overline{u_{n'}}}\overline{u_{n}}\\
\label{eq:df2-3}
	&
	+k_{n}m_{n'}
	\frac{z_{\nb}(u,I(u)-\xi)}{u_{n}(|u_{n'}|^{2}-\xi_{n'})}
	u_{n'}\\
\label{eq:df2-4}
	&+m_nm_{n'} 
	\frac{z_{\nb}(u,I(u)-\xi)}{(|u_{n}|^{2}-\xi_{n})(|u_{n'}|^{2}-\xi_{n'})}
	\overline{u_{n}}u_{n'}\,.
\end{align}
In the subsequent case by case analysis, we estimate each contribution separately.  

\medskip

We deduce from the zero momentum condition that when $\ell_{n'}\geq1$, there exists
$j\in\Z_{M}^{d}\setminus\{n\}$ such that 
\begin{equation}
\label{eq:j}
\max(k_j,\ell_{j})\geq1\,,\quad |j|\geq\frac{|n'|}{2\overline{r}}\,.
\end{equation}
Without loss of generality, we suppose that $k_{j}\geq1$.
\medskip

$\bullet$ {\bf Case 1:} In \eqref{eq:df2-1}, when $k_{n}\ell_{n'}\geq1$ a pair $(u_{n},\overline{u_{n'}})$ goes away. We separate two cases. 
\medskip

\ $\triangleright$ {\bf Case 1a):} If 
\[
|n|\gtrsim |n'|\,,
\]
then
\begin{align*}
	\langle n'\rangle^{s} 
	k_{n}\ell_{n'}
	\mathrm{w}_{\nb}^{0}(\alpha)
	|
	\frac{z_{\nb}(u,I(u)-\xi)}{u_{n}\overline{u_{n'}}}
	w_{n'}
	| 
	&\leq \langle n\rangle^{s}|w_{n}|\mathrm{w}_{\nb}^{0}(\alpha)\ell_nk_{n'}|\frac{z_{\nb}(u,I(u)-\xi)}{u_{n}\overline{u_{n'}}}| \\
	&\lesssim_{r} 
	\operatorname{C}_{n}(\alpha)\mathrm{C}_{n'}(\alpha)
	\mathrm{w}_{\nb'}^{0}(\alpha)
	|z_{\nb'}(u,I(u)-\xi)|\,,
\end{align*}
where 
\[
\nb' =
	 \nb - \mathbf{e}_{\mathfrak{k}}(n)-\mathbf{e}_{\mathfrak{l}}(n') \in \widetilde{\Nb}\,,
\]
and 
\[
	\mathrm{w}_{\nb'}^{0}(\alpha) 
	= \operatorname{C}_{n}(\alpha)^{-1}
	\mathrm{C}_{n'}(\alpha)^{-1}
	\mathrm{w}_{\nb}^{0}(\alpha)\,.
\]
We deduce from the weighted estimate \eqref{eq:wz} for the monomials that, in Case 1a),
\[
\eqref{eq:d2F}\lesssim_{r} 
	N_{\alpha}^{2(s+\tau)}
	\eta^{-2} N_{\alpha}^{-6s}\eta^{6} \lesssim_{r} \eta^{4}N_{\alpha}^{4s-2\tau}
	\,.
\]
\ $\triangleright$ {\bf Case 1b):} In this case, however, we have 
\[
|n|\ll |n'|\,, 
\]
and the factor $w_{n}$ cannot  absorb the derivative $\langle n'\rangle^{s}$. Nevertheless, the frequency $j$ given by \eqref{eq:j} is different from $n$ and it is therefore still available. There holds
\begin{align*}
\langle n'\rangle^{s} \mathrm{w}_{\nb}^{0}(\alpha)
	k_{n}\ell_{n'}
	|
	\frac{z_{\nb}(u,I(u)-\xi)}
	{u_{n}\overline{u_{n'}}}w_{n'}
	| 	
	&\lesssim_{r} \langle j\rangle^{s}|u_{j}| |w_{n'}|
	\mathrm{w}_{\nb}^{0}(\alpha)
	k_{n}\ell_{n'}
	|
	\frac{z_{\nb}(u,I(u)-\xi)}
	{u_{n}\overline{u_{n'}}u_{j}}
	|\\
	&\lesssim_{r}\epsilon\mathrm{C}_j(\alpha)\operatorname{C}_{n}(\alpha)\mathrm{C}_{n'}(\alpha)\mathrm{w}_{\nb'}^{0}(\alpha)|z_{\nb'}(u,I(u)-\xi)|\,,
\end{align*}
where
\[
\nb' = \nb - \mathbf{e}_{\mathfrak{k}}(n)-\mathbf{e}_{\mathfrak{l}}(n')-\mathbf{e}_{\mathfrak{k}}(j) \in \widetilde{\Nb}\,,\quad \mathrm{w}_{\nb'}^{0}(\alpha) = \Big(\operatorname{C}_{n}(\alpha)\mathrm{C}_{n'}(\alpha,\epsilon)\mathrm{C}_{j}(\alpha)\Big)^{-1}\mathrm{w}_{\nb}^{0}(\alpha) 
\]
We deduce from the weighted estimate \eqref{eq:wz} for monomials  that, in Case 1b), 
\[
\eqref{eq:d2F}\lesssim_{ r}
	\epsilon N_{\alpha}^{3(s+\tau)}\eta^{-3}N_{\alpha}^{-6s}\eta^{6}
	\lesssim_{r} \eta^{4}N_{\alpha}^{-3(s-\tau)}\,.
\]
$\bullet$ {\bf Case 2:} In \eqref{eq:df2-2}, when $m_{n}\ell{_n'}\geq1$ a pair $(|u_{n}|^{2}-\xi_{n},\overline{u_{n'}})$ goes away and a term $\overline{u_{n}}$ appears. 
\medskip

\ $\triangleright$ {\bf  Case 2a):} If 
\[
|n|\gtrsim |n'|\,,
\]
then 
\begin{align*}
\langle n'\rangle^{s} 
	\mathrm{w}_{\nb}^{0}(\alpha)
	m_{n}\ell_{n'}
	|
	\frac{z_{\nb}(u,I(u)-\xi)}
	{(|u_{n}|^{2}-\xi_{n})\overline{u_{n'}}}\overline{u_{n}}w_{n'}
	| 
	&
	\lesssim 
	\langle n\rangle^{s}
	|w_{n}u_{n}|
	\mathrm{w}_{\nb}^{0}(\alpha)
	m_{n}\ell_{n'}
	|
	\frac{z_{\nb}(u,I(u)-\xi)}
	{(|u_{n}|^{2}-\xi_{n})\overline{u_{n'}}}
	| \\
	&\lesssim_{r}
	\epsilon
	\operatorname{C}_{n'}(\alpha)\mathrm{D}(\alpha)
	\mathrm{w}_{\nb'}^{0}(\alpha)|z_{\nb'}|\,,
\end{align*}
where 
	\[
\nb' 
	= \nb-\mathbf{e}_{\mathfrak{m}}(n) -\mathbf{e}_{\mathfrak{l}}(n')\in\widetilde{\Nb}\,,
	\quad 
	\mathrm{w}_{\nb'}^{0}(\alpha) = (\operatorname{C}_{n'}(\alpha)\mathrm{D}(\alpha))^{-1}\mathrm{w}_{\nb}^{0}(\alpha)\,.
	\]
We deduce from the weighted estimate for monomials that, in Case 2a),
\[
\eqref{eq:d2F}
	\lesssim_{r}
	\epsilon N_{\alpha}^{3s+\tau}\eta^{-3-\frac{1}{5}}N_{\alpha}^{-6s}\eta^{6}
	\lesssim_{r}\eta^{4-\frac{1}{5}}N_{\alpha}^{-(3s-\tau)}\,.
\]

\ $\triangleright$ {\bf  Case 2b):} If 
\[
|n|\ll |n'|\,,
\]
then we can find $j \in \Z_{M}^{d}\setminus\{n,n'\}$ such that \eqref{eq:j} holds, and deduce that 
\begin{align*}
&\langle n'\rangle^{s} \mathrm{w}_{\nb}^{0}(\alpha)m_{n}\ell_{n'}
	|
	\frac{z_{\nb}(u,I(u)-\xi)}
	{(|u_{n}|^{2}-\xi_{n})\overline{u_{n'}}}
	\overline{u_{n}}w_{n'}
	| \\
	&
	\lesssim_{r}\langle j\rangle^{s}|u_{j}w_{n}u_{n}| \mathrm{w}_{\nb}^{0}(\alpha)m_{n}\ell_{n'}
	|
	\frac{z_{\nb}(u,I(u)-\xi)}
	{(|u_{n}|^{2}-\xi_{n})\overline{u_{n'}}u_{j}}
	| \\
	&\lesssim_{r}
	\epsilon^{2}\mathrm{C}_{j}(\alpha)\operatorname{C}_{n'}(\alpha)\mathrm{D}(\alpha)
	\mathrm{w}_{\nb'}^{0}(\alpha)
	|z_{\nb'}(u,I(u)-\xi)|\,,
\end{align*}
where
\[
	\nb'= 
	\nb-\mathbf{e}_{\mathfrak{m}}(n)-\mathbf{e}_{\mathfrak{l}}(n')-\mathbf{e}_{\mathfrak{k}}(j)\,,
	\quad 
	\mathrm{w}_{\nb'}^{0}(\alpha)=(\operatorname{C}_{n'}(\alpha)\mathrm{C}_{j}(\alpha)\operatorname{D}(\alpha))^{-1}\mathrm{w}_{\nb}^{0}(\alpha)\,.
\]
We deduce from the weighted bound \eqref{eq:wz} for the monomials that, in Case 2b),
\[
	\eqref{eq:d2F}
	\lesssim_{r}\epsilon^{2}
	N_{\alpha}^{-4s-2\tau}\eta^{-4-\frac{1}{5}}N_{\alpha}^{6s}\eta^{6}
	\lesssim_{r}
	\eta^{4-\frac{1}{5}}N_{\alpha}^{-2(s-\tau)}
	\,.
\]
$\bullet$ {\bf Case 3:} In \eqref{eq:df2-3}, when $k_{n}m_{n'}\geq1$ a pair $(u_{n},|u_{n'}|^{2}-\xi_{n'})$ goes away, and a term $u_{n'}$ appears.  It holds
\begin{align*}
&\langle n'\rangle^{s} \mathrm{w}_{\nb}^{0}(\alpha)k_{n}m_{n'}
	|\frac{z_{\nb}(u,I(u)-\xi)}{u_{n}(|u_{n'}|^{2}-\xi_{n'})}u_{n'}w_{n}| \\
	&
	\leq \langle n'\rangle^{s} 
	|u_{n'}w_{n}|\mathrm{w}_{\nb}^{0}(\alpha)k_{n}m_{n'}
	|
	\frac{z_{\nb}(u,I(u)-\xi)}
	{u_{n}(|u_{n'}|^{2}-\xi_{n'})}
	|
	 \\
	&\lesssim_{r}\epsilon\mathrm{D}(\alpha)\mathrm{C}_{n}(\alpha)
	\mathrm{w}^{0}_{\nb'}(\alpha)|z_{\nb'}(u,I(u)-\xi)|\,,
\end{align*}
with 
\[
	\nb' 
	= \nb 
	-\mathbf{e}_{\mathfrak{k}}(n)
	-\mathbf{e}_{\mathfrak{m}}(n')\,,\quad 
	\mathrm{w}_{\nb'}^{0}(\alpha) = (\mathrm{D}(\alpha)\mathrm{C}_{n}(\alpha))^{-1}\mathrm{w}_{\nb}^{0}(\alpha)\,.
\]
We deduce from the weighted estimate \eqref{eq:wz} for monomials that 
\[
\eqref{eq:d2F}
	\lesssim_{r}
	\epsilon N_{\alpha}^{3s+\tau}\eta^{-3-\frac{1}{5}}N_{\alpha}^{-6s}\eta^{6}
	\lesssim_{r} \eta^{4-\frac{1}{5}}N_{\alpha}^{-3s+\tau}\,.
\]
$\bullet$ {\bf Case 4:} The last case \eqref{eq:df2-4} is a priori the worst since a pair of centered actions goes away, and induces a loss $\mathrm{D}(\alpha)^{2}$. Nevertheless, a pair $(\overline{u_{n}},u_{n'})$ appears and will absorb the derivative in $n$. There is also a term $w_{n'}$, which we will not exploit. It holds
\begin{align*}
\langle n'\rangle^{s}
	&\mathrm{w}_{\nb}^{0}(\alpha)
	m_nm_{n'}
	|
	\frac{z_{\nb}(u,I(u)-\xi)}
	{(|u_{n}|^{2}-\xi_{n})(|u_{n'}|^{2}-\xi_{n'})}
	\overline{u_{n}}u_{n'}w_{n'}
	| \\
	&\leq 
	\langle n'\rangle^{s}|u_{n}||u_{n'}||w_{n'}|
	\mathrm{w}_{\nb}^{0}(\alpha)
	m_nm_{n'}
	|
	\frac{z_{\nb}(u,I(u)-\xi)}
	{(|u_{n}|^{2}-\xi_{n})(|u_{n'}|^{2}-\xi_{n'})}
	| \\
	&\lesssim_{r}
	\epsilon^{2}
	\operatorname{D}(\alpha)^{2}
	\mathrm{w}_{\nb'}^{0}(\alpha)
	|z_{\nb'}(u,I(u)-\xi)|\,,
\end{align*}
with 
\[
\nb' = \nb - \mathbf{e}_{\mathfrak{m}}(n)-\mathbf{e}_{\mathfrak{m}}(n')
	\,,\quad 
	\mathrm{w}^{0}_{\nb'}(\alpha)
	=\operatorname{D}(\alpha)^{-2}\mathrm{w}_{\nb'}^{0}(\alpha)\,.
\]
We deduce from the weighted estimate for the monomials \eqref{eq:wz} that this contribution is bounded by 
\[
\eqref{eq:d2F}\lesssim_{r}
	\epsilon^{2}
	N_{\alpha}^{4s}\eta^{-4-\frac{2}{5}}
	N_{\alpha}^{-6s}\eta^{6}
	 \lesssim_{r}\eta^{4-\frac{2}{5}}N_{\alpha}^{-2s}
	\,.
\]
This proves that all contributions to \eqref{eq:wasa} are acceptable to obtain \eqref{eq:d2F}, which completes the proof of Lemma \ref{lem:d2F}.
\end{proof}

\end{document}